\documentclass[reqno,10pt]{amsart}
\usepackage{amssymb,amsmath,amsthm,}
\usepackage{amscd}
\usepackage{amsfonts}
\usepackage{amssymb}
\usepackage{latexsym}
\usepackage{color}
\usepackage{esint}

\usepackage{graphicx}
\graphicspath{ {images/} }

\usepackage[makeroom]{cancel}

\newtheorem{theorem}{Theorem}

\newtheorem{definition}{Definition}
\newtheorem{lemma}{Lemma}
\newtheorem{proposition}[theorem]{Proposition}
\newtheorem{remark}{Remark}

\let\e=\varepsilon

\let\p=\partial

\let\O=\Omega

\numberwithin{equation}{section}

\let\hide\iffalse
\let\unhide\fi

\newcommand{\R}{\mathbb{R}}

\renewcommand{\S}{\mathbb{S}}

\newcommand{\be}{\begin{equation}}
\newcommand{\bm}{\begin{multline}}
\newcommand{\ee}{\end{equation}}
\newcommand{\dd}{\mathrm{d}}

\newcommand{\xb}{x_{\mathbf{b}}}
\newcommand{\tb}{t_{\mathbf{b}}}
\newcommand{\vb}{v_{\mathbf{b}}}

\newcommand{\Bes}{\begin{eqnarray*}}
\newcommand{\Ees}{\end{eqnarray*}}
\newcommand{\Be}{\begin{equation} }
\newcommand{\Ee}{\end{equation}}

\def\p{\partial}

\def\O{\Omega}
\def\R{\mathbb{R}}

\def\B{\begin{equation}}
\def\E{\end{equation}}
\def\BN{\begin{eqnarray*}}
\def\EN{\end{eqnarray*}}

\begin{document}

\title{Local Well-posedness of Vlasov-Poisson-Boltzmann Equation with Generalized Diffuse Boundary Condition}

 \author{Hongxu Chen}

 \author{Chanwoo Kim}

 \author{Qin Li}

 \address{Department of Mathematics, University of Wisconsin, Madison, WI 53706 USA}

 \begin{abstract}
The Vlasov-Poisson-Boltzmann equation is a classical equation governing the dynamics of charged particles with the electric force being self-imposed. We consider the system in a convex domain with the Cercignani-Lampis boundary condition. We construct a uniqueness local-in-time solution based on an $L^\infty$-estimate and $W^{1,p}$-estimate. In particular, we develop a new iteration scheme along the characteristic with the Cercignani-Lampis boundary for the $L^\infty$-estimate, and an intrinsic decomposition of boundary integral for $W^{1,p}$-estimate.
 \end{abstract}

\maketitle
\tableofcontents

 \date{\today}

\section{Introduction}
In this paper we study the Vlasov-Poisson-Boltzmann system, which is a classical model for describing the dynamics of dilute charged particles (such as plasma) with a self-imposed electric field (see ~\cite{GVPB,CKL} and reference therein). We denote $F(t,x,v)$ the phase-space-distribution function of charged particles at time $t$, location $x\in\Omega$, a bounded domain $\mathbb{R}^3$, with velocity $v\in\mathbb{R}^3$. The evolution of the system is described as:
\begin{equation}\label{eqn: VPB equation}
\partial_t F + v\cdot \nabla_x F - E\cdot \nabla_v F = Q(F,F),\ \
F|_{t=0} = F_0(x,v).
\end{equation}
The characteristics solves the following Hamilton ODEs
\begin{equation}
\dot{x} = v ,\quad\dot{v} = -E.\label{hamilton_ODE_1}
\end{equation}
The collision operator $Q$ on the right, as a functional of $F$, describes the binary collisions between particles and takes the form of
\begin{equation}\label{eqn: Q}
\begin{split}
 Q(F_1,F_2)(v)   &=Q_{\text{gain}}(F_1,F_2)(v)-Q_{\text{loss}} (F_1,F_2)(v)= Q_{\text{gain}}(F_1,F_2)-\nu(F_1)F_2 \\
    & :=\iint_{\mathbb{R}^3\times \mathbb{S}^2} B(v-u,\omega)F_1(u')F_2(v') d\omega du-\left(\iint_{\mathbb{R}^3\times \mathbb{S}^2} B(v-u,\omega)F_1(u) d\omega du\right)F_2(v) .
\end{split}
\end{equation}
In the collision process, momentum and energy are conserved, namely,
\begin{equation}\notag\label{eqn: conservation}
  u'+v'=u+v ,\quad |u'|^2+|v'|^2=|u|^2+|v|^2,
\end{equation}
where the post-velocities are denoted as
\begin{equation}\label{eqn: u' v'}
u'=u-[(u-v)\cdot \omega]\omega,\quad \quad v'=v+[(u-v)\cdot \omega]\omega.
\end{equation}
In ~\eqref{eqn: Q}, $B$ is called a collision kernel, and we use the hard potential model in this paper:
\[
B(v-u,\omega)=|v-u|^{\mathcal{K}}q_0\Big(\frac{v-u}{|v-u|}\cdot \omega\Big) ,\quad\text{with}\quad 0< \mathcal{K}\leq 1 ,\quad 0\leq q_0\Big(\frac{v-u}{|v-u|}\cdot \omega\Big)\leq C\Big|\frac{v-u}{|v-u|}\cdot \omega\Big|.
\]
In ~\eqref{eqn: VPB equation}, $E$ denotes the electrostatic field, and we consider a self-imposed electric field in this paper: namely, the charged particles themselves form a potential that in turn drives their own dynamics. This is in particular a relevant model for plasma particles without extra magnetic field. More specifically,
\begin{equation}\label{eqn: E}
  E(t,x)=-\nabla_x \phi(t,x) \,,
\end{equation}
with the electrostatic potential $\phi$ determined by the Poisson equation
\begin{equation}\label{eqn:Possion}
- \Delta_ x \phi (t,x) = \int_{\mathbb{R}^3} F(t,x,v) d v - \rho_0
 \text{ in }  \Omega
, \ \   \frac{\partial \phi}{\partial n}=0 \text{ on } \partial\Omega,
\end{equation}
where $\rho_0$ is a background constant charge density. We set $\rho_0$ as an average of the initial total mass:
\begin{equation}\label{eqn: Background density}
  \rho_0=\frac{1}{|\Omega|}\int_{\Omega\times \mathbb{R}^3} F(0,x,v)dvdx \,.
\end{equation}


A boundary condition of $F$ is determined by an interaction of the charged particles and a physical boundary. We denote a boundary of the phase space as $\gamma:=\{(x,v)\in \partial \Omega\times \mathbb{R}^3\}.$ Let $n=n(x)$ be the outward normal direction at $x\in \p\Omega$. We split the phase boundary into an incoming ($\gamma_-$) and outgoing ($\gamma_+$) set as:
\begin{equation}\label{eqn: incoming outgoing}
\gamma_\mp:=\{(x,v)\in \partial \Omega\times \mathbb{R}^3 :n(x)\cdot v\lessgtr 0\}\quad\text{or}\quad \gamma_\mp(x):=\{v\in \partial \Omega\times \mathbb{R}^3 :n(x)\cdot v\lessgtr 0\} \,.
\end{equation}
The boundary condition determines the distribution on $\gamma_-$, and describes how particles bounces back to the domain, which can be determined by a scattering kernel $R(u \rightarrow v;x,t)$ through a general balance law of
\begin{equation}\begin{split}\label{eqn:BC}
&F(t,x,v) |n(x) \cdot v|= \int_{\gamma_+(x)}
R(u \rightarrow v;x,t) F(t,x,u)
\{n(x) \cdot u\} d u, \quad \text{ on }\gamma_-
.
\end{split}
\end{equation}
Physically, $R(u\to v;x,t)$ represents the probability of a molecule striking in the boundary at $x\in\partial\Omega$ with velocity $u$ to be bounced back to the domain with velocity $v$ at the same location $x$ and time $t$. In this paper we use a model proposed by Cercignani and Lampis in~\cite{CIP,CL}. With two accommodation coefficients
\begin{equation}\label{eqn: r condition}
  0<r_\perp\leq 1,\quad 0<r_\parallel<2 ,
\end{equation}
the Cercignani-Lampis boundary condition (C-L boundary condition) can be written as
\begin{equation}\label{eqn: Formula for R}\begin{split}
&R(u \rightarrow v;x,t)\\
:=& \frac{1}{r_\perp r_\parallel (2- r_\parallel)\pi/2} \frac{|n(x) \cdot v|}{(2T_w(x))^2}
\exp\left(- \frac{1}{2T_w(x)}\left[
\frac{|v_\perp|^2 + (1- r_\perp) |u_\perp|^2}{r_\perp}
+ \frac{|v_\parallel - (1- r_\parallel ) u_\parallel|^2}{r_\parallel (2- r_\parallel)}
\right]\right)\\
& \times  I_0 \left(
 \frac{1}{2T_w(x)}\frac{2 (1-r_\perp)^{1/2} v_\perp u_\perp}{r_\perp}
\right).
\end{split}
\end{equation}
Here $T_w(x)$ is a wall temperature on the boundary and
\begin{equation*}
I_0 (y) := \pi^{-1} \int^{\pi}_0e^{y \cos \phi } d \phi\,.
\end{equation*}
In this formula, $v_\perp$ and $v_\parallel$ denote the normal and tangential components of the velocity respectively:
   \begin{equation}\label{eqn: def of vperppara}
   v_\perp= v\cdot n(x) ,\quad v_\parallel = v- v_\perp n(x)\,.
\end{equation}
Similarly $u_\perp= u\cdot n(x)$ and $u_\parallel = u- u_\perp n(x)$.

This model can be considered as a generalization of fundamental boundary conditions. For instance if we set $r_\perp=1$ and $r_\parallel=1$, the scattering kernel equals
\[R(u\to v;x,t)=\frac{2}{\pi (2T_w(x))^2}e^{-\frac{|v|^2}{2T_w(x)}} |n(x)\cdot v|.\]
This corresponds the so-called diffuse boundary condition:
\begin{equation}\label{eqn: diffuse}
 F(t,x,v)= \frac{2}{\pi (2T_w(x))^2}e^{-\frac{|v|^2}{2T_w(x)}}\int_{n(x)\cdot u>0} F(t,x,u)\{n(x)\cdot u\}du \text{ on }(x,v)\in\gamma_-\,.
\end{equation}
With $r_\perp=0,r_\parallel=0$, the scattering kernel is given by
  \[R(u\to v;x,t)=\delta(u-\mathfrak{R}_xv),\]
  with $\mathfrak{R}_xv=v-2n(x)(n(x)\cdot v)$. This corresponds the specular reflection boundary condition $F(t,x,v)=F(t,x,\mathfrak{R}_xv)$. With $r_\perp=0,r_\parallel=2$, the scattering kernel is given by
  \[R(u\to v;x,t)=\delta(u+v), \]
   which corresponds the bounce-back reflection reflection boundary condition $F(t,x,v)=F(t,x,-v)$. The C-L model is related to the Maxwell boundary condition since both models can describe the intermediate reflection law between diffuse and specular reflection boundary conditions. The comparison of both models is found in~\cite{HC}.

It is important to note that the C-L boundary condition satisfies the reciprocity property
\Be\label{eqn: reciprocity}
    R(u\to v;x,t)=R(-v\to -u;x,t) \frac{e^{-|v|^2/(2T_w(x))}}{e^{-|u|^2/(2T_w(x))}}\frac{|n(x)\cdot v|}{|n(x)\cdot u|}\,,
\Ee
and the normalization property (see the proof in appendix)
\Be\label{eqn: normalization}
\int_{\gamma_-(x)} R(u\to v;x,t) dv=1\,.
\Ee
We note that the normalization~\eqref{eqn: normalization} property immediately leads to the null flux condition for $F$:
\begin{equation}\label{eqn: Null flux condition}
  \int_{\mathbb{R}^3}F(t,x,v)\{n(x)\cdot v\}dv=0 ,\quad \text{for }x\in \partial\Omega.
\end{equation}
This guarantees the conservation of total mass:
\begin{equation}\label{eqn: Mass conservation}
  \int_{\Omega\times \mathbb{R}^3}F(t,x,v)dvdx=\int_{\Omega\times \mathbb{R}^3}F(0,x,v)dvdx \text{ for all }t\geq 0.
\end{equation}
We note that from the conservation of mass~\eqref{eqn: Mass conservation} and our choice~\eqref{eqn: Background density}, we satisfiy the compatibility condition of~\eqref{eqn:Possion} with the Neumann boundary condition.

The generality of the C-L model allows it to be applicable to many problems, including the rarefied gas flow studied in~\cite{KB,SF,SF1}; gas surface interaction model presented in~\cite{L,WR}; and rigid-sphere interaction model investigated in~\cite{CS,Gar}, to name a few. There also emerged many other derivations of C-L model besides the original one, and we refer interested readers to~\cite{C,CIP,CC}.

%

\subsection{Main result}
We now discuss the main result of this paper. Throughout this paper we assume the domain is $C^3$, which means for any $p\in \partial \Omega$, there exists sufficiently small $\delta_1>0$, $\delta_2>0$, and an one-to-one and onto $C^3$-map $\eta_p$ so that
\begin{align}\label{eqn: C3 map}
  \begin{split}
 \eta_p: \{ x_\parallel \in \mathbb{R}^2:|x_\parallel|<\delta_1\}& \rightarrow \partial \Omega \cap B(p,\delta_2), \\
      x_\parallel=(x_{\parallel,1},x_{\parallel,2})   &\mapsto \eta_p(x_{\parallel,1},x_{\parallel,2}).
  \end{split}
\end{align}
We further assume the domain is \textit{convex}: there exists $C_\eta>0$ and $C_\Omega>0$ such that at all $p\in \partial \Omega$, the Hessian of the corresponding $\eta_p$, defined in~\eqref{eqn: C3 map} are upper and lower bounded for all $x_\parallel$ in (\ref{eqn: C3 map}) as
\begin{equation}\label{eqn: convex}
  -C_\eta |\zeta|^2\leq \sum_{i,j=1}^2     \zeta_i \zeta_j \partial_i \partial_j \eta_p(x_\parallel)\cdot n(x_\parallel)\leq -C_\Omega |\zeta|^2 ,\quad\forall \zeta \in \mathbb{R}^2\,.
\end{equation}


%

We define the global Maxwellian using the maximum wall temperature as
\begin{equation}\label{eqn: def for weight}
\mu:=e^{-\frac{|v|^2}{2T_M}}\,,\ \text{ with }T_M:=\max_{x\in \partial \Omega}\{T_w(x)\}.
\end{equation}
We set
\Be\label{F_def}
F=\sqrt{\mu}f.
\Ee
Then $f$ satisfies
\begin{equation}\label{equation for f}
\begin{cases}
  \partial_t f+v\cdot \nabla_x f-\nabla_x \phi\cdot \nabla_v f+\frac{1}{2T_M}fv\cdot \nabla_x \phi=\Gamma(f,f)\\
  f(t=0,x,v) = f_0(x,v):= \mu^{-1/2}F_0\\
  f(t,x,v)|n(x)\cdot v||_{\gamma_-}=\frac{1}{\sqrt{\mu}}\int_{n(x)\cdot u>0}   R(u\to v;x,t)  f(t,x,u)\sqrt{\mu(u)}\{n(x)\cdot u\}du
  \end{cases}
\end{equation}
where the collision operator becomes
\begin{equation}\label{Def: Gamma}
\Gamma(f_1,f_2)=\Gamma_{\text{gain}}(f_1,f_2)-\nu(F_1)F_2/\mu=\frac{1}{\sqrt{\mu}}Q_{\text{gain}}(\sqrt{\mu}f_1,\sqrt{\mu}f_2)-\nu(F_1)f_2 ,
\end{equation}
and $\phi $ solves
\begin{equation}\label{equation for phi_f}
  -\Delta_x\phi(t,x)=\int_{\mathbb{R}^3}f(t,x,v)\sqrt{\mu(v)}dv-\rho_0 \text{ in } \Omega,\quad \frac{\partial \phi}{\partial n}=0 \text{ on } \partial\Omega.
\end{equation}

Now using the reciprocity property (\ref{eqn: reciprocity}) we derive that, for $(x,v)\in \gamma_-$,
\[
f(t,x,v)|n(x)\cdot v|=\frac{1}{\sqrt{\mu}}\int_{n(x)\cdot u>0}   R(-v\to -u;x,t) \frac{e^{-|v|^2/(2T_w(x))}}{e^{-|u|^2/(2T_w(x))}}  f(t,x,u)\sqrt{\mu(u)}\frac{|n(x)\cdot v|}{|n(x)\cdot u|}\{n(x)\cdot u\}du.
\]
Let us denote
\begin{equation}\label{eqn:probability measure}
d\sigma(u,v):=R(-v\to -u;x,t)du.
\end{equation}
Due to the normalization property~\eqref{eqn: normalization}, it is a probability measure in space $\gamma_+(x)$, and then the boundary condition in~\eqref{equation for f} writes:
\begin{equation}\label{eqn:C-L boundary condition in pro measure}
f(t,x,v)|_{\gamma_-}=e^{[\frac{1}{4T_M}-\frac{1}{2T_w(x)}]|v|^2}\int_{n(x)\cdot u>0} f(t,x,u)e^{-[\frac{1}{4T_M}-\frac{1}{2T_w(x)}]|u|^2}d\sigma(u,v) \,.
\end{equation}
We furthermore denote
\begin{equation}\label{Def: w_theta}
w_{\theta}:=e^{\theta |v|^2}\,,\quad \langle v\rangle:=\sqrt{|v|^2+1}\,.
\end{equation}
Now we state our main theorem of the paper:
\begin{theorem}\label{local_existence}
Assume $\Omega \subset \mathbb{R}^3$ is open bounded, and convex $C^3$ domain. A wall temperature $T_w(x)>0$ is defined on $x\in \partial \Omega$ and smooth. We assume that two accommodation coefficients of~\eqref{eqn: r condition} satisfy

\begin{equation}\label{eqn: Constrain on T}
\frac{\min_{x\in \partial \Omega}\{T_w(x)\}}{\max_{x\in \partial \Omega}\{T_w(x)\}}>\max\Big(\frac{1-r_\parallel}{2-r_\parallel},\frac{\sqrt{1-r_\perp}-(1-r_\perp)}{r_\perp}\Big)\,.
\end{equation}
Let
\begin{equation}\label{eqn: Constrain on theta}
0< \tilde{\theta}< \theta <\frac{1}{4\max_{x\in \partial \Omega}\{T_w(x)\}}.
\end{equation}

Assume
\begin{equation}
\| w_\theta f_0 \|_\infty < \infty, \label{eqn: w f_0}
\end{equation}
\begin{equation}
\| w_{\tilde{\theta}} \nabla_v f_0 \|_{L^{3}_{x,v}}<\infty,\label{eqn: assumption nabla v f} \end{equation}
\begin{equation}
\| w_{\tilde{\theta}} \alpha_{f_0, \epsilon }^\beta \nabla_{x,v } f_0 \|_{ {L}^{p } ( {\O} \times \R^3)} <\infty\quad \text{for} \quad 3< p < 6\,,\, 1-\frac{2}{p }< \beta< \frac{2}{3}.\label{eqn: alpha f0}
\end{equation}
 \hide

\begin{equation}\label{eqn: w f_0}
\| w_\theta f_0 \|_\infty < \infty\,,
\end{equation}		
\begin{equation}\label{eqn: assumption nabla v f}
\| w_{\tilde{\theta}} \nabla_v f_0 \|_{L^{3}_{x,v}}<\infty\,,
\end{equation}
\begin{equation}\label{eqn: alpha f0}
\| w_{\tilde{\theta}} \alpha_{f_0, \epsilon }^\beta \nabla_{x,v } f_0 \|_{ {L}^{p } ( {\O} \times \R^3)} <\infty\quad \text{for} \quad 3< p < 6\,,\, 1-\frac{2}{p }< \beta< \frac{2}{3}.
\end{equation}\unhide
Then there is a unique solution $f(t,x,v)$ to~\eqref{equation for f} in a time interval of $t \in [0,\bar{t}]$ with
\begin{equation}\label{eqn: bar t in thm}
\bar{t}=\bar{t}(\Vert w_\theta f_0\Vert_\infty,\| w_{\tilde{\theta}} \alpha_{f_0, \epsilon }^\beta \nabla_{x,v } f_0 \|_{ {L}^{p } ( {\O} \times \R^3)},\Vert w_{\tilde{\theta}}\nabla_v f_0\Vert_{L^3_{x,v}} ,r_\parallel,r_\perp,\Omega,T_M,\min(T_w(x))).
\end{equation}
Moreover, there are $\mathfrak{C}>0$ and $\lambda>0$, so that $f$ satisfies
\begin{equation}\label{infty_local_bound}
\sup_{0 \leq t \leq \bar{t}}\| w_{\theta}e^{-\mathfrak{C} \langle v\rangle^2 t} f  (t) \|_{\infty}\lesssim \| w_\theta f_0 \|_\infty  ,
\end{equation}
\Be\label{31_local_bound}
\sup_{0 \leq t \leq \bar{t}}\| \nabla_v f (t) \|_{L^3_xL^{1+ \delta}_v}< \infty ,
\Ee
\begin{equation}\label{W1p_local_bound}
\sup_{0 \leq t \leq \bar{t}}\Big\{ \| w_{\tilde{\theta}}e^{-\lambda t\langle v\rangle }\alpha_{f,\epsilon }^\beta \nabla_{x,v} f (t) \|_{p} ^p+ \int^t_0  |w_{\tilde{\theta}} e^{-\lambda s\langle v\rangle } \alpha_{f,\epsilon}^\beta \nabla_{x,v} f (t) |_{p,+}^p\Big\}< \infty .
\end{equation}
Here, $\alpha$ is a weight and its precise definition will be seen in~\eqref{alphaweight}.
\end{theorem}

\begin{remark}
We do not assume the smallness of our initial data, but we need the small scale of the time $\bar{t}$. Setting $r_\perp=1$ and $r_\parallel=1$, this theorem also provides the first large date well-posedness of VPB system with the standard diffuse boundary condition (\ref{eqn: diffuse}). A small data result had been established in \cite{CKL}. We use the condition ~\eqref{eqn: Constrain on T} in the proof of the $L^\infty$ bound, which itself serves as an important a-priori estimate for the existence and the $W^{1,p}$ estimate ~\eqref{31_local_bound} and ~\eqref{W1p_local_bound}.
\end{remark}

\begin{remark}As far as the authors know, Theorem 1 provides the \textit{first} local in time solution to the Vlasov-Poisson-Boltzmann system in bounded domains with the Cercignani-Lampis boundary condition. The local in time result for the Boltzmann equation without field can be found in \cite{HC}.
\end{remark}		

\subsection{Strategy of the proof}
In this section we discuss the major difficulties and key ideas to overcome them.

Consider the simple Vlasov-Poisson (VP) equation without the collision:
\begin{equation}\label{eqn: Vlasov-Possion}
  \partial_t f+v\cdot \nabla_x f-\nabla_x \phi_f\cdot \nabla_v f=0\,.
\end{equation}
Suppose one has two solutions $f$ and $g$, then taking the difference we have:
\[\partial_t (f-g)+v\cdot \nabla_x f -\nabla_x \phi_f \cdot \nabla_v (f-g)= (\nabla_x \phi_f -\nabla_x \phi_g)\cdot \nabla_v g.\]
To show the uniqueness using the stability argument one essentially needs to control $\nabla_v g$. This is hard to achieve in general: it is a rather well-known result that transport equation in a bounded domain could potentially form singularities~\cite{ABDG,K}.


To better understand this singularity, we now denote $(X(s;t,x,v), V(s;t,x,v))$ the solution to a trajectory that starts with $(X(t;t,x,v), V(t;t,x,v)) =  (x,v)$. Following the Hamiltonian system~\eqref{hamilton_ODE_1}, we have
\Be\label{hamilton_ODE}
\frac{d}{ds} \left[ \begin{matrix}X(s;t,x,v)\\ V(s;t,x,v)\end{matrix} \right] = \left[ \begin{matrix}V(s;t,x,v)\\
- \nabla_x \phi
(s, X(s;t,x,v))\end{matrix} \right]  \ \ \text{for}   - \infty< s ,  t < \infty  .
\Ee
For $(t,x,v) \in \R  \times  \O \times \R^3$, define \textit{the backward exit time} $\tb(t,x,v)$:
\Be\label{tb}
\tb (t,x,v) := \sup \{s \geq 0 : X(\tau;t,x,v) \in \O \ \ \text{for all } \tau \in (t-s,t) \}\,,
\Ee
and the corresponding existing location and velocity:
\[\xb (t,x,v) := X(t-\tb(t,x,v);t,x,v)\quad\text{and}\quad \vb (t,x,v) := V(t-\tb(t,x,v);t,x,v)\,.
\]
Call the boundary condition $f|_{\gamma_-}=h$, then~\eqref{eqn: Vlasov-Possion} has an explicit solution
\Be\notag
f(t,x,v)= h(t-\tb(t,x,v), \xb(t,x,v), \vb(t,x,v))\,.
\Ee
This leads to a fact that the derivatives of $f$ may contain singularities from a direct computation of $ \nabla_x \xb(t,x,v)$ as
\Be\label{deriv_singular}
\nabla_{x} f(t,x,v)   \sim  \nabla_x \xb(t,x,v)
\sim \frac{1}{n(\xb(t,x,v)) \cdot \vb(t,x,v)}\,.
\Ee
The term blows up as $\vb$ becomes tangential to the surface at the backward exit time. This difficulty sits at the core of many boundary problems of Boltzmann-type equations.

To account for this difficulty, we follow a strategy of a kinetic weight \cite{GKTT,CKL}:
\begin{definition}[Kinetic Weight] For $\epsilon>0$, let $f$ solve~\eqref{eqn: Vlasov-Possion}, define
\Be\label{alphaweight}\begin{split}
\alpha_{f, \epsilon }(t,x,v) : =& \  
\chi \Big(\frac{t-\tb(t,x,v)+\epsilon}{\epsilon}\Big) |n(\xb(t,x,v)) \cdot \vb(t,x,v)| \\
&+ \Big[1- \chi \Big(\frac{t-\tb(t,x,v) +\epsilon}{\epsilon}\Big)\Big].
\end{split}\Ee
Here we use a smooth function $\chi: \R \rightarrow [0,1]$ satisfying
\Be\label{chi}
\begin{split}
\chi(\tau)  =0,  \     \tau\leq 0, \ \text{and} \  \
\chi(\tau)  = 1    ,  \  \tau\geq 1, \\
 \frac{d}{d\tau}\chi(\tau)  \in [0,4] \ \   \text{for all }   \tau \in \R.
\end{split}
\Ee
\end{definition}
Note that $\alpha_{f,\epsilon}(0,x,v)\equiv \alpha_{f_0,\epsilon}(0,x,v)$ is determined by the initial data $f_0$. There are two important features of this weight. First it is invariant under the transport operator, namely:
\begin{equation}\label{eqn: alpha invar}
  [\partial_t +v\cdot \nabla_x -\nabla_x \phi \cdot \nabla_v]\alpha_{f,\epsilon}(t,x,v)=0.
\end{equation}
Second, it picks $|n(\xb^f(t,x,v)) \cdot \vb^f(t,x,v)|$ for $t>\tb^f(t,x,v)$, which is exactly the singularity in (\ref{deriv_singular}).

%

\bigskip

The proof of the main theorem consists two parts: an $L^\infty$-estimate and a weighted $W^{1,p}$ estimate. These estimates are based on the uniform estimates of the following iterative sequence:
\begin{equation}\label{eqn: fm+1}
\partial_t f^{m+1}+v\cdot \nabla_x f^{m+1}-\nabla_x \phi^m\cdot \nabla_v f^{m+1}+\frac{1}{2T_M}f^{m+1}v\cdot \nabla_x \phi^m=\Gamma_{\text{gain}}(f^m,f^m)-\nu(F^m)f^{m+1},
\end{equation}
with boundary condition:
\begin{equation}\label{eqn: fm+1 BC}
  f^{m+1}(t,x,v)\Big|_{\gamma_-}=e^{[\frac{1}{4T_M}-\frac{1}{2T_w(x)}]|v|^2}\int_{n(x)\cdot u>0}f^m(t,x,u)e^{-[\frac{1}{4T_M}-\frac{1}{2T_w(x)}]|u|^2}d\sigma(u,v),
\end{equation}
and initial condition
\[f^{m+1}(0,x,v)=f(0,x,v).\]

Now we separately discuss the roadmap for getting these two types of estimates.

\smallskip

\textit{$L^\infty$ estimate:} For obtaining the $L^\infty$ estimate, we derive the trajectory formula and trace back along the characteristic of the particles til it either hits the boundary or the initial datum for $f^m$.

It may so happen that some particles bounce back and forth in the domain multiple times before tracing back to $t=0$ (say $k$ times), and then a $k$-layered integral will appear. This multiple integral includes $v_i$, the parameter we use to represent the integral variable at the $i$-th iteration with the boundary (see more precise definition in Definition \ref{Def:Back time cycle}), and the integral formula will be derived in Lemma~\ref{lemma: the tracjectory formula for f^(m+1)}. There are two main problems one need to handle here: 1. how to integrate the $k$-fold integral, and 2. what is the probability for a particle to interact with particles finite times?

To deal with the first difficulty amounts to carefully trace and compute the integration. In case of the diffuse boundary condition with constant temperature where $R=\frac{1}{2\pi}e^{-\frac{|v|^2}{2}}|n(x)\cdot v|=c_\mu e^{\frac{-|v|^2}{2}}|n(x)\cdot v|$, the computation can be simplified. According to~\eqref{eqn: fm+1 BC}, the boundary condition here is:
\[f=c_\mu \sqrt{\mu(v_{i-1})}\int_{n\cdot v_i>0}f(v_i)\sqrt{\mu(v_i)}|n\cdot v_i|\dd v_i\,.\]
Trace back further for the next interaction of $i+1$, one arrives at the final integral with respect to $v_i$ to be simply
\[\int_{n\cdot v_i>0}c_\mu \mu(v_i)|n\cdot v_i|dv_i.\]
Since the form of this $v_i$-integral is uniform for all $1\leq i\leq k$, the multiple integral can be treated by Fubini's theorem. Such lucky coincidence no longer holds true for the C-L boundary condition. From~\eqref{eqn:probability measure} the integrand is a function of both $v$ and $u$. As a result the $v_i$-integral is not uniform for all $i$, and the Fubini's theorem is not available. The multiple integral thus needs to be computed with the fixed order $v_k,v_{k-1},\cdots,v_1$, bringing extra computational difficulty. We now perform this integral order by order. To do so we start with $v_k$, the most out layer. The integral contains
\begin{equation}\label{eqn: over k}
\int_{n\cdot u_k>0}   e^{-[\frac{1}{4T_M}-\frac{1}{2T_w(x)}]|u_k|^2}  d\sigma(v_{k},v_{k-1})\,,
\end{equation}
with appropriate $d\sigma(v_k,v_{k-1})$ definition. This integral then becomes a function of $v_{k-1}$, which is then computed in the second outer layer. Using Lemma \ref{Lemma: abc} one can show that~\eqref{eqn: over k} can be approximately explicitly computed -- $e^{c|v_{k-1}|^2}$. We perform this iteratively over $i$ counting back from $k$ to $1$, and inductively compute this $k$-fold integral. This result is presented in Lemma~\ref{lemma: boundedness} .

To deal with the second difficulty, one needs to quantize the probability of a particle that interacts with the wall more than $k$ times, or equivalently, we need to give an estimate the measure $\mathbf{1}_{\{t_k>0\}}$. In~\cite{G,CKL} the authors studied the diffuse boundary condition in which they decompose the boundary as
\[\gamma_+^{\delta}=\{u\in \gamma_+:|n\cdot u|>\delta,|u|\leq \delta^{-1}\}\,,\quad\text{and}\quad \gamma_+\backslash \gamma_+^\delta\,,\]
and show that there can be only finite number of $v_j$ that is belong to $\gamma_+^\delta$. Meanwhile, the integration over $\gamma_+\backslash \gamma_+^\delta$ can be controlled by the small $\delta$. As $k$ increases, one obtains a larger power of $\delta$, leading to a decay factor for the measure of $\mathbf{1}_{\{t_k>0\}}$. When C-L condition is given, the strategy needs to be revised. In particular, the integrand in equation~\eqref{eqn: Formula for R} and~\eqref{eqn:probability measure} contains $e^{-|u_\parallel-(1-r_{\parallel})v_\parallel|^2}$, and even if $|u_\parallel|\gg 1$, $|u_\parallel-(1-r_{\parallel})v_\parallel|$ can still be small, meaning the integration over the $\gamma_+\backslash \gamma_+^\delta$ does not provide the smallness. One key observation here is to realize that depending on the distance between $u_\parallel$ and $(1-r_\parallel)v_\parallel$, cases can be discussed differently. Let $|u_\parallel|$ large enough, with $1-r_{\parallel}<1$. The bad case is when $|u_\parallel-(1-r_{\parallel})v_\parallel|<\delta^{-1}$, then $|v_\parallel|\geq |u_\parallel|+\delta^{-1}$. For example let $1-r_\parallel=1/2$, then if $|u_\parallel-\frac{1}{2}v_\parallel|<\delta^{-1}$, we take $|u_\parallel|\geq 3\delta^{-1}$ to have:
\[\frac{1}{2}|v_\parallel|>|u_\parallel|-\delta^{-1}>\frac{1}{2}|u_\parallel|+\frac{1}{2}\delta^{-1},\quad |v_\parallel|>|u_\parallel|+\delta^{-1}\,,\]
which brings up the value of $v_\parallel$. Consequently, if these `bad' cases $|u_\parallel-(1-r_\parallel)v_\parallel|<\delta^{-1}$ take place many times in the $k$-fold integral, a very big $v_i$ will be generated. Then the application of the boundary condition that provides a fast decay for big $|v_i|$ can be used to balance out all the growing factors, leading to a small measure of $\mathbf{1}_{t_k}>0$ in the end.

Consider this, we further decompose $\gamma_+$ into
\[\gamma_+^{\eta}=\{u\in \gamma_+: |n\cdot u|>\eta\delta,|u|\leq \eta\delta^{-1}\}\,,\quad\text{and}\quad \gamma_+\backslash\gamma_+^{\eta}\,,\]
where $\eta$ is selected to be a small number (depending on $r_\parallel$) so that
\[
|u_\parallel-(1-r_\parallel)v_\parallel|<\delta^{-1}\quad\Rightarrow\quad |v_\parallel|\geq |u_\parallel|+\delta^{-1}\,.
\]
We comment here that such property only works when the coefficient $1-r_\parallel<1$. In the real computation the wall temperature is involved in the boundary condition, thus the actual coefficient contains $T_w(x)$ and is more complicated than $1-r_\parallel$. In order to ensure such constant to be less than $1$, we impose the condition~\eqref{eqn: Constrain on T}. See Lemma~\ref{lemma: t^k} for detail.

\bigskip

\textit{$W^{1,p}$ estimate:} For getting the $W^{1,p}$ estimate~\eqref{W1p_local_bound}, we rely on the energy-type estimate for $\nabla_{x,v} f$ with weight $\alpha^{\beta}_{f,\epsilon}$, for which $\int_0^t \int_{\partial \Omega} \int_{n\cdot v< 0} |\alpha_{f,\epsilon}^\beta \nabla_{x,v}f|^p |n\cdot v| dvdS_x ds$ needs to be controlled. Using the fact that $\alpha_{f,\epsilon}(t,x,v)=|n(x)\cdot v|\text{ on }\gamma_-$, the singularity of~\eqref{deriv_singular} can be controlled by first setting:
\[\beta>\frac{p-2}{p},\quad |n\cdot v|^{p\beta-p+1}\in L_{\text{loc}}^1(\mathbb{R}^3)\,.\]
Then with some further calculation, shown in~\eqref{eqn: derivative bound on the boundary first part}~\eqref{eqn: derivative bound on the boundary second part}, we roughly need to estimate:
\begin{equation}\label{roughly}
\int_{\gamma_-}   |\alpha^\beta \partial f|^p \lesssim \int_{\gamma_-} e^{[\frac{1}{4T_M}-\frac{1}{2T_w(x)}]p|v|^2} \left(\int_{n\cdot u>0}   |\partial f(u)|e^{-[\frac{1}{4T_M}-\frac{1}{2T_w(x)}]p|u|^2} d\sigma(u,v)  \right)^p \,.
\end{equation}
To handle the integration of $u$, in~\cite{CKL}~\cite{GKTT}, the authors studied the diffusion boundary condition and proposed to split the term into the integration over the grazing set
\[\gamma_+^\epsilon=\{(x,u)\in \gamma_+: u\cdot n(x)<\epsilon \text{ or }|u|>1/\epsilon\}\quad\text{and}\quad \gamma_+\backslash\gamma_+^\epsilon\,.\]
However, this is not enough since we do not have direct smallness even for $u$ big (in the grazing set), and thus are not able to bound $\int_{\{(x,u)\in \gamma_+^\epsilon\}}$ by $\e$. To handle C-L boundary condition, we propose in this paper to add another layer of splitting. Besides the standard grazing/non-grazing sets, we also split the $\gamma_+$ integral into the grazing sets defined by $v$, approximately:
\[\gamma_+^{v,x,\epsilon}=\{(x,u)\in \gamma_+: u\cdot n(x)<\epsilon \text{ or }|u-v|>1/\epsilon\}\quad\text{and}\quad\gamma_+\backslash\gamma_+^{v,x,\epsilon}.\]
With this decomposition we have
\begin{align}
  \eqref{roughly}= & \int_{\gamma_-}e^{[\frac{1}{4T_M}-\frac{1}{2T_w(x)}]p|v|^2} \left(\int_{\{u:(x,u)\in \gamma_+^{v,x,\epsilon}\}}   |\partial f(u)|e^{-[\frac{1}{4T_M}-\frac{1}{2T_w(x)}]|u|^2} d\sigma(u,v)  \right)^p\ \notag\\
   & +\int_{\gamma_-}e^{[\frac{1}{4T_M}-\frac{1}{2T_w(x)}]p|v|^2} \left(\int_{\{u:(x,u)\in \gamma_+\backslash\gamma_+^{v,x,\epsilon}\}}   |\partial f(u)|e^{-[\frac{1}{4T_M}-\frac{1}{2T_w(x)}]|u|^2} d\sigma(u,v)  \right)^p \notag\\
   &\lesssim \int_{\gamma_-} e^{[\frac{1}{4T_M}-\frac{1}{2T_w(x)}]p|v|^2} \notag\\
   & \times \bigg[\left(\int_{\{u:(x,u)\in \gamma_+^{v,x,\epsilon}\}} e^{-[\frac{1}{4T_M}-\frac{1}{2T_w(x)}]q|u|^2}(\text{all terms in $d\sigma$~\eqref{eqn:C-L boundary condition in pro measure}})^q du\right)^{p/q}\int_{\gamma_+^{v,x,\epsilon}} |\alpha^\beta \partial f|^p \label{1}\\
   &+ \left(\int_{\{u:(x,u)\in \gamma_+\backslash\gamma_+^{v,x,\epsilon}\}} e^{-[\frac{1}{4T_M}-\frac{1}{2T_w(x)}]q|u|^2}(\text{all terms in $d\sigma$~\eqref{eqn:C-L boundary condition in pro measure}})^q du\right)^{p/q}\int_{\gamma_+\backslash\gamma_+^{v,x,\epsilon}} |\alpha^\beta \partial f|^p \bigg] \label{2}.
\end{align}

Now with the application of C-L boundary condition, one has the smallness in terms of $\epsilon$ for the integral over $\gamma_+^{v,x,\epsilon}$. And after direct computation one has the $L^1_v$ for $dv$ and thus bounds
\[\eqref{1}\lesssim O(\e)\int_{\gamma_+^{v,x,\e}}|\alpha^\beta\partial f|^p\leq O(\e)\int_{\gamma_+}|\alpha^\beta \partial f|^p.\]

On the set of $\gamma_+\backslash \gamma_+^{v,x,\epsilon}$, one still has $L^1_v$ integrand for $dv$ but the smallness is lost. We now recycle the standard grazing/non-grazing set definition, by further splitting the $v$-integration into $\mathbf{1}_{|v|\leq \epsilon^{-1}}$ and $\mathbf{1}_{|v|\geq \epsilon^{-1}}$. While the integration is naturally bounded by $O(\epsilon)$ when integrated on $|v|\geq \epsilon^{-1}$, the $|v|\leq \epsilon^{-1}$ case leads to $|u|\leq 2\epsilon^{-1}$, making $u$ falling in the non-grazing set $\gamma_+\backslash\gamma_+^{\epsilon/2}$. We now stand on the same footing as the situation discussed in~\cite{CKL,GKTT}. Apply Lemma~\ref{lemma: trace thm} we obtain an upper bound for the integration in the bulk (the terms not involving boundaries) and initial data, meaning:
\[\eqref{2}\lesssim O(\e)\int_{\gamma_+}|\alpha^\beta \partial f|^p+ \text{ initial condition }+\text{ bulk }\,.\]
The bulk part is treated similarly as in the proof of~\cite{CKL}. The entire proof for the weighted $W^{1,p}$ estimate is presented in Section 3.

The non-weighted $L^3_xL^{1+\delta}$ bound for the velocity derivative $\nabla_v f$ is discussed in Section 4. Characteristics and the energy-type estimate are the main tools used. The boundary terms are treated similarly as is done for the $W^{1,p}$ estimate, and the bulk terms are similar to those estimated in~\cite{CKL}. This estimate in the end leads to the $L^{1+\delta}$ stability $\Vert f-g\Vert_{L^{1+\delta}}\lesssim \Vert f_0-g_0\Vert_{L^{1+\delta}}$.

\subsection{Outline}
In section 2 we prove the $L^\infty$ bound for the sequence solution $f^m$. In section 3, we prove the weight $W^{1,p}$ estimate for the sequence solution $f^m$. Then we derive the $L_x^{3}L_{v}^{1+\delta}$ estimate and the $L^{1+\delta}$ stability for $f^m$ in section 4. The $L^{1+\delta}$ stability is the key to the well-poseness. In section 5 we combine all the estimates for the sequence solution $f^m$ and conclude the existence and uniqueness. More specifically, in Theorem \ref{local_existence}, the existence is given by Proposition \ref{Prop existence} and the uniqueness is given by Proposition \ref{Prop uniqueness}. In the appendix we prove some necessary estimates.

\section{$L^\infty$ estimate}
For any given constants $\mathfrak{C},\theta\in\mathbb{R}$, define a Gaussian-weighted solution:
\begin{equation}\label{eqn: hm+1}
h^{m+1}(t,x,v)=e^{\theta|v|^2}e^{-\mathfrak{C}t \langle v\rangle^2} f^{m+1}(t,x,v)\,,
\end{equation}
then according to~\eqref{eqn: fm+1}, we have:
\begin{equation}\label{eqn:formula of f^(m+1)}
\begin{split}
   &  \partial_t h^{m+1}+v\cdot \nabla_x h^{m+1}-\nabla_x \phi^m\cdot \nabla_vh^{m+1}+ \nu^mh^{m+1} \\
    & =e^{\theta|v|^2}e^{-\mathfrak{C}t \langle v\rangle^2} \Gamma_{\text{gain}}\left(\frac{h^m}{e^{-\mathfrak{C}\langle v\rangle^2 t}e^{\theta|v|^2}},\frac{h^m}{e^{-\mathfrak{C}\langle v\rangle^2 t}e^{\theta|v|^2}}\right)\,,
\end{split}
\end{equation}
equipped with boundary condition
\begin{equation}\label{eqn: BC for fm+1}
 h^{m+1}|_{\gamma_-(x)}=e^{\theta|v|^2} e^{-\mathfrak{C}t \langle v\rangle^2}e^{[\frac{1}{4T_M}-\frac{1}{2T_w(x)}]|v|^2}\int_{\gamma_+(x)} h^m(t,x,u)e^{-[\frac{1}{4T_M}-\frac{1}{2T_w(x)}]|u|^2} e^{-\theta|u|^2}e^{\mathfrak{C}t \langle u\rangle^2}  d\sigma(u,v)\,,
\end{equation}
where $\gamma_{\pm}$ is defined in~\eqref{eqn: incoming outgoing} and
\begin{equation}\label{eqn: num}
  \nu^m(t)=\mathfrak{C}\langle v\rangle^2+\nabla_x \phi^m\cdot \nabla_v \big(-\mathfrak{C}t\langle v\rangle^2+\theta|v|^2 \big)+\frac{1}{2T_M}v\cdot \nabla_x \phi^m+\nu(F^m)\,.
\end{equation}
This equation is linear for $h^{m+1}$ with $h^m$ serving as a source term, $\nu^m$ serving as a damping coefficient and $\phi^m$ serving as the electric field. The main purpose of this section is to show that $h^m$, and thus $f^m$ form a bounded sequence in $L^\infty$. More precisely:

\begin{proposition}\label{proposition: boundedness}
Let $h^{m+1}$ satisfy~\eqref{eqn: BC for fm+1} with the Cercignani-Lampis boundary condition~\eqref{eqn: BC for fm+1}. Assume the constraints for $\theta$ and $T_w$ hold true (\eqref{eqn: Constrain on theta} and~\eqref{eqn: Constrain on T}) and $\Vert h_0(x,v)\Vert_{L^\infty}<\infty$. Then if
\begin{equation}\label{eqn: fm is bounded}
  \sup_{i\leq m}\Vert h^i(t,x,v)\Vert_{L^\infty}\leq C_\infty \Vert h_0(x,v)\Vert_\infty, \quad t\leq t_\infty\,,
\end{equation}
we have
\begin{equation}\label{eqn: L_infty bound for f^m+1}
\sup_{0\leq t\leq t_\infty}\Vert h^{m+1}(t,x,v)\Vert_{L^\infty} \leq C_\infty\Vert h_0(x,v)\Vert_{L^\infty}\,.
\end{equation}
Here $C_\infty=C_\infty(T_M,\min\{T_w(x)\},\theta,r_\perp,r_\parallel,\Omega)$ is a constant and
\begin{equation}\label{eqn: t_1}
t\leq t_{\infty}=t_{\infty}( \Vert h_0(x,v)\Vert_{L^\infty},T_M,\min\{T_w(x)\},\theta,r_\perp,r_\parallel,\Omega)\ll 1\,.
\end{equation}
\end{proposition}

\begin{remark}
Two remarks are in line:
\begin{itemize}
\item The smallness only depends on the initial data, wall temperature, domain, the accommodation coefficients $r_{\parallel,\perp}$.
\item We will also trace the dependence of the constants $C_\infty$ and $t_\infty$ in the proof. $C_\infty$ will be explicitly defined in~\eqref{eqn: Cinfty}.
\end{itemize}
\end{remark}

This proposition implies the uniform-in-$m$ $L^\infty$ estimate for $h^{m}(t,x,v)$, and this allows us to further bound
\begin{equation}\label{eqn: theta'}
\sup_m\Vert w_{\theta'}f^m\Vert_\infty<\infty\,,
\end{equation}
which lays the foundation for later sections.

To show the proposition, we start with Lemma \ref{lemma: phi_inf} in which we control the acceleration term $\nabla_x \phi$. We then explicitly derive the formula using the information of the trajectory for $h^{m}$, as will be presented in Lemma~\ref{lemma: the tracjectory formula for f^(m+1)}. This will bring a $k$-fold integration for particles that collide with the boundary $k$-times before the final time. We will further show that all the terms in this integration (more precisely, all terms in~\eqref{eqn: Duhamal principle for case1}~\eqref{eqn: Duhamel principle for case 2}), can be bounded in Lemma~\ref{lemma: boundedness} and Lemma~\ref{lemma: t^k}. We then summarize the estimates and give the proof of the proposition.

We now present Lemma \ref{lemma: phi_inf} and Lemma \ref{lemma: the tracjectory formula for f^(m+1)}. Then we split this section into three subsections, the first subsection concludes the proof for Lemma \ref{lemma: the tracjectory formula for f^(m+1)}. We present both Lemma \ref{lemma: boundedness} and Lemma \ref{lemma: t^k} in the second subsection. In last subsection we combine the estimates in Lemma \ref{lemma: boundedness} and Lemma \ref{lemma: t^k} with the formula in Lemma \ref{lemma: the tracjectory formula for f^(m+1)} to conclude Proposition \ref{proposition: boundedness}.

We first give an estimate of the bound of $\Vert \nabla_x \phi^m\Vert_\infty$.
\begin{lemma}\label{lemma: phi_inf} For any $0<\delta<1,\theta<\frac{1}{4T_M}$, $0\leq t\leq 1$, if $(f,\phi)$ satisfy the condition~\eqref{equation for phi_f} then
\begin{equation}\label{}
  \Vert \phi(t)\Vert_{C^{1,1-\delta}(\bar{\Omega})}\leq C\Vert h(t)\Vert_{L^{\infty}}+C\rho_0\,.
\end{equation}
\end{lemma}
\begin{proof} For any $p>1$,
\begin{align*}
 \Big\Vert \int_{\mathbb{R}^3}f(t,x,v)\sqrt{\mu(v)}dv-\rho_0\Big\Vert_{L^p(\Omega)}  & \leq \Big\Vert\int_{\mathbb{R}^3}f(t,x,v)\sqrt{\mu(v)}dv\Big\Vert_{L^p(\Omega)}+\Vert\rho_0\Vert_{L^p(\Omega)} \\
   & \leq |\Omega|^{1/p}\left(\int_{\mathbb{R}^3}e^{-\theta|v|^2}e^{\mathfrak{C}\langle v\rangle^2 t}\sqrt{\mu(v)}dv\right)\Vert h(t)\Vert_{L^{\infty}}+\rho_0\,.
\end{align*}

By the elliptic estimate with condition~\eqref{eqn: Mass conservation}
\[
\Vert \phi(t)\Vert_{W^{2,p}(\Omega)}\leq C\Vert h(t)\Vert_{L^{\infty}}+\rho_0\,,
\]
which further leads to, according to the Morrey inequality for $p>3$, $\Omega\subset \mathbb{R}^3$, and $\partial \Omega$ being $C^1$:
\[
\Vert \phi(t)\Vert_{C^{1,1-3/p}(\Omega)}\leq C\Vert \phi(t)\Vert_{W^{2,p}(\Omega)}\leq C\Vert h(t)\Vert_{L^{\infty}}+C\rho_0\,.
\]
\end{proof}

We represent $h^{m+1}$ with the stochastic cycles defined as follows.
\begin{definition}\label{Def:Back time cycle}
Define an H\"{o}lder continuous characteristics which solves (since $\nabla \phi^m$ is quasi-Lipschitz continuous from Lemma~\ref{lemma: phi_inf}, this is possible, see also chapter 8 of~\cite{MA} for example)
\begin{equation}\label{eqn: trajectory for Xm}
  \frac{d}{ds}\left(
                \begin{array}{c}
                  X^m(s;t,x,v) \\
                  V^m(s;t,x,v) \\
                \end{array}
              \right)=\left(
                        \begin{array}{c}
                          V^m(s;t,x,v) \\
                          -\nabla_x \phi^m(s,X^m(s;t,x,v)) \\
                        \end{array}
                      \right)\,,
\end{equation}
and we trace back in time and determine the boundary-colliding time and location, namely:
\[t_{1}(t,x,v)=\sup\{s<t:X^m(s;t,x,v)\in \partial \Omega\},\quad x_{1}(t,x,v)=X^m\left(t_1(t,x,v);t,x,v\right)\,.\]
We then build the probability measure at $x=x_1$ as $d\sigma(v_1,V^m(t_1;t,x,v))$, supported on $\mathcal{V}_1=\gamma_+(x_1)$:
\[
\int_{\mathcal{V}_1}d\sigma(v_1,V^m(t_1;t,x,v))=1\,.
\]
%

Inductively, define $t_k$ and $x_k$ the time and position of a particle striking the boundary for the $k$-th time:
\[t_k(t,x,v,v_1,\cdots,v_{k-1})=\sup\{s<t_{k-1}:X^{m-k+1}(s;t_{k-1},x_{k-1},v_{k-1})\in \partial \Omega\},\]
\[x_k(t,x,v,v_1,\cdots,v_{k-1})=X^{m-k+1}\left(t_k(t,x,v,v_{k-1});t_{k-1}(t,x,v),x_{k-1}(t,x,v),v_{k-1}\right),\]
and correspondingly build probability measure $d\sigma(v_k,V^{m-k+1}(t_k;t_{k-1},x_{k-1},v_{k-1}))$ at $x_k$ over $\mathcal{V}_k=\gamma_+(x_k)$ for:
\[
\int_{\mathcal{V}_k}d\sigma(v_k,V^{m-k+1}(t_k;t_{k-1},x_{k-1},v_{k-1}))=1\,.
\]
For simplicity, we denote for all $l\leq m$:
\[V^{m-l}(s):=V^{m-l}(s;t_{l},x_{l},v_{l}),\quad X^{m-l}(s):=X^l(s;t_{l},x_{l},v_{l})\,.\]
\end{definition}

\begin{lemma}\label{lemma: the tracjectory formula for f^(m+1)}
Let $h^{m+1}$ satisfies~\eqref{eqn:formula of f^(m+1)} with the Cercignani-Lampis boundary condition~\eqref{eqn: BC for fm+1}, and assume~\eqref{eqn: fm is bounded} holds true, then with properly chosen $\mathfrak{C}$ and $\theta$, point-wise in $(t,x,v)$, one has: if $t_1\leq 0$:
\begin{equation}\label{eqn: Duhamal principle for case1}
\begin{split}
 |h^{m+1}(t,x,v)|\leq  &  |h_0\left(X^m(0),V^m(0)\right)| \\
    & +\int_0^t e^{-\int_s^t \frac{\mathfrak{C}}{2} \langle V^m(\tau)\rangle^2 d\tau} e^{\theta|V^m(s)|^2}   e^{-\mathfrak{C}s\langle V^m(s)\rangle^2}   \Gamma_{\text{gain}}^m(s)ds\,.
\end{split}
\end{equation}
If $t_1>0$, for arbitrary $k\geq 2$, one has:
\begin{equation}\label{eqn: Duhamel principle for case 2}
\begin{split}
    |h^{m+1}(t,x,v)|\leq & \int_{t_1}^t e^{-\int_s^t \frac{\mathfrak{C}}{2} \langle V^m(\tau)\rangle d\tau} e^{\theta|V^m(s)|^2}   e^{-\mathfrak{C}\langle V^m(s)\rangle s}   \Gamma_{\text{gain}}^m(s)ds\\
    & +   e^{\theta|V^m(t_1)|^2} e^{-\mathfrak{C}t_1\langle V^m(t_1)\rangle^2}e^{[\frac{1}{4T_M}-\frac{1}{2T_w(x_1)}]|V^m(t_1)|^2}\int_{\prod_{j=1}^{k-1}\mathcal{V}_j}H\,,
\end{split}
\end{equation}
with $H$ given by
\begin{equation}\label{eqn: formula for H}
\begin{split}
   & \sum_{l=1}^{k-1}\mathbf{1}_{\{t_l>0,t_{l+1}\leq 0\}}|h_0\left(X^{m-l}(0),V^{m-l}(0)\right)|d\Sigma_{l,m}^k(0)  \\
    & +\sum_{l=1}^{k-1}\int_{\max\{0,t_{l+1}\}}^{t_l} e^{\theta|V^{m-l}(s)|^2}    e^{-\mathfrak{C}s\langle V^{m-l}(s)\rangle^2}|\Gamma_{\text{gain}}^{m-l}(s)|d\Sigma_{l,m}^k(s)ds\\
    & +\mathbf{1}_{\{t_k>0\}}|h^{m-k+2}\left(t_k,x_k,V^{m-k+1}(t_k)\right)|d\Sigma_{k-1,m}^k(t_k)\,,
\end{split}
\end{equation}
where
\begin{equation}\label{eqn:trajectory measure}
\begin{split}
 d\Sigma_{l,m}^k(s)=&    \Big\{\prod_{j=l+1}^{k-1}d\sigma\left(v_j,V^{m-j+1}(t_j)\right)\Big\}\\
 & \Big\{e^{-\int_s^{t_l} \frac{\mathfrak{C}}{2} \langle V^{m-l}(\tau)\rangle^2 d\tau} e^{-\theta|v_l|^2} e^{\mathfrak{C}t_l\langle v_l \rangle^2 } e^{-[\frac{1}{4T_M}-\frac{1}{2T_w(x_l)}]|v_l|^2}
d\sigma(v_l,V^{m-l+1}(t_l))\Big\}\\
    & \Big\{\prod_{j=1}^{l-1} 2  e^{\mathfrak{C}(t_j-t_{j+1})\langle v_j\rangle^2}  e^{[\frac{1}{2T_w(x_j)}-\frac{1}{2T_w(x_{j+1})}]|v_j|^2}  d\sigma\left(v_j,V^{m-j+1}(t_j)\right)\Big\}.
\end{split}
\end{equation}
Here we use a notation
\begin{equation}\label{eqn: gamma^m}
\Gamma_{\text{gain}}^m(s):=\Gamma_{\text{gain}}\left(\frac{h^m(s,X^m(s),V^m(s))}{e^{\theta|V^m(s)|^2}e^{-\mathfrak{C}s\langle V^m(s)\rangle^2}},\frac{h^m(s,X^m(s),V^m(s))}{e^{\theta|V^m(s)|^2}e^{-\mathfrak{C}s\langle V^m(s)\rangle^2}}\right).
\end{equation}

\end{lemma}

\subsection{Proof of Lemma \ref{lemma: the tracjectory formula for f^(m+1)}}
We present the proof of Lemma~\ref{lemma: the tracjectory formula for f^(m+1)}. Most of the proof is tedious but straightforward derivation.

\begin{proof}[\textbf{Proof of Lemma \ref{lemma: the tracjectory formula for f^(m+1)}}]
For given $\phi^m$, we choose small enough $t$ and big enough $\mathfrak{C}$:
\begin{equation}\label{eqn: C satisfy 1}
\nabla_x \phi^m\cdot \nabla_v \big(-\mathfrak{C}t\langle v\rangle^2+\theta|v|^2 \big)+\frac{1}{2T_M}v\cdot \nabla_x \phi^m\leq \frac{\mathfrak{C}}{2}\langle v\rangle^2\,,
\end{equation}
and thus, noting $\nu(F^m)\geq 0$, from~\eqref{eqn: num}, we have:
\begin{equation}\label{eqn: choose C}
\nu^m(t)\geq \frac{\mathfrak{C}}{2}\langle v\rangle^2\,.
\end{equation}

From~\eqref{lemma: phi_inf} we first denote
\begin{equation}\label{eqn: cphi}
C_{\phi^m}:=\sup_{0\leq i\leq m}\Vert \nabla_x \phi^i\Vert_\infty\lesssim \sup_{0\leq i\leq m}\Vert h^i\Vert_{L^\infty}<\infty\,,
\end{equation}
and
\begin{equation}\label{eqn: mu tilde}
\tilde{\mu}(t,x,v):=e^{-\theta|v|^2}e^{\mathfrak{C}t\langle v\rangle^2 }e^{-[\frac{1}{4T_M}-\frac{1}{2T_w(x)}]|v|^2}\,.
\end{equation}

If $t_1(t,x,v)\leq 0$, the particle has been following a fixed trajectory without scattering, then according to~\eqref{eqn:formula of f^(m+1)}, for $0\leq s\leq t$,
\begin{equation}\label{eqn: inte factor}
\frac{d}{ds}\left[e^{-\int_s^t \nu^m(\tau) d\tau} h^{m+1}(s,X^m(s),V^m(s))\right]=e^{-\int_s^t \nu^m(\tau) d\tau}e^{\theta|V^m(s)|^2}e^{-\mathfrak{C}s\langle V^m(s)\rangle^2 }\Gamma_{\text{gain}}^m(s)\,,
\end{equation}
then~\eqref{eqn: Duhamal principle for case1} follows by applying~\eqref{eqn: choose C}.

If $t_1(t,x,v)>0$, the trajectory of the particle can be split into a few discontinuous sections. In particular:
\begin{equation}\label{eqn: proof for the intial step in the Duhamul principle}
\begin{split}
   |h^{m+1}(t,x,v)\textbf{1}_{\{t_1>0\}}|\leq & |h^{m+1}\left(t_1,x_1,V^m(t_1)\right)|e^{-\int_{t_1}^t \frac{\mathfrak{C}}{2} \langle V^m(\tau) \rangle^2 d\tau}\\
    & +\int_{t_1}^t e^{-\int_s^t \frac{\mathfrak{C}}{2} \langle V^m(\tau)\rangle^2 d\tau} e^{\theta|V^m(s)|^2} e^{-\mathfrak{C}s\langle V^m(s)\rangle^2 } |\Gamma_{\text{gain}}^m(s)|ds\,.
\end{split}
\end{equation}
Note the first term of the RHS of~\eqref{eqn: proof for the intial step in the Duhamul principle} can be expressed by the boundary condition. In particular, for $1\leq k\leq m$, the boundary condition~\eqref{eqn: BC for fm+1} can be written as, using~\eqref{eqn: mu tilde}:
\begin{equation}\label{eqn: bc for hm+1}
  h^{m-k+2}(t_k,x_k,V^{m-k+1}(t_k))=\frac{1}{\tilde{\mu}(t_k,x_k,V^{m-k+1}(t_k))}\int_{\mathcal{V}_k}h^{m-k+1}(t_k,x_k,v_k)\tilde{\mu}(t_k,x_k,v_k)d\sigma(v_k,V^{m-k+1}(t_k))\,.
\end{equation}

We now use induction on $k$ to show~\eqref{eqn: Duhamel principle for case 2}. Directly applying~\eqref{eqn: bc for hm+1} with $k=1$, the first term of the RHS of~\eqref{eqn: proof for the intial step in the Duhamul principle} is bounded by
\begin{equation}\label{eqn: above term}
\frac{1}{\tilde{\mu}\left(t_1,x_1,V^m(t_1)\right)}\int_{\mathcal{V}_1}h^{m}(t_1,x_1,v_1)\tilde{\mu}(t_1,x_1,v_1)d\sigma(v_1,V^{m}(t_1))\,.
\end{equation}

Noting~\eqref{eqn: Duhamal principle for case1} and~\eqref{eqn: proof for the intial step in the Duhamul principle}, this term is to be controlled by:
\begin{equation*}
\begin{aligned}
\eqref{eqn: above term} & \leq \frac{1}{\tilde{\mu}(t_1,x_1,V^m(t_1))}\Big[\int_{\mathcal{V}_1}\mathbf{1}_{\{t_2\leq 0<t_1\}} e^{-\int_0^{t_1} \frac{\mathfrak{C}}{2}\langle V^{m-1}(\tau)\rangle^2 d\tau} h^{m}(0,X^{m-1}(0),V^{m-1}(0))\tilde{\mu}(t_1,x_1,v_1)d\sigma\left(v_1,V^{m}(t_1)\right)\\
&+\int_0^{t_1}\int_{\mathcal{V}_1}\mathbf{1}_{\{t_2\leq 0<t_1\}}e^{-\int_s^{t_1} \frac{\mathfrak{C}}{2} \langle V^{m-1}(\tau)\rangle^2 d\tau}  e^{\theta|V^{m-1}(s)|^2} e^{-\mathfrak{C}s\langle V^{m-1}(s)\rangle^2 } |\Gamma_{\text{gain}}^{m-1}(s)|\tilde{\mu}(t_1,x_1,v_1)d\sigma\left(v_1,V^{m}(t_1)\right) ds\\
&+\int_{\mathcal{V}_1}\mathbf{1}_{\{t_2>0\}}e^{-\int_{t_2}^{t_1} \frac{\mathfrak{C}}{2} \langle V^{m-1}(\tau)\rangle^2 d\tau} |h^{m}(t_2,x_2,V^{m-1}(t_2))\tilde{\mu}(t_1,x_1,v_1)d\sigma\left(v_1,V^{m}(t_1)\right)\\
&+\int_{t_2}^{t_1}\int_{\mathcal{V}_1}\mathbf{1}_{\{t_2> 0\}}e^{-\int_s^{t_1}\frac{\mathfrak{C}}{2} \langle V^{m-1}(\tau)\rangle^2 d\tau}  e^{\theta|V^{m-1}(s)|^2}  e^{-\mathfrak{C}s\langle V^{m-1}(s)\rangle^2 } |\Gamma_{\text{gain}}^{m-1}(s)|\tilde{\mu}(t_1,x_1,v_1)d\sigma\left(v_1,V^m(t_1)\right) ds \Big]\,,
\end{aligned}
\end{equation*}
showing the validity of~\eqref{eqn: Duhamel principle for case 2} $k=2$. For higher $k$, we use induction. Assume~\eqref{eqn: Duhamel principle for case 2} is valid for $k\geq 2$(induction hypothesis) we prove so for $k+1$. We express the last term in~\eqref{eqn: formula for H} using the boundary condition. In~\eqref{eqn: bc for hm+1}, since $\frac{1}{\tilde{\mu}(t_k,x_k,V^{m-k+1}(t_k))}$ depends on $v_{k-1}$, we move this term to the integration over $\mathcal{V}_{k-1}$ in~\eqref{eqn: Duhamel principle for case 2}. Using the second line of~\eqref{eqn:trajectory measure} with $l=k-1,s=t_k$, the integration over $\mathcal{V}_{k-1}$ is
\begin{equation}\label{eqn: wl}
\int_{\mathcal{V}_{k-1}}  \frac{e^{-\int_{t_k}^{t_{k-1}} \frac{\mathfrak{C}}{2} \langle V^{m-k+1}(\tau)\rangle^2 d\tau}}{\tilde{\mu}\left(t_k,x_k,V^{m-k+1}(t_k)\right)} \tilde{\mu}(t_{k-1},x_{k-1},v_{k-1})    d\sigma(v_{k-1},V^{m-k}(t_{k-1})).
\end{equation}

By Definition \ref{Def:Back time cycle} we have $|V^{m-k+1}(t_{k-1})|=|v_{k-1}|$. By~\eqref{eqn: trajectory for Xm} for $t_{k}\leq \tau\leq t_{k-1}$ we have
\begin{equation}\label{eqn: ..}
\langle v_{k-1}\rangle-C_{\phi^m}(t_{k-1}-\tau)\leq V^{m-k+1}(\tau)\rangle\leq \langle v_{k-1}\rangle+C_{\phi^m}(t_{k-1}-\tau),
\end{equation}
with $C_{\phi^m}$ is defined in~\eqref{eqn: cphi}. This leads to

\begin{equation}\label{eqn: ...}
 |V^{m-k+1}(t_k)|^2 \leq |v_{k-1}|^2+2C_{\phi^m}(t_{k-1}-t_k)|v_{k-1}|+(C_{\phi^m})^2(t_{k-1}-t_k)^2\,,
\end{equation}
and
\begin{equation}\label{eqn: jia}
  \langle V^{m-k+1}(t_k)\rangle^2\geq  \langle v_{k-1}\rangle^2-2C_{\phi^m}(t_{k-1}-t_k)\langle v_{k-1}\rangle\,,
  \end{equation}
  which further suggests:
  \begin{equation}
  e^{-\int_{t_k}^{t_{k-1}} \frac{\mathfrak{C}}{2} \langle V^{m-k+1}(\tau)\rangle^2 d\tau}\leq e^{(t_{k-1}-t_k)\langle v_{k-1}\rangle \Big(-\frac{\mathfrak{C}}{2}\langle v_{k-1}\rangle+\mathfrak{C}C_{\phi^m}(t_{k-1}-t_{k}) \Big)}\,.
\end{equation}

Considering the definition of $\tilde\mu$ in~\eqref{eqn: mu tilde}, and utilizing the inequalities above, we finally arrive at, taking $C_T:=\frac{1}{2\min\{T_w(x)\}}-\frac{1}{4T_M}$:

%
%
%
%

\begin{equation}\label{eqn: 5}
\begin{split}
&e^{-\int_{t_k}^{t_{k-1}} \frac{\mathfrak{C}}{2} \langle V^{m-k+1}(\tau)\rangle^2 d\tau}/\tilde{\mu}\left(t_k,x_k,V^{m-k+1}(t_k)\right)\\
   \leq & e^{-\mathfrak{C}\langle v_{k-1}\rangle^2 t_k+\theta|v_{k-1}|^2}e^{[\frac{1}{4T_M}-\frac{1}{2T_w(x_k)}]|v_{k-1}|^2} \\
    & \times \exp\left(\Big[-\frac{\mathfrak{C}}{2}\langle v_{k-1}\rangle+\mathfrak{C}C_{\phi^m}(t_{k-1}-t_k) +2\mathfrak{C}C_{\phi^m}t_k +2\theta C_{\phi^m}+2C_TC_{\phi^m} \Big](t_{k-1}-t_k)\langle v_{k-1}\rangle\right)\\
    & \times \exp\Big(\big[\theta C_{\phi^m}^2 +C_TC_{\phi^m}^2\big] (t_{k-1}-t_k)^2\Big)\,.
\end{split}
\end{equation}
Set $t=t(T,\theta,\phi^m)$ small enough such that
\[
\exp\Big(\big[\theta C_{\phi^m}^2 +C_TC_{\phi^m}^2\big] (t_{k-1}-t_k)^2\Big)\leq \exp\Big(\big[\theta C_{\phi^m}^2 +C_TC_{\phi^m}^2\big] t\Big)\leq 2\,,\quad \text{and}\quad  C_{\phi^m}t\leq \frac{1}{8}\,.
\]
Furthermore, we take $\mathfrak{C}$ to be big enough so that
\begin{equation}\label{eqn: C satisfy}
\begin{aligned}
&-\frac{\mathfrak{C}}{2}\langle v_{k-1}\rangle+\mathfrak{C}C_{\phi^m}(t_{k-1}-t_k) +2\mathfrak{C}C_{\phi^m}t_k +2\theta C_{\phi^m}+2C_TC_{\phi^m}\\
\leq &-\frac{\mathfrak{C}}{2}+\mathfrak{C}/8 +\mathfrak{C}/4 +2\theta C_{\phi^m}+2C_TC_{\phi^m}\leq -\frac{\mathfrak{C}}{8}  +2\theta C_{\phi^m}+2C_TC_{\phi^m}\leq 0\,.
\end{aligned}
\end{equation}
We simplify~\eqref{eqn: 5}:
\begin{equation}\label{eqn: 5 bound}
\eqref{eqn: 5}\leq 2e^{-\mathfrak{C}\langle v_{k-1}\rangle^2 t_k+\theta|v_{k-1}|^2}e^{[\frac{1}{4T_M}-\frac{1}{2T_w(x_k)}]|v_{k-1}|^2} \,.
\end{equation}
This leads to the boundedness of the integrand in~\eqref{eqn: wl} by:
\begin{equation}\label{eqn: ai}
\eqref{eqn: 5 bound}\times \tilde{\mu}(t_{k-1},x_{k-1},v_{k-1})=2 e^{[\frac{1}{2T_w(x_{k-1})}-\frac{1}{2T_w(x_k)}]|v_{k-1}|^2} e^{\mathfrak{C}(t_{k-1}-t_k)\langle v_{k-1}\rangle^2}\,,
\end{equation}
and in turn gives the estimate shown in~\eqref{eqn:trajectory measure} with $l=k-1$.

For the remaining term in~\eqref{eqn: bc for hm+1}, we split the integration over $\mathcal{V}_k$ into two terms as
\begin{equation}\label{eqn: sh}
\int_{\mathcal{V}_k}h^{m-k+1}(t_k,x_k,v_k) \tilde{\mu}(t_k,x_k,V^{m-k+1}(t_k)) d\sigma(v_k,v_{k-1})=\underbrace{\int_{\mathcal{V}_k}\mathbf{1}_{\{t_{k+1}\leq 0<t_k\}}}_{\eqref{eqn: sh}_1}+\underbrace{\int_{\mathcal{V}_k}\mathbf{1}_{\{t_{k+1}>0\}}}_{\eqref{eqn: sh}_2}.
\end{equation}
We use the similar bound of~\eqref{eqn: Duhamal principle for case1} and derive that
\begin{equation}\label{eqn: wy}
\begin{split}
   &\eqref{eqn: sh}_1\leq  \int_{\mathcal{V}_k} \mathbf{1}_{\{t_{k+1}\leq 0<t_k\}} e^{\int_0^{t_k} -\frac{\mathfrak{C}}{2} \langle V^{m-k}(\tau)\rangle^2 d\tau}h^{m-k+1}(0,X^{m-k}(0),V^{m-k}(0))\tilde{\mu}(t_k,x_k,v_k)d\sigma(v_k,V^{m-k+1}(t_k)) \\
    & +\int_0^{t_k}\int_{\mathcal{V}_k}\mathbf{1}_{\{t_{k+1}\leq 0<t_k\}}e^{\int_{s}^{t_k} -\frac{\mathfrak{C}}{2} \langle V^{m-k}(\tau)\rangle^2 d\tau}e^{-\mathfrak{C}\langle V^{m-k}(s)\rangle^2 s}e^{\theta |V^{m-k}(s)|^2}\Gamma_{\text{gain}}^{m-k}(s)\tilde{\mu}(t_k,x_k,v_k) d\sigma(v_k,V^{m-k+1}(t_k)) ds.
\end{split}
\end{equation}
In the first line of~\eqref{eqn: wy},
is consistent with the second bracket of the first line of~\eqref{eqn:trajectory measure} with $l=k,s=t_k$. In the second line of~\eqref{eqn: wy},
\[e^{\int_s^{t_k} -\frac{\mathfrak{C}}{2} \langle V^{m-k}(\tau)\rangle^2 d\tau}\tilde{\mu}(t_k,x_k,v_k)d\sigma(v_k,V^{m-k+1}(t_k))\]
is consistent with the second line of~\eqref{eqn:trajectory measure} with $l=k$.

From the induction hypothesis(~\eqref{eqn: Duhamel principle for case 2} is valid for $k$), we derive the integration over $\mathcal{V}_j$ for $j\leq k-1$ is consistent with the third line of~\eqref{eqn:trajectory measure}.
After taking integration $\int_{\prod_{j=1}^{k-1} \mathcal{V}_j}$ we change $d\Sigma_{k-1,m}^k$ in~\eqref{eqn:trajectory measure} to $d\Sigma_{k,m}^{k+1}$. Thus~\eqref{eqn: wy} becomes
\begin{equation}\label{eqn: qs}
  \begin{split}
     &\int_{\prod_{j=1}^{k} \mathcal{V}_j} \mathbf{1}_{\{t_{k+1}\leq 0<t_k\}}|h_0\left(X^{k+1}(0),v_k\right)|d\Sigma_{k,m}^{k+1}(0) \\
      & +\int_{\prod_{j=1}^{k} \mathcal{V}_j}\int_{0}^{t_k} e^{-\mathfrak{C}\langle V^{m-k}(s)\rangle^2 s}e^{\theta |V^{m-k}(s)|^2}\Gamma_{\text{gain}}^{m-k}(s)d\Sigma_{k,m}^{k+1}(s)ds.
  \end{split}
\end{equation}

Then we use the same estimate as~\eqref{eqn: proof for the intial step in the Duhamul principle} and derive

\begin{equation}\label{eqn: wyl}
  \begin{split}
    \eqref{eqn: sh}_2  &\leq  \int_{\mathcal{V}_k}\mathbf{1}_{\{t_{k+1}>0\}}e^{\int_{t_{k+1}}^{t_k} -\frac{\mathfrak{C}}{2} \langle V^{m-k}(\tau)\rangle^2 d\tau}h^{m+1}\left(t_{k+1},x_{k+1},V^{m-k}(t_{k+1})\right)\tilde{\mu}(t_k,x_k,v_k)d\sigma\left(v_k,V^{m-k+1}(t_k)\right)\\
      & +\int_{t_{k+1}}^{t_k}\int_{\mathcal{V}_k}\mathbf{1}_{\{t_{k+1}> 0\}}e^{\int_{s}^{t_k} -\frac{\mathfrak{C}}{2}\langle V^{m-k}(\tau)\rangle^2 d\tau}e^{-\mathfrak{C}\langle V^{m-k}(s)\rangle s}e^{\theta |V^{m-k}(s)|^2}\Gamma_{\text{gain}}^{m-k}(s)\tilde{\mu}(t_k,x_k,v_k) d\sigma(v_k,V^{m-k+1}(t_k)) ds.
  \end{split}
\end{equation}
Similarly as~\eqref{eqn: qs}, after taking integration over  $\int_{\prod_{j=1}^{k-1}\mathcal{V}_j}$~\eqref{eqn: wyl} is
\begin{equation}\label{eqn: nxh}
\begin{split}
   &  \int_{\prod_{j=1}^{k} \mathcal{V}_j} \mathbf{1}_{\{t_{k+1}>0\}}|h^{m-k+1}\left(t_{k+1},x_{k+1},V^{m-k}(t_k)\right)|d\Sigma_{k,m}^{k+1}(t_{k+1})\\
    & +\int_{\prod_{j=1}^{k} \mathcal{V}_j} \int_{t_{k+1}}^{t_k}e^{-\mathfrak{C}\langle V^{m-k}(s)\rangle^2 s}e^{\theta |V^{m-k}(s)|^2}\Gamma_{\text{gain}}^{m-k}(s)d\Sigma_{k,m}^{k+1}(s) ds.
\end{split}
\end{equation}

From~\eqref{eqn: nxh}~\eqref{eqn: qs}, the summation in the first and second lines of~\eqref{eqn: formula for H} extends to $k$. And the index of the third line of~\eqref{eqn: formula for H} changes from $k$ to $k+1$.
For the rest terms with index $l\leq k-1$, we haven not done any change to them in the previous step. Thus their integration are over $\prod_{l=1}^{k-1} \mathcal{V}_j$. We add $\int_{\mathcal{V}_k}d\sigma(v_k,V^{m-k+1}(t_k))=1$ to all of them, so that all the integrations are over $\prod_{l=1}^k \mathcal{V}_j$ and we change $d\Sigma_{l,m}^{k-1}$ to $d\Sigma_{l,m}^k$ by
\[d\Sigma_{l,m}^{k}=d\sigma\left(v_k,V^{m-k+1}(t_k)\right)d\Sigma_{l,m}^{k-1}.\]
Therefore, the formula~\eqref{eqn: formula for H} is valid for $k+1$ and we derive the lemma.
\end{proof}

As the lemma implies, to have $L_\infty$ bound of $h^{m+1}$, it is crucial to obtain an estimate of $H$ that is controlled in~\eqref{eqn: formula for H}. It is rather clear that the first two terms in~\eqref{eqn: formula for H} include all finite collisions $<k$ in finite time, while the last term collects all trajectories whose corresponding particles collide with boundaries more than $k$ times within $t$. These two types of estimates will be obtained in Lemma~\ref{lemma: boundedness} and Lemma~\ref{lemma: t^k} respectively in the next subsection. Namely we need only boundedness for the first two terms, but need decaying in $k$ for the third, in which we essentially need to show the chance for a particle to collide with boundaries more than $k$ times within a small time window $t$ is very small.

\subsection{$k$-fold integral}

As a preparation, we first define:\begin{equation}\label{eqn: def of r}
r_{max}:=\max(r_\parallel(2-r_\parallel),r_\perp)\,,\quad r_{min}:=\min(r_\parallel(2-r_\parallel),r_\perp)\,.
\end{equation}
Then immediately, one has $1\geq r_{max}\geq r_{min}>0$. We then define
\begin{equation}\label{eqn: xi}
\xi:=\frac{1}{4T_M\theta}>1\,,
\end{equation}
from~\eqref{eqn: hm+1} considering $\theta<\frac{1}{4T_M}$. In the calculation of the $k$-fold integration over $\prod_{j=1}^{k}\mathcal{V}_j$, we inductively use the following notations:
\begin{equation}\label{eqn: definition of T_p}
T_{l,l}=\frac{2\xi}{\xi+1}T_M\,,    \quad T_{l,l-1}=r_{min}T_M +(1-r_{min})T_{l,l}\,, \cdots,\quad  T_{l,1}= r_{min}T_M +(1-r_{min})T_{l,2}\,,
\end{equation}
and thus naturally for $1\leq i\leq l$, we have
\begin{equation}\label{eqn: formula of Tp}
T_{l,i}=\frac{2\xi}{\xi+1}T_M+(T_M-\frac{2\xi}{\xi+1}T_M)[1-(1-r_{min})^{l-i}]\,.
\end{equation}

Moreover, we will use
\begin{equation}\label{eqn: big phi}
\begin{split}
 d\Phi_{p,m}^{k,l}(s):=  & \Big\{\prod_{j=l+1}^{k-1}d\sigma(v_j,v_{j-1})\Big\}\\
 &\times \Big\{e^{-\int_s^{t_l} \frac{\mathfrak{C}}{2} \langle V^{m-l}(\tau)\rangle^2 d\tau} e^{-\theta|v_l|^2} e^{\mathfrak{C}t_l\langle v_l \rangle^2 } e^{-[\frac{1}{4T_M}-\frac{1}{2T_w(x_l)}]|v_l|^2}
d\sigma(v_l,V^{m-l+1}(t_l))\Big\}  \\
    & \times \Big\{\prod_{j=p}^{l-1} 2  e^{\mathfrak{C}(t_j-t_{j+1})\langle v_j\rangle^2}  e^{[\frac{1}{2T_w(x_j)}-\frac{1}{2T_w(x_{j+1})}]|v_j|^2}  d\sigma\left(v_j,V^{m-j+1}(t_j)\right)\Big\},
\end{split}
\end{equation}
and
\begin{equation}\label{eqn: upsilon}
d\Upsilon_{p}^{p'}:=\{\prod_{j=p}^{p'}  2  e^{\mathfrak{C}(t_j-t_{j+1})\langle v_j\rangle^2}  e^{[\frac{1}{2T_w(x_j)}-\frac{1}{2T_w(x_{j+1})}]|v_j|^2}  d\sigma(v_j,v_{j-1})\},
\end{equation}
to simplify the notation. Note that if $p=1$, $d\Phi_{1,m}^{k,l}(s)=d\Sigma_{l,m}^{k}(s)$, defined in~\eqref{eqn:trajectory measure}, and according to the definition in~\eqref{eqn: big phi} and~\eqref{eqn:trajectory measure}, we have
\begin{equation}\label{eqn: ppt for Phi}
d\Phi_{p,m}^{k,l}(s)=d\Phi_{p',m}^{k,l}(s)d\Upsilon_{p}^{p'-1}\,,\quad\text{and}\quad d\Sigma_{l,m}^k(s)=d\Phi_{p,m}^{k,l}(s)d\Upsilon_1^{p-1}\,.
\end{equation}

Now we state the lemma.
\begin{lemma}\label{lemma: boundedness}
There exists
\begin{equation}\label{eqn: t*}
t^*=t^*(T_M,\xi,\mathcal{C},\mathfrak{C},k)\,,
\end{equation}
such that for $t\leq t^*$, and $0\leq s\leq t_l$, we have
\begin{equation}\label{eqn: boundedness for l-p+1 fold integration}
   \int_{\prod_{j=p}^{k-1}\mathcal{V}_j}     \mathbf{1}_{\{t_l>0\}}  d\Phi_{p,m}^{k,l}(s) \leq (2C_{T_M,\xi})^{2(l-p+1)}\mathcal{A}_{l,p},
\end{equation}
with $C_{T_M,\xi}$ and $\mathcal{C}$ being constants defined in~\eqref{eqn: 1 one} and~\eqref{eqn: cal C} respectively, and
\begin{equation}\label{eqn: Elp}
\mathcal{A}_{l,p}=\exp\left(\big[ \frac{[T_{l,p}-T_w(x_{p})][1-r_{min}]}{2T_w(x_{p})[T_{l,p}(1-r_{min})+r_{min} T_w(x_{p})]} + (2\mathcal{C})^{l-p+1}(\mathfrak{C}t)\big]|V^{m-p+1}(t_p)|^2\right).
\end{equation}
Moreover, for any $p<p'\leq l$, we have
\begin{equation}\label{eqn: structure}
\int_{\prod_{j=p}^{k-1}\mathcal{V}_j} \mathbf{1}_{\{t_l>0\}}  d\Phi_{p,m}^{k,l}(s)\leq (2C_{T_M,\xi})^{2(l-p'+1)} \mathcal{A}_{l,p'}\int_{\prod_{j=p}^{p'-1} \mathcal{V}_j}  \mathbf{1}_{\{t_l>0\}} d\Upsilon_p^{p'-1}\leq (2C_{T_M,\xi})^{2(l-p+1)}\mathcal{A}_{l,p}\,.
\end{equation}

\end{lemma}

\begin{remark}
We comment here that this lemma indeed include the information for the $k$-fold integral in Lemma \ref{lemma: the tracjectory formula for f^(m+1)} by setting $p=1$. To derive the decaying factor in Lemma \ref{lemma: t^k}, we need to extract smallness from the integral over $v_p$ for $p\leq k$, for example, in Lemma \ref{Lemma: (2)} Lemma \ref{Lemma: accumulate}. This is the reason that we introduce the notation~\eqref{eqn: big phi}~\eqref{eqn: ppt for Phi} and incorporate them in this lemma.

\end{remark}

\begin{proof}
From~\eqref{eqn: normalization} and~\eqref{eqn:probability measure}, consider the first bracket of the first line in~\eqref{eqn:trajectory measure}, for $l+1\leq j\leq k-1$ we have
\[\int_{\prod_{j=l+1}^{k-1} \mathcal{V}_j}    \prod_{j=l+1}^{k-1}d\sigma(v_j,V^{m-j+1}(t_j))=1.\]
Without loss of generality we can assume $k=l+1$. Thus $d\Phi_{p,m}^{k,l}=d\Phi_{p,m}^{l+1,l}$. We use an induction of $p$ with $1\leq p\leq l$ to prove~\eqref{eqn: boundedness for l-p+1 fold integration}.

When $p=l$, by the second line of~\eqref{eqn: big phi}, the integration over $\mathcal{V}_l$ is written as
\begin{equation}\label{eqn: dsigmal}
\int_{\mathcal{V}_l} e^{-\theta|v_l|^2} e^{\mathfrak{C}\langle v_l\rangle^2 t_l} e^{\int_{s}^{t_l}-\frac{\mathfrak{C}}{2}
\langle V^{m-l}(\tau)\rangle^2 d\tau} e^{-[\frac{1}{4T_M}-\frac{1}{2T_w(x_l)}]|v_l|^2}d\sigma(v_l,V^{m-l+1}(t_l)).
\end{equation}
In order to compute~\eqref{eqn: dsigmal}, we bound
\begin{equation}\label{eqn: hhh}
e^{-\theta |v_l|^2}e^{\mathfrak{C}\langle v_l\rangle^2 t_l}\leq  2 e^{(-\theta+\mathfrak{C}t_l)|v_l|^2},
\end{equation}
where we take $t\leq t^*=t^*(\mathfrak{C})$ small enough such that $e^{\mathfrak{C}t}\leq 2$ and thus $e^{\mathfrak{C}\langle v_l\rangle^2 t_l}\leq e^{\mathfrak{C}| v_l|^2 t_l}e^{\mathfrak{C}t_l}\leq 2e^{\mathfrak{C}| v_l|^2 t_l}$.

By~\eqref{eqn: hhh} with $\theta=\frac{1}{4T_M\xi}$ in~\eqref{eqn: xi} we have
\begin{equation}\label{eqn: dsigmall}
~\eqref{eqn: dsigmal}\leq 2\int_{\mathcal{V}_l} e^{-[\frac{1}{2T_M}\frac{\xi+1}{2\xi}-\frac{1}{2T_w(x_l)}-\mathfrak{C}t_l]|v_l|^2}d\sigma(v_l,V^{m-l+1}(t_l)).
\end{equation}
Expanding the $d\sigma(v_l,V^{m-l+1}(t_l))$ using~\eqref{eqn: Formula for R} and~\eqref{eqn:probability measure}, we rewrite~\eqref{eqn: dsigmall} as
\begin{equation}\label{eqn: int over V_l}
\begin{split}
   & 2\int_{\mathcal{V}_{l,\perp}} \frac{|v_{l,\perp}|}{r_\perp T_w(x_l)}  e^{-[\frac{1}{2T_M}\frac{\xi+1}{2\xi}-\frac{1}{2T_w(x_l)}-\mathfrak{C}t_l]|v_{l,\perp}|^2}I_0\left(\frac{(1-r_\perp)^{1/2}v_{l,\perp}V^{m-l+1}_\perp(t_l)}{T_w(x_l)r_\perp}\right)e^{-\frac{|v_{l,\perp}|^2+(1-r_\perp)|V^{m-l+1}_\perp(t_l)|^2}{2T_w(x_l)r_\perp}}  dv_{l,\perp} \\
    &\times \int_{\mathcal{V}_{l,\parallel}}\frac{1}{\pi r_\parallel(2-r_\parallel)(2T_w(x_l))}e^{-[\frac{1}{2T_M}\frac{\xi+1}{2\xi}-\frac{1}{2T_w(x_l)}-\mathfrak{C}t_l]|v_{l,\parallel}|^2}e^{-\frac{1}{2T_w(x_l)}\frac{|v_{l,\parallel}-(1-r_\parallel)V^{m-l+1}_{\parallel}(t_l)|^2}{r_\parallel(2-r_\parallel)}}dv_{l,\parallel},
\end{split}
\end{equation}
where $v_{l,\parallel}$, $v_{l,\perp}$, $\mathcal{V}_{l,\perp}$ and $\mathcal{V}_{l,\parallel}$ are defined as
\begin{equation}\label{eqn: Define space}
v_{l,\perp}=v_l\cdot n(x_l),~~ v_{l,\parallel}=v_l-v_{l,\perp}n(x_l),~~\mathcal{V}_{l,\perp}=\{v_{l,\perp}:v_l\in \mathcal{V}_l\},~~\mathcal{V}_{l,\parallel}=\{v_{l,\parallel}:v_l\in \mathcal{V}_l\}.
\end{equation}
$V^{m-l+1}_\perp(t_l)$ and $V^{m-l+1}_\parallel(t_l)$ are defined similarly.

First we compute the integration over $\mathcal{V}_{l,\parallel}$, the second line of~\eqref{eqn: int over V_l}. To apply~\eqref{eqn: coe abc} in Lemma \ref{Lemma: abc}, we set
\[\e=\mathfrak{C}t_l,~w=(1-r_\parallel)V^{m-l+1}_\parallel(t_l)~,v=v_{l,\parallel},\]
\begin{equation}\label{eqn: coefficient a and b}
a=-[\frac{1}{2T_M\frac{2\xi}{\xi+1}}-\frac{1}{2T_w(x_l)}],~ b=\frac{1}{2T_w(x_l)r_\parallel(2-r_\parallel)}.
\end{equation}
By $\xi>1$ in~\eqref{eqn: xi}, we take $t^*=t^*(\mathfrak{C},\xi,T_M)\ll 1$ such that when $t_l<t\leq t^*$, we have
\begin{equation}\label{eqn: b-a-e}
b-a-\e=\frac{1}{2T_w(x_l)r_\parallel(2-r_\parallel)}-\frac{1}{2T_w(x_l)}+\frac{1}{2T_M\frac{2\xi}{\xi+1}}-\mathfrak{C}t_l
\geq \frac{1}{2T_M\frac{2\xi}{\xi+1}}-\mathfrak{C}t\geq \frac{1}{4T_M}.
\end{equation}
Also we take $t^*=t^*(\mathfrak{C},\xi,T_M)$ small enough to obtain $1+4T_M \mathfrak{C}t_l\leq 1+4T_M \mathfrak{C}t\leq 2$ when $t\leq t^*$. Hence
\begin{equation}\label{eqn: 1 one}
\begin{split}
\frac{b}{b-a-\e}   &=\frac{b}{b-a}[1+\frac{\e}{b-a-\e}]\leq \frac{\frac{2\xi}{\xi+1}T_M}{\frac{2\xi}{\xi+1}T_M+[T_w(x_l)-\frac{2\xi}{\xi+1}T_M]r_\parallel(2-r_\parallel)}[1+4T_M\mathfrak{C}t_l]  \\
   & \leq \frac{\frac{4\xi}{\xi+1}T_M}{\frac{2\xi}{\xi+1}T_M+[\min\{T_w(x)\}-\frac{2\xi}{\xi+1}T_M]r_{max}}:=C_{T_M,\xi},
\end{split}
\end{equation}
where we have used~\eqref{eqn: def of r}.

In regard to~\eqref{eqn: coe abc}, we have
\begin{equation}\label{eqn: com}
\frac{(a+\e)b}{b-a-\e}=\frac{ab}{b-a}[1+\frac{\e}{b-a-\e}]+\frac{b}{b-a-\e}\e.
\end{equation}
By~\eqref{eqn: 1 one} and $t_l<t$, we obtain
\[\frac{b}{b-a-\e}\e\leq \frac{\frac{4\xi}{\xi+1}T_M }{\frac{2\xi}{\xi+1}T_M+[\min\{T_w(x)\}-\frac{2\xi}{\xi+1}T_M]r_{\max}}\mathfrak{C}t.\]
By~\eqref{eqn: coefficient a and b}, we have
\[\frac{ab}{b-a}=\frac{\frac{2\xi}{\xi+1}T_M-T_w(x_l)}{2T_w(x_l)[\frac{2\xi}{\xi+1}T_M+[T_w(x_l)-\frac{2\xi}{\xi+1}T_M]r_\parallel(2-r_\parallel)]}.\]
Therefore, by~\eqref{eqn: b-a-e} and~\eqref{eqn: com} we obtain
\begin{equation}\label{eqn: 2 two}
\begin{split}
  &  \frac{(a+\e)b}{b-a-\e}\leq \frac{\frac{2\xi}{\xi+1}T_M-T_w(x_l)}{2T_w(x_l)[\frac{2\xi}{\xi+1}T_M+[T_w(x_l)-\frac{2\xi}{\xi+1}T_M]r_\parallel(2-r_\parallel)]}+\mathcal{C} (\mathfrak{C}t),
\end{split}
\end{equation}
where we define
\begin{equation}\label{eqn: cal C}
\mathcal{C}:=\frac{4T_M\big(\frac{2\xi}{\xi+1}T_M-\min\{T_w(x)\}\big)}{2\min\{T_w(x)\}[\frac{2\xi}{\xi+1}T_M+[\min\{T_w(x)\}-\frac{2\xi}{\xi+1}T_M]r_{max}]}+\frac{\frac{4\xi}{\xi+1}T_M}{\frac{2\xi}{\xi+1}T_M+[\min\{T_w(x)\}-\frac{2\xi}{\xi+1}T_M]r_{max}}.
\end{equation}

By~\eqref{eqn: 1 one},~\eqref{eqn: 2 two} and Lemma~\ref{Lemma: abc}, using $w=(1-r_\parallel)V^{m-l+1}_\parallel(t_l)$ we bound the second line of~\eqref{eqn: int over V_l} by
\begin{align}
   & C_{T_M,\xi}\exp\bigg(\Big[\frac{[\frac{2\xi}{\xi+1}T_M-T_w(x_l)]}{2T_w(x_l)[\frac{2\xi}{\xi+1}T_M(1-r_\parallel)^2+r_\parallel(2-r_\parallel) T_w(x_l)]}+\mathcal{C}(\mathfrak{C}t) \Big]|(1-r_\parallel)V^{m-l+1}_\parallel(t_l)|^2\bigg)\label{eqn: result for para}\\
   & \leq C_{T_M,\xi}\exp\bigg(\Big[\frac{[\frac{2\xi}{\xi+1}T_M-T_w(x_l)][1-r_{min}]}{2T_w(x_l)\big[\frac{2\xi}{\xi+1}T_M(1-r_{min})+r_{min} T_w(x_l)\big]}+\mathcal{C}(\mathfrak{C}t)\Big]|V^{m-l+1}_\parallel(t_l)|^2\bigg)\label{eqn: result of dsigmal},
\end{align}
where we have used~\eqref{eqn: def of r}.

Next we compute the first line of~\eqref{eqn: int over V_l}. To apply~\eqref{eqn: coe abc perp} in Lemma \ref{Lemma: perp abc}, we set
\[\e=\mathfrak{C}t_l,~w=\sqrt{1-r_\perp}V^{m-l+1}_\perp(t_l)~,v=v_{l,\perp},\]
\[a=-[\frac{1}{2T_M\frac{2\xi}{\xi+1}}-\frac{1}{2T_w(x_l)}],~ b=\frac{1}{2T_w(x_l)r_\perp}.\]
Thus we can compute $\frac{b}{b-a-\e}$ and $\frac{(a+\e)b}{b-a-\e}$ using~\eqref{eqn: 1 one} and~\eqref{eqn: 2 two}. Hence replacing $r_\parallel(2-r_\parallel)$ by $r_\perp$ and replacing $V^{m-l+1}_\parallel(t_l)$ by $V^{m-l+1}_\perp(t_l)$ in~\eqref{eqn: result for para}, we bound the first line of~\eqref{eqn: int over V_l} by
\[2C_{T_M,\xi}\exp\bigg(\Big[\frac{[\frac{2\xi}{\xi+1}T_M-T_w(x_l)]}{2T_w(x_l)[\frac{2\xi}{\xi+1}T_M(1-r_\perp)+r_\perp T_w(x_l)]}+\mathcal{C}(\mathfrak{C}t) \Big]|\sqrt{1-r_\perp}V^{m-l+1}_\perp(t_l)|^2\bigg)\]
\begin{equation}\label{eqn: result of dsigmal normal}
\leq 2C_{T_M,\xi}\exp\bigg(\Big[\frac{[\frac{2\xi}{\xi+1}T_M-T_w(x_l)][1-r_{min}]}{2T_w(x_l)\big[\frac{2\xi}{\xi+1}T_M(1-r_{min})+r_{min} T_w(x_l)\big]}+\mathcal{C}(\mathfrak{C}t)\Big]|V^{m-l+1}_\perp(t_l)|^2\bigg).
\end{equation}
where we use~\eqref{eqn: def of r}.

Collecting~\eqref{eqn: result of dsigmal}~\eqref{eqn: result of dsigmal normal}, we derive
\[\eqref{eqn: int over V_l}\leq 2(C_{T_M,\xi})^2\exp\left(\left[\frac{[\frac{2\xi}{\xi+1}T_M-T_w(x_l)][1-r_{min}]}{2T_w(x_l)\big[\frac{2\xi}{\xi+1}T_M(1-r_{min})+r_{min} T_w(x_l)\big]}+\mathcal{C}(\mathfrak{C}t)\right]|V^{m-l+1}(t_l)|^2\right)\leq (2C_{T_M,\xi})^2\mathcal{A}_{l,l},\]
where $\mathcal{A}_{l,l}$ is defined in~\eqref{eqn: Elp} and $T_{l,l}=\frac{2\xi}{\xi+1}T_M$.

Therefore,~\eqref{eqn: boundedness for l-p+1 fold integration} is valid for $p=l$.

Suppose~\eqref{eqn: boundedness for l-p+1 fold integration} is valid for the $p=q+1$(induction hypothesis) with $q+1\leq l$, then
\[\int_{\prod_{j=q+1}^{l}\mathcal{V}_j}        \mathbf{1}_{\{t_l>0\}} d\Phi_{q+1,m}^{l+1,l}(s)\leq (2C_{T_M,\xi})^{2(l-q)}\mathcal{A}_{l,q+1}.\]

We want to show~\eqref{eqn: boundedness for l-p+1 fold integration} holds for $p=q$. By the hypothesis and the third line of~\eqref{eqn: big phi},
\begin{equation}\label{eqn: dsigmaq}
\begin{split}
   &  \int_{\prod_{j=q}^{l}\mathcal{V}_j}        \mathbf{1}_{\{t_l>0\}} d\Phi_{q,m}^{l+1,l}(s) \leq  (2C_{T_M,\xi})^{2(l-q)}\mathcal{A}_{l,q+1}\\
    & \times \int_{\mathcal{V}_q}2  e^{\mathfrak{C}(t_q-t_{q+1})\langle v_q\rangle^2}    e^{[\frac{1}{2T_w(x_{q})}-\frac{1}{2T_w(x_{q+1})}]|v_{q}|^2}  d\sigma(v_q,V^{m-q}(t_{q+1})).
\end{split}
\end{equation}
Using the definition of $\mathcal{A}_{l,q+1}$ in~\eqref{eqn: Elp}, we obtain
\begin{equation}\label{eqn: dsigmap}
\begin{split}
   \eqref{eqn: dsigmaq} &\leq   (2C_{T_M,\xi})^{2(l-q)}  \int_{\mathcal{V}_{q}}2    \exp\bigg(\frac{(T_{l,q+1}-T_w(x_{q+1}))(1-r_{min})}{2T_w(x_{q+1})[T_{l,q+1}(1-r_{min})+r_{min} T_w(x_{q+1})]}|V^{m-q}(t_{q+1})|^2                                     \\
    &  +(2\mathcal{C})^{l-q}(\mathfrak{C}t)|V^{m-q}(t_{q+1})|^2\bigg) e^{\mathfrak{C}(t_q-t_{q+1})\langle v_q\rangle^2}     e^{[\frac{1}{2T_w(x_{q})}-\frac{1}{2T_w(x_{q+1})}]|v_{q}|^2}  d\sigma(v_q,V^{m-q+1}(t_q)).
\end{split}
\end{equation}

Let $\mathfrak{C}$ in~\eqref{eqn: choose C} satisfy
\begin{equation}\label{eqn: chooseC}
2C_{\phi^m}\frac{(T_{l,q+1}-T_w(x_{q+1}))(1-r_{min})}{2T_w(x_{q+1})[T_{l,q+1}(1-r_{min})+r_{min} T_w(x_{q+1})]}\leq \frac{C_{\phi^m}}{\min(T_w(x))(1-r_{min})}\leq \mathfrak{C},\quad (C_{\phi^m})^2\leq \mathfrak{C}.
\end{equation}
Similarly to~\eqref{eqn: ..} and~\eqref{eqn: ...},
\[|V^{m-q}(t_{q+1})|\leq |v_{q}|+C_{\phi^m}(t_{q}-t_{q+1}),\]
\[|V^{m-q}(t_{q+1})|^2\leq C_{\phi^m}^2(t_{q}-t_{q+1})^2+2C_{\phi^m}(t_{q}-t_{q+1})|v_{q}|+|v_{q}|^2.\]
Then we apply~\eqref{eqn: chooseC} to get
\begin{align*}
   & \exp\left[\left(\frac{(T_{l,q+1}-T_w(x_{q+1}))(1-r)}{2T_w(x_{q+1})[T_{l,q+1}(1-r)+r T_w(x_{q+1})]}+(2\mathcal{C})^{l-q}(\mathfrak{C}t)\right)|V^{m-q+1}(t_{q+1})|^2\right]\times e^{\mathfrak{C}(t_q-t_{q+1})\langle v_q\rangle^2} \\
   & \leq\exp\left[\left(\frac{(T_{l,q+1}-T_w(x_{q+1}))(1-r)}{2T_w(x_{q+1})[T_{l,q+1}(1-r)+r T_w(x_{q+1})]}+(2\mathcal{C})^{l-q}(\mathfrak{C}t)\right)|v_{q}|^2\right]\\
   &\times e^{\mathfrak{C}(t_{q}-t_{q+1})|v_q|} e^{\mathfrak{C}(t_{q}-t_{q+1})^2} e^{\mathfrak{C}(t_q-t_{q+1})|v_q|^2}e^{\mathfrak{C}(t_q-t_{q+1})} \\
   & \leq 2\exp\left[\left(\frac{(T_{l,q+1}-T_w(x_{q+1}))(1-r)}{2T_w(x_{q+1})[T_{l,q+1}(1-r)+r T_w(x_{q+1})]}+2(2\mathcal{C})^{l-q}(\mathfrak{C}t)\right)|v_{q}|^2\right] ,
\end{align*}
where we haved use $e^{\mathfrak{C}(t_q-t_{q+1})|v_q|}\leq e^{\mathfrak{C}(t_q-t_{q+1})}e^{\mathfrak{C}(t_q-t_{q+1})|v_q|^2}$ and thus
\[e^{\mathfrak{C}(t_{q}-t_{q+1})|v_q|} e^{\mathfrak{C}(t_{q}-t_{q+1})^2} \times e^{\mathfrak{C}(t_q-t_{q+1})|v_q|^2}e^{\mathfrak{C}(t_q-t_{q+1})}\leq e^{2\mathfrak{C}t|v_q|^2}e^{3\mathfrak{C}(t_q-t_{q+1})}\leq 2e^{(2\mathcal{C})^{l-q}(\mathfrak{C}t)|v_q|^2}.\]
Here we take $t^*=t^*(\mathfrak{C})$ small enough such that when $t\leq t^*$,
\[e^{3\mathfrak{C}(t_q-t_{q+1})}\leq e^{3\mathfrak{C}t}\leq 2.\]
Thus we obtain
\begin{align*}
 \eqref{eqn: dsigmap}  & \leq  4(2C_{T_M,\xi})^{2(l-q)}\int_{\mathcal{V}_{q}}
   \exp\bigg[\Big(\frac{(T_{l,q+1}-T_w(x_{q+1}))(1-r_{min})}{2T_w(x_{q+1})[T_{l,q+1}(1-r_{min})+r_{min} T_w(x_{q+1})]}|v_{q}|^2 \\
   & +2(2\mathcal{C})^{l-q}(\mathfrak{C}t) |v_q|^2\Big)\bigg] e^{[\frac{1}{2T_w(x_{q})}-\frac{1}{2T_w(x_{q+1})}]|v_{q}|^2} d\sigma(v_q,V^{m-q+1}(t_q)).
\end{align*}

We focus on the coefficient of $|v_q|^2$ in~\eqref{eqn: dsigmap}, we derive
\begin{align*}
   &\frac{(T_{l,q+1}-T_w(x_{q+1}))(1-r_{min})}{2T_w(x_{q+1})[T_{l,q+1}(1-r_{min})+r_{min} T_w(x_{q+1})]}|v_q|^2+[\frac{1}{2T_w(x_{q})}-\frac{1}{2T_w(x_{q+1})}]|v_q|^2  \\
   & =       \frac{(T_{l,q+1}-T_w(x_{q+1}))(1-r_{min})-[T_{l,q+1}(1-r_{min})+r_{min} T_w(x_{q+1})]}{2T_w(x_{q+1})[T_{l,q+1}(1-r_{min})+r_{min} T_w(x_{q+1})]}|v_q|^2                             + \frac{|v_q|^2}{2T_w(x_q)}\\
   &=    \frac{ -T_w(x_{q+1})(1-r_{min})-r_{min} T_w(x_{q+1})}{2T_w(x_{q+1})[T_{l,q+1}(1-r_{min})+r_{min} T_w(x_{q+1})]}|v_q|^2                             + \frac{|v_q|^2}{2T_w(x_q)}\\
   &=    \frac{-|v_q|^2}{2[T_{l,q+1}(1-r_{min})+r_{min}T_w(x_{q+1})]}                + \frac{|v_q|^2}{2T_w(x_q)}.
\end{align*}

By the Definition~\ref{Def:Back time cycle}, $x_{q+1}=x_{q+1}(t,x,v,v_1,\cdots,v_{q})$, thus $T_w(x_{q+1})$ depends on $v_{q}$. In order to explicitly compute the integration over $\mathcal{V}_q$, we need to get rid of the dependence of the $T_w(x_{q+1})$ on $v_{q}$, thus we bound
\begin{equation}\label{eqn: Help to integrate x_p over v_p-1}
\begin{split}
   & \exp\left(\frac{-|v_{q}|^2}{2[T_{l,q+1}(1-r_{min})+r_{min} T_w(x_{q+1})]}\right)\leq \exp\left(\frac{-|v_{q}|^2}{2[T_{l,q+1}(1-r_{min})+r_{min}T_M]}\right)= \exp\left(\frac{-|v_{q}|^2}{2T_{l,q}}\right),
\end{split}
\end{equation}
where we use~\eqref{eqn: definition of T_p}.

Expanding $d\sigma(v_q,V^{m-q+1}(t_q))$ by~\eqref{eqn:probability measure} and using~\eqref{eqn: Help to integrate x_p over v_p-1}, we derive
\begin{equation}\label{eqn: int over V_p}
\begin{split}
  & \eqref{eqn: dsigmap}\leq 4(2C_{T_M,\xi})^{2(l-q)}\\
   &  \int_{\mathcal{V}_{q,\perp}} \frac{2}{r_\perp}\frac{|v_{q,\perp}|}{2T_w(x_q)}  e^{-[\frac{1}{2T_{l,q}}-\frac{1}{2T_w(x_q)}-2(2\mathcal{C})^{l-q}(\mathfrak{C}t)]|v_{q,\perp}|^2}I_0\left(\frac{(1-r_\perp)^{1/2}v_{q,\perp}V^{m-q+1}_\perp(t_q)}{T_w(x_q)r_\perp}\right)e^{-\frac{|v_{q,\perp}|^2+(1-r_\perp)|V^{m-q+1}_\perp(t_q)|^2}{2T_w(x_q)r_\perp}}  dv_{q,\perp} \\
    & \times\int_{\mathcal{V}_{q,\parallel}}\frac{1}{\pi r_\parallel(2-r_\parallel)(2T_w(x_q))}e^{-[\frac{1}{2T_{l,q}}-\frac{1}{2T_w(x_q)}-2(2\mathcal{C})^{l-q}(\mathfrak{C}t)]|v_{q,\parallel}|^2}e^{-\frac{1}{2T_w(x_q)}\frac{|v_{q,\parallel}-(1-r_\parallel)V^{m-q+1}_{\parallel}(t_q)|^2}{r_\parallel(2-r_\parallel)}}dv_{q,\parallel}.
\end{split}
\end{equation}
In the third line of~\eqref{eqn: int over V_p}, to apply~\eqref{eqn: coe abc} in Lemma \ref{Lemma: abc}, we set
\[a=-[\frac{1}{2T_{l,q}}-\frac{1}{2T_w(x_{q})}],~ b=\frac{1}{2T_w(x_{q})r_\parallel(2-r_\parallel)},~\e=2(2\mathcal{C})^{l-q}(\mathfrak{C}t),~w=(1-r_\parallel)v_{q-1,\parallel}.\]
Comparing with~\eqref{eqn: coefficient a and b}, we can replace $\frac{2\xi}{\xi+1}T_M$ by $T_{l,q}$ and replace $\mathfrak{C}t$ by $2(2\mathcal{C})^{l-q}(\mathfrak{C}t)$. Then we apply the replacement to~\eqref{eqn: b-a-e} and obtain
\[b-a-\e\geq \frac{1}{2T_{l,q}}-2(2\mathcal{C})^{l-q}(\mathfrak{C}t)\geq \frac{1}{2T_M\frac{2\xi}{\xi+1}}-2(2\mathcal{C})^k (\mathfrak{C}t)\geq \frac{1}{4T_M},\]
where we take $t^*=t^*(T_M,\xi,\mathcal{C},\mathfrak{C},k)$ small enough and $t\leq t^*$. Also we require the $t$ satisfy
\[\frac{\e}{b-a-\e}\leq 4T_M\times 2(2\mathcal{C})^k (\mathfrak{C}t)\leq 2.\]
We conclude the $t^*$ here only depends on the parameters $T_M,\xi,\mathcal{C},\mathfrak{C},k$. Thus by the same computation as~\eqref{eqn: 1 one} we obtain
\[\frac{b}{b-a-\e}\leq \frac{2T_{l,q}}{T_{l,q}+[\min\{T_w(x)\}-T_{l,q}]r_\parallel(2-r_\parallel)}\leq C_{T_M,\xi},\]
where we use $T_{l,q}\leq \frac{2\xi}{\xi+1}T_M$ from~\eqref{eqn: definition of T_p} and \eqref{eqn: def of r}. $C_{T_M,\xi}$ is defined in~\eqref{eqn: 1 one}

By the same computation as~\eqref{eqn: 2 two}, we obtain
\begin{align*}
  \frac{(a+\e)b}{b-a-\e} & = \frac{ab}{b-a}+\frac{ab}{b-a}\frac{\e}{b-a-\e}+\frac{b}{b-a-\e}\e \\
   & \leq  \frac{T_{l,q}-T_w(x_q)}{2T_w(x_q)[T_{l,q}+[T_w(x_q)-T_{l,q}]r_\parallel(2-r_\parallel)]}+(2\mathcal{C})^{l-q+1}(\mathfrak{C}t).
\end{align*}
Here we have used $T_{l,q}\leq \frac{2\xi}{\xi+1}T_M$ and \eqref{eqn: def of r} to obtain
\begin{align*}
\frac{ab}{b-a}\frac{\e}{b-a-\e}+ \frac{b\e}{b-a-\e}   & \leq \frac{4T_M\big(T_{l,q}-\min\{T_w(x)\}\big)}{2\min\{T_w(x)\}[T_{l,q}+[\min\{T_w(x)\}-T_{l,q}]r_\parallel(2-r_\parallel)]}2(2\mathcal{C})^{l-q}(\mathfrak{C}t) \\
   & +\frac{2T_{l,q}}{\frac{2\xi}{\xi+1}T_M+[\min\{T_w(x)\}-T_{l,q}]r_\parallel(2-r_\parallel)}2(2\mathcal{C})^{l-q}(\mathfrak{C}t) \leq (2\mathcal{C})^{l-q+1}(\mathfrak{C}t),
\end{align*}
with $\mathcal{C}$ defined in~\eqref{eqn: cal C}.

Thus by Lemma \ref{Lemma: abc} with $w=(1-r_\parallel)V^{m-q+1}_\parallel(t_{q})$, the third line of~\eqref{eqn: int over V_p} is bounded by
\begin{equation}\label{eqn: Vq para}
\begin{split}
   & C_{T_M,\xi}\exp\left(\big[ \frac{[T_{l,q}-T_w(x_{q})]}{2T_w(x_{q})[T_{l,q}(1-r_\parallel)^2+r(2-r_\parallel) T_w(x_{q})]} + (2\mathcal{C})^{l-q+1}(\mathfrak{C}t)\big]|(1-r_\parallel)V^{m-q+1}_\parallel(t_q)|^2\right)
 \\
    & \leq C_{T_M,\xi}\exp\left(\big[ \frac{[T_{l,q}-T_w(x_{q})][1-r_{min}]}{2T_w(x_{q})[T_{l,q}(1-r_{min})+r_{min} T_w(x_{q})]} + (2\mathcal{C})^{l-q+1}(\mathfrak{C}t)\big]|V^{m-q+1}_\parallel(t_q)|^2\right).
\end{split}
\end{equation}
By the same computation the second line of~\eqref{eqn: int over V_p} is bounded by
\begin{equation}\label{eqn: Vq perp}
C_{T_M,\xi}\exp\left(\big[ \frac{[T_{l,q}-T_w(x_{q})][1-r_{min}]}{2T_w(x_{q})[T_{l,q}(1-r_{min})+r_{min} T_w(x_{q})]} + (2\mathcal{C})^{l-q+1}(\mathfrak{C}t)\big]|V^{m-q+1}_\perp(t_q)|^2\right).
\end{equation}
By~\eqref{eqn: Vq para} and~\eqref{eqn: Vq perp}, we derive that
\[\eqref{eqn: int over V_p}\leq (2C_{T_M,\xi})^{2(l-q+1)}\exp\left(\big[ \frac{[T_{l,q}-T_w(x_{q})][1-r_{min}]}{2T_w(x_{q})[T_{l,q}(1-r_{min})+r_{min} T_w(x_{q})]} + (2\mathcal{C})^{l-q+1}(\mathfrak{C}t)\big]|V^{m-q+1}(t_q)|^2\right)\]
\[=(2C_{T_M,\xi})^{2(l-q+1)}\mathcal{A}_{l,q},\]
which is consistent with~\eqref{eqn: boundedness for l-p+1 fold integration} with $p=q$. The induction is valid we derive~\eqref{eqn: boundedness for l-p+1 fold integration}.

Now we focus on~\eqref{eqn: structure}. The first inequality in~\eqref{eqn: structure} follows directly from~\eqref{eqn: boundedness for l-p+1 fold integration} and~\eqref{eqn: ppt for Phi}. For the second inequality, by~\eqref{eqn: upsilon} we have
\begin{equation}\label{eqn: second ineq}
\begin{split}
   & (2C_{T_M,\xi})^{2(l-p'+1)} \mathcal{A}_{l,p'}\int_{\prod_{j=p}^{p'-1} \mathcal{V}_j}  \mathbf{1}_{\{t_l>0\}} d\Upsilon_p^{p'-1}\\
   &= (2C_{T_M,\xi})^{2(l-p'+1)} \mathcal{A}_{l,p'} \int_{\prod_{j=p}^{p'-2} \mathcal{V}_j}   d\Upsilon_p^{p'-2}\\
    &  \int_{\mathcal{V}_{p'-1}} \mathbf{1}_{\{t_l>0\}}   2 e^{\mathfrak{C}(t_{p'-1}-t_{p'})\langle v_{p'-1}\rangle^2} e^{[\frac{1}{2T_w(x_{p'-1})}-\frac{1}{2T_w(x_{p'})}]|v_{p'-1}|^2} d\sigma(v_{p'-1},V^{m-p'+2}(t_{p'-1})).
\end{split}
\end{equation}

In the proof for~\eqref{eqn: boundedness for l-p+1 fold integration} we have
\[\eqref{eqn: dsigmaq}\leq \eqref{eqn: dsigmap}\leq \eqref{eqn: int over V_p} \leq (C_{T_M,\xi})^{2(l-q+1)}\mathcal{A}_{l,q}.\]
Then by replacing $q$ by $p'-1$ in the estimate $~\eqref{eqn: dsigmaq}\leq (2C_{T_M,\xi})^{2(l-q+1)}\mathcal{A}_{l,q}$ we have

\[(2C_{T_M,\xi})^{2(l-p'+1)} \mathcal{A}_{l,p'}\int_{\prod_{j=p}^{p'-1} \mathcal{V}_j}  \mathbf{1}_{\{t_l>0\}} d\Upsilon_p^{p'-1}=\eqref{eqn: second ineq}\leq (2C_{T_M,\xi})^{2(l-p'+2)}\mathcal{A}_{l,p'-1}\int_{\prod_{j=p}^{p'-2} \mathcal{V}_j} \mathbf{1}_{\{t_l>0\}}  d\Upsilon_p^{p'-2}   .\]
Keep doing this computation until integrating over $\mathcal{V}_p$ we obtain the second inequality in~\eqref{eqn: structure}.

\end{proof}

In the following lemma, we prepare for showing the smallness of the last term in~\eqref{eqn: formula for H}.
\begin{lemma}\label{lemma: t^k}
Assume the constraint for $T_w$ holds true (\eqref{eqn: Constrain on T}). There exists
\begin{equation}\label{eqn: k_0 dependence}
k_0=k_0(\Omega,C_{T_M,\xi},\mathcal{C},T_M,r_\perp,r_\parallel,\min\{T_w(x)\},\xi,\mathfrak{C})\gg 1,
\end{equation}
\begin{equation}\label{eqn: t'}
 t'=t'(k_0,\xi,T_M,\min\{T_w(x)\},\mathcal{C},r_\perp,r_\parallel,\mathfrak{C},C_{\phi^m})\leq t^*\ll 1
\end{equation}
such that for all $t\in [0,t']$, we have
\begin{equation}\label{eqn: 1/2 decay}
   \int_{\prod_{j=1}^{k_0-1}\mathcal{V}_j} \mathbf{1}_{\{t_{k_0}>0\}}d\Sigma_{k_0-1,m}^{k_0}(t_{k_0})  \leq (\frac{1}{2})^{k_0} \mathcal{A}_{k_0-1,1},
\end{equation}
where $\mathcal{A}_{k_0-1,1}$ is defined in~\eqref{eqn: Elp} and $t^*$ is defined in~\eqref{eqn: t*}.
\end{lemma}

\begin{remark}
We comment that the main difference between this lemma and Lemma \ref{lemma: boundedness} is that here a decaying factor $(\frac{1}{2})^{k_0}$ is needed. This lemma implies for $k=k_0$ large enough, the last term of~\eqref{eqn: formula for H} is negligible.
\end{remark}

The main idea to prove Lemma \ref{lemma: t^k} is to use the decomposition~\eqref{eqn: decom} for the integral domain. In Lemma \ref{Lemma: (2)}-\ref{Lemma: Step3} we use Lemma \ref{lemma: boundedness} to show that such decomposition indeed make contribution in obtaining the smallness. Among them Lemma \ref{Lemma: Step3} is the most important one as it summarizes all estimates in Lemma \ref{Lemma: (2)}-\ref{Lemma: accumulate} and directly provides the decaying factor for the $k$-fold integral. Echoing the difficulties for obtaining $L_\infty$ bound as discussed in Section 1.2, where we proposed splitting $\gamma_+$ into $\gamma_+^\delta$ and the remainders, in Lemma~\ref{Lemma: Step3}, we detail such splitting and the trajectories' behavior in these sets.

After proving Lemma \ref{Lemma: (2)}-\ref{Lemma: Step3} as preparation, we will conclude the proof for Lemma \ref{lemma: t^k}.

\begin{lemma}\label{Lemma: (2)}
Recall~\eqref{eqn: boundedness for l-p+1 fold integration} in Lemma \ref{lemma: boundedness}.

For $1\leq i\leq k-1$, if
\begin{equation}\label{eqn: 2 condition}
|v_i\cdot n(x_i)|<\delta,
\end{equation}
then
\begin{equation}\label{eqn: 2}
    \int_{\prod_{j=i}^{k-1} \mathcal{V}_j}  \mathbf{1}_{\{t_k>0\}}   \mathbf{1}_{\{v_i\in \mathcal{V}_i:|v_i\cdot n(x_i)|<\delta\}}   d\Phi_{i,m}^{k,k-1}(t_k) \leq \delta    (2C_{T_M,\xi})^{2(k-i)}\mathcal{A}_{k-1,i}.
\end{equation}

If
\begin{equation}\label{eqn: b condition}
|v_{i,\parallel}-\eta_{i,\parallel}V^{m-i+1}_\parallel(t_i)|>\delta^{-1},
\end{equation}
then
\begin{equation}\label{eqn: case b}
\int_{\prod_{j=i}^{k-1} \mathcal{V}_j}\mathbf{1}_{\{t_k>0\}}\mathbf{1}_{\{|v_{i,\parallel}-\eta_{i,\parallel}V^{m-i+1}_\parallel(t_i)|>\delta^{-1}\}}d\Phi_{i,m}^{k,k-1}(t_k)
 \leq \delta    (2C_{T_M,\xi})^{2(k-i)}\mathcal{A}_{k-1,i}.
\end{equation}
Here $\eta_{i,\parallel}$ is a constant defined in~\eqref{eqn: eta i para}.

If
\begin{equation}\label{eqn: d condition}
|v_{i,\perp}-\eta_{i,\perp}V^{m-i+1}_\perp(t_i)|>\delta^{-1},
\end{equation}
then
\begin{equation}\label{eqn: case d}
\int_{\prod_{j=i}^{k-1} \mathcal{V}_j}\mathbf{1}_{\{t_k>0\}}\mathbf{1}_{\{|v_{i,\perp}-\eta_{i,\perp}V^{m-i+1}_\perp(t_i)|>\delta^{-1}\}}d\Phi_{i,m}^{k,k-1}(t_k)
 \leq \delta    (2C_{T_M,\xi})^{2(k-i)}\mathcal{A}_{k-1,i}.
\end{equation}
Here $\eta_{i,\perp}$ is a constant defined in~\eqref{eqn: eta i perp}.

\end{lemma}

\begin{proof}
First we focus on~\eqref{eqn: 2}. By~\eqref{eqn: int over V_p} in Lemma \ref{lemma: boundedness}, we can replace $l$ by $k-1$ and replace $q$ by $i$ to obtain
\begin{equation}\label{eqn: int V_i}
\begin{split}
  & \int_{\prod_{j=i}^{k-1} \mathcal{V}_j}    \mathbf{1}_{\{t_k>0\}} d\Phi_{i,m}^{k,k-1}(t_k)\leq 4(2C_{T_M,\xi})^{2(k-1-i)} \\
   &  \int_{\mathcal{V}_{i,\perp}} \frac{2}{r_\perp}\frac{|v_{i,\perp}|}{2T_w(x_i)}  e^{-[\frac{1}{2T_{k-1,i}}-\frac{1}{2T_w(x_i)}-2(2\mathcal{C})^{k-1-i}(\mathfrak{C}t)]|v_{i,\perp}|^2}I_0\left(\frac{(1-r_\perp)^{1/2}v_{i,\perp}V^{m-i+1}_\perp(t_i)}{T_w(x_i)r_\perp}\right)e^{\frac{|v_{i,\perp}|^2+(1-r_\perp)|V^{m-i+1}_\perp(t_i)|^2}{2T_w(x)r_\perp}}  dv_{i,\perp} \\
    & \times\int_{\mathcal{V}_{i,\parallel}}\frac{1}{\pi r_\parallel(2-r_\parallel)(2T_w(x_i))}e^{-[\frac{1}{2T_{k-1,i}}-\frac{1}{2T_w(x_i)}-2(2\mathcal{C})^{k-1-i}(\mathfrak{C}t)]|v_{i,\parallel}|^2}e^{-\frac{1}{2T_w(x_i)}\frac{|v_{i,\parallel}-(1-r_\parallel)V^{m-i+1}_\parallel(t_i)|^2}{r_\parallel(2-r_\parallel)}}dv_{i,\parallel}.
\end{split}
\end{equation}
Under the condition~\eqref{eqn: 2 condition}, we consider the second line of~\eqref{eqn: int V_i} with integrating over $\{v_{i,\perp}\in \mathcal{V}_{i,\perp}:|v_i\cdot n(x_i)|<\frac{1-\eta}{2(1+\eta)}\delta\}$. To apply~\eqref{eqn: coe abc perp small} in Lemma \ref{Lemma: perp abc}, we set
\[a=-[\frac{1}{2T_{k-1,i}}-\frac{1}{2T_w(x_{i})}],~ b=\frac{1}{2T_w(x_{i})r_\perp},~\e=(2\mathcal{C})^{k-1-i}(\mathfrak{C}t),~w=\sqrt{1-r_\perp}V^{m-i+1}_\perp(t_i).\]
Under the condition $|v_i\cdot n(x_i)|<\frac{1-\eta}{2(1+\eta)}\delta$, applying~\eqref{eqn: coe abc perp small} in Lemma \ref{Lemma: perp abc} and using~\eqref{eqn: Vq perp} with $q=i,l=k-1$, we bound the second line of~\eqref{eqn: int V_i} by
\begin{equation}\label{eqn: perp small}
\delta C_{T_M,\xi}\exp\left(\big[ \frac{[T_{k-1,i}-T_w(x_{i})][1-r_{min}]}{2T_w(x_{i})[T_{k-1,i}(1-r_{min})+r_{min} T_w(x_{i})]} + (2\mathcal{C})^{k-i}(\mathfrak{C}t)\big]|V^{m-i+1}_\perp(t_i)|^2\right).
\end{equation}
Comparing with~\eqref{eqn: Vq perp}, we conclude the second line of~\eqref{eqn: int V_i} provides one more constant term $\delta$. The third line of~\eqref{eqn: int V_i} is bounded by~\eqref{eqn: Vq para} with $q=i,l=k-1$. Therefore, we derive~\eqref{eqn: 2}.

Then we focus on~\eqref{eqn: case b}. We consider the third line of~\eqref{eqn: int V_i}. To apply~\eqref{eqn: coe abc small} in Lemma \ref{Lemma: abc}, we set
\begin{equation}\label{eqn: abe}
     a=-\frac{1}{2T_{k-1,i}} +\frac{1}{2T_w(x_i)},\quad b=\frac{1}{2T_w(x_i)r_\parallel(2-r_\parallel)},\quad \e=2(2\mathcal{C})^{k-1-i}(\mathfrak{C}t),~w=(1-r_\parallel)V^{m-i+1}_\parallel(t_i).
\end{equation}
We define
\begin{equation}\label{eqn: B i para def}
B_{i,\parallel}:=b-a-\e.
\end{equation}
In regard to~\eqref{eqn: coe abc small},
\[    \frac{b}{b-a-\e}w=\frac{b}{b-a}[1+\frac{\e}{b-a-\e}] w. \]
By~\eqref{eqn: abe},
\[\frac{b}{b-a}=\frac{T_{k-1,i}}{T_{k-1,i}(1-r_\parallel)^2+T_w(x_i)r_\parallel(2-r_\parallel)},\quad \frac{\e}{b-a-\e}=\frac{2(2\mathcal{C})^{k-1-i}(\mathfrak{C}t)}{B_{i,\parallel}}.\]
Thus we obtain
\begin{equation}\label{eqn: constant for the t^k}
    \frac{b}{b-a-\e}w=\eta_{i,\parallel}V^{m-i+1}_\parallel(t_i),
\end{equation}
where we define
\begin{equation}\label{eqn: eta i para}
\eta_{i,\parallel}:=\frac{T_{k-1,i}[1+2(2\mathcal{C})^{k-1-i}(\mathfrak{C}t)/B_{i,\parallel}]}{T_{k-1,i}(1-r_\parallel)^2+T_w(x_i)r_\parallel(2-r_\parallel)}(1-r_\parallel).
\end{equation}
Thus under the condition~\eqref{eqn: b condition}, applying~\eqref{eqn: coe abc small} in Lemma \ref{eqn: coe abc} with $\frac{b}{b-a-\e}w=\eta_{i,\parallel}V^{m-i+1}_\parallel(t_i)$ and using~\eqref{eqn: Vq para} with $q=i,l=k-1$, we bound the third line of~\eqref{eqn: int V_i} by
\[\delta C_{T_M,\xi}\exp\left(\big[ \frac{[T_{k-1,i}-T_w(x_{i})][1-r_{min}]}{2T_w(x_{i})[T_{k-1,i}(1-r_{min})+r_{min} T_w(x_{i})]} + 2(2\mathcal{C})^{k-1-i}(\mathfrak{C}t)\big]|V^{m-i+1}_\parallel(t_i)|^2\right).\]
By the same computation in Lemma \ref{Lemma: (2)}, we derive~\eqref{eqn: case b} because of the extra constant $\delta$.

Last we focus on~\eqref{eqn: case d}. We consider the second line of~\eqref{eqn: int V_i}. To apply~\eqref{eqn: coe abc perp small} in Lemma \ref{Lemma: integrate normal small}, we set
\begin{equation}\label{eqn: abe perp}
     a=-\frac{1}{2T_{k-1,i}} +\frac{1}{2T_w(x_i)},\quad b=\frac{1}{2T_w(x_i)r_\perp},\quad \e=2(2\mathcal{C})^{k-1-i}(\mathfrak{C}t),~ w=\sqrt{1-r_\perp}V^{m-i+1}_\perp(t_i).
\end{equation}
Define
\begin{equation}\label{eqn: B i perp}
B_{i,\perp}:=b-a-\e.
\end{equation}
By the same computation as~\eqref{eqn: constant for the t^k},
\[\frac{b}{b-a-\e}w=\eta_{i,\perp}V^{m-i+1}_\perp(t_i),\]
where we define
\begin{equation}\label{eqn: eta i perp}
\eta_{i,\perp}:= \frac{T_{k-1,i}[1+\frac{2(2\mathcal{C})^{k-1-i}(\mathfrak{C}t)}{B_{i,\perp}}]}{T_{k-1,i}(1-r_\perp)+T_w(x_i)r_\perp}\sqrt{1-r_\perp}.
\end{equation}
Thus under the condition~\eqref{eqn: d condition}, applying~\eqref{eqn: coe perp small 2} in Lemma~\ref{Lemma: integrate normal small} with $\frac{b}{b-a-\e}w=\eta_{i,\perp}V^{m-i+1}_\perp(t_i)$ and using~\eqref{eqn: Vq perp} with $q=i,l=k-1$, we bound the second line of~\eqref{eqn: int V_i} by
\[\delta C_{T_M,\xi}\exp\left(\big[ \frac{[T_{k-1,i}-T_w(x_{i})][1-r_{min}]}{2T_w(x_{i})[T_{k-1,i}(1-r_{min})+r_{min} T_w(x_{i})]} + 2(2\mathcal{C})^{k-i}(\mathfrak{C}t)\big]|V^{m-i+1})_\perp(t_i)|^2\right).\]
Then we derive~\eqref{eqn: case b} because of the extra constant $\delta$.

\end{proof}

\begin{lemma}\label{Lemma:  (a)(c)}
For $\eta_{i,\parallel}$ and $\eta_{i,\perp}$ defined in Lemma \ref{Lemma: (2)}, suppose there exists $\eta<1$ such that
\begin{equation}\label{eqn: eta condition}
  \max\{\eta_{i,\parallel},\eta_{i,\perp}\}<\eta<1.
\end{equation}
If
\begin{equation}\label{eqn: (a) condition}
  |v_{i,\parallel}|>\frac{1+\eta}{1-\eta}\delta^{-1} \text{ and } |v_{i,\parallel}-\eta_{i,\parallel}V^{m-i+1}_\parallel(t_i)|<\delta^{-1},
\end{equation}
then we have
\begin{equation}\label{eqn: (a)}
  |V^{m-i+1}_\parallel(t_i)|>|v_{i,\parallel}|+\delta^{-1}.
\end{equation}

Also if
\begin{equation}\label{eqn: (c) condition}
  |v_{i,\perp}|>\frac{1+\eta}{1-\eta}\delta^{-1} \text{ and } |v_{i,\perp}-\eta_{i,\perp}V^{m-i+1}_\perp(t_i)|<\delta^{-1},
\end{equation}
then we have
\begin{equation}\label{eqn: (c)}
  |V^{m-i+1}_\perp(t_i)|>|v_{i,\perp}|+\delta^{-1}.
\end{equation}

\end{lemma}
\begin{remark}
Lemma \ref{Lemma: (2)} includes all the cases that are controllable since they provides the small number $\delta$, which direct contributes in obtaining the exponential decay in~\eqref{eqn: 1/2 decay}. This lemma discuss the rest cases that does not directly provide the smallness, which are the main difficulty.
\end{remark}

\begin{proof}
Under the condition~\eqref{eqn: (a) condition} we have
\[\eta_{i,\parallel}|V^{m-i+1}_\parallel(t_i)|> |v_{i,\parallel}|-\delta^{-1}.\]
Thus we derive
\begin{align*}
 |V^{m-i+1}_\parallel(t_i)|  &>|v_{i,\parallel}|+\frac{1-\eta_{i,\parallel}}{\eta_{i,\parallel}}|v_{i,\parallel}|-\frac{1}{\eta_{i,\parallel}}\delta^{-1}  \\
   & >|v_{i,\parallel}|+\frac{1-\eta_{i,\parallel}}{\eta_{i,\parallel}}\frac{1+\eta}{1-\eta}\delta^{-1}-\frac{1}{\eta_{i,\parallel}}\delta^{-1}
\\
& >|v_{i,\parallel}|+\frac{1-\eta_{i,\parallel}}{\eta_{i,\parallel}}\frac{1+\eta_{i,\parallel}}{1-\eta_{i,\parallel}}\delta^{-1}-\frac{1}{\eta_{i,\parallel}}\delta^{-1}
\\
&>|v_{i,\parallel}|+\frac{1+\eta_{i,\parallel}}{\eta_{i,\parallel}}\delta^{-1}-\frac{1}{\eta_{i,\parallel}}\delta^{-1}>|v_{i,\parallel}|+\delta^{-1},
\end{align*}
where we use $|v_{i,\parallel}|>\frac{1+\eta}{1-\eta}\delta^{-1}$ in the second line and $1>\eta\geq \eta_{i,\parallel}$ in the third line. Then we obtain~\eqref{eqn: (a)}.

Under the condition~\eqref{eqn: (c) condition}, we apply the same computation above to obtain~\eqref{eqn: (c)}.

\end{proof}

\begin{lemma}\label{Lemma: accumulate}
Suppose there are $n$ number of $v_j$ such that
\begin{equation}\label{eqn: satisfy condition}
|v_{j,\parallel}-\eta_{j,\parallel}V^{m-j+1}_\perp(t_j)|\geq \delta^{-1},
\end{equation}
and also suppose the index $j$ in these $v_j$ are $i_1<i_2<\cdots<i_n$, then
\begin{equation}\label{eqn: claim M}
\int_{\prod_{j={i_1}}^{k-1} \mathcal{V}_j}\mathbf{1}_{\{t_k>0\}}\mathbf{1}_{\{\text{~\eqref{eqn: satisfy condition} holds for $j=i_1,i_2,\cdots, i_n$}\}}d\Phi_{i_1,m}^{k,k-1}(t_k) \leq  (\delta)^{n}   (2C_{T_M,\xi})^{2(k-i_1)}\mathcal{A}_{k-1,i_1}.
\end{equation}

\end{lemma}

\begin{proof}
By~\eqref{eqn: structure} in Lemma 2 with $l=k-1$, $p=i_1$, $p'=i_n$ and using~\eqref{eqn: case b} with $i=i_n$, we have
\begin{equation}\label{eqn: split iM}
\begin{split}
   &  \int_{\prod_{j=i_1}^{k-1} \mathcal{V}_j}\mathbf{1}_{\{t_k>0\}}\mathbf{1}_{\{\text{~\eqref{eqn: satisfy condition} holds for $j=i_1,\cdots,i_n$}\}}d\Phi_{i_1,m}^{k,k-1}(t_k)\\
    & \leq \delta (2C_{T_M,\xi})^{2(k-i_n)} \mathcal{A}_{k-1,i_n}\int_{\prod_{j=i_1}^{i_n-1} \mathcal{V}_j} \mathbf{1}_{\{t_k>0\}}\mathbf{1}_{\{\text{~\eqref{eqn: satisfy condition} holds for $j=i_1,\cdots,i_{n-1}$}\}}        d\Upsilon_{i_1}^{i_n-1}\\
    &=\delta (2C_{T_M,\xi})^{2(k-i_n)} \mathcal{A}_{k-1,i_n}\int_{\prod_{j=i_1}^{i_{n-1}-1}\mathcal{V}_j}\int_{\prod_{j=i_{n-1}}^{(i_n)-1} \mathcal{V}_j} \mathbf{1}_{\{t_k>0\}}\mathbf{1}_{\{\text{~\eqref{eqn: satisfy condition} holds for $j=i_1,\cdots,i_{n-1}$}\}}      d\Upsilon_{i_{n-1}}^{(i_{n})-1}d\Upsilon_{i_1}^{i_{n-1}-1}.
\end{split}
\end{equation}
Again by~\eqref{eqn: structure} and~\eqref{eqn: case b} with $i=i_{n-1}$ we have
\[\eqref{eqn: split iM}\leq \delta^2 (C_{T_M,\xi})^{2(k-i_{n-1})}\mathcal{A}_{k-1,i_{n-1}}\int_{\prod_{j=i_1}^{i_{n-1}-1} \mathcal{V}_j} \mathbf{1}_{\{t_k>0\}}\mathbf{1}_{\{\text{~\eqref{eqn: satisfy condition} holds for $j=i_1,\cdots,i_{n-2}$}\}}        d\Upsilon_{i_1}^{i_{n-1}-1}.\]

Keep doing this computation until integrating over $\mathcal{V}_{i_1}$ we derive~\eqref{eqn: claim M}.

\end{proof}

\begin{lemma}\label{Lemma: Step3}
Assume $t\leq t_*$( so that we can apply Lemma \ref{lemma: boundedness} ) satisfies~\eqref{eqn: t less than 1} and~\eqref{eqn: k_1}. For $0<\delta\ll 1$, we define
\begin{equation}\label{eqn: decom}
  \mathcal{V}_{j}^{\delta}:=\{v_j\in \mathcal{V}_j:|v_j\cdot n(x_j)|>\delta,|v_j|\leq \delta^{-1}\}.
\end{equation}
For the sequence $\{v_1,v_2,\cdots,v_{k-1}\}$, consider a subsequence  $\{v_{l+1},v_{l+2},\cdots,v_{l+L}\}$ with $l+1<l+L\leq k-1$ as following:
\begin{equation}\label{eqn: sequence}
  \underbrace{v_{l}}_{\in \mathcal{V}_l^{\frac{1-\eta}{2(1+\eta)}\delta}},\underbrace{v_{l+1},v_{l+2}\cdots v_{l+L}}_{\text{all}\in \mathcal{V}_{l+j}\backslash \mathcal{V}_{l+j}^{\frac{1-\eta}{2(1+\eta)}\delta}},\quad \quad\underbrace{v_{l+L+1}}_{\in \mathcal{V}_{l+L+1}^{\frac{1-\eta}{2(1+\eta)}\delta}}.
\end{equation}
In~\eqref{eqn: sequence}, if $L\geq 100\frac{1+\eta}{1-\eta}$, then we have
\begin{equation}\label{eqn: Step3}
\int_{\prod_{j={l}}^{k-1} \mathcal{V}_j}\mathbf{1}_{\{t_k>0\}}\mathbf{1}_{\{v_{l+j}\in \mathcal{V}_{l+j}\backslash \mathcal{V}_{l+j}^{\frac{1-\eta}{2(1+\eta)}\delta} \text{ for } 1\leq j\leq L\}}d\Phi_{l,m}^{k,k-1}(t_k) \leq  (3\delta)^{L/2}   (2C_{T_M,\xi})^{2(k-l)}\mathcal{A}_{k-1,l}.
\end{equation}
Here the $\eta$ satisfies the condition~\eqref{eqn: eta condition}.

\end{lemma}
\begin{remark}
In order to apply Lemma \ref{Lemma:  (a)(c)} we need to create the condition~\eqref{eqn: (a) condition} and~\eqref{eqn: (c) condition}. This is the main reason that we consider the space $\mathcal{V}_{j}^{\frac{1-\eta}{2(1+\eta)}\delta}$.

This lemma asserts that implies that when $L$ is large enough, such subsequence~\eqref{eqn: sequence}, without further considering the constraint for $|v_{i,\parallel}-\eta_{i,\parallel}V_{\parallel}^{m-i+1}(t_i)|$ for $l+1\leq i\leq l+L$ as~\eqref{eqn: (a) condition},\eqref{eqn: b condition}, provides a decay factor $(3\delta)^{L/2}$. Such decay factor is the key the obtain the decay factor $(\frac{1}{2})^{k_0}$ in Lemma~\ref{lemma: t^k}. In fact in the proof we consider all possible cases for each $v_{i,\parallel}$ in the subsequence~\eqref{eqn: sequence} and apply the estimates in Lemma \ref{Lemma: (2)}-\ref{Lemma: accumulate} to obtain the decay factor $(3\delta)^{L/2}$ for all cases. We will heavily rely on this lemma to prove Lemma \ref{lemma: t^k}.

\end{remark}

\begin{proof}
By the definition~\eqref{eqn: decom} we have
\[\mathcal{V}_{i}\backslash \mathcal{V}_{i}^{\frac{1-\eta}{2(1+\eta)}\delta}  =\{v_i\in \mathcal{V}_i:|v_i\cdot n(x_i)|<\frac{1-\eta}{2(1+\eta)}\delta \text{ or }|v_i|\geq \frac{2(1+\eta)}{1-\eta}\delta^{-1}\}.\]
Here we summarize the result of Lemma \ref{Lemma: (2)} and Lemma \ref{Lemma:  (a)(c)}.
With $\frac{1-\eta}{1+\eta}\delta<\delta$, when $v_i\in \mathcal{V}_i\backslash \mathcal{V}_i^{\frac{1-\eta}{2(1+\eta)}\delta}$
\begin{enumerate}

  \item When $|v_i\cdot n(x_i)|<\frac{1-\eta}{2(1+\eta)}\delta$, we have~\eqref{eqn: 2}.

  \item When $|v_{i}|>\frac{2(1+\eta)}{1-\eta}\delta^{-1}$,
   \begin{enumerate}
     \item when $|v_{i,\parallel}|>\frac{1+\eta}{1-\eta}\delta^{-1}$, if $|v_{i,\parallel}-\eta_{i,\parallel}V^{m-i+1}_\parallel(t_i)|<\delta^{-1}$, then $|V^{m-i+1}_\parallel(t_i)|>|v_{i,\parallel}|+\delta^{-1}$. \\
     \item when $|v_{i,\parallel}|>\frac{1+\eta}{1-\eta}\delta^{-1}$, if $|v_{i,\parallel}-\eta_{i,\parallel}V^{m-i+1}_\parallel(t_i)|\geq \delta^{-1}$, then we have~\eqref{eqn: case b}. \\
     \item when $|v_{i,\perp}|>\frac{1+\eta}{1-\eta}\delta^{-1}$, if $|v_{i,\perp}-\eta_{i,\perp}V^{m-i+1}_\perp(t_i)|<\delta^{-1}$, then $|V^{m-i+1}_\perp(t_i)|>|v_{i,\perp}|+\delta^{-1}$ .\\
     \item when $|v_{i,\perp}|>\frac{1+\eta}{1-\eta}\delta^{-1}$, if $|v_{i,\perp}-\eta_{i,\perp}V^{m-i+1}_\perp(t_i)|\geq \delta^{-1}$, then we have~\eqref{eqn: case d}.\\
   \end{enumerate}

\end{enumerate}

We define $\mathcal{W}_{i,\delta}$ as the space that provides the smallness:
\begin{align*}
\mathcal{W}_{i,\delta}   & :=\{v_i\in \mathcal{V}_i:|v_{i,\perp}|<\frac{1-\eta}{2(1+\eta)}\delta\}\bigcup \{v_i\in \mathcal{V}_i:|v_{i,\perp}|>\frac{1+\eta}{1-\eta}\delta^{-1}\text{ and }|v_{i,\perp}-\eta_{i,\perp}V^{m-i+1}_\perp(t_i)|>\delta^{-1}\} \\
   & \bigcup \{v_i\in \mathcal{V}_i:|v_{i,\parallel}|>\frac{1+\eta}{1-\eta}\delta^{-1}\text{ and }|v_{i,\parallel}-\eta_{i,\parallel}V^{m-i+1}_\parallel(t_i)|>\delta^{-1}\}.
\end{align*}
Then we have
\begin{equation}\label{eqn: subset}
  \begin{split}
 \mathcal{V}_{i}\backslash \mathcal{V}_{i}^{\frac{1-\eta}{2(1+\eta)}\delta} \subset     & \mathcal{W}_{i,\delta} \bigcup \{v_{i,\perp}\in \mathcal{V}_{i,\perp}:|v_{i,\perp}|>\frac{1+\eta}{1-\eta}\delta^{-1}~\text{and}~|v_{i,\perp}-\eta_{i,\perp}V^{m-i+1}_\perp(t_i)|<\delta^{-1}\}
 \\
      & \bigcup \{v_{i,\parallel}\in \mathcal{V}_{i,\parallel}:|v_{i,\parallel}|>\frac{1+\eta}{1-\eta}\delta^{-1}~\text{and}~|v_{i,\parallel}-\eta_{i,\parallel}V^{m-i+1}_\parallel(t_i)|<\delta^{-1}\}.
  \end{split}
\end{equation}

By~\eqref{eqn: 2},~\eqref{eqn: case b} and~\eqref{eqn: case d} with $\frac{1-\eta}{1+\eta}\delta<\delta$, we obtain
\begin{equation}\label{eqn: V_i,delta}
  \int_{\prod_{j=i}^{k-1}\mathcal{V}_j}  \mathbf{1}_{\{v_i\in \mathcal{W}_{i,\delta}\}}   \mathbf{1}_{\{t_k>0\}}  d\Phi_{i,m}^{k,k-1}(t_k) \leq 3\delta    (2C_{T_M,\xi})^{2(k-i)}\mathcal{A}_{k-1,i}.
\end{equation}

For the subsequence $\{v_{l+1},\cdots,v_{l+L}\}$ in~\eqref{eqn: sequence}, when the number of $v_j\in \mathcal{W}_{j,\delta}$ is larger than $L/2$, by~\eqref{eqn: claim M} in Lemma \ref{Lemma: accumulate} with $n=L/2$ and replacing the condition~\eqref{eqn: satisfy condition} by $v_j\in \mathcal{W}_{j,\delta}$, we obtain
\begin{equation}\label{eqn: 3delta}
\begin{split}
   & \int_{\prod_{j=l}^{k-1} \mathcal{V}_j}   \mathbf{1}_{\{\text{Number of }v_j\in \mathcal{W}_{j,\delta} \text{ is large than }L/2\}}    \mathbf{1}_{\{t_k>0\}} d\Phi_{l,m}^{k,k-1}(t_k) \\
  & \leq (3\delta)^{L/2}    (2C_{T_M,\xi})^{2(k-l_i)}\mathcal{A}_{k-1,l}.
\end{split}
\end{equation}
This finish the discussion with the cases (1),(2b),(2d). Then we focus on the cases (2a),(2c).

When the number of $v_j \notin \mathcal{W}_{j,\delta}$ is larger than $L/2$, by~\eqref{eqn: subset} we further consider two cases. The first case is that the number of $v_j\in \{v_j:|v_{j,\parallel}|>\frac{1+\eta}{1-\eta}\delta^{-1}~\text{and}~|v_{j,\parallel}-\eta_{j,\parallel}V^{m-j+1}_\parallel(t_j)|<\delta^{-1}\}$ is larger than $L/4$. According to the relation of $v_{j,\parallel}$ and $V^{m-j+1}_\parallel(t_j)$, we categorize them into
\begin{description}
  \item[Set1] $\{v_j\notin \mathcal{W}_{j,\delta}:|v_{j,\parallel}|>\frac{1+\eta}{1-\eta}\delta^{-1}~\text{and}~|v_{j,\parallel}-\eta_{j,\parallel}V^{m-j+1}_\parallel(t_j)|<\delta^{-1}\}$.
\end{description}
Denote $M=|\text{Set1}|$ and the corresponding index in Set1 as $j=p_1,p_2,\cdots,p_{M}$. Then we have
\begin{equation}\label{eqn: Mi'}
  L/4\leq M\leq L.
\end{equation}
By~\eqref{eqn: (a)} in Lemma \ref{Lemma:  (a)(c)}, for those $v_{p_j}$, we have
\begin{equation}\label{eqn: increase large}
|v_{p_j,\parallel}|-|V^{m-p_j+1}_\parallel(t_{p_j})|<-\delta^{-1}.
\end{equation}

\begin{description}
  \item[Set2]$\{v_j\in \mathcal{V}_j\backslash \mathcal{V}_j^{\frac{1-\eta}{2(1+\eta)\delta}}:|v_{j,\parallel}|\geq |V^{m-j+1}_\parallel(t_j)|\}$.
\end{description}

Denote $\mathcal{M}=|\text{Set2}|$ and the corresponding index in Set2 as $j=q_1,q_2,\cdots,q_{\mathcal{M}}$. By~\eqref{eqn: Mi'} we have
\begin{equation}\label{eqn: mathcal M}
1\leq \mathcal{M}\leq L-M\leq \frac{3}{4}L.
\end{equation}
Then for those $v_{q_j}$ we define
\begin{equation}\label{eqn: ai def}
a_j:=|v_{q_j,\parallel}|-|V^{m-q_j+1}_\parallel(t_{q_j})|>0.
\end{equation}

 \begin{description}

  \item[Set3] $\{v_j\in \mathcal{V}_j\backslash \mathcal{V}_j^{\frac{1-\eta}{2(1+\eta)\delta}}:|v_{j,\parallel}|\leq |V^{m-j+1}_\parallel(t_j)|\leq |v_{j,\parallel}|+\delta^{-1}\}$.
\end{description}

Denote $N=|\text{Set3}|$ and the corresponding index in Set3 as $j=o_1,o_2,\cdots,o_N$. Then for those $o_j$, we have
\begin{equation}\label{eqn: increase small}
|v_{o_j,\parallel}|\leq |V^{m-o_j+1}_\parallel(t_{o_j})|\leq |v_{o_j,\parallel}|+\delta^{-1}.
\end{equation}

From~\eqref{eqn: sequence}, we have $v_{l}\in \mathcal{V}_{l}^{\frac{1-\eta}{2(1+\eta)}\delta}$, thus we obtain
\begin{align*}
     -\frac{2(1+\eta)}{1-\eta}\delta^{-1}&<|v_{l+L,\parallel}|-|v_{l,\parallel}|= \sum_{j=1}^{L} \big(|v_{l+j,\parallel}|-|v_{l+j-1,\parallel}|\big)\\
   & =\sum_{j=1}^{L} \big(|v_{l+j,\parallel}|-|V^{m-(l+j)+1}_\parallel(t_{l+j})|\big)+\sum_{j=1}^{L} \big(|V^{m-(l+j)+1}_\parallel(t_{l+j})|-|v_{l+j-1,\parallel}|\big)\\
   &\leq \sum_{j=1}^{L} \big(|v_{l+j,\parallel}|-|V^{m-(l+j)+1}_\parallel(t_{l+j})|\big)+\sum_{j=1}^{L}C_{\phi^m}(t_{l+j-1}-t_{l+j})  ,
\end{align*}
where $C_{\phi^m}$ is defined in~\eqref{eqn: cphi}. Take $t=t(\phi^m)$ small enough such that
\begin{equation}\label{eqn: t less than 1}
\sum_{j=1}^{L+1}C_{\phi^m}(t_{l+j-1}-t_{l+j}) \leq C_{\phi^m}t\leq 1.
\end{equation}
By~\eqref{eqn: increase large},~\eqref{eqn: ai def} and~\eqref{eqn: increase small}, we derive that
\begin{align*}
  \frac{-2(1+\eta)}{1-\eta}\delta^{-1}-1 &< \sum_{j=1}^{M} \big(|v_{p_j,\parallel}|-|V^{m-p_j+1}_\parallel(t_{p_j})|\big)+\sum_{j=1}^{\mathcal{M}} \big(|v_{q_j,\parallel}|-|V^{m-q_j+1}_\parallel(t_{q_j})|\big)\\
   & +\sum_{j=1}^{N} \big(|v_{o_j,\parallel}|-|V^{m-o_j+1}_\parallel(t_{o_j})|\big)\leq -M \delta^{-1}+\sum_{j=1}^{\mathcal{M}}a_j.
\end{align*}

Therefore, by $L\geq 100\frac{1+\eta}{1-\eta}$ and~\eqref{eqn: Mi'}, we obtain
\[\frac{2(1+\eta)}{1-\eta}\delta^{-1}+1\leq \frac{L}{10}\delta^{-1}\leq \frac{M}{2}\delta^{-1}\]
and thus
\begin{equation}\label{eqn: ai sum}
 \sum_{j=1}^{\mathcal{M}}a_j\geq M\delta^{-1}-\frac{2(1+\eta)}{1-\eta}\delta^{-1}-1>\frac{M\delta^{-1}}{2}.
\end{equation}
We focus on integrating over $\mathcal{V}_{q_i}$ with $1\leq i\leq \mathcal{M}$, those index satisfy~\eqref{eqn: ai def}. We consider the third line of~\eqref{eqn: int V_i} with $i=q_i$ and with integrating over $\{v_{q_i,\parallel}\in \mathcal{V}_{q_i,\parallel}:|v_{q_i,\parallel}|-|V^{m-q_j+1}_\parallel(t_{q_j})|= a_i\}$. To apply~\eqref{eqn: coe abc smaller} in Lemma \ref{Lemma: abc}, we set
\[a=-\frac{1}{2T_{k-1,q_i}}+\frac{1}{2T_w(x_{q_i})},\quad b=\frac{1}{2T_w(x_{q_i})r_\parallel(2-r_\parallel)}, \quad \e=2(2\mathcal{C})^{k-1-q_i}(\mathfrak{C}t).\]
We take $t=t(\xi,k,T_M,\mathcal{C},\mathfrak{C})$ small enough such that
\begin{equation}\label{eqn: k_1}
  a+\e-b=-\frac{1}{2T_{k-1,q_i}}+\frac{1}{2T_w(x_{q_i})}-\frac{1}{2T_w(x_{q_i})r_\parallel(2-r_\parallel)}+2(2\mathcal{C})^{k-1-q_i}(\mathfrak{C}t) <-\frac{1}{4T_M}.
\end{equation}
Then we use $\eta_{q_i,\parallel}<1$ to obtain
\begin{equation}\label{eqn: bound a_i}
 \mathbf{1}_{\{|v_{q_i,\parallel}|-|V^{m-q_i+1}_\parallel(t_{q_i})|= a_i\}}\leq  \mathbf{1}_{\{|v_{q_i,\parallel}|-\eta_{q_i,\parallel}|V^{m-q_i+1}_\parallel(t_{q_i})|>a_i\}}  \leq  \mathbf{1}_{\{|v_{q_i,\parallel}-\eta_{q_i,\parallel}V^{m-q_i+1}_\parallel(t_{q_i})|>a_i\}}.
\end{equation}
By~\eqref{eqn: coe abc smaller} in Lemma \ref{Lemma: abc} and~\eqref{eqn: bound a_i}, we apply~\eqref{eqn: Vq para} with $q=q_i$ to bound the third line of~\eqref{eqn: int V_i}( the integration over $\mathcal{V}_{q_i,\parallel}$ ) by
\begin{equation}\label{eqn: 4aiTM}
e^{-\frac{a_i^2}{4T_M}} C_{T_M,\xi}\exp\left(\big[ \frac{[T_{k-1,q_i}-T_w(x_{q_i})][1-r_{min}]}{2T_w(x_{q_i})[T_{k-1,q_i}(1-r_{min})+r_{min} T_w(x_{q_i})]} + 2(2\mathcal{C})^{k-q_i}(\mathfrak{C}t)\big]|V^{m-q_j+1}_\parallel(t_{q_i})|^2\right).
\end{equation}
Hence by the constant in~\eqref{eqn: 4aiTM} we draw a similar conclusion as~\eqref{eqn: V_i,delta}:
\begin{equation}\label{eqn: case a_i}
   \int_{\prod_{j=q_i}^{k-1} \mathcal{V}_j}\mathbf{1}_{\{t_k>0\}}\mathbf{1}_{\{|v_{q_i,\parallel}|-|V^{m-q_i+1}_\parallel(t_{q_i})|= a_i\}}d\Phi_{q_i,m}^{k,k-1}(t_k)
\leq e^{-\frac{a_i^2}{4T_M}}  (2C_{T_M,\xi})^{2(k-q_i)}\mathcal{A}_{k-1,q_i}.
\end{equation}
Therefore, by Lemma \ref{Lemma: accumulate}, after integrating over $\mathcal{V}_{q_1,\parallel},\mathcal{V}_{q_2,\parallel},\cdots,\mathcal{V}_{q_\mathcal{M},\parallel}$ we obtain an extra constant
\begin{align*}
 e^{-[a_i^2+a_2^2+\cdots +a_{\mathcal{M}}^2]/4T_M}&\leq e^{-[a_i+a_2+\cdots +a_{\mathcal{M}}]^2/(4T_M\mathcal{M})} \leq e^{-[M\delta^{-1}/2]^2/(4T_M\mathcal{M})}   \\
   & \leq e^{-[\frac{L}{8}\delta^{-1}]^2/(4T_M\frac{3}{4}L)}\leq e^{-\frac{1}{96T_M}L(\delta^{-1})^2}\leq e^{-L\delta^{-1}},
\end{align*}
where we have used~\eqref{eqn: ai sum} in the last step of first line,~\eqref{eqn: Mi'},~\eqref{eqn: mathcal M} in the first step of second line and take $\delta\ll 1$ in the last step of second line. Then $e^{-L\delta^{-1}}$ is smaller than $(3\delta)^{L/2}$ in~\eqref{eqn: 3delta} and we conclude
\begin{equation}\label{eqn: 3delta 1}
  \int_{\prod_{j=l}^{k-1} \mathcal{V}_j}   \mathbf{1}_{\{ M=|\text{Set1}|\geq L/4\}}    \mathbf{1}_{\{t_k>0\}} d\Phi_{l,m}^{k,k-1}(t_k)\leq (3\delta)^{L/2}    (2C_{T_M,\xi})^{2(k-l_i)}\mathcal{A}_{k-1,l}.
\end{equation}

The second case is that the number of $v_j\in \{v_j\notin \mathcal{W}_{j,\delta}:|v_{j,\perp}|>\frac{1+\eta}{1-\eta}\delta^{-1}\}$ is larger than $L/4$. We categorize $v_{j,\perp}$ into

\begin{description}
  \item[Set4] $\{v_j\notin \mathcal{W}_{j,\delta}:|v_{j,\perp}|>\frac{1+\eta}{1-\eta}\delta^{-1}~\text{and}~|v_{j,\perp}-\eta_{j,\perp}V^{m-j+1}_\perp(t_j)|<\delta^{-1}\}$.
\end{description}

\begin{description}
  \item[Set5]$\{v_j\in \mathcal{V}_j\backslash \mathcal{V}_j^{\frac{1-\eta}{2(1+\eta)\delta}}:|v_{j,\perp}|>|V^{m-j+1}_\perp(t_j)|\}$.
\end{description}

 \begin{description}

  \item[Set6] $\{v_j\in \mathcal{V}_j\backslash \mathcal{V}_j^{\frac{1-\eta}{2(1+\eta)\delta}}:|v_{j,\perp}|\leq |V^{m-j+1}_\perp(t_j)|\leq |v_{j,\perp}|+\delta^{-1}\}$.
\end{description}
Denote $|\text{Set4}|=M_1$ with $L/4\leq M_1\leq L$ and the corresponding index as $p'_1,p'_2,\cdots,p'_{M_1}$, $|\text{Set5}|=\mathcal{M}_1$ and the corresponding index as $q'_1,q'_2,\cdots,q'_{\mathcal{M}_1}$, $|\text{Set6}|=N_1$ and the corresponding index as $o'_1,o'_2,\cdots,o'_{N_1}$. Also define $b_j:=|v_{q'_j,\perp}|-|V^{m-q_j'+1}_\perp(t_{q_j'})|$. By the same computation as~\eqref{eqn: ai sum}, we have
\[ \sum_{j=1}^{\mathcal{M}_1}b_j\geq M_1\delta^{-1}-\frac{2(1+\eta)}{1-\eta}\delta^{-1}>\frac{M_1\delta^{-1}}{2}.\]
We focus on the integration over $v_{q'_j}$. Let $1\leq i\leq \mathcal{M}_1$, we consider the second line of~\eqref{eqn: int V_i} with $i=q'_i$ and with integrating over $\{v_{q'_i,\perp}\in \mathcal{V}_{q'_i,\perp}:|v_{q'_i,\perp}|-|V^{m-q_i'+1}_\perp(t_{q_i'})|= b_i\}$. To apply~\eqref{eqn: coe perp smaller 2} in Lemma \ref{Lemma: abc}, we set
\[a=-\frac{1}{2T_{k-1,q_i'}}+\frac{1}{2T_w(x_{q_i'})},\quad b=\frac{1}{2T_w(x_{q_i'})r_\perp}, \quad \e=2(2\mathcal{C})^{k-q_i'-1}(\mathfrak{C}t).\]
By the same computation as~\eqref{eqn: k_1}, we have $a+\e-b<-\frac{1}{4T_M}.$ Similarly to~\eqref{eqn: bound a_i}, we have
\[ \mathbf{1}_{\{|v_{q'_i,\perp}|-|V^{m-q_i'+1}_\perp(t_{q_i'})|= b_i\}}\leq \mathbf{1}_{\{|v_{q'_i,\perp}-\eta_{q'_i,\perp}V^{m-q_i'+1}_\perp(t_{q_i'})|>b_i\}}.\]
Hence by~\eqref{eqn: coe perp smaller 2} in Lemma \ref{Lemma: integrate normal small} and applying~\eqref{eqn: Vq perp}, we bound the integration over $\mathcal{V}_{q'_i,\perp}$ by
\[e^{-\frac{b_i^2}{16T_M}} C_{T_M,\xi}\exp\left(\big[ \frac{[T_{k-1,q_i'}-T_w(x_{q_i'})][1-r_{min}]}{2T_w(x_{q_i'})[T_{k-1,q_i'}(1-r_{min})+r_{min} T_w(x_{q_i'})]} + (2\mathcal{C})^{k-q_i'}(\mathfrak{C}t)\big]|V^{m-q_i'+1}_\perp(t_{q_i'})|^2\right).\]
Therefore,
\[\int_{\prod_{j=q_i'}^{k-1} \mathcal{V}_j}\mathbf{1}_{\{t_k>0\}}\mathbf{1}_{\{|v_{q_i',\perp}|-|V^{m-q_i'+1}_\perp(t_{q_i'})|= b_i\}}d\Phi_{q_i',m}^{k,k-1}(t_k)\leq e^{-\frac{b_i^2}{16T_M}}  (C_{T_M,\xi})^{2(k-q_i')} \mathcal{A}_{k-1,q_i'}.\]
 The integration over $\mathcal{V}_{q'_1,\perp},\mathcal{V}_{q'_2,\perp},\cdots,\mathcal{V}_{q'_{\mathcal{M}_1},\perp}$ provides an extra constant
\[e^{-[b_1^2+b_2^2+\cdots +b_{\mathcal{M}_1}^2]/16T_M}\leq e^{-\frac{1}{400T_M}L (\delta^{-1})^2}\leq e^{-L\delta^{-1}},\]
where we set $\delta\ll 1$ in the last step. Then $e^{-L\delta^{-1}}$ is smaller than $(3\delta)^{L/2}$ in~\eqref{eqn: 3delta} and we conclude
\begin{equation}\label{eqn: 3delta 2}
\int_{\prod_{j=l}^{k-1} \mathcal{V}_j}   \mathbf{1}_{\{ M_1=|\text{Set4}|\geq L/4\}}    \mathbf{1}_{\{t_k>0\}} d\Phi_{l,m}^{k,k-1}(t_k) \leq (3\delta)^{L/2}    (2C_{T_M,\xi})^{2(k-l)}\mathcal{A}_{k-1,l}.
\end{equation}

Finally collecting~\eqref{eqn: 3delta},~\eqref{eqn: 3delta 1} and~\eqref{eqn: 3delta 2} we derive the lemma.

\end{proof}

Now we prove the Lemma \ref{lemma: t^k}.

\begin{proof}[\textbf{Proof of Lemma \ref{lemma: t^k}}]
We mainly apply Lemma \ref{Lemma: Step3} during the proof. In order to apply Lemma \ref{Lemma: Step3}, here we consider the space $\mathcal{V}_i^{\frac{1-\eta}{2(1+\eta)}\delta}$ and ensure $\eta$ satisfy the condition~\eqref{eqn: eta condition}. Also we let $t'=t'(\xi,k,T_M,\mathcal{C},\mathfrak{C},C_{\phi^m})$ ( consistent with~\eqref{eqn: t'} ) satisfy condition~\eqref{eqn: t less than 1} and~\eqref{eqn: k_1}.

In the proof we first construct the $\eta$ that satisfies the condition~\eqref{eqn: eta condition} in Step 1. Then we prove there can be at most finite number of $v_j\in \mathcal{V}\backslash\mathcal{V}_j^{
  \frac{1-\eta}{2(1+\eta)}\delta}$ in Step 2. With such conclusion in Step 2, we apply Lemma \ref{Lemma: Step3} and consider the contribution of all possible subsequence~\eqref{eqn: sequence} in Step 3. In Step 4 we conclude the lemma.

\textbf{Step 1}

In this step we mainly focus on constructing the $\eta$, which is defined in~\eqref{eqn: eta}.

First we consider $\eta_{i,\parallel}$, which is defined in~\eqref{eqn: eta i para}. In regard to~\eqref{eqn: abe} and~\eqref{eqn: B i para def}, by~\eqref{eqn: k_1} with $t\leq t'$,
\begin{equation}\label{eqn: B i para}
B_{i,\parallel}\geq \frac{1}{2T_{k-1,i}}-2(2\mathcal{C})^{k-1-i}t\geq \frac{1}{2\frac{2\xi}{\xi+1}T_M}-(2\mathcal{C})^{k}(\mathfrak{C}t)\geq \frac{1}{4T_M}.
\end{equation}
By~\eqref{eqn: formula of Tp}, $T_{k-1,i}\to T_M$ as $k-i\to \infty$. For any $\e_1>0$, there exists $k_1$ s.t when
\begin{equation}\label{eqn: e1}
k\geq k_1,\quad i\leq k/2, \text{ we have }T_{k-1,i}\leq (1+\e_1)T_M.
\end{equation}
Moreover, by~\eqref{eqn: Constrain on T}, there exists $\e_2$ s.t
\begin{equation}\label{eqn: e2 def}
\frac{\min\{T_w(x)\}}{T_M}>\frac{1-r_\parallel}{2-r_\parallel}(1+\e_2).
\end{equation}
Then we have
\begin{equation}\label{eqn: e2 dependence}
\e_2=\e_2(\min\{T_w(x)\},T_M,r_\parallel,r_\perp).
\end{equation}

We use~\eqref{eqn: e1} and~\eqref{eqn: e2 def} to bound $T_w(x_i)$ in the $\eta_{i,\parallel}$( defined in~\eqref{eqn: eta i para}) below as
\begin{equation}\label{eqn: bound below}
T_w(x_i)=T_{k-1,i}\frac{T_w(x_i)}{T_{k-1,i}}\geq T_{k-1,i}\frac{T_w(x_i)}{T_M}\frac{1}{1+\e_1}> \frac{1-r_\parallel}{2-r_\parallel}T_{k-1,i}\frac{1+\e_2}{1+\e_1}.
\end{equation}
Thus we obtain
\begin{equation}\label{eqn: eta i para bounded}
\eta_{i,\parallel}<\frac{1+2\frac{(2\mathcal{C})^{k}(\mathfrak{C}t)}{B_{i,\parallel}}}{(1-r_\parallel)^2+ \frac{1-r_\parallel}{2-r_\parallel}\frac{1+\e_2}{1+\e_1}r_\parallel(2-r_\parallel)}(1-r_\parallel)= \frac{1+\frac{(2\mathcal{C})^{k}(\mathfrak{C}t)}{B_{i,\parallel}}}{1-r_\parallel+r_\parallel\frac{1+\e_2}{1+\e_1}}                                                                 .
\end{equation}
By~\eqref{eqn: e1}, we take
\begin{equation}\label{eqn: k_1 dependence}
k=k_1=k_1(\e_2,T_M,r_{\min})
\end{equation}
large enough such that $\e_1<\e_2/4$. By~\eqref{eqn: B i para} and~\eqref{eqn: eta i para bounded}, we derive that when $k=k_1$,
\begin{equation}\label{eqn: sup  less 1}
\sup_{i\leq k/2}\eta_{i,\parallel}\leq \frac{1+4T_M(2\mathcal{C})^{k}(\mathfrak{C}t)}{1-r_\parallel+r_\parallel\frac{1+\e_2}{1+\e_2/4}}<\eta_\parallel<1.
\end{equation}
Here we define
\begin{equation}\label{eqn: eta_paral}
\eta_\parallel:=\frac{1}{1-r_\parallel+r_\parallel\frac{1+\e_2}{1+\e_2/2}}<1.
\end{equation}
where we take $t'=t'(k,T_M,\e_2,\mathcal{C},\mathfrak{C},r_\parallel)$ small enough and $t\leq t'$ such that $4T_M(2\mathcal{C})^{k}(\mathfrak{C}t)\ll 1$ to ensure the second inequality in~\eqref{eqn: sup  less 1}.
Combining~\eqref{eqn: e2 dependence} and~\eqref{eqn: k_1 dependence}, we conclude the $t'$ we choose here only depends on the parameter in~\eqref{eqn: t'}.

Then we consider $\eta_{i,\perp}$ which is defined in~\eqref{eqn: eta i perp}. In regard to~\eqref{eqn: abe perp} and~\eqref{eqn: B i perp}, by~\eqref{eqn: B i para} we have $B_{i,\perp}\geq \frac{1}{4T_M}.$ By $\frac{\min\{T_w(x)\}}{T_M}>\frac{\sqrt{1-r_\perp}-(1-r_\perp)}{r_\perp}$ in~\eqref{eqn: Constrain on T} we can use the same computation as~\eqref{eqn: bound below} to obtain
\[T_w(x_i)>  \frac{\sqrt{1-r_\perp}-(1-r_\perp)}{r_\perp}  T_{k-1,i}\frac{1+\e_2}{1+\e_1},\]
with $\e_1<\e_2/4$. Thus we obtain
\[\eta_{i,\perp}<\eta_\perp <1           .                                                     \]
where we define
\begin{equation}\label{eqn: eta perp}
\eta_{\perp}:=\frac{1}{\sqrt{1-r_\perp}+(1-\sqrt{1-r_\perp})\frac{1+\e_2}{1+\e_2/2}}<1,
\end{equation}
with $t'=t'(k,T_M,\e_2,\mathcal{C},\mathfrak{C},r_\parallel)$( consistent with~\eqref{eqn: t'}) small enough and $t\leq t'$.

Finally we define
\begin{equation}\label{eqn: eta}
  \eta:=\max\{\eta_\perp,\eta_\parallel\}<1.
\end{equation}

\textbf{Step 2}

We claim that for $t\ll 1$,
\begin{equation}\label{eqn:claim_delta}
  |t_{j}-t_{j+1}|\gtrsim (\frac{1-\eta}{2(1+\eta)}\delta)^3,\text{ for }v_j\in \mathcal{V}_j^{
  \frac{1-\eta}{2(1+\eta)}\delta}.
\end{equation}
For $t_j\leq 1$,
\begin{align*}
|\int_{t_j}^{t_{j+1}}V^{m-j}(s;t_j,x_j,v_j)ds|^2 & =|x_{j+1}-x_j|^2\gtrsim |(x_{j+1}-x_j)\cdot n(x_j)|  =|\int_{t_j}^{t_{j+1}}V^{m-j}(s;t_j,x_j,v_j)\cdot n(x_j)ds|  \\
   &=|\int_{t_j}^{t_{j+1}} (v_j-\int_{t_j}^s \nabla \phi^{m-j}(\tau,X(\tau;t_j,x_j,v_j))d\tau)\cdot n(x_j)ds|\\
   &\geq |v_j\cdot n(x_j)||t_j-t_{j+1}|-|\int_{t_j}^{t_{j+1}}\int_{t_j}^s \nabla \phi^{m-j}(\tau,X(\tau;t_j,x_j,v_j))d\tau)\cdot n(x_j)ds|.
\end{align*}
Here we have used the fact that if $x,y\in \partial \Omega$ and $\partial \Omega$ is $C^2$ and $\Omega$ is bounded then $|x-y|^2\gtrsim_\Omega |(x-y)\cdot n(x)|$(see the proof in~\cite{EGKM} and~\cite{EGKMA}). Thus
\begin{align*}
  |v_j\cdot n(x_j)| & \lesssim \frac{1}{|t_j-t_{j+1}|}|\int_{t_j}^{t_{j+1}} V(s;t_j,x_j,v_j)ds|^2  \\
   & +\frac{1}{|t_j-t_{j+1}|}|\int_{t_j}^{t_{j+1}}\int_{t_j}^s \nabla \phi^{m-j}(\tau,X(\tau;t_j,x_j,v_j))d\tau)\cdot n(x_j)ds|\\
   &\lesssim |t_j-t_{j+1}|\{|v_j|^2+|t_j-t_{j+1}|^3 \Vert \nabla \phi^{m-j}\Vert^2_\infty +\frac{1}{2}\sup_{t_{j+1}\leq \tau\leq t_j}|\nabla \phi^{m-j}(\tau,X(\tau;t_j,x_j,v_j))\cdot n(x_j)|\}.
\end{align*}
Since $v_j\in \mathcal{V}^{\frac{1-\eta}{2(1+\eta)}\delta}_j$,
\begin{equation}\label{}
  |v_j\cdot n(x_j)|\lesssim |t_j-t_{j+1}|\{\delta^{-2}+t^3 \Vert \nabla \phi^{m-j}\Vert_\infty^2+\Vert \nabla \phi^{m-j}\Vert_\infty\}.
\end{equation}
By $t\ll1$ and $\Vert \nabla \phi^j\Vert_\infty$ is bounded due to Lemma \ref{lemma: phi_inf}, we can prove~\eqref{eqn:claim_delta}.

In consequence, when $t_k> 0$, by~\eqref{eqn:claim_delta} and $t\ll 1$, there can be at most $\{[C_{\Omega}(\frac{2(1+\eta)}{(1-\eta)\delta})^3]+1\}$ numbers of $v_j\in \mathcal{V}_j^{\frac{1-\eta}{2(1+\eta)}\delta}$.
Equivalently there are at least $k-[C_{\Omega}(\frac{2(1+\eta)}{(1-\eta)\delta})^3]$ numbers of $v_j\in \mathcal{V}_j\backslash \mathcal{V}_j^{\frac{1-\eta}{2(1+\eta)}\delta}$.

\textbf{Step 3}

In this step we combine Step 1 and Step 2 and focus on the integration over $\prod_{j=1}^{k-1} \mathcal{V}_j$.

By~\eqref{eqn:claim_delta} in Step 2, we define
\begin{equation}\label{eqn: N}
N:=\Big[C_{\Omega}\big(\frac{2(1+\eta)}{\delta(1-\eta)}\big)^3\Big]+1.
\end{equation}
For the sequence $\{v_1,v_2,\cdots,v_{k-1}\}$, suppose there are $p$ number of $v_j\in \mathcal{V}_j^{\frac{1-\eta}{2(1+\eta)}\delta}$ with $p\leq N$, we conclude there are at $\left(
                                                                                  \begin{array}{c}
                                                                                    k-1 \\
                                                                                    p \\
                                                                                  \end{array}
                                                                                \right)
$ number of these sequences. Below we only consider a single sequence of them.

In order to get~\eqref{eqn: eta_paral},\eqref{eqn: eta perp}$<1$, we need to ensure the condition~\eqref{eqn: e1}. Thus we take $k=k_1(T_M,\xi,r_\perp,r_\parallel)$ and only use the decomposition $\mathcal{V}_j=\Big(\mathcal{V}_j\backslash \mathcal{V}_j^{\frac{1-\eta}{2(1+\eta)}\delta} \Big)   \cup \mathcal{V}_j^{\frac{1-\eta}{2(1+\eta)}\delta}$ for $1\leq j\leq k/2$. Thus we only consider the half sequence $\{v_1,v_2,\cdots,v_{k/2}\}$. We derive that when $t_k>0$, there are at most $N$ number of $v_j\in \mathcal{V}_j^{\frac{1-\eta}{2(1+\eta)}\delta}$ and at least $k/2-N$ number of $v_j\in \mathcal{V}_j\backslash \mathcal{V}_j^{\frac{1-\eta}{2(1+\eta)}\delta}$ in $\prod_{j=1}^{k/2}\mathcal{V}_j$.

In this single half sequence $\{v_1,\cdots, v_{k/2}\}$, in order to apply Lemma \ref{Lemma: Step3}, we only want to consider the subsequence~\eqref{eqn: sequence} with $l+1<l+L\leq k/2$ and $L\geq 100\frac{1+\eta}{1-\eta}$. Thus we need to ignore those subsequence with $L<100\frac{1+\eta}{1-\eta}$. By~\eqref{eqn: sequence} one can see at the end of this subsequence, it is adjacent to a $v_l\in \mathcal{V}_{l}^{\frac{1-\eta}{2(1+\eta)}\delta}$. By~\eqref{eqn: N}, we conclude
\begin{equation}\label{Conclude}
      \textit{There are at most $N$ number of subsequences~\eqref{eqn: sequence} with $L\leq 100\frac{1+\eta}{1-\eta}$}.
\end{equation}
We ignore these subsequences. Then we define the parameters for the remaining subsequence( with $L\geq 100\frac{1+\eta}{1-\eta}$ ) as:
\[M_1:= \text{the number of $v_j\in \mathcal{V}_j\backslash \mathcal{V}_j^{\frac{1-\eta}{2(1+\eta)}\delta}$ in the first subsequence starting from $v_1$}.\]
\[n:= \text{the number of these subsequences}.\]
Similarly we can define $M_2,M_3,\cdots, M_n$ as the number in the second, third, $\cdots$, n-th subsequence. Recall that we only consider $\prod_{j=1}^{k/2} \mathcal{V}_j$, thus we have
\begin{equation}\label{eqn: Mi number}
100\frac{1+\eta}{1-\eta}\leq M_i\leq k/2,   \text{ for } 1\leq i\leq n.
\end{equation}
By~\eqref{Conclude}, we obtain
\begin{equation}\label{eqn: sum of M_i}
  k/2 \geq M_1+\cdots M_n\geq k/2-100\frac{1+\eta}{1-\eta}N.
\end{equation}
Take $M_i$ with $1\leq i\leq n$ as an example. Suppose this subsequence starts from $v_{l_i+1}$ to $v_{l_i+M_i}$, by~\eqref{eqn: Step3} in Lemma \ref{Lemma: Step3} with replacing $l$ by $l_i$ and $L$ by $M_i$, we obtain
\begin{equation}\label{eqn: M_i conclusion}
\int_{\prod_{j={l_i}}^{k-1} \mathcal{V}_j}\mathbf{1}_{\{t_k>0\}}\mathbf{1}_{\{v_{l_i+j}\in \mathcal{V}_{l_i+j}\backslash \mathcal{V}_{l_i+j}^{\frac{1-\eta}{2(1+\eta)}\delta} \text{ for } 1\leq j\leq M_i\}}d\Phi_{l_i,m}^{k,k-1}(t_k) \leq  (3\delta)^{M_i/2}   (2C_{T_M,\xi})^{2(k-l)}\mathcal{A}_{k-1,l_i}.
\end{equation}

Since~\eqref{eqn: M_i conclusion} holds for all $1\leq i\leq n$, by Lemma \ref{Lemma: accumulate} we can draw the conclusion for the Step 3 as following. For a single sequence $\{v_1,v_2,\cdots,v_{k-1}\}$, when there are $p$ number $v_j\in \mathcal{V}_{j}^{\frac{1-\eta}{2(1+\eta)}\delta}$, we have
\begin{equation}\label{eqn: Step 3 conclusion}
\begin{split}
   &\int_{\prod_{j=1}^{k-1} \mathcal{V}_j}   \mathbf{1}_{\{\text{$p$ number $v_j\in \mathcal{V}_{j}^{\frac{1-\eta}{2(1+\eta)}\delta}$ for a single sequence}\}}    \mathbf{1}_{\{t_k>0\}} d\Sigma_{k-1,m}^{k}(t_k) \\
    & \leq (3\delta)^{(M_1+\cdots+M_n)/2}  (2C_{T_M,\xi})^{2k} \mathcal{A}_{k-1,1}.
\end{split}
\end{equation}

\textbf{Step 4}

Now we are ready to prove the lemma. By~\eqref{eqn: N}, we have
\begin{equation}\label{eqn: proof step3}
\int_{\prod_{j=1}^{k-1}\mathcal{V}_j} \mathbf{1}_{\{t_k>0\}} d\Sigma_{k-1,m}^k(t_k)\leq \sum_{p=1}^{N}\int_{\{\text{Exactly $p$ number of $v_j\in \mathcal{V}_j^{\frac{1-\eta}{2(1+\eta)}\delta}$ }\}} \mathbf{1}_{\{t_k>0\}} d\Sigma_{k-1,m}^k(t_k).
\end{equation}
Since~\eqref{eqn: Step 3 conclusion} holds for a single sequence, we derive
\[\eqref{eqn: proof step3}\leq (2C_{T_M,\xi})^{2k}\sum_{p=1}^N\left(
    \begin{array}{c}
      k-1 \\
     p \\
    \end{array}
  \right)(3\delta)^{(M_1+M_2+\cdots M_n)/2} \mathcal{A}_{k-1,1}
\]
\begin{equation}\label{eqn: tk coe}
\leq (2C_{T_M,\xi})^{2k}N(k-1)^N(3\delta)^{k/4-101\frac{1+\eta}{1-\eta}N}\mathcal{A}_{k-1,1},
\end{equation}
where we use~\eqref{eqn: sum of M_i} in the second line.

Take $k=N^3$, the coefficient in~\eqref{eqn: tk coe} is bounded by
\begin{equation}\label{eqn: Finally}
(2C_{T_M,\xi})^{2N^3}N^{3N+1} (3\delta)^{N^3/4-101\frac{1+\eta}{1-\eta}N}\leq (2C_{T_M,\xi})^{2N^3} N^{4N}(3\delta)^{N^3/5},
\end{equation}
where we choose $N=N(\eta)$ large such that $N^3/4-101\frac{1+\eta}{1-\eta}N\geq N^3/5$.

Using~\eqref{eqn: N}, we derive
\[3\delta=C(\Omega,\eta)N^{-1/3}.\]
Finally we bound~\eqref{eqn: Finally} by
\begin{align*}
& (2C_{T_M,\xi})^{2N^3}N^{4N}(C(\Omega,\eta)N^{-1/3})^{N^3/5}\\
   &\leq e^{2N^3\log(2C_{T_M,\xi})} e^{4N\log N}e^{(N^3/5)\log(C(\Omega,\eta)N^{-1/3})}  \\
   & =e^{4N \log N}e^{(N^3/5)(log(C(\Omega,\eta))-\frac{1}{3}\log N)}e^{2N^3 \log(2C_{T_M,\xi})}\\
   &= e^{4N\log N-\frac{N^3}{15}(\log N-3\log C_{\Omega,\eta}-30\log (2C_{T_M,\xi}))}\\
   &\leq e^{4N\log N-\frac{N^3}{30}\log N}\leq e^{-\frac{N^3}{50}\log N}=e^{-\frac{k}{150}\log k}\leq (\frac{1}{2})^k,
\end{align*}
where we choose $\delta$ small enough in the second line such that $N=N(\Omega,\eta,C_{T_M,\xi})$ is large enough to satisfy
\[\log N -3\log C(\Omega,\eta)-30 \log (2C_{T_M,\xi})\geq \frac{\log N}{2},\]
\[4N\log N-\frac{N^3}{30}\log N\leq -\frac{N^3}{50}\log N.\]
And thus we choose $k=N^3=k_2=k_2(\Omega,\eta,C_{T_M,\xi})$ and we also require $\log k>150$ in the last step. Then we get~\eqref{eqn: 1/2 decay}.

Therefore, by the condition~\eqref{eqn: e1}, eventually we choose $k=k_0=\max\{k_1,k_2\}$. By the definition of $\eta$~\eqref{eqn: eta} with~\eqref{eqn: eta_paral} and~\eqref{eqn: eta perp}, we obtain $\eta=\eta(T_M,\mathcal{C},r_\perp,r_\parallel,\e_2)$. Thus by~\eqref{eqn: e2 dependence} and~\eqref{eqn: k_1 dependence}, we conclude the $k_0$ we choose here does not depend on $t$ and only depends on the parameter in~\eqref{eqn: k_0 dependence}. We derive the lemma.

\end{proof}

Now we are ready to prove the Proposition \ref{proposition: boundedness}, we will combine Lemma \ref{lemma: the tracjectory formula for f^(m+1)}-Lemma \ref{lemma: t^k} to close the estimate.

\subsection{Proof of Proposition \ref{proposition: boundedness}}
\begin{proof}[\textbf{Proof of Proposition \ref{proposition: boundedness}}]
First we take
\begin{equation}\label{eqn: first condition for tinf}
t_{\infty}\leq t'.
\end{equation}
with $t'$ defined in~\eqref{eqn: t'}. Then we let $k=k_0$ with $k_0$ defined in~\eqref{eqn: k_0 dependence} so that we can apply Lemma \ref{lemma: t^k} and Lemma \ref{lemma: boundedness}. Define the constant in~\eqref{eqn: fm is bounded} as
\begin{equation}\label{eqn: Cinfty}
  C_\infty=3(2C_{T_M,\xi})^{k_0},
\end{equation}
where $C_{T_M,\xi}$ is defined in~\eqref{eqn: 1 one}, $k_0$ in defined in~\eqref{eqn: k_0 dependence}.

We mainly use the formula given in Lemma \ref{lemma: the tracjectory formula for f^(m+1)} and we use Lemma \ref{lemma: boundedness} \ref{lemma: t^k} to control every term in~\eqref{eqn: formula for H}. We consider two cases.
\begin{description}
\item[Case1] $t_1\leq 0$,
\end{description} We consider~\eqref{eqn: Duhamal principle for case1} in Lemma~\ref{lemma: the tracjectory formula for f^(m+1)}. Since
\[e^{-\int_{s}^t \frac{\mathfrak{C}}{2}\langle V^m(\tau) \rangle^2 d\tau}\leq e^{\frac{\mathfrak{C}}{2}(s-t) \langle v\rangle^2}e^{\mathfrak{C}C_{\phi^m}(t-s)^2\langle v\rangle}\,,\]
by~\eqref{eqn: Duhamal principle for case1} and using the definition of $\Gamma^m_{\text{gain}}(s)$ in~\eqref{eqn: gamma^m} we have
\begin{align}
  |h^{m+1}(t,x,v)| & \leq | h_0(X^m(0;t,x,v),V^m(0;t,x,v))|\label{eqn: first term} \\
   & +\int_0^t e^{\frac{\mathfrak{C}}{2}(s-t) \big(\langle v\rangle^2-2C_{\phi^m}(t-s)\langle v\rangle\big)} e^{-\mathfrak{C}\langle V^m(s)\rangle^2 s}e^{\theta|V^m(s)|^2}   \int_{\mathbb{R}^3\times \mathbb{S}^2}B(V^m(s)-u,w)\sqrt{\mu(u)}\\
   & +\Big|\frac{h^{m}(t,x,v)}{e^{-\mathfrak{C}\langle v\rangle^2 t+\theta|v|^2}}\Big(s,X^{m}(s),u'\big(u,V^m(s)\big)\Big)\Big| \Big|\frac{h^{m}(t,x,v)}{e^{-\mathfrak{C}\langle v\rangle^2 t+\theta|v|^2}}\Big(s,X^{m}(s),v'\big(u,V^m(s)\big)\Big)\Big|      d\omega duds,\label{eqn: second term}
\end{align}
where $u'\big(u,V^m(s)\big)$ and $v'\big(u,V^m(s)\big)$ are defined by~\eqref{eqn: u' v'}. Then we have
\begin{align*}
 \eqref{eqn: second term}  & \leq(\sup_{0\leq s\leq t} \Vert h^m(s)\Vert_{L^{\infty}})^2 \times\int_0^t  \int_{\mathbb{R}^3\times \mathbb{S}^2}e^{\frac{\mathfrak{C}}{2}(s-t) \big(\langle v\rangle^2-2C_{\phi^m}(t-s)\langle v\rangle\big)}  B\big(V^m(s)-u,w\big) \\
   & \times\sqrt{\mu(u)}e^{-\mathfrak{C}\langle V^m(s)\rangle^2 s}e^{\theta|V^m(s)|^2}  e^{-\theta(|u|^2+|V^m(s)|^2)} e^{\mathfrak{C}(\langle u\rangle^2+\langle V^m(s)\rangle^2) s}  d\omega duds \\
   &     \lesssim (\sup_{0\leq s\leq t} \Vert h^m(s)\Vert_{L^{\infty}})^2 \int_0^t \int_{\mathbb{R}^3}e^{\frac{\mathfrak{C}}{2}(s-t) \big(\langle v\rangle^2-2C_{\phi^m}(t-s)\langle v\rangle\big)} |V^m(s)-u|^\mathcal{K}\sqrt{\mu} e^{-\theta |u|^2} e^{\mathfrak{C}\langle u\rangle^2 s}    du ds \\
   &\lesssim_{C_\infty} \Vert h_0\Vert_{L^\infty}^2\int_0^t e^{\frac{\mathfrak{C}}{2}(s-t) \big(\langle v\rangle^2-2C_{\phi^m}(t-s)\langle v\rangle\big)} \langle V^m(s)\rangle^\mathcal{K} ds\\
   &\lesssim \Vert h_0\Vert_{L^\infty}^2\int_0^t e^{\frac{\mathfrak{C}}{2}(s-t) \big(\langle v\rangle^2-2C_{\phi^m}(t-s)\langle v\rangle\big)}\big(\langle v\rangle^\mathcal{K}+(t-s)^{\mathcal{K}}  \big) ds\\
   &\leq\Vert h_0\Vert_{L^\infty}^2\int_0^t e^{\frac{\mathfrak{C}}{2}(s-t) \big(\langle v\rangle^2-2C_{\phi^m}(t-s)\langle v\rangle\big)} \big(\langle v\rangle^\mathcal{K}+1\big) \{\mathbf{1}_{|v|>N}+\mathbf{1}_{|v|\leq N}\}ds\\
   &\lesssim \Vert h_0\Vert_{L^\infty}^2\Big[\int_0^t e^{\frac{\mathfrak{C}}{4}(s-t)\langle v\rangle^2} \langle v\rangle^\mathcal{K}\mathbf{1}_{|v|>N}ds+                \int_0^t   \langle v\rangle^\mathcal{K} \mathbf{1}_{|v|>N}\Big]\\
   &\lesssim_{\Vert h_0\Vert_\infty}\big(\frac{1}{N^2}+Nt\big),
\end{align*}
where $0\leq\mathcal{K}\leq 1$. Therefore, we obtain
\begin{equation}\label{eqn: Gamma bounded by}
~\eqref{eqn: second term}\leq C(C_\infty,\Vert h_0\Vert_\infty)(\frac{1}{N^2}+Nt)\leq \frac{1}{k_0}\Vert h_0\Vert_\infty,
\end{equation}
where we choose
\begin{equation}\label{eqn: second condition for tinf}
N=N(C_\infty,\Vert h_0\Vert_\infty,k_0)\gg 1,\quad
t_\infty=t_\infty(N,C_\infty,\Vert h_0\Vert_\infty,k_0)\ll 1,
\end{equation}
with $t\leq t_\infty$ to obtain the last inequality in~\eqref{eqn: Gamma bounded by}.

Finally collecting~\eqref{eqn: first term} and~\eqref{eqn: second term} we obtain
\begin{equation}\label{eqn: hm+1 bounded case 1}
 \Vert h^{m+1}(t,x,v)\mathbf{1}_{\{t_1\leq 0\}}\Vert_\infty \leq   2\Vert h_0\Vert_\infty\leq C_\infty\Vert h_0\Vert_\infty,
\end{equation}
where $C_\infty$ is defined in~\eqref{eqn: Cinfty}.

\begin{description}
\item[Case2] $t_1\geq 0$,
\end{description}
We consider~\eqref{eqn: Duhamel principle for case 2} in Lemma \ref{lemma: the tracjectory formula for f^(m+1)}. First we focus on the first line. By~\eqref{eqn: Gamma bounded by} we obtain
\begin{equation}\label{eqn: first line bounded}
\int_{t_1}^t e^{-\int_s^t \frac{\mathfrak{C}}{2} \langle V^m(\tau)\rangle^2 d\tau} e^{-\mathfrak{C}\langle V^m(s)\rangle^2 s}e^{\theta|V^m(s)|^2}   \Gamma_{\text{gain}}^m(s)ds \leq \frac{1}{k_0}\Vert h_0\Vert_\infty.
\end{equation}
Then we focus on the second line of~\eqref{eqn: Duhamel principle for case 2}. Using $\theta=\frac{1}{4T_M\xi}$ we bound the second line of~\eqref{eqn: Duhamel principle for case 2} by
\begin{equation}\label{eqn: extra term to cancel}
\exp\bigg(\big[\frac{1}{2T_M\frac{2\xi}{\xi+1}}-\frac{1}{2T_w(x_1)}\big]|V^m(t_1)|^2\bigg)\int_{\prod_{j=1}^{k_0-1}\mathcal{V}_j}H.
\end{equation}
Now we focus on $\int_{\prod_{j=1}^{k_0-1}\mathcal{V}_j}H$. We compute $H$ term by term with the formula given in~\eqref{eqn: formula for H}. First we compute the first line of~\eqref{eqn: formula for H}. By Lemma~\ref{lemma: boundedness} with $p=1$, for every $1\leq l\leq k_0-1$, we have
\begin{equation}\label{eqn: l term}
\begin{split}
& \int_{\prod_{j=1}^{k_0-1}\mathcal{V}_j}  \mathbf{1}_{\{t_{l+1}\leq 0<t_l\}} |h_0\big(X^{m-l}(0),V^{m-l}(0)\big)|  d\Sigma_{l,m}^{k_0}(0) \\
  & \leq \Vert h_0\Vert_\infty  \int_{\prod_{j=1}^{k_0-1}\mathcal{V}_j}  \mathbf{1}_{\{t_{l+1}\leq 0<t_l\}}   d\Sigma_{l,m}^{k_0}(0)\\
  &\leq (2C_{T_M,\xi})^l\Vert h_0\Vert_{\infty}\exp\bigg(\frac{(T_{l,1}-T_w(x_1))(1-r_{min})}{2T_w(x_1)[T_{l,1}(1-r_{min})+r_{min} T_w(x_1)]}|V^{m}(t_1)|^2+(2\mathcal{C})^{l}(\mathfrak{C}t)|V^{m}(t_1)|^2\bigg).
\end{split}
\end{equation}

In regard to~\eqref{eqn: extra term to cancel} we have
\begin{align*}
   & \exp\bigg(\big[\frac{1}{2T_M\frac{2\xi}{\xi+1}}-\frac{1}{2T_w(x_1)}\big]|V^m(t_1)|^2\bigg)\times ~\eqref{eqn: l term} \\
   & =(2C_{T_M,\xi})^l\Vert h_0\Vert_{\infty}\exp\bigg(\Big[\frac{-1}{2\big(T_w(x_1)r_{min}+T_{l,1}(1-r_{min})\big)}+\frac{1}{2T_M\frac{2\xi}{\xi+1}}\Big]|V^{m}(t_1)|^2+(2\mathcal{C})^{l}(\mathfrak{C}t)|V^{m}(t_1)|^2\bigg).
\end{align*}

Using the definition~\eqref{eqn: definition of T_p} we have $T_w(x_1)<\frac{2\xi}{\xi+1}T_M$ and $T_{l,1}<\frac{2\xi}{\xi+1}T_M$, then we take
\begin{equation}\label{eqn: third condition for tinfty}
t_\infty=t_\infty(T_M,k_0,\xi,\mathcal{C},\mathfrak{C})
\end{equation}
small enough and $t\leq t_\infty$ so that the coefficient for $|V^m(t_1)|^2$ is
\begin{equation}\label{eqn: less than 0}
\begin{split}
   &  \frac{-1}{2\big(T_w(x_1)r_{min}+T_{l,1}(1-r_{min})\big)}+\frac{1}{2T_M\frac{2\xi}{\xi+1}}+(2\mathcal{C})^{l}(\mathfrak{C}t) \\
    & \leq \frac{-1}{2\big(T_Mr_{min}+T_{l,1}(1-r_{min})\big)}+\frac{1}{2T_M\frac{2\xi}{\xi+1}}+(2\mathcal{C})^{k_0}(\mathfrak{C}t)\leq 0.
\end{split}
\end{equation}
Since~\eqref{eqn: l term} holds for all $1\leq l\leq k_0-1$, by~\eqref{eqn: less than 0} the contribution of the first line of~\eqref{eqn: formula for H} in ~\eqref{eqn: extra term to cancel} is bounded by
\begin{equation}\label{eqn: first term bounded}
(2C_{T_M,\xi})^{k_0}\Vert h_0\Vert_{\infty}.
\end{equation}

Then we compute the second line of~\eqref{eqn: formula for H}. For each $1\leq l\leq k_0-1$ such that $\max\{0,t_{l+1}\}\leq s\leq t_l$, by~\eqref{eqn:trajectory measure}, we have
\[d\Sigma_{l,m}^{k_0}(s)=e^{-\int_s^{t_l} \frac{\mathfrak{C}}{2} \langle V^{m-l}(\tau)\rangle^2 d\tau} d\Sigma_{l,m}^{k_0}(t_l).\]
Therefore, we derive
\begin{equation}\label{eqn: last line second}
\begin{split}
   & \int_{\max\{0,t_l\}}^{t_l}\int_{\prod_{j=1}^{k_0-1}\mathcal{V}_{j}}  e^{-\mathfrak{C}\langle V^{m-l}(s)\rangle^2 s}e^{\theta|V^{m-l}(s)|^2} |\Gamma_{\text{gain}}^{m-l}(s)|d\Sigma_{l,m}^{k_0}(s)ds \\
   & \leq \int_{\prod_{j=1}^{k_0-1}\mathcal{V}_{j}}\int_{\max\{0,t_l\}}^{t_l} e^{-\int_{s}^{t_l} \frac{\mathfrak{C}}{2} \langle V^{m-l}(\tau)\rangle^2 d\tau} e^{-\mathfrak{C}\langle V^{m-l}(s)\rangle^2 s}e^{\theta|V^{m-l}(s)|^2} |\Gamma_{\text{gain}}^{m-l}(s)| ds d\Sigma_{l,m}^{k_0}(t_l)\\
   &\leq \frac{1}{k_0}\Vert h_0\Vert_\infty\int_{\prod_{j=1}^{k_0-1}\mathcal{V}_j} \Sigma_{l,m}^{k_0}(t_l)\\
   &\leq \frac{1}{k_0}\Vert h_0\Vert_\infty (2C_{T_M,\xi})^l \exp\bigg(\frac{(T_{l,1}-T_w(x_1))(1-r_{min})}{2T_w(x_1)[T_{l,1}(1-r_{min})+r_{min} T_w(x_1)]}|V^{m}(t_1)|^2+(2\mathcal{C})^{l}(\mathfrak{C}t)|V^{m}(t_1)|^2\bigg),
\end{split}
\end{equation}
where we apply~\eqref{eqn: Gamma bounded by} in the third line and apply Lemma \ref{lemma: boundedness} in the last line.

In regard to~\eqref{eqn: extra term to cancel}, by~\eqref{eqn: less than 0} we obtain
\[\exp\bigg(\big[\frac{1}{2T_M\frac{2\xi}{\xi+1}}-\frac{1}{2T_w(x_1)}\big]|V^m(t_1)|^2\bigg)\times ~\eqref{eqn: last line second}\leq \frac{1}{k_0}(2C_{T_M,\xi})^l \Vert h_0\Vert_\infty.\]

Since~\eqref{eqn: last line second} holds for all $1\leq l\leq k_0-1$, the contribution of the second line of~\eqref{eqn: formula for H} in~\eqref{eqn: extra term to cancel} is bounded by
\begin{equation}\label{eqn: second term bounded}
  \frac{k_0-1}{k_0}(2C_{T_M,\xi})^{k_0}\Vert h_0\Vert_\infty.
\end{equation}

Last we compute the third term of~\eqref{eqn: formula for H}. By Lemma \ref{lemma: t^k} and the assumption~\eqref{eqn: fm is bounded} we obtain
\begin{equation}\label{eqn: third term bounded in H}
\begin{split}
   & \int_{\prod_{j=1}^{k_0-1}\mathcal{V}_j}  \mathbf{1}_{\{0<t_{k_0}\}} |h^{m-k_0+2}\big(t_{k_0},x_{k_0},V^{m-k_0+1}(t_{k_0})\big)|  d\Sigma_{k_0-1,m}^{k_0}(t_{k_0}) \\
    & \leq \Vert h^{m-k_0+2}\Vert_\infty  \int_{\prod_{j=1}^{k_0-1}\mathcal{V}_j}  \mathbf{1}_{\{0<t_{k_0}\}}   d\Sigma_{k_0-1,m}^{k_0}(t_{k_0})\\
    &\leq      3(2C_{T_M,\xi})^{k_0} (\frac{1}{2})^{k_0}\Vert h_0\Vert_\infty\exp\bigg(\frac{(T_{l,1}-T_w(x_1))(1-r_{min})}{2T_w(x_1)[T_{l,1}(1-r_{min})+r_{min} T_w(x_1)]}|V^{m}(t_1)|^2+(2\mathcal{C})^{l}(\mathfrak{C}t)|V^{m}(t_1)|^2\bigg).
\end{split}
\end{equation}

In regard to~\eqref{eqn: extra term to cancel}, by~\eqref{eqn: less than 0} we have
\[\exp\bigg(\big[\frac{1}{2T_M\frac{2\xi}{\xi+1}}-\frac{1}{2T_w(x_1)}\big]|V^m(t_1)|^2\bigg)\times ~\eqref{eqn: third term bounded in H}\leq (2C_{T_M,\xi})^{k_0} \Vert h_0\Vert_\infty.\]
Thus the contribution of the third line of~\eqref{eqn: formula for H} in~\eqref{eqn: extra term to cancel} is bounded by
\begin{equation}\label{eqn: third term bounded}
(2C_{T_M,\xi})^{k_0} \Vert h_0(x,v)\Vert_\infty.
\end{equation}

Collecting~\eqref{eqn: first term bounded}~\eqref{eqn: second term bounded}~\eqref{eqn: third term bounded} we conclude that the second line of~\eqref{eqn: Duhamel principle for case 2} is bounded by
\begin{equation}\label{eqn: second line bounded}
(2C_{T_M,\xi})^{k_0}\times (2+\frac{k_0-1}{k_0}) \Vert h_0\Vert_\infty.
\end{equation}
Adding~\eqref{eqn: second line bounded} to~\eqref{eqn: first line bounded} we use~\eqref{eqn: Duhamel principle for case 2} to derive
\begin{equation}\label{eqn: hm+1 bounded case 2}
\Vert h^{m+1}(t,x,v)\mathbf{1}_{\{t_{1}\geq 0\}}\Vert_\infty\leq 3(2C_{T_M,\xi})^{k_0} \Vert h_0\Vert_\infty.
\end{equation}

Combining~\eqref{eqn: hm+1 bounded case 1} and~\eqref{eqn: hm+1 bounded case 2} we derive~\eqref{eqn: L_infty bound for f^m+1}.

Last we focus the parameters for $t_\infty$ in~\eqref{eqn: t_1}. In the proof the constraints for $t_\infty$ are~\eqref{eqn: first condition for tinf},~\eqref{eqn: second condition for tinf} and~\eqref{eqn: third condition for tinfty}. We obtain
\[t_\infty=t_\infty(t',N,C_\infty,\Vert h_0\Vert_\infty,T_M,k_0,\xi,\mathcal{C},\mathfrak{C})=t_\infty(k_0,\xi,T_M,\min\{T_w(x)\},\mathcal{C},r_\perp,r_\parallel,\mathfrak{C},C_{T_M,\xi},\Vert h_0\Vert_\infty,C_{\phi^m}).\]
By the definition of $k_0$ in~\eqref{eqn: k_0 dependence}, definition of $C_{T_M,\xi}$ in~\eqref{eqn: 1 one}, definition of $\mathcal{C}$ in~\eqref{eqn: cal C}, definition of $C_{\phi^m}$ in~\eqref{eqn: cphi} and the condition for $\mathfrak{C}$ in~\eqref{eqn: C satisfy},~\eqref{eqn: C satisfy 1},~\eqref{eqn: chooseC}, we derive~\eqref{eqn: t_1}.

\end{proof}

\section{Weighted $W^{1,p}$ estimate for $f^{m+1}$}
For proving the uniqueness of the solution as mentioned in the introduction, we rely on the estimate for $\nabla_x f$. In this section we prove the weighted $W^{1,p}$ estimates for $f^{m+1}=F^{m+1}/\sqrt{\mu}$ that satisfies~\eqref{eqn: fm+1} with boundary condition~\eqref{eqn: fm+1 BC}. We will be proving the following proposition.

\begin{proposition}\label{Prop W1p}
Assume all the assumption in Proposition \ref{proposition: boundedness} holds true ( so that we have~\eqref{eqn: theta'} ). Let $f^{m+1}$ solving~\eqref{eqn: fm+1} with boundary condition~\eqref{eqn: fm+1 BC}. Define
\Be\label{mathcal_E}
\begin{split}
   &  \mathcal{E}^m(t):=\sup_{l\leq m}\Big[
		\lambda\| w_{\tilde{\theta}}e^{-\lambda t\langle v\rangle} f^l (t) \|_p^p
		+\lambda\int_0^t   |w_{\tilde{\theta}}e^{-\lambda s\langle v\rangle} f^l(s)|_{p,+}^p+
		\frac{1}{2}\| e^{-\lambda t\langle v\rangle}w_{\tilde{\theta}}\alpha_{f^{l-1 },\epsilon}^\beta \nabla_{x,v} f^l(t) \|_{p}^p\\
    & +\int^t_0  | e^{-\lambda s\langle v\rangle}w_{\tilde{\theta}}\alpha_{f^{l-1 },\epsilon}^\beta \nabla_{x,v} f^l(s) |_{p,+}^p+\frac{\lambda}{4}\int_0^t \Vert e^{-\lambda s\langle v\rangle} \langle v\rangle w_{\tilde{\theta}}\alpha_{f^{l-1 },\epsilon}^\beta \nabla_{x,v} f^l(s)\Vert_p^p\Big]\,.
\end{split}
\Ee
Then for small enough $\tilde{\theta}$ so that $0<\tilde{\theta}<\theta\ll 1$ and $\lambda\gg 1$,
there exists $t_{W} \ll 1$ ($t_{W} \leq t_\infty$) and $C_{W}\gg 1$ such that for
\begin{equation}\label{Condition for p}
  \frac{p-2}{p}<\beta<\frac{2}{3}\quad \text{  for  } \quad 3<p<6\,,
\end{equation}
if
\begin{equation}
\sup_{0 \leq t \leq t_{W}} \mathcal{E}^m(t)\leq 2C_W\{ \| w_{\tilde{\theta}} f _0 \|_p^p+\| w_{\tilde{\theta}} \alpha_{f_{0 },\epsilon}^\beta \nabla_{x,v} f  _0 \|_{p}^p \}< \infty\,,\\
\end{equation}
then
\Be \label{induc_hypo}
\sup_{0 \leq t \leq t_{W}}\mathcal{E}^{m+1} (t)\leq 2C_W \{ \| w_{\tilde{\theta}} f _0 \|_p^p + \| w_{\tilde{\theta}} \alpha_{f_{0 },\epsilon}^\beta \nabla_{x,v} f_0 \|_{p}^p \}\,.
\Ee

Here $C_W$ is a constant defined in~\eqref{eqn: C_W}, and $t_W$ satisfies the condition~\eqref{eqn: lesssim}.

\end{proposition}

This proposition implies the uniform in $m$ bound of the weighted $W^{1,p}$ norm of $f^m$. This gives us an a-priori estimate for the later proof for the uniqueness. The ``energy" term defined in~\eqref{mathcal_E} has two components that depend on $p$-norm of $f$, and three components on $p$-norm of $\partial f$. Therefore in the proof for the proposition, we need to provide the estimates for these components. Some lemmas from~\cite{CKL}~\cite{GKTT} will be repeatedly used and we cite them here first.

We note that in bulk this part of proof is rather similar to that in~\cite{CKL}. The main difficulty comes in through the boundary treatment as mentioned in the introduction, and this complexity is reflected Step 1 for $f$ and Step 5 for $\partial f$ in the proof.

For the initial problems of the transport equation with time-independent field $E(t,x)$, source $H(t,x,v)$, and damping term $\psi(t,x,v)\geq 0$, let $h$ solves:

\begin{equation}\label{eqn: eqn in section 2}
  \partial_t h+v\cdot \nabla_x h+E\cdot \nabla_v h+\psi h=H\,,
\end{equation}
then we have the following estimates for $h$:

\begin{lemma}\label{lemma: Green's indentity}(Lemma 5 in~\cite{CKL})\\
For $p\in[1,\infty)$ assume that $h,\partial_t h + v\cdot
	\nabla_x h-\nabla\phi \cdot \nabla_v h \in L^p ([0,T];L^p(\Omega\times\mathbb{R}^3))$ and $%
	h_{\gamma_-} \in L^p ([0,T];L^p(\gamma))$. Then $h  \in C^0( [0,T];
	L^p(\Omega\times\mathbb{R}^3))$ and $h_{\gamma_+} \in
	L^p ([0,T];L^p(\gamma))$ and for almost every $t\in [0,T]$ :
	\Be\begin{split}
\| h(t) \|_{p}^p + \int_0^t |h |_{\gamma_+,p}^p =\  \| h(0)\|_p^p + \int_0^t
		|h |_{\gamma_-,p}^p + \int_0^t \iint_{\Omega\times\mathbb{R}^3 }
		\{\partial_t h + v\cdot \nabla_x h+E \cdot \nabla_v h+\psi h\} |h|^{p-1}.
	\end{split}\Ee

\end{lemma}

\begin{lemma}\label{lemma: trace thm}(Lemma 6 in~\cite{CKL})\\
Assume $E \in L^\infty$, then for $t\ll 1$, $\varepsilon > 0$,
\begin{equation} \label{case:nondecay}
\int_{0}^{t}\int_{\gamma _{+}\setminus \gamma _{+}^{\varepsilon }}|h|\mathrm{%
		d}\gamma \mathrm{d}s \leq C(\e)\left\{ \
	\|h_{0}\|_{1}+\int_{0}^{t}  \| h(s)\|_{1}+\big{\Vert}[\partial_{t}+v\cdot \nabla _{x}+E\cdot \nabla_v + \psi ]h(s)\big{\Vert} _{1} \mathrm{d}s\ \right\}\,,
\end{equation}
where
\begin{equation}\label{eqn: gamma+ e}
\gamma_+^\epsilon=   \{(x,v)\in \gamma_+:           n(x)\cdot v<\epsilon \text{ or }|v|>\epsilon^{-1}\}\,.
\end{equation}
\end{lemma}

The next result is about the integrability of $\alpha$,
\begin{proposition}\label{prop_int_alpha} (Proposition 3 in~\cite{CKL})\\
	Assume $E(t,x) \in C^1_x$ is given.
	Then for $0\leq s\leq t\ll 1$, $0<\sigma<1$ and $N>1$ and $x \in \bar{\O}$,
	\Be\label{NLL_split2}
	\int_{|u| \leq N }
	\frac{   \dd u
	}{\alpha_{f,\epsilon}(s,x,u)^{ \sigma}}
	\lesssim_{\sigma, \O,N}  1, 
	\Ee
	and, for any $0< \kappa\leq2$,
	\Be\label{NLL_split3}
	\int_{  |u|\geq N} \frac{e^{-C|v-u|^2}}{|v-u|^{2-\kappa}} \frac{1}{\alpha_{f,\epsilon}(s,x,u)^\sigma} \dd u
	\lesssim_{\sigma, \O,N,\kappa}  1.
	\Ee
\end{proposition}

We will also need the $C^2$ estimate for $\phi$:
\begin{lemma}
Assume~\eqref{Condition for p}. If $\phi$ solves~\eqref{equation for phi_f} then
\begin{equation}\label{C2 estimate for phi}
  \Vert \phi(t)\Vert_{C^{2,1-\frac{3}{p}}}\leq (C_1)^{1/p}\{\Vert f(t)\Vert_p+\Vert e^{-\lambda t\langle v\rangle} \alpha_{f,\epsilon}^\beta \nabla_x f(t)\Vert_p\} \quad\text{ for }\quad p>3\,.
\end{equation}
\end{lemma}

\begin{proof}
Applying the Schauder estimate to~\eqref{equation for phi_f} we deduce
\begin{equation}\label{}
  \Vert \phi\Vert_{C^{2,1-\frac{3}{p}}} \lesssim_{p,\Omega}\Big\Vert \int_{\mathbb{R}^3}f(t)\sqrt{\mu}dv\Big\Vert_{C^{0,1-\frac{3}{p}}(\bar{\Omega})} \quad \text{for}\quad p>3\,.
\end{equation}
By the Morrey inequality, $W^{1,p}\subset C^{0,1-\frac{3}{p}}$ for $p>3$, we derive
\begin{align*}
\Big\Vert \int_{\mathbb{R}^3} f(t)\sqrt{\mu}dv\Big\Vert_{C^{0,1-\frac{n}{p}}(\bar{\Omega})}   & \lesssim \Big\Vert \int_{\mathbb{R}^3} f(t)\sqrt{\mu}dv\Big\Vert_{W^{1,p}(\Omega)} \\
   & \lesssim \Big(\int_{\mathbb{R}^3} \mu^{q/2}dv\Big)^{1/q}\Vert f(t)\Vert_{L^p(\Omega \times \mathbb{R}^3)}+\Big\Vert \int_{\mathbb{R}^3}\nabla_x f(t)\sqrt{\mu}dv\Big\Vert_{L^p(\Omega)}.
\end{align*}
By the H\"{o}lder inequality,we have
\begin{align*}
  \Big|\int_{\mathbb{R}^3} \nabla_x f(t,x,v)\sqrt{\mu(v)}dv\Big| & \leq \big\Vert \frac{\sqrt{e^{\lambda t\langle \cdot\rangle}\mu(\cdot)}}{\alpha_{f,\epsilon}(t,x,\cdot)^\beta}\Vert_{L^{\frac{p}{p-1}}(\mathbb{R}^3)}\big\Vert e^{-\lambda t\langle \cdot\rangle}\alpha_{f,\epsilon}(t,x,\cdot)^\beta \nabla_xf(t,x,\cdot)\Vert_{L^p(\mathbb{R}^3)} \\
   & =\Big(\int_{\mathbb{R}^3}\frac{\mu(v)^{\frac{p}{2(p-1)}}}{\alpha(t,x,v)^{\frac{\beta p}{p-1}}}dv\Big)^{\frac{p-1}{p}} \Vert e^{-\lambda t\langle \cdot\rangle}\alpha_{f,\epsilon}^\beta \nabla_xf(t,x,\cdot)\Vert_{L^p(\mathbb{R}^3)}.
\end{align*}

From the assumption $\frac{p-2}{p-1}<\frac{\beta p}{p-1}<\frac{2}{3}\frac{p}{p-1}<1$. We draw the conclusion.

\end{proof}

We need some estimates about the collision operator $\Gamma$ for Proposition~\ref{Prop W1p}. Define a notation
\begin{equation}\label{eqn: k_rho}
  \mathbf{k}_\rho(v,u)=\frac{1}{|v-u|}\exp\{-\rho|v-u|^2-\rho\frac{||v|-|u||^2}{|v-u|^2}\}.
\end{equation}
The velocity derivative for the nonlinear Boltzmann operator reads
\begin{equation}\begin{split}\label{nabla_Gamma}
			&\nabla_v \left( \Gamma_{\text{gain}}(f^m,f^m)-\Gamma_{\text{loss}}(f^m,f^{m+1}) \right)\\
			= & \ \Gamma_{\textrm{gain}} (\nabla_v f^m,f^m) + \Gamma_{\textrm{gain}} ( f^m,\nabla_v f^m)
			- \Gamma_{\textrm{loss}} (\nabla_v f^m,f^{m+1}) -\Gamma_{\textrm{loss}} ( f^m,\nabla_v f^{m+1})\\
            &  + \Gamma_{v,\text{gain}} (f^m,f^m)-\Gamma_{v,\text{loss}}(f^m,f^{m+1}).
		\end{split}\end{equation}
		Here we have defined
		\begin{equation}\begin{split}\label{Gamma_v}
          & \Gamma_{v,\textrm{gain}}(f^m,f^m) - \Gamma_{v,\textrm{loss}}(f^m,f^{m+1}) \\
			&:=\int_{\mathbb{R}^{3}}   \int_{\S^2}   | u \cdot \omega|
			f^m(v+ u_\perp) f^m(v + u_\parallel)
			\nabla_v\sqrt{\mu(v+u)} \dd \omega  \mathrm{d}u  \\
			&\quad - \int_{\mathbb{R}^{3}}   \int_{\S^2}   | u \cdot \omega|
			f^m(v+u) f^{m+1}(v  )
			\nabla_v \sqrt{\mu(v+u)} \dd \omega  \mathrm{d}u.
		\end{split}\end{equation}

\begin{lemma}\label{Lemma: lots of estimates}
For $0<\frac{\theta}{4}<\rho$, if $0<\tilde{\rho}<  \rho- \frac{\theta}{4}$, $0\leq s\leq t\ll \tilde{\rho}$, then
		\begin{equation}\label{k_theta_comparision}
		\mathbf{k}_{  \varrho}(v,u) \frac{e^{{\theta} |v|^2}}{e^{\mathcal{\theta} |u|^2}}\frac{e^{\lambda s\langle u\rangle}}{e^{\lambda s\langle v\rangle}} \lesssim  \mathbf{k}_{\tilde{\varrho}}(v,u) .
		\end{equation}

		Moreover,
		\begin{equation}\label{grad_estimate}
		\int_{\mathbb{R}^3}\mathbf{k}_{  \tilde{\rho}}(v,u) d u  \lesssim \langle v\rangle^{-1}.
		\end{equation}

For the nonlinear Boltzmann operator we have
		\begin{equation}\begin{split}\label{bound_Gamma_k}
		|\Gamma_{\text{gain}}(f^m, f^m)-\Gamma_{\text{loss}}(f^m,f^{m+1})|   \lesssim \left(\| w_{\theta'} f^m \|_\infty+\|w_{\theta'} f^{m+1} \| \right)\int_{\R^3} \mathbf{k}_{{\varrho }} (v,u) |f^m(u) | \dd u.
		\end{split}\end{equation}

For~\eqref{nabla_Gamma} we have
		\begin{equation}\label{bound_Gamma_nabla_vf1}
		|w_{\tilde{\theta}} \Gamma_{\textrm{gain}} (\nabla_v f^m,f^m) |+ |w_{\tilde{\theta}}\Gamma_{\textrm{gain}} ( f^m,\nabla_v f^m)|
\lesssim \| w_{\theta'} f^m \|_\infty \int_{\R^3} \mathbf{k}_{\tilde{\varrho}} (v,u)  |w_{\tilde{\theta}}\nabla_v f^m(u)| \dd u  .
		\end{equation}

		\begin{equation}\label{bound_Gamma_nabla_vf2}
		\begin{split}
			|w_{\tilde{\theta}}\Gamma_{\textrm{loss}}(\nabla_v f^m, f^{m+1})| &\lesssim \|w_{\theta'} f^{m+1}\|_{\infty} \int_{\R^3} \mathbf{k}_{\tilde{\varrho}} (v,u) |w_{\tilde{\theta}}\nabla_v f^m(u)| \dd u , \\
			|w_{\tilde{\theta}}\Gamma_{\textrm{loss}}( f^m, \nabla_v f^{m+1})| &\lesssim \langle v\rangle  \| w_{\theta'} f^m \|_\infty
			|w_{\tilde{\theta}}\nabla_v f^{m+1} (v)|  .
		\end{split}
\end{equation}

		\begin{equation} \label{Gvloss}
			|w_{\tilde{\theta}}\Gamma_{v,\textrm{loss}} (f^m,f^{m+1})|\lesssim  \langle v\rangle \frac{w_{\tilde{\theta}}(v)}{w_{\theta'} (v)} \|w_{\theta'} f^{m+1}\|_{\infty}\Vert e^{-\lambda\langle u\rangle s}w_{\tilde{\theta}}(u) f^m\Vert_p.
		\end{equation}	

		\begin{equation}\label{Gvgain}
			|w_{\tilde{\theta}}\Gamma_{v,\textrm{gain}} (f^m,f^m)| \lesssim   \|w_{\theta'} f^m\|_{\infty} \int_{\R^3} \mathbf{k}_{\tilde{\varrho}} (v,u) | w_{\theta'} f^m(u)| \dd u. \\
		\end{equation}

For $(x,v)\in \gamma_-$, we have the following bound for $\nabla_{x,v}f^{m+1}$ on the boundary:
\begin{equation}\label{eqn: derivative bound on the boundary first part}
|\nabla_{x,v}f^{m+1}(t,x,v)|\lesssim \langle v\rangle^2e^{[\frac{1}{4T_M}-\frac{1}{2T_w(x)}]|v|^2}\left(1+\frac{1}{|n(x)\cdot v|}\right)\times~\eqref{eqn: derivative bound on the boundary second part}
\end{equation}
with
\begin{equation}\label{eqn: derivative bound on the boundary second part}
  \begin{split}
     & \int_{n(x)\cdot u>0} \Big[\langle u\rangle|\nabla_{x,v}f^m(t,x,u)|\\
      & + \langle u\rangle^2|f^m|+ \Vert w_{\theta'} f^m\Vert_\infty \int_{\mathbb{R}^3}\mathbf{k}_\rho(u,u')|f^{m-1}(u')|du'   +\langle u\rangle |f^m|\Vert\nabla_x \phi^{m-1}\Vert_\infty \Big]                                                                   \\
      &  e^{-[\frac{1}{4T_M}-\frac{1}{2T_w(x)}]|u|^2}d\sigma(u,v).
  \end{split}
\end{equation}

\end{lemma}

\begin{proof}

The proof of~\eqref{k_theta_comparision} is given in appendix.

		The nonlinear Boltzmann operator~\eqref{Def: Gamma} equals
		\begin{equation}\label{carleman}
		\begin{split}		
		& 
		\int_{\mathbb{R}^{3}}   \int_{\S^2}   | u \cdot \omega|
		g_1(v+ u_1) g_2(v + u_2)
		\sqrt{\mu(v+u)} \dd \omega  \mathrm{d}u \\
		& - \int_{\mathbb{R}^{3}}   \int_{\S^2}   | u \cdot \omega|
		g_1(v+u) g_2(v  )
		\sqrt{\mu(v+u)} \dd \omega  \mathrm{d}u,		
		\end{split}
		\end{equation}
where $u_1=(u\cdot \omega)\omega$ and $u_2=u-u_1$. By exchanging the role of $\sqrt{\mu}$ and $w^{-1}$, we conclude~\eqref{bound_Gamma_k}.

The estimates~\eqref{bound_Gamma_nabla_vf1}--~\eqref{Gvgain} follows from the standard way using~\eqref{carleman}. The readers can also find them in chapter 4 of~\cite{CKL}.

Then we focus on the derivative on the boundary. By~\eqref{equation for f} we have
\begin{equation}\label{eqn: normal derivative}
\begin{split}
   & \partial_n f^{m+1}(t,x,v) =\frac{-1}{n(x)\cdot v}\{\partial_t f^{m+1}+\sum_{i=1}^2 (v\cdot \tau_i)\partial_{\tau_i}f^{m+1}-\nabla_x \phi^m\cdot \nabla_v f^{m+1}\\
    &+(\frac{v}{2T_M}\cdot \nabla_x \phi^m) f^{m+1}-\Gamma_{\text{gain}}(f^m,f^m)+\Gamma_{\text{loss}}(f^m,f^{m+1})\}.
\end{split}
\end{equation}
Let $\tau_1(x)$ and $\tau_2(x)$ be unit tangential vectors to $\partial \Omega$ satisfying $\tau_1(x)\cdot n(x)=0=\tau_2(x)\cdot n(x)$ and $\tau_1(x)\times \tau_2(x)=n(x)$. Define the orthonormal transformation from $\{n,\tau_1,\tau_2\}$ to the standard bases $\{e_1,e_2,e_3\}$,
\[\mathcal{T}(x)n(x)=e_1,~~ \mathcal{T}(x)\tau_1(x)=e_2,~~\mathcal{T}(x)\tau_2(x)=e_3,~~\mathcal{T}^{-1}=\mathcal{T}^T.\]
By a change of variable $u'=\mathcal{T}(x)u$, $v'=\mathcal{T}(x)v$ we have
\[u_\perp=n(x)\cdot u=n(x)\cdot \mathcal{T}^T(x)u'=n(x)^T \mathcal{T}^T(x)u'=[\mathcal{T}(x)n(x)]^T u'=e_1\cdot u'=u_1',\]
\begin{align*}
 u_\parallel=  & [\tau_1(x)\cdot u]\tau_1(x)+[\tau_2(x)\cdot u]\tau_2(x)=[\tau_1(x)\cdot \mathcal{T}^T(x)u^\prime]\tau_1(x)+[\tau_2(x)\cdot \mathcal{T}^T(x)u^\prime]\tau_2(x) \\
   & =\{[\mathcal{T}\tau_1(x)]^Tu'\}\tau_1(x)+\{[\mathcal{T}\tau_2(x)]^Tu'\}\tau_2(x)=u_2^\prime \tau_1(x)+u_3^\prime \tau_2(x)=u_2'\mathcal{T}^T(x)e_2+u_3'\mathcal{T}^T(x)e_3,
\end{align*}
\[v_\perp=v_1',\quad v_\parallel=v_2'\mathcal{T}^T(x)e_2+v_3'\mathcal{T}^T(x)e_3.\]
Then the boundary condition becomes
\[f^{m+1}(t,x,v)=e^{[\frac{1}{4T_M}-\frac{1}{2T_w(x)}]|v|^2}\int_{u_1^\prime>0}f^m(t,x,\mathcal{T}^T(x)u^\prime) e^{-[\frac{1}{4T_M}-\frac{1}{2T_w(x)}]|u^\prime|^2}d\sigma^\prime(u,v),\]
where we define
\begin{equation*}
\begin{split}
   & d\sigma^\prime(u,v)=  \frac{1}{r_\perp r_\parallel (2- r_\parallel)\pi/2} \frac{|u_1^\prime|}{(2T_w(x))^2}I_0 \left(
 \frac{1}{2T_w(x)}\frac{2 (1-r_\perp)^{1/2} v_1' u_1^\prime}{r_\perp}\right) \\
    & \exp\Big(-\frac{1}{2T_w(x)}[\frac{|u_1^\prime|^2+(1-r_\perp)|v_1'|^2}{r_\perp}+\frac{|u_2^\prime \mathcal{T}^T(x)e_2+u_3^\prime \mathcal{T}^T(x) e_3-(1-r_\parallel)[v_2'\mathcal{T}^T(x)e_2+v_3'\mathcal{T}^T(x)e_3]|^2}{r_\parallel(2-r_\parallel)}]\Big)du^\prime.
\end{split}
\end{equation*}
We can further take the tangential derivatives $\partial_{\tau_i}$, for $(x,v)\in \gamma_-$,
\begin{equation}\label{eqn:tangential derivative on boundary}
  \begin{split}
     & |\partial_{\tau_i}f^{m+1}(t,x,v)|  \\
      & \lesssim e^{[\frac{1}{4T_M}-\frac{1}{2T_w(x)}]|v|^2}\bigg(  \Big|\frac{|v|^2\partial_{\tau_i}T_w(x)}{2[T_w(x)]^2}\int_{n(x)\cdot u>0}f^m(t,x,u)e^{-[\frac{1}{4T_M}-\frac{1}{2T_w(x)}]|u|^2}d\sigma(u,v) \Big|\\
      &+ \Big|\int_{n(x)\cdot u>0} \partial_{\tau_i}f^m(t,x,u)e^{-[\frac{1}{4T_M}-\frac{1}{2T_w(x)}]|u|^2}d\sigma(u,v)\Big|\\
      &+ \Big| \int_{n(x)\cdot u>0} \nabla_v f^m(t,x,u)\partial_{\tau_i}\mathcal{T}^T(x)\mathcal{T}(x)ue^{-[\frac{1}{4T_M}-\frac{1}{2T_w(x)}]|u|^2}d\sigma(u,v)\Big|\\
      &+ \Big|\int_{n(x)\cdot u>0}f^m(t,x,u) e^{-[\frac{1}{4T_M}-\frac{1}{2T_w(x)}]|u|^2} \frac{-|u|^2\partial_{\tau_i}T_w(x)}{2[T_w(x)]^2}d\sigma(u,v)\Big|\\
      &+ \Big|\int_{n(x)\cdot u>0}f^m(t,x,u) e^{-[\frac{1}{4T_M}-\frac{1}{2T_w(x)}]|u|^2}\frac{-2\partial_{\tau_i}T_w(x)}{T_w(x)}d\sigma(u,v)\Big|\\
      &+ \Big|\int_{n(x)\cdot u>0}f^m(t,x,u) e^{-[\frac{1}{4T_M}-\frac{1}{2T_w(x)}]|u|^2}  [-\frac{\partial_{\tau_i}T_w(x)}{T_w^2(x)}\frac{(1-r_\perp)^{1/2}v_1' u_1'}{r_\perp}]d\sigma(u,v)\Big|\\
      &+ \Big|\int_{n(x)\cdot u>0}f^m(t,x,u)e^{-[\frac{1}{4T_M}-\frac{1}{2T_w(x)}]|u|^2}\frac{\partial_{\tau_i}T_w(x)}{2(T_w(x))^2}(|v|^2+|u|^2) d\sigma(u,v)\Big|\\
      &+ \Big|\int_{n(x)\cdot u>0}f^m(t,x,u)e^{-[\frac{1}{4T_M}-\frac{1}{2T_w(x)}]|u|^2}\frac{1}{2T_w(x)}\big[|u|^2+\frac{|v|^2\partial_{\tau_i} \mathcal{T}^T(x)}{r_\parallel(2-r_\parallel)} \big]   d\sigma(u,v)\Big|   \bigg).
  \end{split}
\end{equation}

Then we take velocity derivatives and obtain for $(x,v)\in \gamma_-$,
\begin{equation}\label{eqn: velocity derivative on the boundary}
\begin{split}
   & \nabla_v f^{m+1}(t,x,v) \\
    &  = v[\frac{1}{2T_M}-\frac{1}{T_w(x)}]e^{[\frac{1}{4T_M}-\frac{1}{2T_w(x)}]|v|^2}  \int_{n(x)\cdot u>0}f^m(t,x,u)e^{-[\frac{1}{4T_M}-\frac{1}{2T_w(x)}]|u|^2}d\sigma(u,v)\\
    & +e^{[\frac{1}{4T_M}-\frac{1}{2T_w(x)}]|v|^2}\int_{n(x)\cdot u>0}f^m(t,x,u)e^{-[\frac{1}{4T_M}-\frac{1}{2T_w(x)}]|u|^2} \frac{(1-r_\perp)^{1/2}u_\perp}{T_w(x)r_\perp}n(x)d\sigma(u,v)\\
    &+e^{[\frac{1}{4T_M}-\frac{1}{2T_w(x)}]|v|^2}\int_{n(x)\cdot u>0}f^m(t,x,u)e^{-[\frac{1}{4T_M}-\frac{1}{2T_w(x)}]|u|^2} \\
    & \bigg(-\frac{1}{T_w}\Big[\frac{(1-r_\perp)v_\perp}{r_\perp}n(x)-\frac{u_\parallel-(1-r_\parallel)v_\parallel}{r_\parallel(2-r_\parallel)}(1-r_\parallel)(\mathbf{I_{3\times 3}}-n(x)\otimes n(x))\Big]\bigg)d\sigma(u,v),
\end{split}
\end{equation}
where we use
\[\nabla_v v_\perp=n(x),\nabla_v v_\parallel=\nabla_v (v-v_\perp\cdot n(x))=\mathbf{I}_{3\times 3}-n(x)\otimes n(x).\]

From~\eqref{equation for f}, the temporal derivative is
\begin{equation}\label{eqn: time derivative on the boundary}
\begin{split}
   &  \partial_t f^{m+1}(t,x,v) =e^{[\frac{1}{4T_M}-\frac{1}{2T_w(x)}]|v|^2}  \int_{n(x)\cdot u>0}\partial_t f^m(t,x,u)e^{-[\frac{1}{4T_M}-\frac{1}{2T_w(x)}]|u|^2}d\sigma(u,v)\\
    & =e^{[\frac{1}{4T_M}-\frac{1}{2T_w(x)}]|v|^2}  \int_{n(x)\cdot u>0}\Big[ -u\cdot \nabla_x f^m+\nabla_x\phi^{m-1}\cdot \nabla_v f^m\\
    &    -(\frac{u}{2T_M}\cdot \nabla_x \phi^{m-1}) f^m +\Gamma_{\text{gain}}(f^{m-1},f^{m-1})-\Gamma_{\text{loss}}(f^{m-1},f^{m})\Big] e^{-[\frac{1}{4T_M}-\frac{1}{2T_w(x)}]|u|^2}d\sigma(u,v).
\end{split}
\end{equation}
Combine~\eqref{eqn: normal derivative}-\eqref{eqn: time derivative on the boundary} we conclude~\eqref{eqn: derivative bound on the boundary first part}, where we use $T_w\in C^1_x$.

\end{proof}

We are now ready to show Proposition~\ref{Prop W1p}.
\begin{proof}[\textbf{{Proof of Proposition~\ref{Prop W1p}}}]
Setting $t\leq t_{W}\leq t_\infty$ so that Proposition \ref{proposition: boundedness} holds valid. In the following, we first examine the terms related to $p$-norm of $f$ in Step 1, and it will be followed by Step 2, in which we examine the boundedness of $\partial f$ terms. In Step 3 we collect these estimates to form the conclusion. The Green's identity used in Step 2 leads to two terms (bulk and boundary), to bound which, heavy computation is involved and we present the details in Step 4 and 5 respectively.
\begin{description}

\item[Step 1: estimate of $p$-norm of $f$]
\end{description}
Since $f^{m+1}$ solves~\eqref{eqn: fm+1}, its weighted version then satisfies:
\begin{equation}\label{eqn: fm+1 with w and C}
\begin{split}
   &  \partial_t \big[e^{-\lambda t\langle v\rangle} w_{\tilde{\theta}}f^{m+1} \big]+ v\cdot \nabla_x \big[e^{-\lambda t\langle v\rangle}w_{\tilde{\theta}}f^{m+1}\big]-\nabla_x \phi^m \cdot \nabla_v\big[e^{-\lambda t\langle v\rangle}w_{\tilde{\theta}}f^{m+1} \big] +\big[\nu(F^m)+ \\
    & \lambda\langle v\rangle+\frac{v}{2T_M}\cdot \nabla_x \phi^m-\lambda t\partial_v \langle
 v\rangle+2\tilde{\theta}v\cdot \nabla_x \phi^m\big]\big[e^{-\lambda t\langle v\rangle}w_{\tilde{\theta}}f^{m+1} \big]=e^{-\lambda t \langle v\rangle}w_{\tilde {\theta}}\Gamma_{\text{gain}}(f^m,f^m)\,.
\end{split}
\end{equation}
For $t_W=t_W(\lambda)\ll 1$, one can take $\lambda=\lambda(t,C_{\phi^m})$ large enough so that
\begin{equation}\label{eqn: C large}
\nu(F^m)+ \lambda\langle v\rangle+\frac{v}{2T_M}\cdot \nabla_x \phi^m-\lambda t\partial_v \langle
 v\rangle+2\tilde{\theta}v\cdot \nabla_x \phi^m \geq \nu(F^m)+\frac{\lambda}{2}\langle v\rangle>\frac{\lambda}{2}\langle v\rangle\,,
\end{equation}
and thus we apply Lemma~\ref{lemma: Green's indentity}, and combine with~\eqref{bound_Gamma_k}  to have:

\begin{equation}\label{eqn: Green's identity for fp}
  \begin{split}
     & \Vert e^{-\lambda t\langle v\rangle}w_{\tilde{\theta}}f^{m+1}(t)\Vert_p^p +\int_0^t |e^{-\lambda s\langle v\rangle}w_{\tilde{\theta}}f^{m+1}|_{p,+}^p+\frac{\lambda}{2}\int_0^t \Vert \langle v\rangle^{1/p}   e^{-\lambda s\langle v\rangle}w_{\tilde{\theta}} f^{m+1}\Vert_p^p\\
      & \lesssim_p \Vert w_{\tilde{\theta}}f(0)\Vert_p^p +\int_0^t |e^{-\lambda s\langle v\rangle}w_{\tilde{\theta}}f^{m+1}|_{p,-}^p\\
       &+\Vert w_{\theta'} f^m\Vert_\infty\int_0^t\int_{\Omega\times \mathbb{R}^3} |e^{-\lambda s\langle v\rangle}w_{\tilde{\theta}}f^{m+1}(v)|^{p-1}e^{-\lambda s\langle v\rangle}w_{\tilde{\theta}}\int_{\mathbb{R}^3}\mathbf{k}_\varrho (v,u)|f^m(u)|du\,.
  \end{split}
\end{equation}
To deal with the last term in~\eqref{eqn: Green's identity for fp}, we note that by using H\"{o}lder inequality and Young's inequality, with~\eqref{k_theta_comparision}, we have:
\Be\begin{split}\label{eqn: k_p_bound}
				&\int_{\R^3}|e^{-\lambda s\langle v\rangle}w_{\tilde{\theta}} f^{m+1}(v)|^{p-1} \int_{\R^3}  \mathbf{k}_{{\varrho}} (v,u)\frac{w_{\tilde{\theta}}(v)}{w_{\tilde{\theta}}(u)}\frac{e^{\lambda s\langle u\rangle}}{e^{\lambda s\langle v\rangle}}  |e^{-\lambda s\langle u\rangle}w_{\tilde{\theta}}(u) f^m(u)|\dd u \dd v  \\
				\lesssim_{p} & \ \|e^{-\lambda s\langle v\rangle}w_{\tilde{\theta}}  f^{m+1}  \|_{L^p_v}^{p-1}  \left\| \int_{\R^3} \mathbf{k}_{{\tilde{\varrho}}} (v,u)^{1/q} \mathbf{k}_{{\tilde{\varrho}}} (v,u)^{1/p}  | e^{-\lambda s\langle u\rangle}w_{\tilde{\theta}}f^m(u)|\dd u  \right\|_{L^p_v }\\
				\lesssim & \  \|e^{-\lambda s\langle v\rangle}w_{\tilde{\theta}}  f^{m+1} \|_{L^p_v}^{p-1} \left( \int_{\R^3} \mathbf{k}_{{\tilde{\varrho}}} (v,u) \dd u\right)^{1/q}\left\|
				\left( \int_{\R^3} \mathbf{k}_{{\tilde{\varrho}}}  (v,u) |e^{-\lambda s\langle u\rangle}w_{\tilde{\theta}}f^m(u)|^p \dd u \right)^{1/p} \right\|_{L^p_v }\\
				\lesssim & \ \|e^{-\lambda s\langle v\rangle}w_{\tilde{\theta}} f^{m+1}  \|_{L^p_v}^ {p-1}\|e^{-\lambda s\langle v\rangle}w_{\tilde{\theta}} f^m\|_{L_v^p}  \left( \int_{\R^3} \mathbf{k}_{{\tilde{\varrho}}} (v,u) \dd u\right)^{1/q} \left( \int_{\R^3} \mathbf{k}_{{\tilde{\varrho}}}  (v,u)   \dd v \right)^{1/p} \\
				\lesssim &  \ \| e^{-\lambda s\langle v\rangle}w_{\tilde{\theta}} f^{m+1}  \|_{L^p_v}^ {p-1} \|e^{-\lambda s\langle v\rangle}w_{\tilde{\theta}} f^m \|_{L_v^p}\lesssim_{p} \Vert e^{-\lambda s\langle v\rangle}w_{\tilde{\theta}} f^{m+1}\Vert_{L_v^p}^p+ \Vert e^{-\lambda s\langle v\rangle}w_{\tilde{\theta}} f^m\Vert_{L_v^p}^p\,,
			\end{split}\Ee
which gives a bound for the last term in~\eqref{eqn: Green's identity for fp} as:
\begin{equation}\label{eqn: 1st term}
C(p)\sup_m \Vert w_{\theta'}f^m\Vert_\infty \Big( \int_0^t  \Vert e^{-\lambda s\langle v\rangle}w_{\tilde{\theta}}f^{m+1} \Vert_p^p+\int_0^t\Vert e^{-\lambda s\langle v\rangle}w_{\tilde{\theta}} f^m\Vert_{L_v^p}^p\Big)\,.
\end{equation}
One can further absorb the first term above to the left hand side of~\eqref{eqn: Green's identity for fp} by choosing large enough $\lambda$:
\begin{equation}\label{eqn: lambda condition 1}
\frac{\lambda}{4}\geq C(p)\sup_m \Vert w_{\theta'}f^m\Vert_\infty \,.
\end{equation}

To deal with $\int_0^t |e^{-\lambda s\langle v\rangle}w_{\tilde{\theta}} f^{m+1}|_{p,-}^p$ in~\eqref{eqn: Green's identity for fp}, we first decompose
\[\gamma_+=\gamma_{+,1}^{v,x,\epsilon}\cup \Big( \gamma_{+}(x)\backslash\gamma_{+,1}^{v,x,\epsilon} \Big)\,,\]
where
\begin{equation}\label{eqn: gamma_+^v}
\begin{split}
   & \gamma_{+,1}^{v,x,\epsilon}=\{(x,u)\in \gamma_+:|n(x)\cdot u|\leq \epsilon \text{ or }|u_\parallel-\frac{2T_M(1-r_\parallel)}{2T_M+(T_w(x)-2T_M)r_\parallel(2-r_\parallel)}v_\parallel|\geq \epsilon^{-1}  \\
    & \text{ or }|u_\perp-\frac{2T_M\sqrt{1-r_\perp}}{2T_M+\big(T_w(x)-2T_M\big)r_\perp}v_\perp|\geq \epsilon^{-1}\}\,.
\end{split}
\end{equation}
This leads to
   \[ \int_0^t   |e^{-\lambda s\langle v\rangle}w_{\tilde{\theta}} f^{m+1}|_{p,-}^p =\int_0^t \int_{\partial \Omega}\int_{n(x)\cdot v <0} |n(x)\cdot v| e^{-p\lambda s\langle v\rangle}w^p_{\tilde{\theta}}|f^{m+1}|^p \]
   \begin{equation}\label{eqn: Lp bdr first}
     \begin{split}
         & \lesssim_{\tilde{\theta}} \int_0^t |n(x)\cdot v|\int_{\partial \Omega}\int_{n(x)\cdot v<0} e^{-p\lambda s\langle v\rangle}w^p_{\tilde{\theta}}(v) e^{p[\frac{1}{4T_M}-\frac{1}{2T_w(x)}]|v|^2} \\
          & \times \bigg(\Big[\int_{\gamma_{+,1}^{v,x,\epsilon}}+\int_{\gamma_+\backslash \gamma_{+,1}^{v,x,\epsilon}}\Big] |e^{-\lambda s\langle u\rangle}w_{\tilde{\theta}}(u) f^m|e^{-[\frac{1}{4T_M}-\frac{1}{2T_w(x)}]|u|^2}d\sigma(u,v) \bigg)^p ,
     \end{split}
   \end{equation}
 where we used
\[|f^m|=|e^{-\lambda s\langle u\rangle }w_{\tilde{\theta}}f^m||e^{\lambda s\langle u\rangle}w^{-1}_{\tilde{\theta}}(u)|\lesssim_{\tilde{\theta}}|e^{-\lambda s\langle u\rangle }w_{\tilde{\theta}}f^m|\,.\]
We further expand $d\sigma(v,u)$ by~\eqref{eqn:probability measure} and apply H\"{o}lder inequality using $1=\frac{1}{p}+\frac{1}{p^*}$ for:
    \begin{equation}\label{eqn: bound the f on negative bdr}
\begin{split}
   & \eqref{eqn: Lp bdr first}\lesssim_p \int_0^t  \int_{\partial \Omega}\bigg(\int_{\gamma_{+,1}^{v,x,\epsilon}(x)} |e^{-\lambda s\langle u\rangle}w_{\tilde{\theta}}(u)f^m|^p \{n(x)\cdot u\} du \bigg)\int_{n(x)\cdot v<0}|n(x)\cdot v| e^{-p\lambda s\langle v\rangle}w^p_{\tilde{\theta}}(v)  e^{p[\frac{1}{4T_M}-\frac{1}{2T_w(x)}]|v|^2}   \\
    & \bigg(\int_{\gamma_{+,1}^{v,x,\epsilon}(x)} |n(x)\cdot u|e^{-p^*[\frac{1}{4T_M}-\frac{1}{2T_w(x)}]|u|^2} I_0\Big(p^*\frac{(1-r_\perp)^{1/2}v_\perp u_\perp}{T_w(x)r_\perp} \Big)e^{-\frac{p^*}{2T_w(x)}\big[\frac{|u_\perp|^2+(1-r_\perp)|v_\perp|^2}{r_\perp}+\frac{|u_\parallel-(1-r_\parallel)v_\parallel|^2}{r_\parallel(2-r_\parallel)}\big]}           du \bigg)^{p/p^*}\\
    \end{split}
\end{equation}

    \begin{equation}\label{eqn: bound the f on negative bdr 2nd}
\begin{split}
    & +\int_0^t\int_{\partial \Omega}\bigg(\int_{\gamma_+\backslash \gamma_{+,1}^{v,x,\epsilon}} |e^{-\lambda s\langle u\rangle}w_{\tilde{\theta}}(u)f^m|^p\{n(x)\cdot u\}\bigg)\int_{n(x)\cdot v<0}|n(x)\cdot v|  e^{-p\lambda s\langle v\rangle}w^p_{\tilde{\theta}}(v) e^{p[\frac{1}{4T_M}-\frac{1}{2T_w(x)}]|v|^2}   \\
    & \bigg(\int_{\gamma_+\backslash\gamma_{+,1}^{v,x,\epsilon}(x)} |n(x)\cdot u|e^{-p^*[\frac{1}{4T_M}-\frac{1}{2T_w(x)}]|u|^2} I_0\Big(p^*\frac{(1-r_\perp)^{1/2}v_\perp u_\perp}{T_w(x)r_\perp} \Big)e^{-\frac{p^*}{2T_w(x)}\big[\frac{|u_\perp|^2+(1-r_\perp)|v_\perp|^2}{r_\perp}+\frac{|u_\parallel-(1-r_\parallel)v_\parallel|^2}{r_\parallel(2-r_\parallel)}\big]}           du \bigg)^{p/p^*},\\
    \end{split}
\end{equation}
where we apply H\"{o}lder inequality for $I_0$ to have
\[I_0^{p^*}(y)=(\frac{1}{\pi}\int_0^\pi e^{y\cos\theta}d\theta)^{p^*}\leq \frac{1}{\pi^{p^*}}\int_0^\pi e^{p^* y\cos\theta}d\theta \pi^{1/p}=\pi^{1/p-1/p^*}I_0(p^* y).\]
We now separate the discussion of~\eqref{eqn: bound the f on negative bdr} and~\eqref{eqn: bound the f on negative bdr 2nd}.
\begin{itemize}
\item[--]{Estimate of~\eqref{eqn: bound the f on negative bdr}:} to control this term, we will first control the integrand, which itself is an integration in $u$, shown on the second line, and with this term bounded, we move forward to control the next layer integration in $v$.
\begin{itemize}
\item[--] Based on the decomposition~\eqref{eqn: gamma_+^v}, the $u$-integration in the second line of~\eqref{eqn: bound the f on negative bdr} is further split into
\begin{equation}\label{eqn: three terms}
\underbrace{\int_{|n(x)\cdot u|\leq \e}}_\text{term I}+\underbrace{\int_{|u_\parallel-\frac{2T_M(1-r_\parallel)}{2T_M+(T_w(x)-2T_M)r_\parallel(2-r_\parallel)}v_\parallel|\geq \e^{-1}}}_\text{term II}+\underbrace{\int_{|u_\perp-\frac{2T_M\sqrt{1-r_\perp}}{2T_M+(T_w(x)-2T_M)r_\perp}v_\perp|\geq \e^{-1}}}_\text{term III}\,.
\end{equation}
To control term I, we draw the similarity to~\eqref{eqn: 2} in Lemma \ref{Lemma: (2)}. To be more specific, we apply~\eqref{eqn: coe abc perp small} with
\[a=-[\frac{p^*}{4T_M}-\frac{p^*}{2T_w(x)}],b=\frac{p^*}{2T_w(x)r_\perp},\e=0,w=\sqrt{1-r_\perp}v_\perp\,.\]
Thus by~\eqref{eqn: perp small} with $T_{k-1,i}$ replaced by $2T_M$, term I is bounded by
\begin{equation}\label{eqn: order ep}
\e\exp\bigg(\frac{p^*[2T_M-T_w(x)][1-r_{min}]}{2T_w(x)[2T_M(1-r_{min})+r_{min}T_w(x)]} |v|^2\bigg)\,.
\end{equation}
Similar techniques can be applied to analyze term II and term III. With
\[a=-[\frac{p^*}{4T_M}-\frac{p^*}{2T_w(x)}],b=\frac{p^*}{2T_w(x)r_\parallel(2-r_\parallel)},\e=0,w=(1-r_\parallel)v_\parallel\,\]
and
\[a=-[\frac{p^*}{4T_M}-\frac{p^*}{2T_w(x)}],b=\frac{p^*}{2T_w(x)r_\perp},\e=0,w=\sqrt{1-r_\perp}v_\perp\,\]
respectively, we have either
\[\frac{b}{b-a-\e}w=\frac{2T_M(1-r_\parallel)}{2T_M+(T_w(x)-2T_M)r_\parallel(2-r_\parallel)}v_\parallel\,,\]
or
\[\frac{b}{b-a-\e}w=\frac{2T_M\sqrt{1-r_\parallel}}{2T_M+(T_w(x)-2T_M)r_\perp}v_\perp\,,\]
which further bound the two terms by~\eqref{eqn: order ep}. Putting them back into~\eqref{eqn: three terms} we have:
\begin{equation}\label{eqn: fl}
\eqref{eqn: three terms}\lesssim \e\exp\bigg(\frac{p[2T_M-T_w(x)][1-r_{min}]}{2T_w(x)[2T_M(1-r_{min})+r_{min}T_w(x)]} |v|^2\bigg)\,.
\end{equation}
\item[--] With the integrand controlled, we move to the $v$-integration in~\eqref{eqn: bound the f on negative bdr}. Plugging~\eqref{eqn: fl} into~\eqref{eqn: bound the f on negative bdr}, we have the boundedness of the integrand:
\begin{equation}\label{eqn: v integrand for 1}
 \lesssim_{p}\e |n(x)\cdot v|  e^{-p\lambda s\langle v\rangle}w^p_{\tilde{\theta}}(v)\exp\Big(p\big[\frac{1}{4T_M}-\frac{1}{2[2T_M(1-r_{min})+r_{min} T_w(x)\big]}]|v|^2\Big)\,.
\end{equation}
Taking $\tilde{\theta}=\tilde{\theta}(T_M,r_{min})\ll 1$ such that
\begin{align*}
&p\tilde{\theta}+p\big[\frac{1}{4T_M}-\frac{1}{2[2T_M(1-r_{min})+r_{min} T_w(x)\big]}]
\end{align*}
\begin{equation}\label{eqn: dominated}
\leq p\tilde{\theta}+p\big[\frac{1}{4T_M}-\frac{1}{2[2T_M(1-r_{min})+r_{min} T_M\big]}]< 0\,,
\end{equation}

one has $\eqref{eqn: v integrand for 1}\in L^1_v(\mathbb{R}^3)$.
\end{itemize}
Pull out the constant we finally conclude with
\begin{equation}\label{eqn: bdr 1}
~\eqref{eqn: bound the f on negative bdr} \lesssim_{p,T_M,r_{min}} \epsilon\int_0^t |e^{-\lambda s \langle u\rangle}w_{\tilde{\theta}}(u)f^{m}|^{p}_{p,+}\,.
\end{equation}

\item[--]{Estimate of~\eqref{eqn: bound the f on negative bdr 2nd}:} note that comparing with the integrand in~\eqref{eqn: bound the f on negative bdr}, here the integration in $u$ is taken on $\gamma_+\backslash \gamma_{+,1}^{v,x,\epsilon}$, which does not provide a small $\epsilon$. With brute-force calculation we only get:
\[
\lesssim_{p} |n(x)\cdot v| e^{-p\lambda s\langle v\rangle}w^p_{\tilde{\theta}}(v) \exp\Big(p\big[\frac{1}{4T_M}-\frac{1}{2[2T_M(1-r_{min})+r_{min} T_w(x)]}\big]|v|^2\Big)\in L^1_v(\mathbb{R}^3)\,.
\]

Now we decompose the $v$-integration into
\[\int_{n(x)\cdot v<0}=\int_{n(x)\cdot v<0} \mathbf{1}_{|v|> \epsilon^{-1}}+\mathbf{1}_{|v|\leq \epsilon^{-1}}.\]
\begin{itemize}
\item[--] When $|v|>\e^{-1}$, using the exponential decaying function~\eqref{eqn: dominated} we obtain,
\begin{equation}\label{eqn: bdr 2}
\eqref{eqn: bound the f on negative bdr 2nd}\mathbf{1}_{|v|>\e^{-1}}\lesssim_{p,T_M,r_{min}} \epsilon\int_0^t |e^{-\lambda s\langle u\rangle}w_{\tilde{\theta}}(u) f^m|^p_{p,+}.
\end{equation}

\item[--] When $|v|\leq \epsilon^{-1}$, since $u\in \gamma_+\backslash \gamma_{+,1}^{v,x,\epsilon}$, for any $x\in \Omega$,
\begin{equation}\label{eqn: u bounded}
\begin{split}
    &  |u|\leq |u_\perp|+|u_\parallel|\leq |u_\perp-\frac{2T_M\sqrt{1-r_\perp}}{2T_M+\big(T_w(x)-2T_M\big)r_\perp}v_\perp|+|\frac{2T_M\sqrt{1-r_\perp}}{2T_M+\big(T_w(x)-2T_M\big)r_\perp}v_\perp|+\\
     & |u_\parallel-\frac{2T_M(1-r_\parallel)}{2T_M+(T_w(x)-2T_M)r_\parallel(2-r_\parallel)}v_\parallel|+|\frac{2T_M(1-r_\parallel)}{2T_M+(T_w(x)-2T_M)r_\parallel(2-r_\parallel)}v_\parallel|\leq 6\epsilon^{-1}\,.
\end{split}
\end{equation}
In the derivation we used~\eqref{eqn: Constrain on T}, $r_\parallel\leq 1$ in the assumption~\eqref{eqn: r condition} for
\begin{align}\label{eqn: TM less than 2}
\frac{2T_M(1-r_\parallel)}{2T_M+(T_w(x)-2T_M)r_\parallel(2-r_\parallel)}&<\frac{1}{(1-r_\parallel)+\frac{T_w(x)r_\parallel (2-r_\parallel)}{2T_M(1-r_\parallel)}}  \leq \frac{1}{1-\frac{1}{2}r_\parallel}\leq 2\,,
\end{align}
and similarly to have
\[\frac{2T_M\sqrt{1-r_\perp}}{2T_M+[T_w(x)-2T_M]r_\perp}\leq 2\,.\]
Then $u\in \gamma_+(x)\backslash \gamma_{+}^{\epsilon/6}(x)$, where $\gamma_+^{\e/6}$ is defined in~\eqref{eqn: gamma+ e}. By Lemma~\ref{lemma: trace thm} we obtain
\[\eqref{eqn: bound the f on negative bdr 2nd}\mathbf{1}_{|v|\leq \e^{-1}}\lesssim_\e\int_0^t \int_{\partial \Omega}\int_{\gamma_+(x)/\gamma_+^{\epsilon/6}(x)}   |e^{-\lambda s\langle u\rangle}w_{\tilde{\theta}}(u)  f^m|^p\{n(x)\cdot u\} dudS_x ds\]
\begin{equation}\label{eqn: Trace thm for f part1}
\lesssim_\e\Big[  \Vert w_{\tilde{\theta}}(v) f(0)\Vert_p^p +\int_0^t \Vert e^{-\lambda s\langle v\rangle}w_{\tilde{\theta}}(v) f^m\Vert_p^p + \eqref{eqn: Trace thm for f part2}\Big]
\end{equation}
with
\begin{equation}\label{eqn: Trace thm for f part2}
\begin{split}
   & \int_0^t\iint_{\Omega\times \mathbb{R}^3}          [\partial_t +v\cdot \nabla_x -\nabla_x \phi^{m-1}\cdot \nabla_v+\text{ LHS of }\eqref{eqn: C large}]|e^{-\lambda s\langle u\rangle}w_{\tilde{\theta}}(u) f^m|^p \\
    & \lesssim_{p} \sup_m\Vert w_{\theta'}f^m\Vert_\infty\Big(\int_0^t  \Vert e^{-\lambda s\langle v\rangle}w_{\tilde{\theta}}(v) f^m\Vert_p^p+\int_0^t  \Vert e^{-\lambda s\langle v\rangle}w_{\tilde{\theta}}(v) f^{m-1}\Vert_p^p\Big),
\end{split}
\end{equation}
where we apply~\eqref{eqn: 1st term} and replace $m+1$, $m$ by $m$, $m-1$ respectively.
\end{itemize}

Adding~\eqref{eqn: bdr 1},~\eqref{eqn: bdr 2} and~\eqref{eqn: Trace thm for f part1} back into~\eqref{eqn: Lp bdr first}, one has:
\begin{equation}\label{eqn: f on nega bound}
\begin{split}
   & \int_0^t |e^{-\lambda s\langle v\rangle}w_{\tilde{\theta}}(v) f^{m+1}|_{p,-}^p \\
  \leq& C(p,T_M,r_{min})\times\epsilon\int_0^t |e^{-\lambda s\langle v\rangle}w_{\tilde{\theta}}(v) f^{m}|_{p,+}^p +  \\
    & C(p,T_M,r_{min})C(\e)\sup_m\Vert w_{\theta'}f^m\Vert_\infty\\
    &\times\Big( \Vert w_{\tilde{\theta}}(v) f(0)\Vert^p_p +\int_0^t \Vert e^{-\lambda s\langle v\rangle}w_{\tilde{\theta}}(v) f^m\Vert_p^p+\int_0^t \Vert e^{-\lambda s\langle v\rangle}w_{\tilde{\theta}}(v) f^{m-1}\Vert_p^p\Big),
\end{split}
\end{equation}
where $C(\e)$ comes from~\eqref{eqn: Trace thm for f part1}, $C(p,T_M,r_{min})$ comes from $\lesssim_{p,T_M,r_{min}}$ and $C(p)$ comes in~\eqref{eqn: 1st term}. 
\end{itemize}

We finally plug~\eqref{eqn: 1st term} and~\eqref{eqn: f on nega bound} back in~\eqref{eqn: Green's identity for fp}, with condition for $\lambda$ in~\eqref{eqn: lambda condition 1} satisfied, we conclude with
\begin{equation}\label{eqn: Lp bound for f}
\begin{split}
   &   \Vert e^{-\lambda t\langle v\rangle}w_{\tilde{\theta}}(v)f^{m+1}(t)\Vert_p^p+\int_0^t|e^{-\lambda s\langle v\rangle}w_{\tilde{\theta}}(v) f^{m+1}|_{p,+}^p + \frac{\lambda}{4}\int_0^t \Vert \langle v\rangle^{1/p} e^{-\lambda s\langle v\rangle}w_{\tilde{\theta}}(v)f^{m+1}(t)\Vert_p^p
\\
  &   \leq    C(p,T_M,r_{min})\times \epsilon\int_0^t |e^{-\lambda s\langle v\rangle}w_{\tilde{\theta}}(v) f^m|_{p,+}^p+C(p,T_M,r_{min})C(\e)\sup_m\Vert w_{\theta'}f^m\Vert_\infty \Big(\Vert w_{\tilde{\theta}}(v)f(0)\Vert_p^p\\
    &+t\sup_{0\leq s\leq t}\Vert e^{-\lambda s\langle v\rangle }w_{\tilde{\theta}} f^{m-1}(s)\Vert_p^p+t\sup_{0\leq s\leq t}\Vert e^{-\lambda s\langle v\rangle }w_{\tilde{\theta}} f^m(s)\Vert_p^p\Big)\,.
\end{split}
\end{equation}

\begin{description}
\item[Step 2: estimate of $p$-norm of $\partial f$]
\end{description}
We first write down the equation for $e^{-\lambda t\langle v\rangle}w_{\tilde{\theta}}\partial f^{m+1}$ with $\partial \in \{\nabla_{x_i},\nabla_{v_i}\}$. According to~\eqref{eqn: fm+1} one has
\begin{equation}\label{eqn: formula for partial f}
  [\partial_t+v\cdot \nabla_x -\nabla_x \phi^m\cdot \nabla_v+\nu_{\lambda,\phi^m,w}](e^{-\lambda t\langle v\rangle}w_{\tilde{\theta}}\partial f^{m+1})=e^{-\lambda t\langle v\rangle}w_{\tilde{\theta}}\mathcal{G}^m\,,
\end{equation}
with
\begin{equation}\label{eqn: formula for G}
  \mathcal{G}^m=-\partial v \cdot \nabla_x f^{m+1}+\partial \nabla \phi^m \cdot \nabla_v f^{m+1}+\partial \Gamma_{\text{gain}}(f^m,f^m)-\partial \Gamma_{\text{loss}}(f^m,f^{m+1})-\partial\left(\frac{v}{2T_M}\cdot \nabla \phi^m(t,x)\right)f^{m+1}.
\end{equation}
Considering~\eqref{Gamma_v} we have:
\begin{equation}\label{eqn: G bounded}
\begin{split}
   & |\mathcal{G}^m|\lesssim |\nabla_x f^{m+1}|+|\nabla^2 \phi^m||\nabla_v f^{m+1}| +|\Gamma_{\text{gain}}(\partial f^m,f^m)|+|\Gamma_{\text{gain}}(f^m,\partial f^m)|+|\Gamma_{v,\text{gain}}(f^m,f^m)|\\
    & +|\Gamma_{\text{loss}}(\partial f^m,f^{m+1})|+|\Gamma_{\text{loss}}(f^m,\partial f^{m+1})|+|\Gamma_{v,\text{loss}}(f^m,f^{m+1})|+w^{-1/2}_\theta(|\nabla \phi^m|+|\nabla^2 \phi^m|)\Vert w_{\theta'} f^{m+1}\Vert_\infty.
\end{split}
\end{equation}
By~\eqref{eqn: C large}, we have
\begin{equation}\label{eqn: formula for the nu in the uniqueness}
  \nu_{\lambda,\phi^m,w}:=\lambda\langle v\rangle+\frac{v}{2T_M}\cdot \nabla_x\phi^m(t,x)+\frac{\nabla_x \phi^m\cdot \nabla_v [e^{-\lambda t\langle v\rangle}w_{\tilde{\theta}}]}{e^{-\lambda t\langle v\rangle}w_{\tilde{\theta}}}\geq \frac{\lambda}{2}\langle v\rangle.
\end{equation}
Since $\alpha$ is invariant to the transport equation, according to~\eqref{eqn: alpha invar}, we have
\begin{equation}\label{eqn: formula for partial f with weight}
\begin{split}
   & p|e^{-\lambda t\langle v\rangle}w_{\tilde{\theta}}\alpha_{f^{m},\epsilon}^\beta \partial f^{m+1}|^{p-1}[\partial_t +v\cdot \nabla_x -\nabla_x \phi \cdot \nabla_v +\nu_{\lambda,\phi^m,w}] |e^{-\lambda t\langle v\rangle}w_{\tilde{\theta}}\alpha_{f^{m},\epsilon}^\beta \partial f^{m+1}|\\
    & =pe^{-\lambda pt\langle v\rangle}w_{\tilde{\theta}}^p \alpha_{f^{m},\epsilon}^{\beta p}|\partial f^{m+1}|^{p-1}\mathcal{G}^m.
\end{split}
\end{equation}

These allow us to apply Lemma~\ref{lemma: Green's indentity} to~\eqref{eqn: formula for partial f with weight} for
\begin{equation}\label{eqn: Green's iden for partial f}
\begin{split}
   & \Vert e^{-\lambda t\langle v\rangle}w_{\tilde{\theta}}\alpha_{f^m,\epsilon}^\beta \partial f^{m+1}(t)\Vert_p^p +\frac{\lambda}{2}\int_0^t\Vert e^{-\lambda s\langle v\rangle} \langle v\rangle^{1/p} w_{\tilde{\theta}}\alpha_{f^m,\epsilon}^\beta \partial f^{m+1}\Vert_p^p   +\int_0^t | e^{-\lambda s\langle v\rangle} w_{\tilde{\theta}}\alpha_{f^m,\epsilon}^\beta \partial f^{m+1}|_{p,+}^p\\
    & \leq \Vert  w_{\tilde{\theta}}\alpha_{f,\epsilon}^\beta \partial f(0)\Vert_p^p +\underbrace{\int_0^t | e^{-\lambda s\langle v\rangle} w_{\tilde{\theta}}\alpha_{f^m,\epsilon}^\beta \partial f^{m+1}|_{p,-}^p}_{\eqref{eqn: Green's iden for partial f} \gamma_-}+\underbrace{\int_0^t \iint_{\Omega\times \mathbb{R}^3}pe^{-\lambda ps\langle v\rangle}\alpha_{f^m,\epsilon}^{\beta p}w_{\tilde{\theta}}^p |\partial f^{m+1}|^{p-1}|\mathcal{G}^m|}_{  \eqref{eqn: Green's iden for partial f}  \mathcal{G}^m}.
\end{split}
\end{equation}

The two terms will be separately considered in the later steps (Step 4 and 5 respectively). In the end we will obtain:
\begin{equation}\label{eqn: bound for negative bdy deri}
\begin{split}
   & \eqref{eqn: Green's iden for partial f}\gamma_- \\
    & \leq  C(p,T_M,r_{min}) \times\epsilon\int_0^t |e^{-\lambda s\langle v\rangle}w_{\tilde{\theta}} \alpha_{f^{m-1},\epsilon}^\beta \partial f^m|^p_{p,+} + C(p,T_M,r_{min})C(\e)\Vert w_{\tilde{\theta}}\alpha_{f^{m-1},\epsilon}^\beta \nabla_{x,v}f(0)\Vert_p^p+\\
    &+C(p,T_M,r_{min})C(\e)\sup_m\Vert w_{\theta'} f^m\Vert_\infty\Big(\int_0^t \Vert e^{-\lambda s\langle v\rangle}\langle v\rangle^{1/p}w_{\tilde{\theta}}\alpha_{f^{m-1},\epsilon}^\beta \nabla_{x,v} f^m \Vert_p^p     +\int_0^t|e^{-\lambda s\langle v\rangle} w_{\tilde{\theta}} f^m|_{p,+}^p \Big)  \\
    &+C(p,T_M,r_{min})C(\e)\times (\sup_m\Vert w_{\theta'} f^m\Vert_\infty+\sup_{l\leq m}\Vert \nabla^2 \phi^{l}\Vert_\infty)\times  \\
   &\Big( \int_0^t\Vert e^{-\lambda s\langle v\rangle} w_{\tilde{\theta}} f^{m-1}\Vert_p^p     +\int_0^t \Vert e^{-\lambda s\langle v\rangle }w_{\tilde{\theta}} \alpha_{f^{m-2},\epsilon}^\beta \nabla_{x,v} f^{m-1}\Vert_p^p +\int_0^t \Vert e^{-\lambda s\langle v\rangle }w_{\tilde{\theta}} \alpha_{f^{m-1},\epsilon}^\beta \nabla_{x,v} f^{m}\Vert_p^p  \Big),
\end{split}
\end{equation}

and that
\begin{equation}\label{eqn: bound of G}
\begin{split}
\eqref{eqn: Green's iden for partial f}\mathcal{G}^m  \leq& C(p)\bigg[\Big(1+\sup_m\Vert w_{\theta'} f^{m}\Vert_\infty\Big)\int_0^t \Vert e^{-\lambda s\langle v\rangle }\langle v\rangle^{1/p}w_{\tilde{\theta}} \alpha_{f^m,\epsilon}^\beta \partial f^{m+1}\Vert_p^p \\
    & +\Big(1+\sup_m\Vert w_{\theta'} f^m\Vert_\infty+\Vert \nabla^2 \phi^m\Vert_\infty\Big)\int_0^t \Big(\Vert e^{-\lambda s\langle v\rangle }w_{\tilde{\theta}}\alpha_{f^{m-1},\epsilon}^\beta \partial f^m  \Vert_p^p+\Vert e^{-\lambda s\langle v\rangle }w_{\tilde{\theta}}\alpha_{f^{m},\epsilon}^\beta \partial f^{m+1}  \Vert_p^p\Big)\\
   & +\Big(1+\sup_m\Vert w_{\theta'} f^{m}\Vert_\infty\Big)\int_0^t \Vert e^{-\lambda s\langle v\rangle}w_{\tilde{\theta}}f^m\Vert_p^p\bigg].
\end{split}
\end{equation}

Inserting these back in~\eqref{eqn: Green's iden for partial f} and using $\sup_{l\leq m}\sup_{0\leq s\leq t}\Vert \nabla^2 \phi^l(s)\Vert_\infty \lesssim \mathcal{E}^m<\infty$ and $\sup_{m}\Vert w_{\theta'} f^m\Vert_\infty<\infty$ according to Proposition \ref{proposition: boundedness}, we have:
\begin{equation}\label{eqn: bound for partial f}
\begin{split}
   & \Vert e^{-\lambda t\langle v\rangle}w_{\tilde{\theta}}\alpha_{f^m,\epsilon}^\beta \partial f^{m+1}(t)\Vert_p^p +  \frac{\lambda}{4} \int_0^t\Vert e^{-\lambda s\langle v\rangle} \langle v\rangle^{1/p} w_{\tilde{\theta}}\alpha_{f^m,\epsilon}^\beta \partial f^{m+1}\Vert_p^p   +\int_0^t | e^{-\lambda s\langle v\rangle} w_{\tilde{\theta}}\alpha_{f^m,\epsilon}^\beta \partial f^{m+1}|_{p,+}^p\\
& \leq  C(p,T_M,r_{min})\times\epsilon\int_0^t |e^{-\lambda s\langle v\rangle}w_{\tilde{\theta}} \alpha_{f^{m-1},\epsilon}^\beta \partial f^m|^p_{p,+}+C(p,T_M,r_{min})C(\e)\Vert w_{\tilde{\theta}}\alpha_{f^{m-1},\epsilon}^\beta \partial f(0)\Vert_p^p  \\
&+C(p,T_M,r_{min})C(\e) \sup_m\Vert w_{\theta'} f^m\Vert_\infty(\int_0^t|e^{-\lambda s\langle v\rangle}w_{\tilde{\theta}} f^m|_{p,+}^p+    \int_0^t \Vert e^{-\lambda s\langle v\rangle }\langle v\rangle^{1/p}w_{\tilde{\theta}} \alpha_{f^{m-1},\epsilon}^\beta \partial f^{m}\Vert_p^p)\\
&   +tC(p,T_M,r_{min},\e)\times (\sup_m\Vert w_{\theta'} f^m\Vert_\infty+\mathcal{E}^m)\times \sup_{0\leq s\leq t}\big( \Vert e^{-\lambda s\langle v\rangle} w_{\tilde{\theta}}f^{m-1}\Vert_p^p +\Vert e^{-\lambda s\langle v\rangle} w_{\tilde{\theta}}f^m\Vert_p^p  \\
& + \Vert e^{-\lambda s\langle v\rangle }w_{\tilde{\theta}}\alpha_{f^{m},\epsilon}^\beta \partial f^{m+1}  \Vert_p^p+\Vert e^{-\lambda s\langle v\rangle }w_{\tilde{\theta}}\alpha_{f^{m-1},\epsilon}^\beta \partial f^{m}  \Vert_p^p+ \Vert e^{-\lambda s\langle v\rangle}w_{\tilde{\theta}}\alpha_{f^{m-2},\epsilon}^\beta \partial f^{m-1} \Vert_p^p\big)\,.
\end{split}
\end{equation}

\begin{description}
\item[Step 3: summarize (collecting~\eqref{eqn: Lp bound for f} and~\eqref{eqn: bound for partial f} for the conclusion)]
\end{description}

Multiplying $\lambda\gg 1$ to~\eqref{eqn: Lp bound for f} and adding to~\eqref{eqn: bound for partial f} we derive that
\begin{equation}\label{eqn: final result of W1p}
\begin{split}
   & \lambda\Vert e^{-\lambda t\langle v\rangle} w_{\tilde{\theta}} f^{m+1}(t)\Vert_p^p+  \Vert e^{-\lambda t\langle v\rangle}w_{\tilde{\theta}}\alpha_{f^m,\epsilon}^\beta \partial f^{m+1}(t)\Vert_p^p + \frac{\lambda}{4}\int_0^t\Vert e^{-\lambda s\langle v\rangle} \langle v\rangle^{1/p} w_{\tilde{\theta}}\alpha_{f^m,\epsilon}^\beta \partial f^{m+1}(s)\Vert_p^p \\
    & +\int_0^t | e^{-\lambda s\langle v\rangle} w_{\tilde{\theta}}\alpha_{f^m,\epsilon}^\beta \partial f^{m+1}(s)|_{p,+}^p+ \lambda\int_0^t |e^{-\lambda s\langle v\rangle} w_{\tilde{\theta}} f^{m+1}|_{p,+}^p\\
    & \leq  C(p,T_M,r_{min})C(\e)\lambda\big(\Vert w_{\tilde{\theta}} f(0)\Vert_p^p+ \Vert w_{\tilde{\theta}}\alpha_{f,\epsilon}^\beta \partial f(0)\Vert_p^p\big)+C(p,T_M,r_{min})\times \e\int_0^t |e^{-\lambda s\langle v\rangle}w_{\tilde{\theta}} \alpha_{f^{m-1},\epsilon}^\beta \partial f^m|^p_{p,+}\\
    &   + C(p,T_M,r_{min})     \Big(\frac{C(\e)\sup_m\Vert w_{\theta'} f^m\Vert_\infty}{\lambda}     +\e\Big)   \lambda\int_0^t |e^{-\lambda s\langle v\rangle} w_{\tilde{\theta}} f^{m}|_{p,+}^p                       \\
&+C(p,T_M,r_{min})\frac{4C(\e)}{\lambda}\frac{\lambda}{4}\int_0^t \Vert e^{-\lambda s\langle v\rangle }\langle v\rangle^{1/p}w_{\tilde{\theta}} \alpha_{f^{m-1},\epsilon}^\beta \partial f^{m}\Vert_p^p)\\
&+tC(p,T_M,r_{min})C(\e)\times (\sup_m \Vert w_{\theta'}f^m\Vert_\infty+\mathcal{E}^m)\lambda  \\
&\times\bigg[ \sup_{l=m,m-1}\sup_{0\leq s\leq t}\Big(\Vert e^{-\lambda s\langle v\rangle}w_{\tilde{\theta}}f^{l}\Vert_p^p+\Vert e^{-\lambda s\langle v\rangle}w_{\tilde{\theta}}\alpha^\beta_{f^{l-1},\e}\partial f^l\Vert_p^p \Big)+\Vert e^{-\lambda s\langle v\rangle}w_{\tilde{\theta}}\alpha^\beta_{f^{m},\e}\partial f^{m+1}\Vert_p^p\bigg].
\end{split}
\end{equation}
Recall the definition of $\mathcal{E}^m$ in~\eqref{induc_hypo}, we have
\begin{equation}\label{eqn: final final W1p}
\begin{split}
  \eqref{eqn: final result of W1p}  &\leq C(p,T_M,r_{min})C(\e)\lambda\big(\Vert w_{\tilde{\theta}} f(0)\Vert_p^p+ \Vert w_{\tilde{\theta}}\alpha_{f,\epsilon}^\beta \partial f(0)\Vert_p^p\big)+\mathcal{E}^m\times C(p,T_M,r_{min}) \\
    & \Big[\e+\big(\frac{C(\e)\sup_m\Vert w_{\theta'} f^m\Vert_\infty}{\lambda}     +\e\big)+\frac{4C(\e)}{\lambda}+t\times C(\e)(\sup_m \Vert w_{\theta'}f^m\Vert_\infty+\mathcal{E}^m)\lambda \Big]\\
    &+C(p,T_M,r_{min})tC(\e)(\sup_m \Vert w_{\theta'}f^m\Vert_\infty+\mathcal{E}^m)\lambda \Vert e^{-\lambda s\langle v\rangle}w_{\tilde{\theta}}\alpha^\beta_{f^{m},\e}\partial f^{m+1}\Vert_p^p .
\end{split}
\end{equation}

First we take $\e=\e(p,T_M,r_{min})\ll 1$ such that $2\e C(p,T_M,r_{min})\leq \frac{1}{10}$. Then with $\e$ fixed we let $\lambda=\lambda(p,T_M,r_{min},\e)\gg 1$ satisfy
\begin{equation}\label{eqn: lambda condition 3}
C(p,T_M,r_{min})\times\Big(\frac{C(\e)\sup_m\Vert w_{\theta'} f^m\Vert_\infty}{\lambda}+\frac{4C(\e)}{\lambda}\Big)\leq \frac{1}{10}.
\end{equation}
Then with $\e,\lambda$ fixed we can define the constant $C_W$ in~\eqref{induc_hypo} as
\begin{equation}\label{eqn: C_W}
  C_W:=C(p,T_M,r_{min})C(\e)\lambda\gg 1,
\end{equation}
where $C(p,T_M,r_{min})C(\e)\lambda$ is the coefficient for the first term in the RHS of~\eqref{eqn: final final W1p}.

Last we take $t_{W}=t_{W}(p,T_M,r_{min},\e,\lambda,C_W,f_0)$ small with $t\leq t_{W}$ and apply the assumption in~\eqref{induc_hypo} such that
\begin{equation}\label{eqn: lesssim}
t_W\times C(p,T_M,r_{min})C(\e)\times(\sup_m\Vert w_{\theta'} f^m\Vert_\infty+\mathcal{E}^m)\lambda \leq \frac{1}{10}.
\end{equation}
Finally collecting~\eqref{eqn: final result of W1p},~\eqref{eqn: lambda condition 3},~\eqref{eqn: C_W} and~\eqref{eqn: lesssim}, since~\eqref{eqn: final result of W1p} holds for all $0<t\leq t_W$, we obtain
\begin{align*}
   &  \sup_{t\leq t_W}\Big(\lambda\Vert e^{-\lambda t\langle v\rangle} w_{\tilde{\theta}} f^{m+1}(t)\Vert_p^p+  \frac{9}{10}\Vert e^{-\lambda t\langle v\rangle}w_{\tilde{\theta}}\alpha_{f^m,\epsilon}^\beta \partial f^{m+1}(t)\Vert_p^p + \frac{\lambda}{4}\int_0^t\Vert e^{-\lambda s\langle v\rangle} \langle v\rangle^{1/p} w_{\tilde{\theta}}\alpha_{f^m,\epsilon}^\beta \partial f^{m+1}(s)\Vert_p^p\\
   & +\int_0^t | e^{-\lambda s\langle v\rangle} w_{\tilde{\theta}}\alpha_{f^m,\epsilon}^\beta \partial f^{m+1}(s)|_{p,+}^p+ \lambda\int_0^t |e^{-\lambda s\langle v\rangle} w_{\tilde{\theta}} f^{m+1}|_{p,+}^p\Big)\\
   &\leq C(p,T_M,r_{min})C(\e)\lambda\big(\Vert w_{\tilde{\theta}} f(0)\Vert_p^p+ \Vert w_{\tilde{\theta}}\alpha_{f,\epsilon}^\beta \partial f(0)\Vert_p^p \big)+\frac{3}{10}\sup_{t\leq t_W}\mathcal{E}^m(t)\\
   &\leq (C_W+\frac{3\times 2}{10}C_W)\big(\Vert w_{\tilde{\theta}} f(0)\Vert_p^p+ \Vert w_{\tilde{\theta}}\alpha_{f,\epsilon}^\beta \partial f(0)\Vert_p^p \big)\leq 2C_W \big(\Vert w_{\tilde{\theta}} f(0)\Vert_p^p+ \Vert w_{\tilde{\theta}}\alpha_{f,\epsilon}^\beta \partial f(0)\Vert_p^p \big).
\end{align*}

Thus we prove~\eqref{induc_hypo} and conclude Proposition \ref{Prop W1p}.

\begin{description}
\item[Step 4: estimate of~\eqref{eqn: Green's iden for partial f}$\mathcal{G}^m$]
\end{description}

First we consider~\eqref{eqn: Green's iden for partial f}$\mathcal{G}^m$. Directly the first two terms $|\nabla_x f^{m+1}|+|\nabla^2\phi^m||\nabla_v f^{m+1}|$ of~\eqref{eqn: formula for G} in~\eqref{eqn: Green's iden for partial f} is bounded by
\begin{equation}\label{eqn: first two terms of G}
  (1+\Vert \nabla^2 \phi^m\Vert_\infty )\int_0^t \Vert  e^{-\lambda s\langle v\rangle}w_{\tilde{\theta}}\alpha_{f^m,\epsilon}^\beta \partial f^{m+1}(s)\Vert_p^p.
\end{equation}
From~\eqref{bound_Gamma_nabla_vf1}~\eqref{bound_Gamma_nabla_vf2}, the contribution of
\[|\Gamma_{\text{gain}}(\partial f^m,f^m)|+|\Gamma_{\text{loss}}(\partial f^m,f^{m+1})|+|\Gamma_{\text{gain}}(f^m,\partial f^m)|\]
of~\eqref{eqn: formula for G} in~\eqref{eqn: Green's iden for partial f}$\mathcal{G}^m$ is bounded by
\begin{equation}\label{eqn: bound gamma(parf,f) and gamma_gain(f,parf)}
\begin{split}
   & \Big(1+\Vert w_{\theta'} f^m\Vert_\infty + \Vert w_{\theta'} f^{m+1}\Vert_\infty \Big)\\
    & \times \int_0^t \iint_{\Omega\times \mathbb{R}^3} |e^{-\lambda s\langle v\rangle}\alpha_{f^m,\epsilon}^\beta w_{\tilde{\theta}}\partial f^{m+1}(v)|^{p-1}\int_{\mathbb{R}^3}\alpha_{f^m,\epsilon}^\beta(v)\mathbf{k}_\rho(v,u)w_{\tilde{\theta}}(v)|\partial f^m(u)|dudvdxds.
\end{split}
\end{equation}
The estimate of~\eqref{eqn: bound gamma(parf,f) and gamma_gain(f,parf)} will be carried out in Step3.

From~\eqref{bound_Gamma_nabla_vf2}, the contribution of $|\Gamma_{\text{loss}}(f^m,\partial f^{m+1})|$ of~\eqref{eqn: formula for G} in~\eqref{eqn: Green's iden for partial f}$\mathcal{G}^m$ is bounded by
\begin{equation}\label{eqn: bound gamma(f,partf)}
 \Vert w_{\theta'} f^m\Vert_\infty \int_0^t \Vert e^{-\lambda s\langle v\rangle}\langle v\rangle^{1/p} w_{\tilde{\theta}}\alpha_{f^m,\epsilon}^\beta \partial f^{m+1}\Vert_p^p.
\end{equation}

From~\eqref{Gvloss}, the contribution of $|\Gamma_{v,\text{loss}}(f^m,f^{m+1})|$ of~\eqref{eqn: formula for G} in~\eqref{eqn: Green's iden for partial f}$\mathcal{G}^m$ is bounded by
\begin{equation}\label{eqn: bound gamma_v loss of G}
\begin{split}
& \Vert w_{\theta'} f^{m+1}\Vert_\infty \int_0^t\iint_{\Omega\times \mathbb{R}^3}pe^{-\lambda (p-1)s\langle v\rangle}|\alpha_{f^{m},\epsilon}^\beta w_{\tilde{\theta}}\partial f^{m+1}|^{p-1} e^{-\lambda s\langle v\rangle}\alpha_{f^m,\epsilon}^\beta \langle v\rangle \frac{w_{\tilde{\theta}}}{w_{\theta'}}\Vert e^{-\lambda s\langle u\rangle }w_{\tilde{\theta}}(u) f^m(s,x,u) \Vert_{L^p(\mathbb{R}^3)} \\
& \lesssim  \Vert w_{\theta'} f^{m+1}\Vert_\infty\Big(\int_0^t \iint_{\Omega\times \mathbb{R}^3}|e^{-\lambda s\langle v\rangle}\alpha_{f^m,\epsilon}^\beta \partial f^{m+1}|^p +\int_0^t\iint_{\Omega\times \mathbb{R}^3}|e^{-\lambda s \langle u\rangle}w_{\tilde{\theta}}(u) f^m(u)|^p\Big),
\end{split}
\end{equation}
where we use
\[e^{-\lambda s\langle v\rangle}\alpha^\beta_{f^m,\epsilon}\langle v\rangle \frac{w_{\tilde{\theta}}}{w_{\theta'}}\lesssim w_{\theta'}^{-1/2}.\]

From~\eqref{Gvgain}, the contribution of $|\Gamma_{v,\text{gain}}(f^m,f^m)|$ in~\eqref{eqn: Green's iden for partial f}$\mathcal{G}$ is bounded by
\begin{equation}\label{eqn: bound Gamma_v gain of G}
\begin{split}
   &    \Vert w_{\theta'} f^m\Vert_\infty \int_0^t \iint_{\Omega\times \mathbb{R}^3} e^{-\lambda (p-1)s\langle v\rangle} \alpha_{f^{m},\epsilon}^{\beta p}|w_{\tilde{\theta}} \partial f^{m+1}(v)|^{p-1}\int_{\mathbb{R}^3}\mathbf{k}_\rho(v,u)\frac{e^{\lambda s\langle u\rangle }w_{\tilde{\theta}}(v)}{e^{\lambda s\langle v\rangle }w_{\tilde{\theta}}(u)}e^{-\lambda s\langle u\rangle }w_{\tilde{\theta}}(u)|f^m(u)|
\\
    & \lesssim_p \Vert w_{\theta'} f^m\Vert_\infty\Big( \int_0^t \Vert \langle v\rangle^{1/p} e^{-\lambda s \langle v\rangle} \alpha_{f^m,\epsilon}^\beta w_{\tilde{\theta}}\partial f^{m+1}\Vert_p^p+\int_0^t \Vert e^{-\lambda s\langle v\rangle } w_{\tilde{\theta}}f^m\Vert_p^p    \Big),
\end{split}
\end{equation}
where we have used, for $1/p+1/p^*=1$ and $0<\tilde{\rho}\ll \rho$, from~\eqref{k_theta_comparision} and~\eqref{grad_estimate}
\begin{align*}
   & \int_0^t \iint_{\Omega\times \mathbb{R}^3} e^{-\lambda (p-1)s\langle v\rangle} \alpha_{f^m,\epsilon}^{\beta p}|w_{\tilde{\theta}} \partial f^{m+1}(v)|^{p-1}\int_{\mathbb{R}^3}\mathbf{k}_\rho(v,u)\frac{e^{\lambda s\langle u\rangle }w_{\tilde{\theta}}(v)}{e^{\lambda s\langle v\rangle }w_{\tilde{\theta}}(u)}e^{-\lambda s \langle u\rangle}w_{\tilde{\theta}}(u)|f^m(u)| \\
   & \leq \int_0^t \iint_{\Omega\times \mathbb{R}^3}   \frac{\alpha_{f^{m},\epsilon}^\beta(v)}{\langle v\rangle^{(p-1)/p}}  |\langle v \rangle^{1/p} e^{-\lambda s\langle v\rangle}\alpha^\beta_{f^{m},\epsilon}w_{\tilde{\theta}}\partial f^{m+1}|^{p-1}\\
   &\times \Big(\int_{\mathbb{R}^3}\mathbf{k}_{\tilde{\rho}}(v,u)du \Big)^{1/p^*}          \Big(\mathbf{k}_{\tilde{\rho}}(v,u)|e^{-\lambda s\langle u\rangle}w_{\tilde{\theta}}f^m(u)|^p du \Big)^{1/p}          dv\\
   &\leq \int_0^t \int_{\Omega}\Big(\int_{\mathbb{R}^3}  |\langle v \rangle^{1/p} e^{-\lambda s\langle v\rangle}\alpha^\beta_{f^{m},\epsilon}w_{\tilde{\theta}}\partial f^{m+1}|^{p} \Big)^{\frac{p-1}{p}}   \Big(\int_{\mathbb{R}^3}\int_{\mathbb{R}^3}\mathbf{k}_{\tilde{\rho}}(v,u)|e^{-\lambda s\langle u\rangle}w_{\tilde{\theta}}f^m(u)|^p \Big)^{1/p}\\
   &\lesssim_p \int_0^t \iint_{\Omega\times \mathbb{R}^3} |e^{-\lambda s\langle v\rangle}\langle v\rangle^{1/p} \alpha_{f^{m},\epsilon}^\beta w_{\tilde{\theta}}\partial f^{m+1}|^p   +|e^{-\lambda s\langle u\rangle} w_{\tilde{\theta}} f^{m}|^p.
\end{align*}
In the last step we have applied the Young's inequality.

We focus on~\eqref{eqn: bound gamma(parf,f) and gamma_gain(f,parf)}. We split the u-integration of~\eqref{eqn: bound gamma(parf,f) and gamma_gain(f,parf)} into the integration over $\{|u|\leq N\}$ and $\{|u|>N\}$.

The contribution of $\{|u|\geq N\}$ in~\eqref{eqn: bound gamma(parf,f) and gamma_gain(f,parf)} is bounded by
\begin{equation}\label{eqn: bound gamma(parf,f) and gamma_gain(f,parf) u large}
\begin{split}
   & \int_0^t\int_{\Omega\times {\mathbb{R}^3}} |e^{-\lambda s\langle v\rangle }\langle v\rangle^{1/p} w_{\tilde{\theta}}\alpha_{f^{m},\epsilon}^\beta \partial f^{m+1}(v)|^{p-1}\frac{\alpha_{f^{m},\epsilon}^\beta }{\langle v\rangle^{p/(p-1)}}\\
    & \times \int_{|u|\geq N}\mathbf{k}_\rho(v,u)\frac{w_{\tilde{\theta}}(v)}{w_{\tilde{\theta}}(u)}\frac{e^{-\lambda s\langle v\rangle }}{e^{-\lambda s\langle u\rangle }}|e^{-\lambda s\langle u\rangle }w_{\tilde{\theta}}\partial f^m(u)|dudvdxds\\
    & \leq \int_0^t\int_{\Omega}\left(\int_{\mathbb{R}^3} |e^{-\lambda s\langle v\rangle }\langle v\rangle^{1/p} w_{\tilde{\theta}}\alpha_{f^{m},\epsilon}^\beta \partial f^{m+1}(v)|^{p} \right)^{1/p^*}\left( \int_{|u|\geq N} |e^{-\lambda s\langle u\rangle }w_{\tilde{\theta}} \alpha_{f^{m-1},\epsilon}^\beta \partial f^m(u)|^p \int_v \mathbf{k}_{\tilde{\rho}}(v,u)   \right)^{1/p}\\
    & \leq \int_0^t \Vert e^{-\lambda s\langle v\rangle }\langle v\rangle^{1/p}w_{\tilde{\theta}}\alpha_{f^m,\epsilon}^\beta \partial f^{m+1}(s)\Vert_p^p ds+\int_0^t \Vert e^{-\lambda s\langle v\rangle }w_{\tilde{\theta}}\alpha_{f^{m-1},\epsilon}^\beta \partial f^m(s)\Vert_p^p ds,
\end{split}
\end{equation}
where we use H\"{o}lder inequality, Proposition~\ref{prop_int_alpha} with $\beta \frac{p}{p-1}<1$, $\frac{\alpha_{f^{m},\epsilon}^\beta}{\langle v\rangle^{p/(p-1)}}\leq 1$. And we apply~\eqref{k_theta_comparision} to get
\begin{align*}
   & \int_{|u|\geq N}\mathbf{k}_\rho(v,u)\frac{w_{\tilde{\theta}}(v)}{w_{\tilde{\theta}}(u)}\frac{e^{-\lambda s\langle v\rangle }}{e^{-\lambda s\langle u\rangle }}|e^{-\lambda s\langle u\rangle }w_{\tilde{\theta}}\partial f^m(u)|du \\
   & \leq \int_{|u|\geq N}\mathbf{k}_{\tilde{\rho}}(v,u)\frac{1}{\alpha_{f^{m-1},\epsilon}^\beta}|e^{-\lambda s\langle u\rangle }\alpha_{f^{m-1},\epsilon}^\beta w_{\tilde{\theta}}(u)\partial f^m(u)|du\\
   &\lesssim \left( \int_{|u|\geq N} \mathbf{k}_{\tilde{\rho}}(v,u) \frac{1}{\alpha_{f^{m-1},\epsilon}^{\beta p^*}(u)}\right)^{1/p^*} \left(\int_{|u|\geq N}\mathbf{k}_{\tilde{\rho}}(v,u)|e^{-\lambda s\langle u\rangle }\alpha_{f^{m-1},\epsilon}^\beta w_{\tilde{\theta}}(u)\partial f^m(u)|^p \right)^{1/p}\\
   &\lesssim \left(\int_{|u|\geq N}\mathbf{k}_{\tilde{\rho}}(v,u)|e^{-\lambda s\langle u\rangle }\alpha_{f^{m-1},\epsilon}^\beta w_{\tilde{\theta}}(u)\partial f^m(u)|^p \right)^{1/p}.
\end{align*}

The contribution of $\{|u|\leq N\}$ in~\eqref{eqn: bound gamma(parf,f) and gamma_gain(f,parf)} is bounded by, from H\"{o}lder inequality,
\begin{align*}
   & \int_{0}^t \iint_{\Omega\times \mathbb{R}^3} |e^{-\lambda s\langle v\rangle }\langle v\rangle^{1/p} w_{\tilde{\theta}}\alpha_{f^{m},\epsilon}(v)^\beta \partial f^{m+1}(v)|^{p-1} \\
   & \times \int_{|u|\leq N}\mathbf{k}_\rho(v,u)\frac{w_{\tilde{\theta}}(v)}{w_{\tilde{\theta}}(u)}\frac{e^{-\lambda s\langle v\rangle }}{e^{-\lambda s\langle u\rangle }}\frac{\alpha_{f^m,\epsilon}(v)^\beta |e^{-\lambda s\langle u\rangle } w_{\tilde{\theta}}\alpha_{f^{m-1},\epsilon}^\beta \partial f^m(u)|}{\langle v\rangle^{(p-1)/p}\alpha_{f^{m-1},\epsilon}^\beta(u)}dudvdxds\\
   &\leq \int_{0}^t \Vert e^{-\lambda s\langle v\rangle } \langle v\rangle^{1/p} w_{\tilde{\theta}}\alpha_{f^{m},\epsilon}^\beta \partial f^{m+1}(v)\Vert_p^{p-1}
\end{align*}
\begin{equation}\label{eqn: bound gamma(part,f) u bounded part 1}
\times \bigg[\iint_{\Omega\times \mathbb{R
}^3}\left(\int_{|u|\leq N} \mathbf{k}_{\tilde{\rho}}(v,u) \frac{|e^{-\lambda s \langle u\rangle}w_{\tilde{\theta}}\alpha_{f^{m-1},\epsilon}^\beta \partial f^m(u)|}{\alpha_{f^{m-1},\epsilon}^\beta(u)}du \right)^pdvdx \bigg]^{1/p} ds.
\end{equation}
By the H\"{o}lder inequality, the u-integration part of~\eqref{eqn: bound gamma(part,f) u bounded part 1} as
\begin{equation}\label{eqn: bound gamma(part,f) u bounded part 2}
  \Vert e^{-\lambda s\langle v\rangle }w_{\tilde{\theta}}\alpha_{f^{m-1},\epsilon}^\beta \partial f^m(\cdot)\Vert_{L^p(\mathbb{R}^3)}\times \left(\int_{\mathbb{R}^3} \frac{e^{-p^* \tilde{\rho}|v-u|^2}}{|v-u|^{p^*}}\frac{\mathbf{1}_{|u|\leq N}}{\alpha_{f^{m-1},\epsilon}^{\beta p^*}(u)} \right)^{1/p^*}.
\end{equation}
Note that
\[\left(\int_{\mathbb{R}^3}\frac{e^{-p^*\tilde{\rho}|v-u|^2}}{|v-u|^{p^*}}\frac{\mathbf{1}_{|u|\leq N}}{\alpha_{f^{m-1},\epsilon}^{\beta p^*}(u)}  \right)^{1/p^*}\leq \Big|\frac{1}{|\cdot|^{p^*}}* \frac{\mathbf{1}_{|\cdot|
\leq N}}{\alpha_{f^{m-1},\epsilon}^{\beta p^*}}\Big|^{1/p^*}.\]
By the Hardy-Littlewood-Sobolev inequality with
\[1+\frac{1}{p/p^*}=\frac{1}{3/p^*}+\frac{1}{\frac{3}{2}\frac{p-1}{p}},\]
we have
\begin{equation}\label{eqn: bound gamma(part,f) u bounded part 3}
\begin{split}
   & \Big\Vert \big|\frac{1}{|\cdot|^{p^*}}* \frac{\mathbf{1}_{|\cdot|\leq N}}{\alpha_{f^{m-1},\epsilon}^{\beta p^*}(u)}\big| \Big\Vert_{L^p(\mathbb{R}^3)}= \Big\Vert \frac{1}{|\cdot|^{p^*}}* \frac{\mathbf{1}_{|\cdot|\leq N}}{\alpha_{f^{m-1},\epsilon}^{\beta p^*}(u)}\Big\Vert_{L^{p/p^*}(\mathbb{R}^3)} \\
    & \lesssim \Vert \frac{\mathbf{1}_{|
   \cdot|\leq N}}{\alpha_{f^{m-1},\epsilon}^{\beta p^*}(\cdot)}\Vert_{L^{\frac{3(p-1)}{2p}}(\mathbb{R}^3)} \lesssim \left(\int_{\mathbb{R}^3} \frac{\mathbf{1}_{|v|\leq N}}{\alpha_{f^{m-1},\epsilon}(v)^{\frac{p}{p-1}\beta \frac{3(p-1)}{2p}}}dv \right)^{\frac{2p}{3(p-1)}\frac{p-1}{p}}\\
   & \lesssim \left(\int_{\mathbb{R}^3} \frac{\mathbf{1}_{|v|\leq N}}{\alpha_{f^{m-1},\epsilon}^{3\beta/2}(v)} \right)^{2/3}\lesssim 1,
\end{split}
\end{equation}
where we use $3\beta/2<1$ and Proposition~\ref{prop_int_alpha}. Using~\eqref{eqn: bound gamma(part,f) u bounded part 3}~\eqref{eqn: bound gamma(part,f) u bounded part 2}~\eqref{eqn: bound gamma(parf,f) and gamma_gain(f,parf) u large} we have
\begin{equation}\label{eqn: bound gamma(parf,f)gamma_gain(f,parf) final result}
\eqref{eqn: bound gamma(parf,f) and gamma_gain(f,parf)}\lesssim (1+\sup_m\Vert w_{\theta'} f^m\Vert_\infty)\Big[ \int_0^t \Vert e^{-\lambda s\langle v\rangle }\langle v\rangle^{1/p}w_{\tilde{\theta}}\alpha_{f^m,\epsilon}^\beta \partial f^{m+1}\Vert_p^p+\int_0^t \Vert e^{-\lambda s\langle v\rangle }w_{\tilde{\theta}}\alpha_{f^{m-1},\epsilon}^\beta \partial f^{m}\Vert_p^p\Big].
\end{equation}
Finally from~\eqref{eqn: first two terms of G}~\eqref{eqn: bound gamma(f,partf)}~\eqref{eqn: bound gamma(parf,f) and gamma_gain(f,parf) u large}~\eqref{eqn: bound Gamma_v gain of G}~\eqref{eqn: bound gamma_v loss of G}~\eqref{eqn: bound gamma(parf,f)gamma_gain(f,parf) final result},
~\eqref{eqn: Green's iden for partial f}$\mathcal{G}^m$ has a bound as
\begin{equation}\label{eqn: bound of G}
\begin{split}
   & C(p)\bigg[\Big(1+\sup_m\Vert w_{\theta'} f^{m}\Vert_\infty\Big)\int_0^t \Vert e^{-\lambda s\langle v\rangle }\langle v\rangle^{1/p}w_{\tilde{\theta}} \alpha_{f^m,\epsilon}^\beta \partial f^{m+1}\Vert_p^p \\
    & +\Big(1+\sup_m\Vert w_{\theta'} f^m\Vert_\infty+\Vert \nabla^2 \phi^m\Vert_\infty\Big)\int_0^t \Big(\Vert e^{-\lambda s\langle v\rangle }w_{\tilde{\theta}}\alpha_{f^{m-1},\epsilon}^\beta \partial f^m  \Vert_p^p+\Vert e^{-\lambda s\langle v\rangle }w_{\tilde{\theta}}\alpha_{f^{m},\epsilon}^\beta \partial f^{m+1}  \Vert_p^p\Big)\\
   & +\Big(1+\sup_m\Vert w_{\theta'} f^{m}\Vert_\infty\Big)\int_0^t \Vert e^{-\lambda s\langle v\rangle}w_{\tilde{\theta}}f^m\Vert_p^p\bigg].
\end{split}
\end{equation}
In order to control the first line in~\eqref{eqn: bound of G} by $\frac{\lambda}{2}\int_0^t\Vert e^{-\lambda s\langle v\rangle} \langle v\rangle^{1/p} w_{\tilde{\theta}}\alpha_{f^m,\epsilon}^\beta \partial f^{m+1}\Vert_p^p$ in~\eqref{eqn: Green's iden for partial f}, we require the $\lambda$ satisfy
\begin{equation}\label{eqn: lambda condition 2}
\frac{\lambda}{4}\geq C(p)(1+\sup_m\Vert w_{\theta'}f^{m}\Vert_\infty ).
\end{equation}

\begin{description}
\item[Step 5: estimate of~\eqref{eqn: Green's iden for partial f}$\gamma_-$]
\end{description}
We focus on~\eqref{eqn: Green's iden for partial f}$\gamma_-$. The overall strategy is similar to~\eqref{eqn: f on nega bound}. From~\eqref{eqn: derivative bound on the boundary first part}~\eqref{eqn: derivative bound on the boundary second part}
\begin{equation}\label{eqn: negative bdr in green's indentity}
\begin{split}
   &\int_{0}^{t}|e^{-\lambda s \langle v\rangle}w_{\tilde{\theta}}\alpha_{f^m,\epsilon}^\beta \partial f^{m+1}|^p_{p,-}  \\
    & =\int_0^t \int_{\partial \Omega}\int_{n(x)\cdot v<0} |n(x)\cdot v|^{\beta p} |e^{-\lambda s \langle v\rangle}w_{\tilde{\theta}}\nabla_{x,v}f^{m+1}(t,x,v)|^p |n(x)\cdot v|dv\\
    &\lesssim \int_0^t \int_{\partial \Omega}\int_{n(x)\cdot v<0}  \langle v\rangle^{2p} e^{-\lambda ps \langle v\rangle}e^{p[\frac{1}{4T_M}-\frac{1}{2T_w(x)}]|v|^2} w_{\tilde{\theta}}^p \left(|n(x)\cdot v|^{\beta p+1}+|n(x)\cdot v|^{(\beta-1)p+1}\right)\times|\eqref{eqn: derivative bound on the boundary second part} |^pdv.
\end{split}
\end{equation}

Now we bound $|\eqref{eqn: derivative bound on the boundary second part}|^p$.

\begin{itemize}
\item[--]
First line of~\eqref{eqn: derivative bound on the boundary second part}, we split the $u$-integration into $\gamma_{+,2}^{v,x,\epsilon}\cup \Big(\gamma_+(x)\backslash\gamma_{+,2}^{v,x,\epsilon}\Big)$, where
\end{itemize}
\begin{equation}\label{eqn: gamma_+^v 2}
\begin{split}
   & \gamma_{+,2}^{v,x,\epsilon}=\{(x,u)\in \gamma_+:|n(x)\cdot u|\leq \epsilon \text{ or }|u_\parallel-\frac{2T_\zeta(1-r_\parallel)}{2T_\zeta+(T_w(x)-2T_\zeta)r_\parallel(2-r_\parallel)}v_\parallel|\geq \epsilon^{-1}  \\
    & \text{ or }|u_\perp-\frac{2T_\zeta\sqrt{1-r_\perp}}{2T_\zeta+\big(T_w(x)-2T_\zeta\big)r_\perp}v_\perp|\geq \epsilon^{-1}\},
\end{split}
\end{equation}
and $T_\zeta$ will be defined later in~\eqref{eqn: Tzeta}.

By the H\"{o}lder inequality
\[\left(\int_{n(x)\cdot u>0}|e^{-\lambda s \langle u\rangle}w_{\tilde{\theta}} \alpha_{f^{m-1},\epsilon}^\beta \nabla_{x,v}f^m(s,x,u)|\{e^{\lambda s \langle v\rangle}w_{\tilde{\theta}}^{-1}\alpha^{-\beta}_{f^{m-1},\epsilon}(u)\} \langle u\rangle e^{-[\frac{1}{4T_M}-\frac{1}{2T_w(x)}]|u|^2} d\sigma(u,v)   \right)^p\]
\begin{equation}\label{eqn: first term of the boundary derivative 1st}
\begin{split}
   & \lesssim\left(\int_{\gamma_{+,2}^{v,x,\epsilon}(x)} |e^{-\lambda s \langle u\rangle}w_{\tilde{\theta}} \alpha_{f^{m-1},\epsilon}^\beta \nabla_{x,v}f^m(s,x,u)|^p \{n(x)\cdot u\}  du\right) \\
    & \times     \bigg(      \int_{\gamma_{+,2}^{v,x,\epsilon}(x)}\underbrace{\{e^{-\lambda s \langle u\rangle} w_{\tilde{\theta}} \alpha^\beta_{f^{m-1},\epsilon}(u)\}^{- p^*}}_{(I)} |n(x)\cdot u|\langle u\rangle^{p^*} e^{-p^*[\frac{1}{4T_M}-\frac{1}{2T_w(x)}]|u|^2} \\
   &\times I_0 \Big( \frac{ (1-r_\perp)^{1/2} u_\perp v_\perp}{T_w(x)r_\perp} \Big)^{p^*}e^{-\frac{p^*}{2T_w(x)}\big[\frac{|u_\perp|^2+(1-r_\perp)|v_\perp|^2}{r_\perp}+\frac{|u_\parallel-(1-r_\parallel)v_\parallel|^2}{r_\parallel(2-r_\parallel)}\big]}           du \bigg)^{p/p^*}
\end{split}
\end{equation}

\begin{equation}\label{eqn: first term of the boundary derivative 2nd}
\begin{split}
& + \left(  \int_{\gamma_+(x)\backslash \gamma_{+,2}^{v,x,\epsilon}(x)} |e^{-\lambda s\langle u\rangle}w_{\tilde{\theta}}\alpha_{f^{m-1},\epsilon}^\beta \nabla_{x,v}f^m(s,x,u)|^p \{n(x)\cdot u\}du  \right) \\
 &  \times \bigg(\int_{\gamma_+(x)\backslash\gamma_{+,2}^{v,x,\epsilon}(x)}  \{e^{-\lambda s\langle u\rangle}w_{\tilde{\theta}}\alpha^\beta_{f^m,\epsilon}(s,x,u)\}^{- p^*} |n(x)\cdot u|\langle u\rangle^{p^*} e^{-p^*[\frac{1}{4T_M}-\frac{1}{2T_w(x)}]|u|^2} \\
 &  \times I_0 \Big( \frac{ (1-r_\perp)^{1/2} u_\perp v_\perp}{T_w(x)r_\perp} \Big)^{p^*}e^{-\frac{p^*}{2T_w(x)}\big[\frac{|u_\perp|^2+(1-r_\perp)|v_\perp|^2}{r_\perp}+\frac{|u_\parallel-(1-r_\parallel)v_\parallel|^2}{r_\parallel(2-r_\parallel)}\big]}           du\bigg)^{p/p^*}.
\end{split}
\end{equation}
Similar to Step 1, we separate the discussion of~\eqref{eqn: first term of the boundary derivative 1st} and~\eqref{eqn: first term of the boundary derivative 2nd}.
\begin{itemize}
\item[--]{estimate of~\eqref{eqn: first term of the boundary derivative 1st}}.
\begin{itemize}
\item[--]
To compute the $u$-integration, for any $c>0$ we bound
\begin{equation}\label{eqn: bound use zeta}
e^{\lambda sp^* \langle u\rangle} |n(x)\cdot u|^{c} \langle u\rangle^{p^*}\lesssim e^{p^* \zeta|u|^2},
\end{equation}
where $\zeta$ will be defined later in~\eqref{eqn: zeta}. Then we introduce $c_1>1$ with $1=\frac{1}{c_1}+\frac{1}{c_1^*}$ to deal with the $\alpha_{f^{m-1},\e}$ in (I). Then the $u$-integration is
\begin{equation}\label{eqn: bound use c1}
\begin{split}
& \lesssim\int_{\gamma_{+,2}^{v,x,\epsilon}(x)}\{e^{-\lambda s\langle u\rangle} w_{\tilde{\theta}} \alpha^\beta_{f^{m-1},\epsilon}(u)\}^{- p^*} |n(x)\cdot u|^{1/c_1}|n(x)\cdot u|^{1/c_1^*}\langle u\rangle^{p^*} e^{-p^*[\frac{1}{4T_M}-\frac{1}{2T_w(x)}]|u|^2}\\
&\times I_0 \Big( \frac{ (1-r_\perp)^{1/2} u_\perp v_\perp}{T_w(x)r_\perp} \Big)^{p^*}e^{-\frac{p^*}{2T_w(x)}\big[\frac{|u_\perp|^2+(1-r_\perp)|v_\perp|^2}{r_\perp}+\frac{|u_\parallel-(1-r_\parallel)v_\parallel|^2}{r_\parallel(2-r_\parallel)}\big]}           du\\
& \lesssim \int_{\gamma_{+,2}^{v,x,\epsilon}(x)}\Big[ w_{\tilde{\theta}} \alpha^\beta_{f^{m-1},\epsilon}(u)\Big]^{- p^*}    |n(x)\cdot u|^{1/c_1^*}e^{p^* \zeta |u|^2} e^{-p^*\big[\frac{1}{4T_M}-\frac{1}{2T_w(x)}\big]|u|^2} \\
&\times I_0 \Big( p^*\frac{ (1-r_\perp)^{1/2} u_\perp v_\perp}{T_w(x)r_\perp} \Big)e^{-\frac{p^*}{2T_w(x)}\big[\frac{|u_\perp|^2+(1-r_\perp)|v_\perp|^2}{r_\perp}+\frac{|u_\parallel-(1-r_\parallel)v_\parallel|^2}{r_\parallel(2-r_\parallel)}\big]}           du,
\end{split}
\end{equation}
where we have applied~\eqref{eqn: bound use zeta}. Applying the H\"{o}lder inequality once more with $1=\frac{1}{c_1}+\frac{1}{c_1^*}$, we obtain
\begin{align}
\eqref{eqn: bound use c1}   &\lesssim\Big(\int_{\gamma_{+,2}^{v,x,\epsilon}(x)}  \big[ w_{\tilde{\theta}}\alpha_{f^{m-1},\epsilon}^\beta(u)\big]^{-p^*c_1}     du\Big)^{\frac{1}{c_1}}\label{eqn: alpha betapc_1 bounded}  \\
   & \times \bigg(\int_{\gamma_{+,2}^{v,x,\epsilon}(x)}|n(x)\cdot u| e^{-c_1^*p^*[\frac{1}{4T_M}-\frac{1}{2T_w(x)}-\zeta]|u|^2}
I_0 \Big( c_1^*p^*\frac{ (1-r_\perp)^{1/2} u_\perp v_\perp}{T_w(x)r_\perp} \Big) \\
& e^{-\frac{c_1^*p^*}{2T_w(x)}\big[\frac{|u_\perp|^2+(1-r_\perp)|v_\perp|^2}{r_\perp}+\frac{|u_\parallel-(1-r_\parallel)v_\parallel|^2}{r_\parallel(2-r_\parallel)}\big]}            du\bigg)^{\frac{1}{c_1^*}}.\label{eqn: small term of gamma_+}
\end{align}
We choose $c_1$ to be close to $1^+$ to guarantee $\beta p^*c_1<1$. Using Proposition~\ref{prop_int_alpha} with $v=0$ and $w_{\tilde{\theta}}^{-p^* c_1}=e^{-\tilde{\theta}p^* c_1 |u|^2}$, we have$~\eqref{eqn: alpha betapc_1 bounded}\lesssim_p 1$.

For~\eqref{eqn: small term of gamma_+} we let $\zeta<\frac{1}{4T_M}$ and denote
\begin{equation}\label{eqn: Tzeta}
\frac{1}{4T_{\zeta}}=\frac{1}{4T_M}-\zeta,\quad     T_{\zeta}>T_M.
\end{equation}
By $0<r_{min}\leq 1$, we choose $\zeta=\zeta(T_M,r_{min})$ to be small such that
\begin{equation}\label{eqn: zeta}
2T_{\zeta}(1-r_{min})+T_Mr_{min}<2T_M,\quad \frac{T_M}{T_\zeta}>1/2.
\end{equation}
\end{itemize}
\begin{itemize}
\item[--]To control~\eqref{eqn: small term of gamma_+}, recall the definition of~\eqref{eqn: gamma_+^v 2}. Here we simply replace the $T_M$ in~\eqref{eqn: gamma_+^v} by $T_\zeta$. Thus we can apply the same decomposition as in~\eqref{eqn: three terms} and obtain the result as~\eqref{eqn: fl} in Step 1 with replacing $T_M$ by $T_{\zeta}$. We get
\begin{equation}\label{eqn: bound of small term of gamma+}
\eqref{eqn: small term of gamma_+}\lesssim \e\exp\Big(\frac{2T_{\zeta}-T_w(x)}{2T_w(x)[2T_\zeta(1-r_{min})+r_{min} T_w(x)]}(1-r_{min})p^*|v|^2\Big).
\end{equation}

Thus we obtain
\begin{equation}\label{eqn: extra term after integrating the derivative on the boudnary}
~\eqref{eqn: first term of the boundary derivative 1st}\lesssim_p \exp\Big(\frac{2T_{\zeta}-T_w(x)}{2T_w(x)[2T_{\zeta}(1-r_{min})+r T_w(x)]}(1-r_{min})p|v|^2\Big)
\end{equation}
\[\quad\quad\quad\quad\times \epsilon\int_{\gamma_{+,2}^{v,x,\epsilon}(x)} |e^{-\lambda s \langle u\rangle}w_{\tilde{\theta}} \alpha_{f^{m-1},\epsilon}^\beta \nabla_{x,v}f^m(s,x,u)|^p \{n(x)\cdot u\}  du .\]
\end{itemize}
\begin{itemize}
\item[--]With the integrand controlled, we move to the $v$-integration~\eqref{eqn: negative bdr in green's indentity}.
Plugging~\eqref{eqn: extra term after integrating the derivative on the boudnary} into~\eqref{eqn: negative bdr in green's indentity} we have the boundedness of the integrand:
\begin{equation}\label{eqn: integrand dv}
\begin{split}
   & \langle v\rangle^{2p}e^{-\lambda ps\langle v\rangle} \left(|n(x)\cdot v|^{\beta p+1}+|n(x)\cdot v|^{(\beta-1)p+1}\right) w^p_{\tilde{\theta}} \\
    & \times\exp\bigg(p\big[\frac{1}{4T_M}-\frac{1}{2[2T_{\zeta}(1-r_{min})+r_{min} T_w(x)]}\big]|v|^2\bigg),
\end{split}
\end{equation}
where we apply the same computation as~\eqref{eqn: v integrand for 1}.

By~\eqref{Condition for p}, $(\beta-1)p+1>-1$, thus $|n(x)\cdot v|^{(\beta-1)p+1}\in L_{\text{loc}}^1$. Using~\eqref{eqn: zeta} we derive
\[2[2T_{\zeta}(1-r_{min})+T_w(x)r_{min}]<4T_M,\]
We take $\tilde{\theta}=\tilde{\theta}(\zeta,r_{min},T_M)\ll 1$ to obtain
\begin{equation}\label{eqn: thete tilde conditiopn}
p\tilde{\theta}+p\big[\frac{1}{4T_M}-\frac{1}{2[2T_{\zeta}(1-r_{min})+r_{min} T_w(x)]}\big]<0.
\end{equation}
Thus we derive$~\eqref{eqn: integrand dv}\in L_v^1$.

Therefore, the contribution of~\eqref{eqn: first term of the boundary derivative 1st} in~\eqref{eqn: Green's iden for partial f}$\gamma_-$ is
\begin{equation}\label{eqn: result of the first term in the decompo for bdy deri}
\lesssim_{p,T_M,r_{min}} \e\int_0^t |e^{-\lambda s \langle u\rangle}w_{\tilde{\theta}} \alpha_{f^{m-1},\epsilon}^\beta \nabla_{x,v}f^m|^p_{p,+}.
\end{equation}

\end{itemize}
\end{itemize}

\begin{itemize}
\item[--]{estimate of~\eqref{eqn: first term of the boundary derivative 2nd}}. Similar to the Step 1, the integration in $u$ does not provide a small $\e$. Thus we have
    \[\eqref{eqn: first term of the boundary derivative 2nd}\lesssim\exp\Big(\frac{2T_{\zeta}-T_w(x)}{2T_w(x)[2T_{\zeta}(1-r_{min})+r_{min} T_w(x)]}(1-r_{min})p|v|^2 \Big)\]
\begin{equation}\label{eqn: second term of the first term}
\quad\quad\quad\quad\quad\times\Big(\int_{\gamma_+(x)\backslash\gamma_{+,2}^{v,x,\epsilon}(x)} |e^{-\lambda s \langle u\rangle}w_{\tilde{\theta}} \alpha_{f^{m-1},\epsilon}^\beta \nabla_{x,v}f^m(s,x,u)|^p \{n(x)\cdot u\}  du\Big).
\end{equation}
Plugging~\eqref{eqn: second term of the first term} into~\eqref{eqn: negative bdr in green's indentity} we conclude the integrand is given by$~\eqref{eqn: integrand dv}\in L^1_v$.
    Again we decompose $v$ into $\mathbf{1}_{|v|\leq \epsilon^{-1}}$ and $\mathbf{1}_{|v|> \epsilon^{-1}}$.

\begin{itemize}
\item[--]{When $|v|>\e^{-1}$,} by the exponential decaying function in$~\eqref{eqn: integrand dv}$ the contribution of $\eqref{eqn: first term of the boundary derivative 2nd}\mathbf{1}_{|v|> \epsilon^{-1}}$ in~\eqref{eqn: Green's iden for partial f}$\gamma_-$
\begin{equation}\label{eqn: bound of first line of bdry deri v large}
\lesssim_{p,T_M,r_{min}} \epsilon\int_0^t \int_{\partial \Omega}\eqref{eqn: second term of the first term}\leq \epsilon\int_0^t |e^{-\lambda s \langle u\rangle}w_{\tilde{\theta}} \alpha_{f^m,\epsilon}^\beta \nabla_{x,v}f^m|^p_{p,+}.
\end{equation}

\item[--]{When $|v|\leq\epsilon^{-1}$,} since $u\in\gamma_+\backslash \gamma_{+,2}^{v,x,\e}$, for any $x\in \partial\Omega$ we have
\[|u|\leq 2\e^{-1}+ |\frac{2T_\zeta\sqrt{1-r_\perp}}{2T_\zeta+(T_w(x)-2T_\zeta)r_\perp}v_\parallel|+|\frac{2T_\zeta(1-r_\parallel)}{2T_\zeta+(T_w(x)-2T_\zeta)r_\parallel(2-r_\parallel)}v_\parallel|\leq 10 \e^{-1}.\]
In the derivation we used~\eqref{eqn: Tzeta} and~\eqref{eqn: TM less than 2} to conclude
\[\frac{2T_\zeta\sqrt{1-r_\perp}}{2T_\zeta+(T_w(x)-2T_\zeta)r_\perp}\leq \frac{1}{(1-r_\parallel)+\frac{T_w(x)r_\parallel(2-r_\parallel)}{4T_M(1-r_\parallel)}}\leq 4,\]
and similarly to have
\[\frac{2T_\zeta\sqrt{1-r_\perp}}{2T_\zeta+[T_w(x)-2T_\zeta]r_\perp}\leq 4.\]
Thus $u\in \gamma_+(x)/ \gamma_+^{\epsilon/10}$, where $\gamma_+^{\e/6}$ is defined in~\eqref{eqn: gamma+ e}. From Lemma~\ref{lemma: trace thm} the contribution of $\eqref{eqn: first term of the boundary derivative 2nd}\mathbf{1}_{|v|\leq\epsilon^{-1}}$ in~\eqref{eqn: Green's iden for partial f}$\gamma_-$ is
\begin{equation}\label{eqn: bound of first line of bdry deri v bounded}
\begin{split}
   & \lesssim_\e\int_0^t \int_{\partial \Omega}\int_{\gamma_+(x)/\gamma_+^{\epsilon/10}(x)}|e^{-\lambda s \langle u\rangle}w_{\tilde{\theta}} \alpha_{f^{m-1},\epsilon}^\beta \nabla_{x,v}f^m(s,x,u)|^p \{n(x)\cdot u\}  dudS_x ds \\
    & \lesssim_\e \Vert w_{\tilde{\theta}} \alpha_{f^{m-1},\epsilon}^\beta \nabla_{x,v}f(0)\Vert_p^p+ \int_0^t \Vert e^{-\lambda s \langle v\rangle}w_{\tilde{\theta}} \alpha_{f^{m-1},\epsilon}^\beta \nabla_{x,v}f^m\Vert_p^p + \eqref{eqn: bdy deri bounded by G}
\end{split}
\end{equation}
with
\begin{align}
   &  \int_0^t \iint_{\Omega\times \mathbb{R}^3}[\partial_t +v\cdot \nabla_x -\nabla_x \phi^{m-1} \cdot \nabla_v +\nu_{\phi^{m-1},\lambda,w_{\tilde{\theta}}}]|e^{-\lambda s \langle v\rangle}w_{\tilde{\theta}} \alpha_{f^{m-1},\epsilon}^\beta \nabla_{x,v}f^m|^p\label{eqn: bdy deri bounded by G}\\
   & \leq \int_0^t \iint_{\Omega \times \mathbb{R}^3}p e^{-\lambda ps \langle v\rangle}w^p_{\tilde{\theta}} \alpha_{f^{m-1},\epsilon}^{\beta p} |\nabla_{x,v}f^m|^{p-1}|\mathcal{G}^{m-1}|.\label{eqn: bdy deri bounded by G part2}
\end{align}
Clearly~\eqref{eqn: bdy deri bounded by G part2}$\lesssim$~\eqref{eqn: bound of G} with replacing all $m+1$ by $m$ and $m$ by $m-1$.
\end{itemize}
\end{itemize}

Collecting~\eqref{eqn: result of the first term in the decompo for bdy deri}~\eqref{eqn: bound of first line of bdry deri v large}~\eqref{eqn: bound of first line of bdry deri v bounded}~\eqref{eqn: bdy deri bounded by G}, the contribution of the first line~\eqref{eqn: derivative bound on the boundary second part} in~\eqref{eqn: Green's iden for partial f}$\gamma_-$ is
\begin{equation}\label{eqn: contribution of first line}
\begin{split}
   &\lesssim_{p,T_M,r_{min}} \epsilon\int_0^t |e^{-\lambda s\langle v\rangle}w_{\tilde{\theta}}\alpha_{f^{m-1},\epsilon}^\beta \partial f^{m}|_{p,+}^p+C(\e)\Big[\Vert w_{\tilde{\theta}} \alpha_{f^m,\epsilon}^\beta \nabla_{x,v}f(0)\Vert_p^p\\
    & + \int_0^t \Vert e^{-\lambda s \langle v\rangle}w_{\tilde{\theta}} \alpha_{f^{m-1},\epsilon}^\beta \nabla_{x,v}f^m\Vert_p^p+~\eqref{eqn: Green's iden for partial f}\mathcal{G}^{m-1}\Big],
\end{split}
\end{equation}
where $C(\e)$ comes from~\eqref{eqn: bound of first line of bdry deri v bounded}.

\begin{itemize}
\item[--]{Second line of~\eqref{eqn: derivative bound on the boundary second part}.} By the H\"{o}lder inequality, we have
\end{itemize}
\begin{equation}\label{eqn: second line contribution}
\begin{split}
   & \bigg(\int_{n(x)\cdot u>0} \Big(\langle u\rangle^2 \Vert \nabla_x \phi^{m-1}\Vert_{L^\infty} |f^m|\Big)+\Vert w_{\theta'} f^m\Vert_{L^\infty} \int_{\mathbb{R}^3}\mathbf{k}_{\rho}(u,u')|f^{m-1}(u')|du' e^{-[\frac{1}{4T_M}-\frac{1}{2T_w(x)}]|u|^2} d\sigma(u,v)\bigg)^p \\
   &\lesssim  \bigg(\int_{n(x)\cdot u>0} \Big(\langle u\rangle^2 \Vert \nabla_x \phi^{m-1}\Vert_{L^\infty} |e^{-\lambda s\langle u\rangle}w_{\tilde{\theta}}(u) f^m|\Big)e^{\lambda s\langle u\rangle}w_{\tilde{\theta}}^{-1}(u)d\sigma(u,v) \\
   &+\Vert w_{\theta'} f^m\Vert_{L^\infty} \int_{\mathbb{R}^3}\mathbf{k}_{\rho}(u,u')|e^{-\lambda s\langle u'\rangle}w_{\tilde{\theta}}(u') f^{m-1}(u')|du' e^{-[\frac{1}{4T_M}-\frac{1}{2T_w(x)}]|u|^2} e^{\lambda s\langle u\rangle}w_{\tilde{\theta}}^{-1}(u) d\sigma(u,v)\bigg)^p.
   \end{split}
\end{equation}
Similarly to~\eqref{eqn: bound use zeta}, we bound $\langle u\rangle^2$ as $\langle u \rangle^2\lesssim e^{\zeta|u|^2}$ with the same $\zeta$ satisfying~\eqref{eqn: zeta}. Using $e^{\lambda s\langle u\rangle}w_{\tilde{\theta}}^{-1}(u) \lesssim 1$ we obtain
\begin{align*}
\eqref{eqn: second line contribution}   &  \lesssim \int_{n(x)\cdot u>0} |e^{-\lambda s\langle u\rangle}w_{\tilde{\theta}}(u) f^m|^p \{n\cdot u\}du \times     \bigg(      \int_{n(x)\cdot u>0}e^{p^*\zeta |u|^2} e^{-p^*[\frac{1}{4T_M}-\frac{1}{2T_w(x)}]|u|^2}\\
   & \times I_0 \Big(\frac{ p^*(1-r_\perp)^{1/2} u_\perp v_\perp}{T_w(x) r_\perp} \Big)e^{-\frac{p^*}{2T_w(x)}\big[\frac{|u_\perp|^2+(1-r_\perp)|v_\perp|^2}{r_\perp}+\frac{|u_\parallel-(1-r_\parallel)v_\parallel|^2}{r_\parallel(2-r_\parallel)}\big]}           du \bigg)^{p/p^*}\\
   &+\Vert w_{\theta'} f^m\Vert_{L^\infty} \bigg(\int_{n(x)\cdot u>0} \big(\int_{\mathbb{R}^3} \mathbf{k}_\rho(u,u')du'\big)^{p/q} \big(\int_{\mathbb{R}^3} \mathbf{k}_\rho(u,u')|e^{-\lambda s\langle u'\rangle}w_{\tilde{\theta}}(u') f^{m-1}(u')|^p du'\big)^{1/p}\\
   &\times e^{-[\frac{1}{4T_M}-\frac{1}{2T_w(x)}]|u|^2} \{n(x)\cdot u\} I_0e^{-\frac{1}{2T_w(x)}\Big[\frac{|u_\perp|^2+(1-r_\perp)|v_\perp|^2}{r_\perp}+\frac{|u_\parallel-(1-r_\parallel)v_\parallel|^2}{r_\parallel(2-r_\parallel)}\Big]}            du                         \bigg)^p.
\end{align*}

Then we can apply the same computation as in~\eqref{eqn: first term of the boundary derivative 1st},~\eqref{eqn: first term of the boundary derivative 2nd} for the $u$-integration. Thus by~\eqref{eqn: extra term after integrating the derivative on the boudnary} we derive
\begin{equation*}
\begin{split}
\eqref{eqn: second line contribution}&\lesssim   \int_{n(x)\cdot u>0} |e^{-\lambda s\langle u\rangle}w_{\tilde{\theta}}(u)f^m|^p \{n\cdot u\}du \times \eqref{eqn: extra term after integrating the derivative on the boudnary}\\
&+\Vert w_{\theta'} f^m\Vert_{L^{\infty}} \Big(\int_{\mathbb{R}^3}\int_{\mathbb{R}^3}\mathbf{k}_\rho(u,u')|e^{-\lambda s\langle u'\rangle}w_{\tilde{\theta}}(u') f^{m-1}(u')|^p du'du\Big) \times \eqref{eqn: extra term after integrating the derivative on the boudnary}\\
& \lesssim \Big(\int_{n(x)\cdot u>0} |e^{-\lambda s\langle u\rangle}w_{\tilde{\theta}}(u) f^m|^p \{n\cdot u\}du+ \Vert e^{-\lambda s\langle u\rangle}w_{\tilde{\theta}}(u) f^{m-1}\Vert_p^p \Big)   \times \eqref{eqn: extra term after integrating the derivative on the boudnary}.
\end{split}
\end{equation*}
By exactly the same computation as~\eqref{eqn: integrand dv}, the integrand for the $v$-integration in~\eqref{eqn: negative bdr in green's indentity} is $\in L_v^1$. Thus the contribution of the second line of~\eqref{eqn: derivative bound on the boundary second part} in~\eqref{eqn: Green's iden for partial f}$\gamma_-$ is
\begin{equation}\label{eqn: seond line contribution}
\lesssim_{p,T_M,r_{min}}\Vert w_{\theta'} f^m\Vert_\infty\Big(\int_0^t |e^{-\lambda s\langle u\rangle}w_{\tilde{\theta}}(u) f^m|_{p,+}^p + \int_0^t \Vert e^{-\lambda s\langle u\rangle}w_{\tilde{\theta}}(u) f^{m-1}\Vert_p^p\Big).
\end{equation}

Collecting~\eqref{eqn: contribution of first line}~\eqref{eqn: seond line contribution} we conclude that
\begin{equation}\label{eqn: bound for negative bdy deri}
\begin{split}
   & \int_0^t |e^{-\lambda s\langle v\rangle} w_{\tilde{\theta}} \alpha^\beta_{f^{m},\epsilon} \partial f^{m+1}|_{p,-}^p \\
    & \leq  C(p,T_M,r_{min}) \times\epsilon\int_0^t |e^{-\lambda s\langle v\rangle}w_{\tilde{\theta}} \alpha_{f^{m-1},\epsilon}^\beta \partial f^m|^p_{p,+} + C(p,T_M,r_{min})C(\e)\Vert w_{\tilde{\theta}}\alpha_{f^{m-1},\epsilon}^\beta \nabla_{x,v}f(0)\Vert_p^p+\\
    &+C(p,T_M,r_{min})C(\e)\sup_m\Vert w_{\theta'} f^m\Vert_\infty\Big(\int_0^t \Vert e^{-\lambda s\langle v\rangle}\langle v\rangle^{1/p}w_{\tilde{\theta}}\alpha_{f^{m-1},\epsilon}^\beta \nabla_{x,v} f^m \Vert_p^p     +\int_0^t|e^{-\lambda s\langle v\rangle} w_{\tilde{\theta}} f^m|_{p,+}^p \Big)  \\
    &+C(p,T_M,r_{min})C(\e)\times (\sup_m\Vert w_{\theta'} f^m\Vert_\infty+\sup_{l\leq m}\Vert \nabla^2 \phi^{l}\Vert_\infty)\times  \\
   &\Big( \int_0^t\Vert e^{-\lambda s\langle v\rangle} w_{\tilde{\theta}} f^{m-1}\Vert_p^p     +\int_0^t \Vert e^{-\lambda s\langle v\rangle }w_{\tilde{\theta}} \alpha_{f^{m-2},\epsilon}^\beta \nabla_{x,v} f^{m-1}\Vert_p^p +\int_0^t \Vert e^{-\lambda s\langle v\rangle }w_{\tilde{\theta}} \alpha_{f^{m-1},\epsilon}^\beta \nabla_{x,v} f^{m}\Vert_p^p  \Big),
\end{split}
\end{equation}
where $C(p,T_M,r_{min})$ comes from $\lesssim_{p,T_M,r_{min}}$ and $C(p)$ in~\eqref{eqn: bound of G}. Similar to Step 1, here it's important to note that in the second line, the first term has $\e\ll 1$, while the second term has $C(\e)$, which can be large number depends on $\e$.

\begin{remark}
We comment the largeness of $\lambda$ comes from~\eqref{eqn: C large}, the boundedness of $f$~\eqref{eqn: lambda condition 1}, the boundedness of $\partial f$~\eqref{eqn: lambda condition 2} and in~\eqref{eqn: lambda condition 3}. The smallness of $\tilde{\theta}$ comes from~\eqref{eqn: dominated} and~\eqref{eqn: thete tilde conditiopn}. The largeness of the constant $C_W$ in~\eqref{induc_hypo} comes from~\eqref{eqn: C_W}. The smallness the time $t_W$ in~\eqref{induc_hypo} comes from~\eqref{eqn: lesssim}.
\end{remark}

\end{proof}

\section{$L_x^3L_v^{1+}$-Estimate of $\nabla_v f$ and $L^{1+}$-Stability}
As we mention in the introduction, to conclude the uniqueness we need to control $\nabla_v f$ with certain norm. With $W^{1,p}$ estimate for $\nabla_x f$ in section 3, we will establish the $L_x^3L_v^{1+}$-estimate for the sequence solution $\nabla_v f^{m+1}$ in Proposition \ref{proposition: L3L1 estimate} in the section. With such estimate for $\nabla_v f$, we then show the sequence $f^{m+1}$ is $L^{1+}$ Cauchy in Proposition \ref{Prop: L1+stability Cauchy}. The $L^{1+}$ Cauchy is crucial to show the existence of the VPB equation. These two propositions lead to the $L^3_x L_v^{1+}$-estimate for $\nabla_v f$ and the $L^{1+}$-stability for $f$ that satisfies~\eqref{eqn: VPB equation} under good initial condition. These two propositions are given in Proposition \ref{Prop: L3L1 for nabla f}, \ref{L1+stability} respectively. The $L^{1+}$ stability directly leads to the uniqueness of VPB system.

\begin{proposition}\label{proposition: L3L1 estimate}
Assume $f^{m+1}$ solves~\eqref{eqn:formula of f^(m+1)} and satisfy all assumptions in Proposition \ref{Prop W1p}. We also assume extra initial condition
		\Be \label{Extra_uniq}
		\| w_{\tilde{\theta}} \nabla_v f_{0} \|_{L^3_{x,v}} < \infty.
		\Ee
There exists $t_{\delta}\ll 1$ ($t_{\delta}<t_W$) and $C_\delta$ such that when $0\leq t<t_{\delta}$, if
\begin{equation}\label{eqn: nabla v fm}
\begin{split}
    & \sup_{0\leq s\leq t}\Vert e^{-\lambda s\langle v\rangle}\nabla_v f^m(s)\Vert_{L_x^3L_v^{1+\delta}} \\
     &\leq \underbrace{2C_{\delta}\Big[\Vert w_{\tilde{\theta}}\nabla_v f(0) \Vert_{L_{x,v}^3}+\sup_{n}\sup_{0\leq s\leq t}\Vert w_{\theta'} f^n(s)\Vert_\infty+ \sup_n\sup_{0\leq s\leq t}\Vert e^{-\lambda s\langle u\rangle }w_{\tilde{\theta}}\alpha_{f^{n-1},\epsilon}^\beta \nabla_{x,v}f^n(s)\Vert_p \Big]}_{\eqref{eqn: nabla v fm}_*},
\end{split}
\end{equation}
then we have
\Be\label{bound_nabla_v_g_global}
\| e^{-\lambda t \langle v\rangle}\nabla_v f^{m+1}(t) \|_{L^3_x L^{1+\delta}_v} \leq~\eqref{eqn: nabla v fm}_*.
\Ee

Here $C_\delta$ is defined in~\eqref{eqn: C delta} and $t_\delta$ satisfies~\eqref{eqn: t delta}.

\end{proposition}


\begin{proof}[\textbf{Proof of Proposition \ref{proposition: L3L1 estimate}}]
First we take $t_{\delta}\leq t_{W}$ with $t_{W}$ defined in Proposition \ref{Prop W1p} so that we can apply Proposition \ref{Prop W1p} and Proposition \ref{proposition: boundedness}. We have
		\Be\begin{split}\label{eqtn_g_v}
			&[\p_t + v\cdot \nabla_x - \nabla_x \phi^m \cdot \nabla_v   ] (e^{-\lambda t\langle v\rangle}\p_v f^{m+1})+[\lambda \langle v\rangle-\lambda t +\frac{v}{2T_M} \cdot \nabla_x \phi^m+\nu(F^m)] (e^{-\lambda t\langle v\rangle}\p_v f^{m+1})\\
			& = e^{-\lambda t\langle v\rangle}\times \Big[-\nabla_x f^{m+1}- \frac{1}{2T_M} \nabla_x \phi^m f^{m+1}  + \p_v \Big(\Gamma_{\text{gain}}(f^m,f^m)\Big)\Big].
		\end{split}\Ee
By~\eqref{eqn: velocity derivative on the boundary}, we have boundary bound for $(x,v) \in\gamma_-$
		\Be\label{bdry_g_v}
		\big|\p_v f^{m+1}(t,x,v)  \big| \lesssim   |v|^2 e^{[\frac{1}{4T_M}-\frac{1}{2T_w(x)}]|v|^2} \int_{n \cdot u>0} |f^m(t,x,u)| |u| e^{-[\frac{1}{4T_M}-\frac{1}{2T_w(x)}]|u|^2} \dd \sigma(u,v) \ \ \text{on } \ \gamma_-.
		\Ee
By~\eqref{eqn: C large} we have
\[\lambda \langle v\rangle-\lambda t\partial_v \langle v\rangle+\frac{v}{2T_M}\cdot \nabla_x \phi^m +\nu(F^m)\geq \frac{\lambda}{2}\langle v\rangle.\]
Then we bound $|e^{-\lambda t\langle v\rangle}\nabla_v f^{m+1}|$  along the characteristics
		\begin{eqnarray}
		&&|e^{-\lambda t\langle v\rangle}\p_v f^{m+1}(t,x,v)|\nonumber
		\\
		&\leq &   \mathbf{1}_{\tb(t,x,v)> t}
		|\p_v f(0,X(0;t,x,v), V(0;t,x,v))|\label{g_initial}\\
		& +&   \ \mathbf{1}_{\tb(t,x,v)<t} |v_b|^2 e^{[\frac{1}{4T_M}-\frac{1}{2T_w(x)}]|v_b|^2}\label{g_bdry1} \\
&&\int_{n(\xb) \cdot u>0}
		| f^m(t-\tb, \xb, u) | |u| e^{-[\frac{1}{4T_M}-\frac{1}{2T_w(x)}]|u|^2} \dd \sigma(u,v_b)\label{g_bdry}\\
&+ &
		\int^t_{\max\{t-\tb, 0\}}
		\Vert \nabla_x \phi^m\Vert_\infty |v| |{w}_{{\theta'}}( V(s;t,x,v))|^{-1} \Vert w_{\theta'} f^{m+1}\Vert_\infty
		\dd s
		\label{g_infty}\\
		&  +&   \int^t_{\max\{t-\tb, 0\}}
		|\nabla_x f^{m+1}(s, X(s;t,x,v),V(s;t,x,v))|
		\dd s\label{g_x}\\
		&   + &\int^t_{\max\{t-\tb, 0\}}
		(1+ \| w_{\theta'} f^m \|_\infty)
		\int_{\R^3} \mathbf{k}_\varrho (V(s;t,x,v),u) |\p_v f^m(s,X(s),u)| \dd u
		\dd s\label{g_K}.
		\end{eqnarray}
		
We will discuss every term in~\eqref{g_initial}-\eqref{g_K} separately. In Step 1 we analyze~\eqref{g_initial}-\eqref{g_infty}. In Step 2,3 we analyze~\eqref{g_x},\eqref{g_K} respectively. In Step 4 we conclude this lemma by summarizing all the estimates in previous steps.

\textbf{Step 1.}

		Note that if $|v| > 2C_{\phi^m}t $, for $0 \leq s \leq t$,
		\Be\label{V_lower_bound_v}
		\begin{split}
			|V(s;t,x,v)| &\geq |v| - \int^t_0 |\nabla_x \phi^m (\tau;t,x,v)| \dd \tau \geq |v|
			- C_{\phi^m}t \geq \frac{|v|}{2}.
		\end{split}
		\Ee
		Therefore
		\Be\label{tilde_w_integrable}
		\sup_{s,t,x}\left\|    \frac{1}{{w}_{\tilde{\theta}} (V(s;t,x,v)) }    \right\|_{ L^{r}_v} \lesssim_{\tilde{\theta}} 1 \  \text{ for any }    1 \leq r \leq \infty .
		\Ee

\begin{itemize}
  \item[--]{Estimate of~\eqref{g_initial}.} We derive
		\Be\begin{split}\label{est_g_initial}
			&\| (\ref{g_initial})\|_{L^3_x L^{1+ \delta}_v}\\
			\lesssim & \ \left(
			\int_{\O}
			\left(\int_{\R^3} |w_{\tilde{\theta}} \p_v f(0,X(0 ), V(0 ))|^3
			\right)
			\left(
			\int_{\R^3} \frac{1}{|w_{\tilde{\theta}} (V(0 ))|^{(1+ \delta) \frac{3}{2-\delta}}}
			dv\right)^{\frac{2-\delta}{1+ \delta}}
			\right)^{1/3} \\
			\lesssim_{\tilde{\theta}} & \
			\left(\iint_{\O \times \R^3} |w_{\tilde{\theta}}(V(0;t,x,v)) \p_v f(0,X(0;t,x,v), V(0;t,x,v))|^3 \dd v \dd x\right)^{1/3}= \| w_{\tilde{\theta}} \p_v f (0) \|_{L^3_{x,v}},
		\end{split}
		\Ee
		where we have used a change of variables $(x,v) \mapsto (X(0;t,x,v), V(0;t,x,v))$ and (\ref{tilde_w_integrable}).

		\hide
		Also we use $|V(0;t,x,v)| \gtrsim |v|$ for $|v|\gg 1$, from (\ref{decay_phi}), and hence $\tilde{w}(V(0;t,x,v))^{- (1+ \delta) \frac{3}{2-\delta}} \in L_v^1 (\R^3)$.\unhide
		
\item[--]{Estimate of~\eqref{g_infty}.} Clearly with $\theta'>0$,
		\Be\label{est_g_bdry}
\|(\ref{g_infty})\|_{L^3_x L^{1+ \delta}_v} \lesssim_{\theta'} \sup_{0 \leq s \leq t} \| w_{\theta'} f^{m+1}(s) \|_\infty.\Ee

\item[--]{Estimate of~\eqref{g_bdry1},\eqref{g_bdry}.} We bound~\eqref{g_bdry1}~\eqref{g_bdry} by
\begin{align*}
   &\Vert w_{\theta'} f^m\Vert_\infty |v_b|^2 e^{[\frac{1}{4T_M}-\frac{1}{2T_w(x)}]|v_b|^2}\int_{n(x_b)\cdot u>0}  e^{-\theta' u^2}|u|e^{-[\frac{1}{4T_M}-\frac{1}{2T_w(x_b)}]|u|^2}d\sigma(u,v_b)  \\
   & \lesssim \Vert w_{\theta'} f^m\Vert_\infty |v_b|^2 e^{[\frac{1}{4T_M}-\frac{1}{2T_w(x)}]|v_b|^2}\int_{n(x_b)\cdot u>0}e^{-[\frac{1}{4T_M}-\frac{1}{2T_w(x_b)}]|u|^2}d\sigma(u,v_b)\\
   &\lesssim \Vert w_{\theta'} f^m\Vert_\infty |v_b|^2 e^{[\frac{1}{4T_M}-\frac{1}{2T_w(x_b)}]|v_b|^2}
\exp\bigg(\Big[\frac{[2T_M-T_w(x_b)][1-r_{min}]}{2T_w(x_b)\big[2T_M(1-r_{min})+r_{min} T_w(x_b)\big]}\Big]|v_b|^2\bigg)\\
&=\Vert w_{\theta'} f^m\Vert_\infty |v_b|^2 \exp\Big(\big[\frac{1}{4T_M}-\frac{1}{2[2T_M(1-r_{min})+r_{min}T_w(x_b)]} \big]|v_b|^2 \Big),
\end{align*}
where we use~\eqref{eqn: int over V_l} and directly apply~\eqref{eqn: result for para} with replacing $\frac{2\xi}{\xi+1}T_M$ by $2T_M$, $t$ by $0$ in the third line for the $u$-integration. Using
\[\frac{1}{4T_M}-\frac{1}{2[2T_M(1-r_{min})+r_{min}T_w(x_b)]}<0,\]
we obtain
\begin{equation}\label{eqn: g bdry esti}
\Vert \eqref{g_bdry1}~\eqref{g_bdry}\Vert_{L^3_x L^{1+\delta}_v} \lesssim_{T_M,r_{min}} \sup_{0\leq s\leq t}\Vert w_{\theta'} f^m(s) \Vert_\infty.
\end{equation}

\end{itemize}

\textbf{Step 2. }
\begin{itemize}
\item[--]{Estimate of~\eqref{g_x}.} We claim
		\Be\label{est_g_x}
		\|(\ref{g_x})\|_{L^3_x L^{1+ \delta}_v} \lesssim_{\tilde{\theta},\beta,p} \int^t_0 \| e^{-\lambda s\langle v\rangle} w_{\tilde{\theta}} \alpha_{f^m,\epsilon}^\beta \nabla_x f^{m+1} (s) \|_{L^p_{x,v}}.
		\Ee
		For $3<p<6$, by the H\"older inequality $\frac{1}{1+ \delta}= \frac{1}{ \frac{p+p \delta}{p-1 - \delta}}+ \frac{1}{p}$,
		\Be\label{init_p_xf}
		\begin{split}
			&\left\|\left\| \int^t_{\max\{t-\tb, 0\}}
			\p_x f^{m+1}(s,X(s;t,x,v),V(s;t,x,v)) \dd s
			\right\|_{L_{v}^{1+ \delta}(  \R^3)}\right\|_{L^{3}_x}\\
			\lesssim & \ \left\|\left\| \int^t_{\max\{t-\tb, 0\}}  \frac{e^{-\lambda s\langle V(s;t,x,v)\rangle} w_{\tilde{\theta}} \alpha_{f^{m},\epsilon}^{\beta} \partial_x f^{m+1}(s,X(s;t,x,v),V(s;t,x,v))}{e^{-\lambda s\langle V(s;t,x,v)\rangle}w_{\tilde{\theta}}\alpha^{\beta}_{f^{m},\epsilon}(s,X(s;t,x,v),V(s;t,x,v)) }
			\dd s
			\right\|_{L_{v}^{1+ \delta}(  \R^3)}\right\|_{L^{3}_x}
			\\
			\lesssim & \ \sup_{s,x}\left\|  \frac{ e^{\lambda s\langle v\rangle} w_{\tilde{\theta}}(v)^{-1}}{\alpha_{f^{m},\epsilon}(s,x,v)^\beta}\right\|_{L_v^{\frac{p+p \delta}{p-1 - \delta}} (\R^3) }\\
			& \times \left\|
			\left\| \int^t_0 e^{-\lambda s\langle V(s;t,x,v)\rangle}w_{\tilde{\theta}}\alpha_{f^m,\epsilon}^\beta \partial_x f^{m+1}(s,X(s;t,x,v), V(s;t,x,v)) \dd s\right\|_{L_{v}^p( \R^3)}\right\|_{L^{3}_x}\\
			\lesssim &  \sup_{s,x}\left\| \frac{ e^{\lambda s\langle v\rangle}w_{\tilde{\theta}}(v)^{-1}}{\alpha_{f^{m},\epsilon}(s,x,v)^\beta}\right\|_{L_v^{\frac{p+p \delta}{p-1 - \delta}} (\R^3) }
			\times \int^t_0 \Vert  e^{-\lambda s\langle v\rangle}w_{\tilde{\theta}} \alpha_{f^m,\epsilon}^\beta \partial_x f^{m+1} (s) \Vert_{L^p_{x,v}} ds,
	\end{split}\Ee
		where we have used
$\alpha_{f^m,\epsilon}(t,x,v)=\alpha_{f^m,\epsilon}(s,X(s;t,x,v),V(s;t,x,v))$ for $t-\tb(t,x,v)\leq s \leq t$ and the change of variables $(x,v) \mapsto (X(s;t,x,v), V(s;t,x,v))$ and the Minkowski inequality.
		
		For $\beta$ in~\eqref{Condition for p}, we have $
		\beta \frac{p}{p-1} < 1$ since $\frac{2}{3} < \frac{p-1}{p}$ for $3 < p$. Therefore, we can choose $0 < \delta=\delta(\beta,p) \ll 1$ so that $\beta$ satisfies
		\Be\label{less_than_1}
		\beta \times \frac{p+p \delta}{p-1 - \delta} < 1.
		\Ee
		\hide
		For $\beta$ in (\ref{W1p_initial}) we have $
		\beta \frac{p+p \delta}{p-1 - \delta}> \frac{p+p\delta}{p} \frac{p-2}{p-1-\delta}.$ Therefore for $0<\delta\ll1$ we can choose $\beta$ in (\ref{W1p_initial}) and satisfies
		\Be\label{less_than_1}
		\beta \times \frac{p+p \delta}{p-1 - \delta} < 1.
		\Ee\unhide
		We apply Proposition \ref{prop_int_alpha} to conclude that
		\Be\label{bound_g_x}
		\sup_{s,x} \left\| \frac{ e^{\lambda s\langle v\rangle}w_{\tilde{\theta}} (v)^{-1}}{\alpha_{f^m,\epsilon}(s,x,v)^\beta}\right\|_{L_v^{\frac{p+p \delta}{p-1 - \delta}} (\R^3) }^{\frac{p+p \delta}{p-1 - \delta}}
		\lesssim_{\tilde{\theta}} \sup_{s,x} \int_{\R^3}  \frac{1}
		{\alpha_{f^m,\epsilon}(s,x,v)^{\beta \frac{p+p \delta}{p-1 - \delta}   }}   \dd v \lesssim_{\beta,p} 1.
		\Ee

		Finally, from (\ref{init_p_xf}), (\ref{bound_g_x}), we conclude the claim (\ref{est_g_x}).

\end{itemize}

\textbf{Step 3. }
 \begin{itemize}
 \item[--]{Estimate of~\eqref{g_K}.} We consider (\ref{g_K}). We split the $u$-integration of (\ref{g_K}) into two parts with $N\gg 1$ as
		\begin{eqnarray}
		&&\int_{|u| \leq N} \mathbf{k}_\varrho (V(s;x,t,v), u) |\nabla_v f^m(s,X(s ), u) | \dd u \label{g_K_split1}\\
		&+& \int_{|u| \geq N}  \mathbf{k}_\varrho (V(s;t,x,v), u) |\nabla_v f^m(s,X(s ), u) | \dd u .\label{g_K_split2}
		\end{eqnarray}

		First we bound (\ref{g_K_split1}). From the change of variables $(x,v) \mapsto (X(s;t,x,v), V(s;t,x,v))$ for $t-\tb(t,x,v)\leq s \leq t$
		\Be\label{g_K_COV}
		\begin{split}
			&\left\|  \int_{|u| \leq N} \mathbf{k}_\varrho (V(s;t,x,v), u) |\nabla_v f^m(s,X(s;t,x,v ), u) | \dd u  \right\|_{L^3_x L^{1+\delta}_v}\\
			\lesssim & \ \left\|\int_{|u| \leq N}\mathbf{k}_\varrho (v, u) |e^{-\lambda s\langle u\rangle}\nabla_v f^m(s,x, u) | \dd u  \right\|_{L^3_x L^{1+\delta}_v},
		\end{split}
		\Ee
where we use $e^{\lambda s\langle u\rangle}\lesssim 1$ when $|u|\leq N$. If $|v|\geq 2N$ then $|v-u|^2\gtrsim |v|^2$ and $|v-u|\geq N$, thus for $|v|\geq 2N$ and $|u| \leq N$,
\[\mathbf{k}_\varrho (v,u) \lesssim \frac{e^{-C_1|v|^2}}{|v-u|}=O(1/N).\]
If $|v|\leq 2N$, for $0 < \delta \ll 1$ with $\frac{3(1+ \delta)}{1-2\delta}>3$,
		\Be
		\begin{split}\label{g_K_COV1}
	(\ref{g_K_COV})		&\lesssim  \ \left\| \left\| \int_{|u| \leq N} \mathbf{k}_\varrho (v, u) |e^{-\lambda s\langle u\rangle}\nabla_v f^m(s,x, u) | \dd u \right\|_{L^{1+\delta}_v
				(\{|v| \geq 2N\})
			} \right\|_{L^3_x  }\\
			&+ \left\| \left\|  e^{-C|v|^2}\right\|_{L^{3/2}_v} \left\| \int_{|u| \leq N} \frac{1}{|v-u|} |e^{-\lambda s\langle u\rangle}\nabla_v f^m(s,x, u) | \dd u \right\|_{L^{\frac{3(1+ \delta)}{1-2\delta}}_v
				(\{|v| \leq 2N\})
			} \right\|_{L^3_x  }\\
			 &\lesssim \| e^{-\lambda s\langle v\rangle}\nabla_v f^m(s) \|_{L^3_x L^{1+ \delta}_v}+ \left\|   \left\|   \frac{\mathbf{1}_{|v|\leq 2N}}{|v- \cdot |} * |e^{-\lambda s\langle \cdot\rangle}\nabla_v f^m(s,x, \cdot ) |   \right\|_{L^{\frac{3(1+ \delta)}{1-2\delta}}_v
			} \right\|_{L^3_x  }.
		\end{split}
		\Ee
		Then by the Hardy-Littlewood-Sobolev inequality with $1+ \frac{1}{\frac{3(1+ \delta)}{1- 2\delta}} = \frac{1}{3} + \frac{1}{1+ \delta}$, we derive that
		\Be\begin{split}\notag
			(\ref{g_K_COV1}) 
			\lesssim_\delta & \ \left\| \| e^{-\lambda s\langle v\rangle}\nabla_v f^m(s, x,v)  \|_{L^{1+ \delta}_v  }\right\|_{L^3_x}
			= \| e^{-\lambda s\langle  v\rangle}\nabla_v f^m(s) \|_{L^3_x L^{1+ \delta}_v}
			.
		\end{split}\Ee
		Combining the last estimate with (\ref{g_K_COV}), (\ref{g_K_COV1}), we prove that
		\Be\label{est_g_K_split1}
		\|(\ref{g_K_split1})\|_{L^3_x L^{1+ \delta}_v} \lesssim_{\delta}    \| e^{-\lambda s\langle v\rangle}\nabla_v f^m(s) \|_{L^3_x L^{1+ \delta}_v}.
		\Ee
		
		Now we consider (\ref{g_K_split2}). We have
		\Be\notag
		\begin{split}
		(\ref{g_K_split2})	&
			\lesssim \int_{|u| \geq N} \frac{1}{w_{\tilde{\theta}-t} (V(s;t,x,v))^{1/2}} \frac{w_{\tilde{\theta}-t}(V(s;t,x,v))  }{w_{\tilde{\theta}}(u)} \frac{\mathbf{k}_\varrho (V(s;t,x,v), u)e^{\lambda s \langle u\rangle}}{\alpha_{f^{m-1},\epsilon}(s,X(s;t,x,v), u)^\beta} \\
			& \ \  \times \frac{1}{w_{\tilde{\theta}-t}(V(s;t,x,v))^{ 1/2}} e^{-\lambda s\langle u\rangle}w_{\tilde{\theta}}(u)\alpha_{f^{m-1},\epsilon}(s,X(s;t,x,v), u)^\beta | \nabla_v f^m(s,X(s;t,x,v), u)| \dd u .
		\end{split}
		\Ee

		By the H\"older inequality with $\frac{1}{p} + \frac{1}{p^*}=1$ with $3<p<6$,
		\Be
		\begin{split}\label{g_K_split2_Holder}
	|(\ref{g_K_split2})|		& \lesssim  \frac{1}{w_{\tilde{\theta}-t} (V(s;t,x,v))^{1/2}} \left\|  \frac{w_{\tilde{\theta}-t}(V(s;t,x,v))  }{w_{\tilde{\theta}}(u)} \frac{\mathbf{k}_\varrho (V(s;t,x,v), u)e^{\lambda s \langle u\rangle}}{\alpha_{f^{m-1},\epsilon}(s,X(s;t,x,v), u)^\beta} \right\|_{L^{p^*} (\{ |u|\geq N \})} \\
			&  \times  \left\|  \frac{e^{-\lambda s\langle u\rangle} w_{\tilde{\theta}}(u) }{w_{\tilde{\theta}-t}(V(s;t,x,v))^{ 1/2}} \alpha_{f^{m-1},\epsilon}(s,X(s;t,x,v), u)^\beta | \nabla_v f^m(s,X(s;t,x,v), u)| \right\|_{L^p_u (\R^3)}.
		\end{split}
		\Ee
		
		Then by the H\"older inequality with $\frac{1}{1+ \delta} = \frac{1}{p} + \frac{1}{\frac{(1+ \delta)p}{p - (1+ \delta)}}$,
		\Be
		\begin{split}\notag
	\|(\ref{g_K_split2})\|_{L^{1+ \delta}_v}		&\lesssim  \left\| \frac{1}{w_{\tilde{\theta}-t} (V(s;t,x,v))^{1/2}} \right\|_{L_v^{\frac{(1+ \delta) p}{ p- (1+ \delta)}}} \\
			& \times \sup_v \left\|  \frac{w_{\tilde{\theta}-t}(V(s;t,x,v))  }{w_{\tilde{\theta}}(u)} \frac{\mathbf{k}_\varrho (V(s;t,x,v), u) e^{\lambda s\langle u\rangle}}{\alpha_{f^{m-1},\epsilon}(s,X(s;t,x,v), u)^\beta} \right\|_{L^{p^*} (\{ |u|\geq N \})} \\
			& \times \left\| \left\|  \frac{e^{-\lambda s\langle u\rangle}w_{\tilde{\theta}}(u) }{w_{\tilde{\theta}-t}(V(s;t,x,v))^{ 1/2}} \alpha_{f^{m-1},\epsilon}(s,X(s;t,x,v), u)^\beta | \nabla_v f^{m}(s,X(s;t,x,v), u)| \right\|_{L^p_u  }\right\|_{L^{p}_v}.
		\end{split}
		\Ee
		
From (\ref{k_theta_comparision}), for some $0<\tilde{\varrho}< \varrho$ we have
\[\mathbf{k}_\varrho (v,u) \frac{e^{(\tilde{\theta} -t)|v|^2}e^{\lambda s\langle u\rangle}}{e^{\tilde{\theta} |u|^2}} \lesssim \mathbf{k}_\varrho (v,u) \frac{e^{(\tilde{\theta} -t)|v|^2}}{e^{(\tilde{\theta} -t)|u|^2}} \frac{e^{\lambda t\langle u\rangle}}{e^{t|u|^2}}  \lesssim \mathbf{k}_{\tilde{\varrho}} (v,u).\]

Hence using (\ref{tilde_w_integrable}) we derive,
		\Be
		\begin{split}\notag
			&\left\|\|(\ref{g_K_split2})\|_{L^{1+ \delta}_v}\right\|_{L^3_x}\lesssim_{\O,\beta,p}   \ \sup_{X,V} \left\| \frac{e^{- \frac{\tilde{\varrho}}{10}|V-u|^2  }}{|V-u|}  \frac{1}{\alpha_{f^{m-1},\epsilon}(s,X , u)^\beta} \right\|_{L^{p^*} (\{ |u|\geq N \})} \\
			&\times  \left\|  \frac{ e^{-\lambda s\langle u\rangle} w_{\tilde{\theta}}(u) }{w_{\tilde{\theta}-t}(V(s;t,x,v))^{1/2}} \alpha_{f^{m-1},\epsilon}(s,X(s;t,x,v), u)^\beta | \nabla_v f^m(s,X(s;t,x,v), u)| \right\|_{L^p_{u,v,x}  } .
		\end{split}
		\Ee
		By (\ref{NLL_split3}) in Proposition \ref{prop_int_alpha} with $ \frac{p-2}{p-1}<\beta p^*< 1$ from~\eqref{Condition for p} and applying the change of variables $(x,v) \mapsto (X(s;t,x,v), V(s;t,x,v))$, we derive that
		\Be
		\begin{split}\label{g_K_split2_Holder2}
			&\left\|\|(\ref{g_K_split2})\|_{L^{1+ \delta}_v}\right\|_{L^3_x}\lesssim_{\Omega,p,\beta}
			\left\| \left\|  \frac{e^{-\lambda s\langle u\rangle} }{w_{\tilde{\theta}-t}(v)^{ 1/2}}w_{\tilde{\theta}}(u) \alpha_{f^{m-1},\epsilon}(s,x, u)^\beta | \nabla_v f^m(s,x, u)| \right\|_{L^p_v} \right\|_{L^p_{u ,x}  } \\
			  &\lesssim \left\|  \frac{1 }{w_{\tilde{\theta}-t}(v)^{1/2}} \right\|_{L^p_v}    \left\| e^{-\lambda s\langle u\rangle} w_{\tilde{\theta}}(u) \alpha_{f^{m-1},\epsilon}(s,x, u)^\beta | \nabla_v f^m(s,x, u)|   \right\|_{L^p_{u ,x}  }\\
			 & \lesssim_{\tilde{\theta}}  \left\|e^{-\lambda s\langle u\rangle} w_{\tilde{\theta}} \alpha_{f^{m-1},\epsilon}^\beta | \nabla_v f^m(s )|   \right\|_{L^p  }.
		\end{split}
		\Ee
		Combining (\ref{g_K_split2_Holder}) and (\ref{g_K_split2_Holder2}) we conclude that
		\Be
		\|(\ref{g_K_split2})\|_{L^3_x L^{1+ \delta}_v} \lesssim_{\Omega,\beta,p,\tilde{\theta}}   \| e^{-\lambda s\langle u\rangle} w_{\tilde{\theta}} \alpha_{f^{m-1},\epsilon}^\beta   \nabla_v f^m(s )  \|_{L^p_{x,v}}.\label{est_g_K_split2}
		\Ee
		
		Finally from (\ref{est_g_K_split1}) and (\ref{est_g_K_split2}), and using the Minkowski inequality, we conclude that
		\Be\label{est_g_K}\begin{split}
			& \| (\ref{g_K}) \|_{L^3_v L^{1+ \delta}_x} \\
			 & \lesssim_{\Omega,\beta,p,\tilde{\theta}}  (1+ \| w_{\theta'} f^m\|_\infty) \int^t_0  \big[
			\|e^{-\lambda s\langle v\rangle} \nabla_v f^m (s) \|_{L^3_x L^{1+ \delta}_v}
			+
			\|e^{-\lambda s\langle v\rangle} w_{\tilde{\theta}} \alpha_{f^{m-1},\epsilon}^\beta  \nabla_v f^m(s ) \|_{L^p_{x,v}}\big] \dd s. 
		\end{split}\Ee
\end{itemize}

	\textbf{Step 4.}	
Since all assumption in Proposition \ref{proposition: boundedness},\ref{Prop W1p} are satisfied, we have the uniform in $n$ bound
\[\sup_n \Vert w_{\theta'} f^n\Vert_\infty<\infty,\quad \sup_n \Vert e^{-\lambda s\langle v\rangle} w_{\tilde{\theta}}\alpha^\beta_{f^{m-1},\e}\nabla_{x,v}f^{n}\Vert_p<\infty.\]

		Collecting terms from~\eqref{est_g_initial},~\eqref{est_g_bdry},\eqref{eqn: g bdry esti},~\eqref{est_g_x}, and (\ref{est_g_K}),
 we derive
		\Be\begin{split}\label{bound_nabla_v_g}
			&\| e^{-\lambda t\langle v\rangle}\nabla_vf^{m+1}(t) \|_{L^3_xL^{1+ \delta}_v} \\
			\leq & \ C(\Omega,p,\beta,\tilde{\theta})\times (1+\Vert w_{\theta'}f^m\Vert_\infty) \Big[
			\| w_{\tilde{\theta}} \nabla_v f(0) \|_{L^3_{x,v}} +
		\sup_n	\sup_{0 \leq s \leq t}
			\| w_{\theta'} f^n(s) \|_\infty
			\\
			&
			+ \sup_n\sup_{0\leq s\leq t}\Vert e^{-\lambda s\langle u\rangle }w_{\tilde{\theta}}\alpha_{f^{n-1},\epsilon}^\beta \nabla_{x,v}f^n(s)\Vert_p+t\sup_{0\leq s\leq t}
			\| e^{-\lambda s\langle v\rangle}\nabla_v f^m (s) \|_{L^3_x L^{1+ \delta}_v}
			\Big].
		\end{split} \Ee
Now we define the constant in~\eqref{eqn: nabla v fm} as
\begin{equation}\label{eqn: C delta}
  C_\delta=C(\Omega,p,\beta,\tilde{\theta})(1+\sup_n\Vert w_{\theta'}f^n\Vert_\infty).
\end{equation}
For the last term in~\eqref{bound_nabla_v_g} by the assumption~\eqref{eqn: nabla v fm}, we take $t<t_{\delta}=t_{\delta}(C_\delta)\ll 1$ small enough such that
\begin{align}
   & C(\Omega,p,\beta,\tilde{\theta})(1+\sup_n\Vert w_{\theta'}f^n\Vert_\infty)t\sup_{0\leq s\leq t}
			\|e^{-\lambda s\langle v\rangle} \nabla_v f^m (s) \|_{L^3_x L^{1+ \delta}_v} \notag\\
   & \leq t_\delta C_{\delta}^2\Big[\Vert w_{\tilde{\theta}}\nabla_v f(0) \Vert_{L_{x,v}^3}+\sup_{n}\sup_{0\leq s\leq t}\Vert w_{\theta'} f^n(s)\Vert_\infty+ \sup_n\sup_{0\leq s\leq t}\Vert e^{-\lambda s\langle u\rangle }w_{\tilde{\theta}}\alpha_{f^{n-1},\epsilon}^\beta \nabla_{x,v}f^n(s)\Vert_p \Big]\notag\\
   &\leq C_{\delta}\Big[\Vert w_{\tilde{\theta}}\nabla_v f(0) \Vert_{L_{x,v}^3}+\sup_{n}\sup_{0\leq s\leq t}\Vert w_{\theta'} f^n(s)\Vert_\infty+ \sup_n\sup_{0\leq s\leq t}\Vert e^{-\lambda s\langle u\rangle }w_{\tilde{\theta}}\alpha_{f^{n-1},\epsilon}^\beta \nabla_{x,v}f^n(s)\Vert_p \Big].\label{eqn: t delta}
\end{align}

Finally we get

\begin{align*}
   & \| e^{-\lambda t\langle v\rangle}\nabla_vf^{m+1}(t) \|_{L^3_xL^{1+ \delta}_v}  \\
   & \leq 2C_\delta\Big[\Vert w_{\tilde{\theta}}\nabla_v f(0) \Vert_{L_{x,v}^3}+\sup_{n}\sup_{0\leq s\leq t}\Vert w_{\theta'} f^n(s)\Vert_\infty+ \sup_n\sup_{0\leq s\leq t}\Vert e^{-\lambda s\langle u\rangle }w_{\tilde{\theta}}\alpha_{f^{n-1},\epsilon}^\beta \nabla_{x,v}f^n(s)\Vert_p \Big].
\end{align*}

We prove~\eqref{bound_nabla_v_g_global} and derive the proposition.
\end{proof}

The next proposition follows from the Proposition \ref{proposition: L3L1 estimate}.

\begin{proposition}\label{Prop: L1+stability Cauchy}
Suppose $f^{m+1}$ and $f^{m}$ solve~\eqref{eqn: fm+1} with boundary condition~\eqref{eqn: fm+1 BC}, and satisfy all assumption in Proposition \ref{proposition: boundedness} \ref{Prop W1p} \ref{proposition: L3L1 estimate}. Then there exists $\bar{t}\ll 1$ ($\bar{t}\leq t_{\delta}$) with $t\leq \bar{t}$ such that
\Be\label{1+delta_stability}
\begin{split}
    & \sup_{0\leq s\leq t}\|e^{-\lambda s\langle v\rangle} \big[f^{m+1}(s) - f^m(s)\big] \|_{L^{1+\delta}(\O \times \R^3)}+\int_{0}^{t}|e^{-\lambda s\langle v\rangle}(f^{m+1}-f^m)(s)|_{1+\delta,+}^{1+\delta} \\
     & \leq \frac{1}{2}\sup_{0\leq s\leq t}\|e^{-\lambda s\langle v\rangle} \big[f^m(s) - f^{m-1}(s)\big] \|_{L^{1+\delta}(\O \times \R^3)}+\frac{1}{2}\int_{0}^{t}|e^{-\lambda s\langle v\rangle}(f^{m}-f^{m-1})(s)|_{1+\delta,+}^{1+\delta}.
\end{split}		
		\Ee

Here $\bar{t}$ satisfies~\eqref{eqn: bar t}.

	\end{proposition}
\begin{remark}
This proposition is crucial to show the existence of the solution. In Proposition \ref{Prop existence} we will use the $L^{1+}$ Cauchy with~\eqref{eqn: theta'} to conclude the existence of the solution $f$.

\end{remark}

\begin{proof}
First we take $\bar{t}\leq t_{\delta}$ with $t_{\delta}$ defined in Proposition \ref{proposition: L3L1 estimate} so that we can apply all the previous Propositions.

Assume $f^{m+1}$ and $f^m$ solve~\eqref{eqn: fm+1}, then
\begin{equation}\label{eqn: fm+1-fm}
  \begin{split}
      & \partial_t\big[e^{-\lambda t\langle v\rangle}(f^{m+1}-f^m) \big]+v\cdot \nabla_x\big[e^{-\lambda t\langle v\rangle}(f^{m+1}-f^m) \big]-\nabla_x \phi^m \cdot\nabla_v\big[e^{-\lambda t\langle v\rangle}(f^{m+1}-f^m) \big]\\
     & +\Big(\lambda\langle v\rangle+\frac{v}{2T_M}\cdot \nabla_x \phi^m-\lambda t\partial_v \langle v\rangle+\nu(F^m) \Big) \big[e^{-\lambda t\langle v\rangle}(f^{m+1}-f^m) \big] \\
       & =(\nabla_x \phi^m-\nabla_x \phi^{m-1})\nabla_v (e^{-\lambda t\langle v\rangle}f^m) -\frac{v}{2T_M}\cdot (\nabla_x \phi^m-\nabla_x \phi^{m-1})(e^{-\lambda t\langle v\rangle} f^{m}) \\
       & +e^{-\lambda t\langle v\rangle}   \bigg[\Gamma_{\text{gain}}(f^m,f^m)-\Gamma_{\text{gain}}(f^{m-1},f^{m-1})+f^m\Big(\nu(F^{m-1})-\nu(F^m) \Big) \bigg].
  \end{split}
\end{equation}
By~\eqref{eqn: C large} we have
\[
  \lambda\langle v\rangle+\frac{v}{2T_M}\cdot \nabla_x \phi^m-\lambda t\partial_v \langle v\rangle +\nu(F^m)\geq \frac{\lambda}{2}\langle v\rangle.
\]

Then using Lemma~\ref{lemma: Green's indentity} for $L^{1+\delta}$-space with $0<\delta\ll 1$, we obtain
  \begin{equation}\label{eqn: Greens 1+delta}
    \begin{split}
    &   \Vert e^{-\lambda t\langle v\rangle}(f^{m+1}-f^m) (t)\Vert_{1+\delta}^{1+\delta} +\int_0^t |e^{-\lambda s\langle v\rangle}(f^{m+1}-f^m) (s)|_{1+\delta,+}^{1+\delta}+\frac{\lambda}{2}\int_0^t \Vert \langle v\rangle e^{-\lambda s\langle v\rangle}(f^{m+1}-f^m) (s)  \Vert_{1+\delta}^{1+\delta} \\
         & \leq \Vert [f^{m+1}-f^m](0)\Vert_{1+\delta}^{1+\delta}+\int_0^t \iint_{\Omega\times \mathbb{R}^3}|\text{RHS of~\eqref{eqn: fm+1-fm}}|e^{-\lambda s\langle v\rangle}(f^{m+1}-f^m) |^{\delta}+\int_0^t |e^{-\lambda s\langle v\rangle}(f^{m+1}-f^m) |_{1+\delta,-}^{1+\delta}.
    \end{split}
  \end{equation}

We now analyze the three terms in RHS of~\eqref{eqn: fm+1-fm}.
\begin{itemize}
\item[--]{Estimate of the first term.} For $0 < \delta \ll 1$, by the H\"older inequality with $1=\frac{1}{\frac{3(1+ \delta)}{2- \delta}} + \frac{1}{3} +  \frac{1}{\frac{1+ \delta}{\delta}}$ and the Sobolev embedding $W^{1, 1+ \delta} (\O)\subset L^{\frac{3(1+ \delta)}{2- \delta}}(\O)$ when $\O \subset \R^3$, the contribution of the first term of the RHS of~\eqref{eqn: fm+1-fm} is bounded by

\begin{equation}\label{eqn: bound nabla_vg}
  \begin{split}
      & \int_0^t \int_{\Omega\times \mathbb{R}^3}|(\nabla_x \phi^m-\nabla_x \phi^{m-1})\cdot \nabla_v\big(e^{-\lambda s\langle v\rangle}f^m\big)||e^{-\lambda s\langle v\rangle}(f^{m+1}-f^m)|^\delta \\
       & \lesssim \int_0^t \Vert \nabla_x \phi^m-\nabla_x \phi^{m-1}\Vert_{L_x^{\frac{3(1+\delta)}{2-\delta}}} \Vert e^{-\lambda s\langle v\rangle}\nabla_v f^m\Vert_{L_x^3 L_v^{1+\delta}}\Vert e^{-\lambda s\langle v\rangle}[f^{m+1}-f^m]^{\delta}\Vert_{L_{x,v}^{\frac{1+\delta}{\delta}}}\\
       & \lesssim \sup_{0\leq s\leq t}\Vert e^{-\lambda s\langle v\rangle}\nabla_v f^m(s)\Vert_{L_x^3 L_v^{1+\delta}}\times \int_0^t \Vert e^{-\lambda s\langle v \rangle}(f^{m+1}-f^m)(s)\Vert_{1+\delta}^{1+\delta} ds.
  \end{split}
\end{equation}

\item[--]{Estimate of the second term.} By the H\"{o}lder inequality with $1=\frac{\delta}{1+\delta}+\frac{1}{1+\delta}$, the contribution of the second term of the RHS of~\eqref{eqn: fm+1-fm} is bounded by
\begin{equation}\label{eqn: 2nd RHS}
\begin{split}
    & \int_0^t \int_{\Omega \times \mathbb{R}^3} \frac{v}{2T_M}\cdot (\nabla_x \phi^m-\nabla_x \phi^{m-1})(e^{-\lambda s\langle v\rangle} f^{m}) |e^{-\lambda s\langle v\rangle}(f^{m+1}-f^m)|^\delta  \\
     & \lesssim  \int_0^t  \sup_{x} \Vert \langle v\rangle f^m\Vert_{L_v^{1+\delta}}\Vert \nabla_x \phi^m-\nabla_x \phi^{m-1}\Vert_{L_x^{1+\delta}}\Vert e^{-\lambda s\langle v\rangle}(f^{m+1}-f^m)^\delta\Vert_{L_{x,v}^{\frac{1+\delta}{\delta}}}\\
     & \lesssim_\Omega \big[\Vert w_{\theta'} f^m\Vert_\infty+\Vert w_{\theta'} f^{m+1}\Vert_\infty \big]\int_{0}^{t} \Vert e^{-\lambda s\langle v\rangle}(f^{m+1}-f^m)(s)\Vert_{1+\delta}^{1+\delta}ds.
\end{split}
\end{equation}

\item[--]{Estimate of the third term.} By~\eqref{bound_Gamma_k}, using the H\"{o}lder inequality with $1=\frac{\delta}{1+\delta}+\frac{1}{1+\delta}$, the contribution of the last term of the RHS of~\eqref{eqn: fm+1-fm} is bounded by
\begin{equation}\label{eqn: bound gamma g}
  \begin{split}
  & \int_0^t \int_{\Omega\times \mathbb{R}^3} e^{-\lambda s\langle v\rangle}\big[\Gamma_{\text{gain}}(f^m,f^m)-\Gamma_{\text{gain}}(f^{m-1},f^{m-1})+\Gamma_{
  \text{loss}}(f^{m-1},f^m)-\Gamma_{\text{loss}}(f^m,f^m)\big]|e^{-\lambda s\langle v\rangle}(f^{m+1}-f^m)|^\delta  \\
  &    \lesssim \int_0^t \int_{\Omega\times\mathbb{R}^3} \big[\Gamma_{\text{gain}}(f^m,f^m)-\Gamma_{\text{gain}}(f^{m},f^{m-1})+ \Gamma_{\text{gain}}(f^m,f^{m-1})-\Gamma_{\text{gain}}(f^{m-1},f^{m-1})            \\
  &    +\Gamma_{
  \text{loss}}(f^{m-1},f^m)-\Gamma_{\text{loss}}(f^m,f^m) \big]|e^{-\lambda s\langle v\rangle}(f^{m+1}-f^m)|^\delta  \\
 &  \lesssim \big[ \Vert w_{\theta'} f^{m}\Vert_\infty +\Vert w_{\theta'} f^{m-1}\Vert_\infty \big]\int_0^t \int_x \int_v \int_u e^{-\lambda s\langle v\rangle}\mathbf{k}_\rho(v,u)|f^m(u)-f^{m-1}(u)||e^{-\lambda s\langle v\rangle}\big[ f^{m+1}(v)-f^{m}(v)\big]|^{\delta}   \\
 & \lesssim \int_0^t \int_x \int_v \int_u   \big[\Vert w_{\theta'} f^{m}\Vert_\infty+\Vert w_{\theta'} f^{m-1}\Vert_\infty \big] \mathbf{k}_\rho(v,u) |e^{-\lambda s\langle v\rangle}\big[ f^{m+1}(v)-f^{m}(v)\big]|^\delta   \\
 & \lesssim \int_0^t \int_x \int_v |e^{-\lambda s\langle v\rangle}\big[ f^{m+1}(v)-f^{m}(v)\big]|^{1+\delta} \int_u \big(\mathbf{k}_\rho(v,u) \big)^{1+\delta}\\
 & \lesssim \int_0^t \Vert e^{-\lambda s\langle v\rangle}\big[ f^{m+1}(v)-f^{m}(v)\big]\Vert_{1+\delta}^{1+\delta}.
  \end{split}
\end{equation}

\end{itemize}

Since all assumptions in Proposition~\ref{proposition: boundedness} and Proposition~\ref{proposition: L3L1 estimate} are satisfied, we have
\begin{equation}\label{eqn: sup nablav+infty}
\sup_{0 \leq s \leq t}\sup_n\big\{  \| e^{-\lambda s\langle v\rangle}\nabla_v f^n (s) \|_{L^{3}_xL^{1+ \delta}_v} + \| w_{\theta'} f^n(s) \|_\infty
			\big\}<\infty.
\end{equation}

Collecting~\eqref{eqn: bound nabla_vg}~\eqref{eqn: 2nd RHS} and~\eqref{eqn: bound gamma g}, in~\eqref{eqn: Greens 1+delta} we have
\begin{equation}\label{eqn: RHS delta m}
\begin{split}
    & \int_0^t \iint_{\Omega\times \mathbb{R}^3} |\text{RHS of}~\eqref{eqn: fm+1-fm}| |e^{-\lambda s\langle v\rangle}(f^{m+1}-f^m)|^\delta \\
     & \lesssim_\Omega  \eqref{eqn: sup nablav+infty}\times\int_0^t \Vert e^{-\lambda s\langle v\rangle}(f^{m+1}-f^{m})\Vert_{1+\delta}^{1+\delta}.
\end{split}
\end{equation}

Following the proof of the Step 1 in Proposition~\ref{Prop W1p}, we apply the same decomposition~\eqref{eqn: bound the f on negative bdr} to $\gamma_+(x)$. By~\eqref{eqn: RHS delta m}, we can obtain
		\Be\label{eqn: bdr delta m}
		\begin{split}
			&\int_0^t |e^{-\lambda s\langle v\rangle}[f^{m+1}-f^m]|_{1+ \delta, - }^{1 + \delta} \\
			&\lesssim_{\delta,T_M,r_{min},\Omega}   \e \int_0^t |e^{-\lambda s\langle v\rangle}[f^{m}-f^{m-1}]|_{1+ \delta, + }^{1 + \delta}+C(\e)\| [f^{m}- f^{m-1}](0)\|_{1+ \delta}^{1+ \delta} \\
			&  +C(\e)
			\eqref{eqn: sup nablav+infty}\times  \int^t_0 \|e^{-\lambda s\langle v\rangle} (f^{m}-f^{m-1}) \|_{1+ \delta}^{1+ \delta}.
		\end{split}
		\Ee

By~\eqref{eqn: Greens 1+delta}~\eqref{eqn: RHS delta m} and~\eqref{eqn: bdr delta m}, using $f^m(0)=f^{m+1}(0)=f_0$ and $\frac{\lambda}{2}\int_0^t \Vert \langle v\rangle e^{-\lambda s\langle v\rangle}(f^{m+1}-f^m) (s)  \Vert_{1+\delta}^{1+\delta}\geq 0$,
\begin{equation}\label{eqn: final 1+delta m}
  \begin{split}
     &  \Vert e^{-\lambda t\langle v\rangle}(f^{m+1}-f^m)(t)\Vert_{1+\delta}^{1+\delta}+\int_0^t |e^{-\lambda s\langle v\rangle} (f^{m+1}-f^m)(s)|^{1+\delta}_{1+\delta,+}      \\
      & \leq C(\delta,T_M,r_{min},\Omega)\eqref{eqn: sup nablav+infty} \times \Big( t\sup_{0\leq s\leq t} \Vert e^{-\lambda s\langle v\rangle}(f^{m+1}-f^m)(s)\Vert_{1+\delta}^{1+\delta}\\
&+C(\e)t\sup_{0\leq s\leq t} \Vert e^{-\lambda s\langle v\rangle}(f^{m}-f^{m-1})(s)\Vert_{1+\delta}^{1+\delta}+\e \int_0^t |e^{-\lambda s\langle v\rangle}[f^{m}-f^{m-1}]|_{1+ \delta, + }^{1 + \delta}\Big).
  \end{split}
\end{equation}
Now we take $\e$ and $t<\bar{t}(\delta,T_M,r_{min},\Omega,\e)$ small enough such that
\[C(\delta,T_M,r_{min},\Omega)\e \times\eqref{eqn: sup nablav+infty}<\frac{1}{10},\]
\begin{equation}\label{eqn: bar t}
C(\delta,T_M,r_{min},\Omega)C(\e) \bar{t}\times\eqref{eqn: sup nablav+infty}\leq \frac{1}{10} ,
\end{equation}
we derive~\eqref{1+delta_stability} and prove the Proposition.

\end{proof}

The Proposition \ref{proposition: L3L1 estimate} suggests, according to~\eqref{bound_nabla_v_g_global}, that the $L^3_x L_v^{1+\delta}$ estimate of $\nabla_v f$ is obtained upon a good initial condition, the boundedness in $L^\infty$ and the weighted $W^{1,p}$ estimate. In particular, we have the following proposition.

\begin{proposition}\label{Prop: L3L1 for nabla f}
Assume $f$ and $\phi$ solve~\eqref{equation for f}~\eqref{equation for phi_f}~\eqref{eqn:BC}, and satisfy estimates
\begin{equation}\label{eqn: L_infty f}
\Vert w_{\theta'} f\Vert<\infty,
\end{equation}
\begin{equation}\label{eqn: W1p f}
\Vert w_{\tilde{\theta}}e^{-\lambda t\langle v\rangle }\alpha^\beta_{f,\epsilon}\nabla_{x,v} f\Vert_p<\infty.
\end{equation}

We also assume extra initial condition
		\Be \label{Extra_uniq}
		\| w_{\tilde{\theta}} \nabla_v f_{0} \|_{L^3_{x,v}} < \infty.
		\Ee
Then
\begin{equation}\label{eqn: L3L1 f bounded}
\Vert e^{-\lambda t\langle v\rangle} \nabla_v f\Vert_{L_x^3 L_v^{1+\delta}}<\infty.
\end{equation}

\end{proposition}

\begin{proof}
By replacing $f^{m+1}$ and $f^m$ by $f$ in~\eqref{eqtn_g_v}, we obtain bound for $\partial_v f$ using~\eqref{g_initial}-\eqref{g_infty} with replacing $f^m$ and $f^{m+1}$ into $f$. Following exactly the same proof in Proposition \ref{proposition: L3L1 estimate}, by~\eqref{bound_nabla_v_g}, we obtain
\begin{equation}\label{eqn: result of L3L1 f}
\begin{split}
   &  \Vert e^{-\lambda t\langle v\rangle}\nabla_v f\Vert_{L_x^3 L_v^{1+\delta}}\\
    & \leq C(\Omega,p,\beta,\tilde{\theta})\times (1+\Vert w_{\theta'}f\Vert_\infty)\Big[\Vert w_{\tilde{\theta}}\nabla_v f(0)\Vert_{L_{x,v}^3}+\Vert w_{\theta'} f\Vert_\infty+\Vert e^{-\lambda s\langle v\rangle} w_{\tilde{\theta}}\alpha_{f,\epsilon}^\beta \nabla_{x,v}f\Vert_p\\
    &+\int_0^t \Vert e^{-\lambda s\langle v\rangle
    }\nabla_v f(s)\Vert_{L_x^3L_v^{1+\delta}}ds\Big].
\end{split}
\end{equation}
By assumption~\eqref{eqn: L_infty f} and~\eqref{eqn: W1p f}, the first line of the RHS of~\eqref{eqn: result of L3L1 f} is bounded. We derive the proposition by the Gronwall's inequality.

\end{proof}

The Proposition \ref{Prop: L1+stability Cauchy} suggests that the $L^{1+}$ stability of $f$ can be also obtained upon a good initial condition.
\begin{proposition}\label{L1+stability}
Suppose $f$ and $g$ solve~\eqref{equation for f}~\eqref{equation for phi_f}~\eqref{eqn:BC}, and satisfy all assumption in Proposition \ref{Prop: L3L1 for nabla f}. Then
\Be\label{1+delta_stability for f}
		\|e^{-\lambda t\langle v\rangle} \big[f(t) - g(t)\big] \|_{L^{1+\delta}(\O \times \R^3)}\lesssim \| f_0 - g_0 \|_{L^{1+\delta}(\O \times \R^3)}.
		\Ee

	\end{proposition}

\begin{remark}
Clearly, this proposition serves as a criteria for showing the uniqueness of the solution. The main assumption is that the solution needs to satisfy the initial condition~\eqref{Extra_uniq} and the estimates~\eqref{eqn: L_infty f} and~\eqref{eqn: W1p f}. In Section 5 where we show the uniqueness in Proposition 10, the effort is devoted to bounding~\eqref{eqn: W1p f} for the solution $f$.
\end{remark}

\begin{proof}[\textbf{{Proof of Proposition \ref{L1+stability}}}]

By replacing $f^{m+1}=f$ and $f^m=g$ in Proposition~\ref{Prop: L1+stability Cauchy}, using~\eqref{eqn: Greens 1+delta}~\eqref{eqn: RHS delta m} and~\eqref{eqn: bdr delta m}, we obtain

\begin{equation}\label{eqn: final 1+delta}
  \begin{split}
     &  \Vert e^{-\lambda t\langle v\rangle}(f-g)(t)\Vert_{1+\delta}^{1+\delta} +\int_0^t |e^{-\lambda s\langle v\rangle} (f-g)(s)|^{1+\delta}_{1+\delta,+}      \\
      & \leq C(\delta,T_M,r_{min},\Omega)C(\e)\Vert f_0-g_0\Vert_{1+\delta}^{1+\delta}+C(\delta,T_M,r_{min},\Omega)\int_0^t \Vert  e^{-\lambda s\langle v\rangle}(f-g)(s)  \Vert_{1+\delta}^{1+\delta}\\
      & + C(\delta,T_M,r_{min})\e \int_0^t |e^{-\lambda s\langle v\rangle}(f-g)|_{1+ \delta, + }^{1 + \delta}.
  \end{split}
\end{equation}
We pick $\e\ll 1$ such that $ C(\delta,T_M,r_{min},\Omega)\e <\frac{1}{10}$. With $\e$ fixed we applying the Gronwall inequality and derive the $L^{1+\delta}$-stability (\ref{1+delta_stability for f}).\end{proof}

\section{Existence and Uniqueness}

In this section we prove the existence and uniqueness of the VPB system. The existence is given by Proposition \ref{Prop existence} and the uniqueness is given by Proposition \ref{Prop uniqueness}. The combination of these two Propositions finish the proof of Theorem \ref{local_existence}.

To show the existence, we first realize that due to the linearity of the boundary condition, the boundary contribution of integration of the equation over the domain with a test function will converge to that of the weak limit. The strategy of the proof for the case of diffuse boundary condition thus can be carried over. Below we adapt the proof of Theorem 6 in~\cite{CKL} to fit our setting.

In order to apply all the propositions in the previous sections, we let $t\leq \bar{t}$ with  $\bar{t}$ given in Proposition \ref{Prop: L1+stability Cauchy}. Then from the assumption in Proposition~\ref{proposition: boundedness}, \ref{Prop W1p}, \ref{proposition: L3L1 estimate}, \ref{Prop: L1+stability Cauchy}, we have
  \[\bar{t}\leq t_\delta\leq t_W\leq t_\infty.\]
The condition for these four terms are~\eqref{eqn: bar t},\eqref{eqn: t delta},\eqref{eqn: lesssim} and~\eqref{eqn: t_1} respectively. Thus we conclude the $\bar{t}$ only depends on the variable in~\eqref{eqn: bar t in thm} in Theorem \ref{local_existence}.

\begin{proposition}\label{Prop existence}
Given the assumption in Proposition~\ref{proposition: boundedness} and Proposition \ref{Prop: L1+stability Cauchy}, for $t\leq \bar{t}$ there exists at least one solution $f$ that satisfies
\[\partial_t f+v\cdot \nabla_x f-\nabla_x \phi\cdot \nabla_v f+\frac{v}{2T_M}\cdot \nabla_x \phi f=\Gamma(f,f).\]
Moreover, we have
\begin{equation}\label{eqn: w_theta f infty}
  \Vert w_{\theta'} f\Vert_\infty<\infty.
\end{equation}

\end{proposition}

To prove this proposition we first cite a lemma. This lemma will be used to apply the average lemma in~\eqref{avergae}.
\begin{lemma}\label{extension} Lemma 14 of~\cite{CKL}\\
Assume $f(s,x,v)=e^s f_0(x,v)$ for $s<0$. Assume $\Omega$ is convex and $\sup_{0\leq t\leq T}\Vert E\Vert_{L^\infty(\Omega)}<\infty$. Let $\bar{E}(t,x)=1_{\Omega}(x)E(t,x)$ for $x\in \mathbb{R}^3$. There exists $\bar{f}(t,x,v)\in L^2(\mathbb{R}\times \mathbb{R}^3\times \mathbb{R}^3)$, an extension of $f_\delta$, such that
\[f|_{\Omega\times \mathbb{R}^3}=f_\delta \text{ and } \bar{f}|_{\gamma}=f_{\delta}|_{\gamma},\text{ and } \bar{f}|_{t=0}=f_{\delta}|_{t=0}.\]
%
\end{lemma}

\begin{proof}[\textbf{Proof of Proposition \ref{Prop existence}}]
Since assumptions on Proposition 6 are all satisfied, we apply the result for:
\begin{equation}\label{eqn: Cauchy sequence}
 \sup_{0\leq s\leq t} \Vert e^{-\lambda s\langle v\rangle}(f^l-f^m)(s)\Vert_{L^{1+\delta}}\leq (\frac{1}{2})^{\min\{l,m\}}.
\end{equation}
Thus $e^{-\lambda s\langle v\rangle}f^m$ is a Cauchy sequence in $L^{1+\delta}$ and there exists $f$ such that
\begin{equation}\label{eqn: Strong converge}
  e^{-\lambda t\langle v\rangle}f^m\to e^{-\lambda t\langle v\rangle}f \text{ strongly in }L^{1+\delta}(\Omega\times \mathbb{R}^3).
\end{equation}

By~\eqref{eqn: L_infty bound for f^m+1} and~\eqref{eqn: Strong converge}, there is a unique weak-* limit(up to subsequence) $(w_{\theta'} f^m,w_{\theta'} f^{m+1})\rightharpoonup^* (w_{\theta'} f,w_{\theta'} f)$ weakly-* in $L^\infty(\mathbb{R}\times \Omega \times \mathbb{R}^3)\cap L^\infty(\mathbb{R}\times \gamma)$ with $\Vert w_{\theta'} f\Vert_\infty<\infty$. For $\varphi\in C_c^\infty([0,\bar{t}]\times \bar{\Omega}\times \mathbb{R}^3)$,
\begin{align*}
   & \int_0^{\bar{t}} \int_{\Omega\times \mathbb{R}^3}f^{m+1}[-\partial_t-v\cdot \nabla_x]\varphi+f^{m+1}\{\nabla_x \phi^m\cdot \nabla_v \varphi+\frac{v}{2T_M}\cdot \nabla_x \phi^m\varphi\} \\
   & +\int_{\Omega\times \mathbb{R}^3}f^{m+1}(\bar{t},x,v)\varphi(\bar{t},x,v)-\int_{\Omega\times \mathbb{R}^3}f_0(x,v)\phi(0,x,v)\\
   &=\int_0^{\bar{t}}\int_{\Omega\times \mathbb{R}^3}\Gamma_{\text{gain}}(f^m,f^m)\varphi-\Gamma_{\text{loss}}(f^m,f^{m+1})\varphi \\ &+\int_0^{\bar{t}}\int_{\gamma_+}f^{m+1}\varphi-\int_0^{\bar{t}}\int_{\gamma_-}e^{[\frac{1}{4T_M}-\frac{1}{2T_w(x)}]|v|^2}\int_{n\cdot u>0}f^m(u)e^{[\frac{1}{2T_w(x)}-\frac{1}{4T_M}]|u|^2}d\sigma(u,v)\varphi.
\end{align*}

From~\eqref{eqn: L_infty bound for f^m+1} all the terms except $f^{m+1}\{\nabla_x \phi^m\cdot \nabla_v \varphi+\frac{v}{2T_M}\cdot \nabla_x \phi^m\varphi\}, \Gamma_{\text{gain}}(f^m,f^m)\varphi, \Gamma_{\text{loss}}(f^m,f^{m+1})\varphi$ converge to limit with $f$ instead of $f^{m+1}$ or $f^m$.

We define, for $(t,x,v)\in \mathbb{R}\times \bar{\Omega}\times \mathbb{R}^3$ and for $0<\delta\ll 1$,
\begin{equation}\label{eqn:delta cut off}
  f_\delta^m(t,x,v)=\mathrm{k}_\delta(x,v)f^m(t,x,v)=\chi(\frac{|n(x)\cdot v|}{\delta})[1-\chi(\delta|v|)]\chi(\frac{|v|}{\delta}-1)f^m(t,x,v)
\end{equation}
with smooth function
\begin{equation}\label{}
  \chi(x)=\left\{
            \begin{array}{ll}
              0, & \hbox{$x<0$;} \\
              1, & \hbox{$x\geq 1$.}
            \end{array}
          \right.
\end{equation}
Then $f_\delta(t,x,v)=0$ if either $|n(x)\cdot v|\leq \delta$, $|v|\geq \frac{1}{\delta}$, or $|v|<\delta$. First we consider the term $\Gamma_{\text{loss}}(f^m,f^{m+1})\varphi$.
\begin{align*}
   & |\int_0^{\bar{t}} \int_{\Omega\times \mathbb{R}^3}\Gamma_{\text{loss}}(f^m,f^{m+1})\varphi-\Gamma_{\text{loss}}(f,f)\varphi| \\
   & \leq |\int_0^{\bar{t}} \int_{\Omega\times \mathbb{R}^3}\int_{\mathbb{R}^3}|v-u|[f^m(u)-f(u)]\sqrt{\mu(u)}duf^{m+1}(v)\varphi(t,x,v)dvdxdt|\\
   &+|\int_0^{\bar{t}} \int_{\Omega\times \mathbb{R}^3}\int_{\mathbb{R}^3}|v-u|f(u)\sqrt{\mu(u)}du[f^{m+1}(v)-f(v)]\varphi(t,x,v)dvdxdt|.
\end{align*}

The second term converges to zero from~\eqref{eqn: L_infty bound for f^m+1}. Applying the H\"{o}lder inequality the first term is bounded by
\begin{equation}\label{eqn:Gamma_loss}
\begin{split}
&\int_0^{\bar{t}} \int_{\Omega\times \mathbb{R}^3}\int_{\mathbb{R}^3} (|v|+|u|)e^{-\theta'|v|^2}\big(\mathrm{k}_\delta(x,u)+1-\mathrm{k}_\delta(x,u)\big)[f^m(u)-f(u)]\sqrt{\mu(u)}du\\
&\times\varphi(t,x,v)dvdxdt\sup_{0\leq t\leq \bar{t}} \Vert  e^{\theta'|v|^2}f^{m+1}(t)\Vert_\infty\\
&\lesssim \Big[\int_0^{\bar{t}}\int_{\Omega\times \mathbb{R}^3}|v|e^{-\theta'|v|^2} \big(\int_{\mathbb{R}^3} \langle u\rangle \sqrt{\mu(u)} \mathrm{k}_\delta(x,u)[f^m(u)-f(u)]du\big)^2dvdxdt   \Big]^{1/2}\\
&\times \Big[\int_0^{\bar{t}}\int_{\Omega\times \mathbb{R}^3} \varphi^2(t,x,v)dvdxdt \Big]^{1/2}+O(\delta)\\
&\lesssim \Big[\int_{\mathbb{R}^3} |v|e^{-\theta'|v|^2}\big\Vert \int_{\mathbb{R}^3}\mathrm{k}_\delta(x,u)[f^m(t,x,u)-f(t,x,u)]\langle u\rangle\sqrt{\mu}(u)du\big\Vert_{L^2(\Omega\times [0,T])}dv\Big]^{1/2}+O(\delta).
\end{split}
\end{equation}
The $O(\delta)$ is due to the integration with $1-\mathrm{k}_\delta(x,u)$ which is not 0 only when $|u|\leq 2\delta$ or $|u|\leq \frac{1}{\delta}$, or $|n(x)\cdot u|\leq \delta$, and thus
\[\int_0^{\bar{t}} \int_{\Omega\times \mathbb{R}^3}\int_{\mathbb{R}^3}|u|\sqrt{\mu}\mathbf{1}_{|u|\leq 2\delta \text{ or }|u|\geq \delta^{-1}}[\cdots] =O(\delta).\]

From Lemma~\ref{extension} we have an extension $\bar{f}^m(t,x,v)$ of $\mathrm{k}_\delta(x,u)f^m(t,x,v)$. We apply the average lemma in~\cite{Gl} to $f^m(t,x,v)$,
\begin{equation}\label{avergae}
  \sup_m \Vert \int_{\mathbb{R}^3} \bar{f}^m(t,x,u)\langle u\rangle\sqrt{\mu(u)}du\Vert_{H_{t,x}^{1/4}(\mathbb{R}\times \mathbb{R}^3)}<\infty.
\end{equation}
By $H^{1/4}\subset \subset L^2$, up to subsequence we conclude that
\[\int_{\mathbb{R}^3} \mathrm{k}_\delta(x,u)f^m(t,x,u)\langle u\rangle\sqrt{\mu(u)}du\to \int_{\mathbb{R}^3} \mathrm{k}_\delta(x,u)f(t,x,u)\langle u\rangle\sqrt{\mu(u)}du \text{ strongly in }  L^2_{t,x}.\]
We conclude that~\eqref{eqn:Gamma_loss} goes to 0 as $m\to \infty$.

For $\Gamma_{\text{gain}}(f^m,f^m)\varphi$ we use a test function $\varphi_1(v)\varphi_2(t,x)$. By the standard change of variables $(v,u)\to (v',u')$ and $(v,u)\to (u',u)$, we get
\begin{equation}\label{eqn:v to u}
\begin{split}
   & \int_0^{\bar{t}} \int_{\Omega \times \mathbb{R}^3} \Gamma_{\text{gain}}(f^m,f^m)\varphi-\Gamma_{\text{gain}}(f,f)\varphi \\
    & =\int_0^{\bar{t}} \int_{\Omega\times \mathbb{R}^3} \Gamma_{\text{gain}}(f^m-f,f^m)\varphi+\int_0^{\bar{t}} \iint_{\Omega\times \mathbb{R}^3}\Gamma_{\text{gain}}(f,f^m-f)\varphi\\
    &=\int_0^{\bar{t}} \int_{\Omega\times \mathbb{R}^3}\Big(\int_{\mathbb{R}^3}\int_{\mathbb{S}^2}\big(f^m(t,x,u)-f(t,x,u)\big)\sqrt{\mu(u')}|(v-u)\cdot \omega|\varphi_1(u)d\omega du\Big)\times f^m(t,x,v)\varphi_2(t,x)dvdxdt
\end{split}
\end{equation}
\begin{equation}\label{eqn:v to u'}
    +\int_0^{\bar{t}} \int_{\Omega\times \mathbb{R}^3}\Big(\int_{\mathbb{R}^3}\int_{\mathbb{S}^2}\big(f^m(t,x,u)-f(t,x,u)\big)\sqrt{\mu(u)}|(v-u)\cdot \omega|\varphi_1(u')d\omega du\Big)\times f(t,x,v)\varphi_2(t,x)dvdxdt.
\end{equation}
For $N\gg 1$ we decompose the integration of~\eqref{eqn:v to u'} and~\eqref{eqn:v to u} using
\begin{equation}\label{}
  1=\{1-\chi(|u|-N)\}\{1-\chi(|v|-N)\}+\chi(|u|-N)+\chi(|v|-N)-\chi(|u|-N)\chi(|v|-N).
\end{equation}
Note that $\{1-\chi(|u|-N)\}\{1-\chi(|v|-N)\}\neq 0$ if $|v|\leq N+1$ and $|u|\leq N+1$, and if $\chi(|u|-N)+\chi(|v|-N)-\chi(|u|-N)\chi(|v|-N)\neq 0$ then either $|v|\geq N$ or $|u|\geq N$. Take ~\eqref{eqn:v to u} for example, based on the decomposition, the second part of~\eqref{eqn:v to u} are bounded by
\begin{equation}\label{eqn:first b}
\begin{split}
   & \int_0^{\bar{t}} \int_{\Omega\times \mathbb{R}^3}\int_{\mathbb{R}^3}\int_{\mathbb{S}^2}[\cdot\cdot\cdot]\times\{\chi(|u|-N)+\chi(|v|-N)-\chi(|u|-N)\chi(|v|-N)\} \\
   & \lesssim \sup_m\Vert e^{\theta'|v|^2}f^m\Vert_\infty \Vert e^{\theta'|v|^2}f\Vert_\infty \Big(\int_0^{\bar{t}} \int_{\Omega\times \mathbb{R}^3}\int_{\mathbb{R}}e^{-\frac{\theta'}{2}|v|^2}e^{-\frac{\theta'}{2}|u|^2}[\mathbf{1}_{|v|\geq N}+\mathbf{1}_{|u|\geq N}]dudvdxdt \Big)\leq O(\frac{1}{N}).
\end{split}
\end{equation}
Thus we only need to consider the part with $\{1-\chi(|u|-N)\}\{1-\chi(|v|-N)\}$. First we consider the contribution of this part in~\eqref{eqn:v to u}, which equals to
\begin{equation}\label{eqn:Gamma_gain}
\begin{split}
    &   \int_0^{\bar{t}} \int_{\Omega\times \mathbb{R}^3}\int_{\mathbb{R}^3}\Big(f^m(t,x,v)-f(t,x,v)\Big) \\
     &  \times \{1-\chi(|u|-N)\}\Big(\int_{\mathbb{S}^2}\sqrt{\mu(u')}|(v-u)\cdot \omega|\varphi_1(u)d\omega \Big)du\\
     & \times\{1-\chi(|v|-N)\}f^m(t,x,v)\varphi_2(t,x)dvdxdt.
\end{split}
\end{equation}
Define
\begin{equation}\label{eqn:Phi}
  \Phi_v(u)=\{1-\chi(|u|-N)\}\int_{\mathbb{S}^2}\sqrt{\mu(u')}|(v-u)\cdot \omega|\varphi_1(u)d\omega \text{ for } |v|\leq N+1.
\end{equation}

Then~\eqref{eqn:Gamma_gain} is further written as
\begin{equation}\label{eqn:Gamma gain further}
\begin{split}
   & \int_0^{\bar{t}}\int_\Omega \int_{\mathbb{R}^3}
(1-\mathrm{k}_\delta)\Big(f^m(t,x,v)-f(t,x,v)\Big)\Phi_v(u)\{1-\chi(|v|-N)\}f^m(t,x,v)\varphi_2(t,x)dvdxdt. \\
  & + \mathrm{k}_\delta \Big(f^m(t,x,v)-f(t,x,v)\Big)\Phi_v(u)\{1-\chi(|v|-N)\}f^m(t,x,v)\varphi_2(t,x)dvdxdt.
\end{split}
\end{equation}
The first term in~\eqref{eqn:Gamma gain further} is bounded by $O(\delta)\sup_m \Vert e^{\theta'|v|^2}f^m\Vert_\infty,$ introducing $O(\delta)$ error.

To handle the second term in~\eqref{eqn:Gamma gain further}, we form an open cover of $\{v\in \mathbb{R}^3:|v|\leq N+1\}\subset \bigcup_{i=1}^{O(N^3/\delta^3)}B(v_i,\delta)$. We choose $\delta$ to be small enough so that
\begin{equation}\label{eqn:difference}
  |\Phi_v(u)-\Phi_{v_i}(u)|<\epsilon, \text{ if } v\in B(v_i,\delta).
\end{equation}
This leads to
\begin{align*}
   &   \int_0^{\bar{t}}\int_{\Omega}\int_{\mathbb{R}^3}\sum_{i}1_{v\in B(v_i,\delta)}\int_{\mathbb{R}^3}\big(f^m(t,x,u)-f(t,x,u)\big)\big(\Phi_v(u)-\Phi_{v_i}(u)\big)du\\
   & \times \{1-\chi(|v|-N)\}f^m(t,x,v)\varphi_2(t,x)dvdxdt=O(\e).
\end{align*}

%
By rewriting $\Phi_v(u)$ in the second term of~\eqref{eqn:Gamma gain further} as $\Phi_v(u)-\Phi_{v_i}(u)+\Phi_{v_i} u$ we finally obtain
\begin{equation}\label{eqn:Gamma gain bound}
\begin{split}
 \eqref{eqn:Gamma gain further}  &\leq O(\epsilon)+O(\delta)+\int_0^{\bar{t}} \int_{\Omega}\sum_i \int_{\mathbb{R}^3}1_{v\in B(v_i,\delta)}\int_{\mathbb{R}^3}\mathrm{k}_\delta(x,u)(f^m(t,x,u)-f(t,x,u))\Phi_{v_i}(u)du  \\
   &   \times \{1-\chi(|v|-N)\}f^m(t,x,v)\varphi_2(t,x)dvdxdt.
\end{split}
\end{equation}
By the average lemma we can conclude
\begin{equation}\label{eqn:Average lemma for Gamma gain}
\max_{1\leq i \leq O(\frac{N^3}{\delta^3})}\sup_m \Vert \int_{\mathbb{R}^3}\mathrm{k}_\delta(x,u)f^m(t,x,u)\Phi_{v_i}(u)du\Vert_{H_{t,x}^{1/4}(\mathbb{R}\times \mathbb{R}^3)}<\infty.
\end{equation}
For $i=1$ we extract a subsequence $m_1\subset \mathcal{M}_1$ such that
\begin{equation}\label{eqn: Convergent subsequence}
  \int_{\mathbb{R}^3}\mathrm{k}_\delta(x,u)f^{m}(t,x,u)\Phi_{v_i}(u)du\to \int_{\mathbb{R}^3}\mathrm{k}_\delta(x,u)f(t,x,u)\Phi_{v_i}(u)du \text{ strongly in }L_{t,x}^2.
\end{equation}
By the Cantor diagonal argument we can extract convergent subsequences $\mathcal{M}_{O(\frac{N^3}{\delta^3})}\subset \cdots \subset \mathcal{M}_2\subset \mathcal{M}_1$. We take the last subsequence $m\in \mathcal{M}_{O(\frac{N^3}{\delta^3})}$ and define it as $f^m$. Then we have~\eqref{eqn: Convergent subsequence} for all $i$. Thus we conclude
\begin{equation}\label{eqn: second b}
\eqref{eqn:Gamma_gain}\leq C_{\varphi_2,N}\sup_m\Vert e^{\theta'|v|^2}f^m\Vert_\infty \max_i\int_0^{\bar{t}} \Big\Vert\int_{\mathbb{R}^3}\mathrm{k}_\delta(x,u)(f^m(t,x,u)-f(t,x,u))\Phi_{v_i}(u)du\Big\Vert_{L^2_{t,x}} \to 0.
\end{equation}
Combining~\eqref{eqn:first b}~\eqref{eqn: second b} we conclude $\eqref{eqn:v to u}$ goes to 0.

Similarly we can prove~\eqref{eqn:v to u'} goes to 0. This in the end gives the convergence of $\Gamma_{\text{gain}}(f^m,f^m)$

Finally for the last term $f^{m+1}\{\nabla_x \phi^m\cdot \nabla_v \varphi+\frac{v}{2T_M}\cdot \nabla_x \phi^m \varphi\}$, from
\[-(\Delta \phi^m-\Delta \phi)=\int \mathrm{k}_\delta (f^m-f)\sqrt{\mu}+\int(1-\mathrm{k}_\delta)(f^m-f)\sqrt{\mu}.\]
By the standard elliptic estimate we have
\begin{equation}\label{eqn:phi^m}
  \Vert \nabla_x \phi^m-\nabla_x \phi\Vert_{L^2_{t,x}}\leq \Vert \mathrm{k}_\delta (f^m-f)\sqrt{\mu}\Vert_{L_{t,x}^2}+O(\delta)\sup_m\Vert e^{\theta|v|^2}f^m\Vert_\infty.
\end{equation}
By the previous argument with the average lemma we can prove $\nabla_x \phi^m\to \nabla_x \phi$ strongly in $L_{t,x}^2$ as $m\to \infty$. We have
\begin{align*}
   & \int_0^{\bar{t}} \int_{\Omega\times \mathbb{R}^3} f^{m+1}\{\nabla_x \phi^m\cdot \nabla_v \varphi+\frac{v}{2}\cdot \nabla_x \phi^m\varphi\}-f\{\nabla_x \phi\cdot \nabla_v \varphi+\frac{v}{2}\cdot \nabla_x \phi\varphi\} dvdxdt \\
   & \leq \int_0^{\bar{t}} \int_{\Omega\times \mathbb{R}^3}(f^{m+1}-f)\{\nabla_x \phi^m\cdot \nabla_v \varphi+\frac{v}{2}\cdot \nabla_x \phi^m\varphi\}    dvdxdt\\
   &+\int_0^{\bar{t}}\int_{\Omega\times \mathbb{R}^3}f\{\nabla_x(\phi^m-\phi)\cdot \nabla_v\varphi+\frac{v}{2}\cdot \nabla_x (\phi^m-\phi)\varphi\}dvdxdt.
\end{align*}

The first term goes to 0 by the weak* convergence of $e^{\theta|v|^2}f$ in $L^\infty$, the second term goes to 0 by~\eqref{eqn:phi^m}. This proves the existence of a weak solution $f\in L^\infty$.

\end{proof}

The next proposition gives the uniqueness of the VPB system.

\begin{proposition}\label{Prop uniqueness}
Assume $\Vert w_\theta f_0\Vert_\infty<\infty$ and $T_w$ satisfies~\eqref{eqn: Constrain on T}. Assume $ \| w_{\tilde{\theta}} \alpha_{f_0 ,\epsilon}^\beta \nabla_{x,v} f _0 \|_{p} <\infty$ for $0< \epsilon \ll1$ , and $(p, \beta)$ satisfy~\eqref{Condition for p}.

Also if $\| w_{\tilde{\theta}} \nabla_v f_0 \|_{L^{3}_{x,v}}<+\infty,$ then for $t\leq \bar{t}$ there is a unique solution to~\eqref{equation for f} satisfying~\eqref{31_local_bound} and~\eqref{W1p_local_bound}.

\end{proposition}

Before proving the proposition we need this lemma. Such lemma will be used in proving the convergence of~\eqref{af_l_3}.
\begin{lemma}\label{cannot_graze}Assume that $\O$ is convex~\eqref{eqn: convex}. Suppose that $\sup_t\| E(t) \|_{C^1_x} < \infty$ and
	\Be\label{nE=0}
	n(x) \cdot E(t,x) =0 \ \ \text{for } x \in \p\O \ \text{and for all  } t.
	\Ee
	Assume $(t,x,v) \in \R_+ \times \bar{\O} \times \R^3$ and $t+1 \geq \tb(t,x,v)$. If $x \in \p\O$ then we further assume that $n(x) \cdot v > 0$. Then we have
	\Be
	n(\xb(t,x,v)) \cdot \vb(t,x,v) <0.\label{no_graze}
	\Ee
\end{lemma}
\begin{proof}\hide  \textit{Step 1.}  We claim that for all $(t,x) \in [ 0, \infty) \times \bar{\O}$ as $N \rightarrow \infty$
	\Be\label{conv_NN}
	\mathbf{1}_{\tb(t,x,u)< N}
	\mathbf{1}_{n(\xb(t,x,u)) \cdot \vb(t,x,u) < -\frac{1}{N}}
	\nearrow \mathbf{1}_{\tb(t,x,u)< \infty} \ \ \text{almost every } u \in \R^3.
	\Ee
	First we prove that, for fixed $N \in \mathbb{N}$, as $M \rightarrow \infty$
	\Be\label{conv_NM}
	\mathbf{1}_{\tb(t,x,u)< N}
	\mathbf{1}_{n(\xb(t,x,u)) \cdot \vb(t,x,u) < -\frac{1}{M}}
	\nearrow \mathbf{1}_{\tb(t,x,u)< N} \ \ \text{almost every } u \in \R^3.
	\Ee
	Since $\mathbf{1}_{\tb(t,x,u)< N}$ converges to $\mathbf{1}_{\tb(t,x,u)< \infty}$ as $N \rightarrow \infty$ we can apply Cantor's diagonal argument to conclude (\ref{conv_NN}) from (\ref{conv_NM}).\unhide
	

	\textit{Step 1.} Note that locally we can parameterize the trajectory (see Lemma 15 in \cite{GKTT} for details). We consider local parametrization~\eqref{eqn: C3 map}. We drop the subscript $p$ for the sake of simplicity. If $X(s;t,x,v)$ is near the boundary then we can define $(X_n, X_\parallel)$ to satisfy
	\Be\label{X_local}
	X(s;t,x,v)  =   \eta (X_\parallel (s;t,x,v)) + X_n(s;t,x,v) [- n(X_\parallel(s;t,x,v))].
	\Ee
	
	\hide $\eta$ near $\p\O$ such that
	\Be
	\eta : (x,y,z)\in B(0,r_{1})\cap \R^{3}_{+} \mapsto B(p,r_{2})\cap\O,
	\Ee
	where $\eta(0)=p\in\p\O$ and $\eta(x,y,0) \in \p\O$ for some $r_{1}, r_{2} > 0$. Then for $X(s;t,x,v)$, there exist a unique $x_{*}\in B(p,r_{2})\cap \p\O$ such that
	\Be \label{X def}
	|X(s;t,x,v)-x_{*}| \leq \sup_{x\in B(p,r_{2})\cap \p\O} |X(s;t,x,v)-x|,
	\Ee
	since $\eta$ is bijective. We define
	\[
	(X_{\parallel},0) = \eta^{-1}(x_{*}) \quad \text{and} \quad  X_{n} := |X(s;t,x,v)-x_{*}|.
	\]\unhide
	For the normal velocity 
	we define
	\Be\label{def_V_n}
	V_{n}(s;t,x,v) := V(s;t,x,v)\cdot [-n(X_{\parallel}(s;t,x,v))].
	\Ee
	We define $V_{\parallel}$ tangential to the level set $  \big( \eta(X_{\parallel}) + X_{n}(-n(X_{\parallel})) \big)$ for fixed $X_{n}$. Note that
	\[
	\frac{\p   \big( \eta(x_{\parallel} ) + x_{n}(-n(x_{\parallel})) \big) }{\p {x_{\parallel, i}}}  \perp n(x_{\parallel}) \ \  \text{for} \ i=1,2.
	\]
	We define $(V_{\parallel,1}, V_{\parallel,2})$ as
	\Be\label{def_V_parallel}
	V_{\parallel, i} :=\Big( V - V_{n}[-n(X_{\parallel})]\Big) \cdot   \Big(
	\frac{\p   \eta(X_{\parallel} )}{\p {x_{\parallel, i}} } + X_{n}\Big[- \frac{\p n(X_{\parallel})}{\p x_{\parallel, i}}\Big] \Big)  .
	\Ee
	\hide \Be\notag
	\begin{split}
		&\sum_{i=1,2} V_{\parallel,i} \p_{i} \big( \eta(X_{\parallel},0) + X_{n}(-n(X_{\parallel})) \big)  \\
		&= \nabla_{\parallel} \big( \eta(X_{\parallel},0) + X_{n}(-n(X_{\parallel})) \big) V_{\parallel} = V - V_{n}(-n(X_{\parallel}))  .
	\end{split}
	\Ee\unhide
	Therefore we obtain
	\Be \label{V_local}
	V(s;t,x,u)   = V_n [- n(X_\parallel)] + V_\parallel \cdot \nabla_{x_\parallel} \eta (X_\parallel )
	- X_n V_\parallel \cdot \nabla_{x_\parallel} n (X_\parallel).
	\Ee
	
	Directly we have
	\Be \begin{split}\notag
		\dot{X}(s;t,x,u) &=\dot{X}_{\parallel} \cdot \nabla_{x_\parallel}\eta (X_\parallel) + \dot{X}_n [- n(X_\parallel)] - X_{n}\dot{X}_{\parallel}  \cdot \nabla_{x_\parallel} n(X_{\parallel}) .
	\end{split}\Ee
	Comparing coefficients of normal and tangential components, we obtain that
	\Be\label{dot_Xn_Vn}
	\dot{X}_{n}(s;t,x,v) = V_{n}(s;t,x,v) , \ \  \dot{X}_{\parallel}(s;t,x,v) = V_{\parallel}(s;t,x,v).
	\Ee
	
	On the other hand, from (\ref{V_local}),
	\Be \begin{split} \label{Vdotn}
		\dot{V} (s) &=  \dot{V}_{n} [-n(X_{\parallel})] - V_{n} \nabla_{x_\parallel} n(X_{\parallel})\dot{X}_{\parallel} + V_{\parallel}\cdot\nabla^{2}_{x_\parallel}\eta(X_{\parallel}) \dot{X}_{\parallel} + \dot{V}_{\parallel} \cdot \nabla_{x_\parallel}\eta(X_{\parallel})  \\
		&\quad - \dot{X}_{n}\nabla_{x_\parallel} n(X_{\parallel})V_{\parallel} - X_{n}\nabla_{x_\parallel} n(X_{\parallel})\dot{V}_{\parallel} - X_{n} V_{\parallel}\cdot\nabla_{x_\parallel}^{2}n(X_{\parallel})\dot{X}_{\parallel}.
	\end{split}\Ee
	From $(\ref{Vdotn})\cdot [-n(X_{\parallel})]$, (\ref{dot_Xn_Vn}), and $\dot{V}=E$, we obtain that
	\Be \begin{split}\label{hamilton_ODE_perp}
		\dot{V}_n (s)
		&=  [V_\parallel (s)\cdot \nabla^2 \eta (X_\parallel(s)) \cdot V_\parallel(s) ] \cdot n(X_\parallel(s))
		+  E (s , X (s ) ) \cdot [-n(X_\parallel(s)) ] \\
		&\quad - X_n (s) [V_\parallel(s) \cdot \nabla^2 n (X_\parallel(s)) \cdot V_\parallel(s)]  \cdot n(X_\parallel(s)) .
	\end{split}\Ee
	
	\vspace{4pt}
	
	\textit{Step 2.} We prove (\ref{no_graze}) by the contradiction argument. Assume we choose $(t,x,v)$ satisfying the assumptions of Lemma \ref{cannot_graze}. Let us assume
	\Be\label{initial_00}
	X_n (t-\tb;t,x,v)  +V_n (t-\tb;t,x,v) =0.
	\Ee
	
	First we choose $0<\epsilon \ll 1$ such that $X_n(s;t,x,v) \ll 1$ and
	\Be\label{Vn_positive}
	V_n (s;t,x,v) \geq0 \ \ \text{for} \  t- \tb(t,x,v)<s<t-\tb(t,x,v) + \epsilon.
	\Ee
	The sole case that we cannot choose such $\epsilon>0$ is when there exists $0< \delta\ll1$ such that $V_n(s;t,x,v)<0$ for all $s \in ( t-\tb(t,x,v), t-\tb(t,x,v) + \delta)$. But from (\ref{dot_Xn_Vn}) for $s \in ( t-\tb(t,x,v), t-\tb(t,x,v) + \delta)$
	$$
	0 \leq X_n(s;t,x,v)   =  X_n(t-\tb(t,x,v);t,x,v)  +  \int^s_{t-\tb(t,x,v)} V_n (\tau; t,x,v) \dd \tau <  0.$$
	
	Now with $\epsilon>0$ in (\ref{Vn_positive}), temporarily we define that $t_* := t-\tb(t,x,v) + \epsilon$, $x_* = X(t-\tb(t,x,v) + \epsilon; t,x,v),$ and $v_* = V(t-\tb(t,x,v) + \epsilon; t,x,v)$. Then $(X_n(s;t,x,v), X_\parallel (s;t,x,v)) = (X_n(s; t_*, x_*, v_*), X_\parallel (s; t_*, x_*, v_*))$ and \\$(V_n(s;t,x,v), V_\parallel (s;t,x,v)) = (V_n(s; t_*, x_*, v_*), V_\parallel (s; t_*, x_*, v_*))$.
	\hide
	From (\ref{hamilton_ODE}), for $(X_n(s), X_\parallel (s)) = (X_n(s; t_*, x_*, u_*), X_\parallel (s; t_*, x_*, u_*))$ and $(V_n(s), V_\parallel (s)) = (V_n(s; t_*, x_*, u_*), V_\parallel (s; t_*, x_*, u_*))$,
	\Be \begin{split}\label{hamilton_ODE_perp}
		\dot{X}_n (s)  =& V_n (s),\\
		\dot{V}_n (s)  =&
		[V_\parallel (s)\cdot \nabla^2 \eta (X_\parallel(s)) \cdot V_\parallel(s) ] \cdot n(X_\parallel(s))
		+   \nabla\phi (s , X (s ) ) \cdot n(X_\parallel(s))  \\
		&+   X_n (s) [V_\parallel(s) \cdot \nabla^2 n (X_\parallel(s)) \cdot V_\parallel(s)]  \cdot n(X_\parallel(s)) .
	\end{split}\Ee\unhide
	
	Now we consider the RHS of (\ref{hamilton_ODE_perp}). From~\eqref{eqn: convex}, the first term $[V_\parallel(s) \cdot \nabla^2 \eta (X_\parallel(s)) \cdot V_\parallel(s) ] \cdot n(X_\parallel(s))\leq 0$. By an expansion and (\ref{nE=0}) we can bound the second term
	\Be\begin{split}\label{expansion_E}
		&E (s , X(s )) \cdot n(X_\parallel(s ) )\\
		=&   \ E (s , X_n(s ), X_\parallel(s ) ) \cdot n(X_\parallel (s )) \\
		=& \  E (s , 0, X_\parallel(s ) ) \cdot n(X_\parallel (s ))
		+ \| E (s) \|_{C_x^1}  O( |X_n(s )| )\\
		=&  \ \| E (s) \|_{C_x^1}  O( |X_n(s )| ).
	\end{split}\Ee
	From (\ref{hamilton_ODE}) and assumptions of Lemma \ref{cannot_graze},
	$$|V_\parallel (s;t,x,v )|\leq |v| + \tb(t,x,v) \| E \|_\infty  \leq |v| +  (1+t) \| E \|_\infty.$$
	Combining the above results with (\ref{hamilton_ODE_perp}), we conclude that
	\Be\label{eqn: dot V_n}
	\dot{V}_n(s;t_*,x_*,v_*) \lesssim  ( |v| + (1+ t) \| E\|_\infty  )^2X_n(s;t_*,x_*,v_*) ,
	\Ee
	and hence from (\ref{dot_Xn_Vn}) for $t-\tb(t,x,v)\leq s \leq t_*$
	\Be\label{ODE_X+V}
	\begin{split}
		&\frac{d}{ds} [X_n (s;t_*,x_*,v_* )  +V_n (s ;t_*,x_*,v_*) ]\\
		\lesssim & \ ( |v| + (1+ t) \| E\|_\infty  )^2  [X_n (s;t_*,x_*,v_* )  +V_n (s;t_*,x_*,v_* ) ].\end{split}
	\Ee
	By the Gronwall inequality and (\ref{initial_00}), for $t-\tb(t,x,v)\leq s \leq t_*$
	\Be \begin{split}\notag
		& [X_n (s;t_*,x_*,v_*)  +V_n (s;t_*,x_*,v_*) ]  \\
		\lesssim & \   [X_n (t-\tb(t,x,u))  +V_n (t-\tb(t,x,u)) ] e^{C \epsilon ( |v| + (1+ t) \| E\|_\infty  )^2) }\\
		=&  \ 0.
	\end{split}\Ee
	
	From (\ref{Vn_positive}) we conclude that $X_n (s;t,x,v) \equiv 0$ and $V_n (s;t,x,v) \equiv 0$ for all $s \in [t-\tb(t,x,u), t-\tb(t,x,u) + \epsilon]$. We can continue this argument successively to deduce that $X_n (s;t,x,v) \equiv 0$ and $V_n (s;t,x,v) \equiv 0$ for all $s \in [t-\tb(t,x,v), t]$. Therefore $x_n =0 = v_n$ which implies $x \in \p\O$ and $n(x) \cdot v =0$. This is a contradiction since we chose $n(x) \cdot v>0$ if $x \in \p\O$.\end{proof}

\hide\begin{lemma} If $n(\xb(t,x,v)) \cdot \vb(t,x,v) \neq 0$ then $(\tb,\xb,\vb)$ is differentiable and
	\Be\begin{split}\label{computation_tb_x}
		\frac{\p\tb}{\p x_i}  =& \  \frac{1}{n(\xb) \cdot \vb}n(\xb) \cdot  \left[
		e_i + \int^{t-\tb}_t \int^s_t \Big(\frac{\p X(\tau )}{\p x_i} \cdot \nabla\Big) E(\tau, X(\tau )) \dd \tau \dd s
		\right]  ,\\
		\frac{ \p\xb}{\p x_i} = & \  e_i - \frac{\p \tb}{\p x_i} \vb + \int^{t-\tb}_{t} \int^s_t    \Big(\frac{\p X(\tau )}{\p x_i} \cdot \nabla\Big) E(\tau, X(\tau ))  \dd \tau \dd s,\\
		\frac{\p \vb}{\p x_i} = & \ - \frac{\p \tb}{\p x_i} E(t-\tb, \xb) + \int^{t-\tb}_t   \Big(\frac{\p X(\tau )}{\p x_i} \cdot \nabla\Big) E(\tau, X(\tau )) \dd \tau,\\
		\frac{\p\tb}{\p v_i}  =& \  \frac{1}{n(\xb) \cdot \vb}n(\xb) \cdot  \left[
		e_i + \int^{t-\tb}_t \int^s_t \Big(\frac{\p X(\tau )}{\p v_i} \cdot \nabla\Big) E(\tau, X(\tau )) \dd \tau \dd s
		\right]  ,\\
		\frac{\p \xb}{\p v_i} = & \ - \tb e_i - \frac{
			\p \tb}{\p v_i} \vb + \int^{t-\tb}_{t} \int^s_t    \Big(\frac{\p X(\tau )}{\p v_i} \cdot \nabla\Big) E(\tau, X(\tau ))  \dd \tau \dd s ,\\
		\frac{\p \vb}{\p v_i} = & \ e_i - \frac{\p \tb}{\p v_i} E(t-\tb, \xb) + \int^{t-\tb}_t
		\Big(\frac{\p X(\tau )}{\p v_i} \cdot \nabla\Big) E(\tau, X(\tau ))  \dd \tau.
	\end{split} \Ee
	where $(X(\tau),V(\tau))= (X(\tau;t,x,v), V(\tau;t,x,v))$.
\end{lemma}
\begin{proof}
	These equalities can be derived from direct computations and the implicit function theorem. See \cite{KL1} for details.
\end{proof}\unhide

\begin{proof}[\textbf{Proof of Proposition \ref{Prop uniqueness}}]

Since all the assumption in Proposition \ref{proposition: boundedness} \ref{Prop W1p} \ref{Prop: L1+stability Cauchy} are valid, from Proposition \ref{Prop existence} we have the existence of the solution $f$ to~\eqref{equation for f}. To conclude the uniqueness, we aim to apply Proposition \ref{L1+stability}. Thus we need to verify the condition~\eqref{eqn: L_infty f} and~\eqref{eqn: W1p f}. The first condition is already given in~\eqref{eqn: w_theta f infty} in Proposition \ref{Prop existence}. Thus we focus on establishing the second condition.

For $f$ satisfying~\eqref{equation for f}, we claim
		\Be\label{eqn: claim w1p}\begin{split}
			&\sup_{0 \leq t \leq \bar{t}}  \|w_{\tilde{\theta}} f  (t) \|_p^p
			+
			\sup_{0 \leq t \leq \bar{t}}  \|e^{-\lambda t\langle v\rangle} w_{\tilde{\theta}}\alpha_{f ,\epsilon}^\beta \nabla_{x,v} f (t) \|_{p}^p
			+
			\int^{\bar{t}}_0  |e^{-\lambda t\langle v\rangle} w_{\tilde{\theta}}\alpha_{f ,\epsilon}^\beta \nabla_{x,v} f (t) |_{p,+}^p \\
			&\lesssim  \   \| w_{\tilde{\theta}}f_0 \|_p^p
			+
			\| w_{\tilde{\theta}}\alpha_{f_{0 },\epsilon}^\beta \nabla_{x,v} f  _0 \|_{p}^p
			.
		\end{split}\Ee

By the weak lower-semicontinuity of $L^p$ we know that
		\[
			w_{\tilde{\vartheta}}\alpha^\beta_{f^\ell, \e} \nabla_{x,v} f^{\ell+1} \rightharpoonup \mathcal{F},\quad
			\sup_{0 \leq t \leq \bar{t}}\|\mathcal{F}(t)\|_p^p \leq \liminf   \sup_{0 \leq t \leq \bar{t}}  \| w_{\tilde{\vartheta}}\alpha_{f^{\ell} ,\e}^\beta \nabla_{x,v} f^{\ell+1} (t) \|_{p}^p,
		\]
		and
		\[
		\int^{\bar{t}}_0  | \mathcal{F}  |_{p,+}^p \leq \liminf
		\int^{\bar{t}}_0  | w_{\tilde{\vartheta}}\alpha_{f^{\ell} ,\e}^\beta \nabla_{x,v} f^{\ell+1} (t) |_{p,+}^p.
		\]
		We need to prove that
		\Be\label{a_nabla_f_seq}
			\mathcal{F}
			=w_{\tilde{\vartheta}}\alpha^{\beta}_{f,\e} \nabla_{x,v}f  \ \ \text{almost everywhere except } \gamma_0.
		\Ee

		We claim that, up to some subsequence, for any given smooth test function $\psi \in C_c^\infty(\bar{\O} \times \R^3\backslash \gamma_0)$
\Be\label{lim_Fell=F}
\lim_{\ell  \rightarrow \infty}\int^t_0\iint_{\O \times \R^3} w_{\tilde{\vartheta}} \alpha^\beta_{f^{\ell },\e} \nabla_{x,v} f^{\ell +1} \psi \dd x\dd v
= \int^t_0\iint_{\O \times \R^3} w_{\tilde{\vartheta}} \alpha^\beta_{f ,\e} \nabla_{x,v} f  \psi \dd x\dd v.
\Ee
We note that we need to extract a single subsequence, let say $\{\ell_*\} \subset \{\ell\}$, satisfying (\ref{lim_Fell=F}) for all test functions in $C_c^\infty(\bar{\O} \times \R^3\backslash \gamma_0)$.

We will exam (\ref{lim_Fell=F}) by the identity obtained from the integration by parts
		\begin{eqnarray}
			&&\int^t_0\iint_{\O \times \R^3} w_{\tilde{\vartheta}} \alpha^\beta_{f^{\ell },\e} \nabla_{x,v} f^{{\ell }+1} \psi \dd x\dd v
			\notag
			\\
			&= &
			-  \int^t_0\iint_{\O \times \R^3} \alpha^\beta_{f^{\ell },\e}  f^{{\ell }+1} \nabla_{x,v}(w_{\tilde{\vartheta}}\psi) \dd x\dd v\label{af_l_1}
			  \\
			&& + \int^t_0 \iint_{\gamma} n  \alpha^\beta_{f^{\ell },\e}   f^{{\ell }+1} (w_{\tilde{\vartheta}}\psi)
			\label{af_l_2}\\
			&&-  {\int^t_0\iint_{\O \times \R^3} \nabla_{x,v}\alpha^\beta_{f^{\ell },\e}  f^{{\ell }+1} (w_{\tilde{\vartheta}}\psi) \dd x\dd v}. \label{af_l_3}
		\end{eqnarray}

For each $N\in\mathbb{N}$ we define a set
\Be\label{set_N}
\mathcal{S}_N:= \Big\{ (x,v) \in \bar{\O} \times \R^3: \text{dist}(x, \p\O) \leq \frac{1}{N}  \ \text{and} \ |n(x) \cdot v| \leq \frac{1}{N}  \Big\}
		 \cup \{|v|>N\}.
\Ee
For a given test function we can always find $N\gg1$ such that
		\Be\label{test_function_condition}
		supp(\psi) \subset (\mathcal{S}_N)^c:=\bar{\O} \times \R^3 \backslash \mathcal{S}_N.
		 \Ee

	We focus on proving the convergence of (\ref{af_l_1}) and (\ref{af_l_2}). From (\ref{alphaweight}), Lemma \ref{lemma: phi_inf} and the uniform in $\ell$ estimate~\eqref{eqn: L_infty bound for f^m+1} , if $(x,v) \in (\mathcal{S}_N)^c$ then
		\Be
		\begin{split}\notag
		\sup_{\ell \geq 0}	| \alpha^\beta_{f^\ell,\e}(t,x,v)|  \lesssim |v|^\beta + (t+\e)^\beta\sup_{\ell \geq 0} \| \nabla \phi_{f^\ell} \|_\infty^\beta \lesssim N^\beta  +(\bar{t}+\e)^\beta \sup_{\ell \geq 0} \| w_{\vartheta} f^\ell \|_\infty^\beta\leq C_N <+\infty  .
		\end{split}
		\Ee
	Hence we extract a subsequence (let say $\{\ell_N\}$) out of subsequence in Proposition \ref{Prop existence} such that $
	 \alpha_{f^{\ell_N},\e}^\beta  \overset{\ast}{\rightharpoonup} A \in L^\infty  \textit{ weakly}-* \textit{ in }  L^\infty ((0,\bar{t}) \times (\mathcal{S}_N)^c) \cap L^\infty ((0,\bar{t})  \times (\gamma \cap  (\mathcal{S}_N)^c)).$ Note that $\alpha_{f^{\ell_N},\e}^\beta$ satisfies $[\p_t + v\cdot \nabla_x - \nabla_x \phi^{\ell_N}\cdot \nabla_v]  \alpha_{f^{\ell_N},\e}^\beta=0$ and $\alpha_{f^{\ell_N},\e}^\beta|_{\gamma_-} =|n \cdot v |^\beta$. By passing a limit in the weak formulation we conclude that $[\p_t + v\cdot \nabla_x - \nabla_x \phi_f \cdot \nabla_v]  A =0$ and $A |_{\gamma_-} =|n \cdot v |^\beta$. By the uniqueness of the Vlasov equation ($\nabla \phi_f \in W^{1,p}$ for any $p<\infty$) we derive $A=  \alpha_{f ,\e}^\beta$ almost everywhere and hence conclude that
	\Be\label{converge_alpha_ell}
	 \alpha_{f^{\ell_N},\e}^\beta  \overset{\ast}{\rightharpoonup}   \alpha_{f ,\e}^\beta  \textit{ weakly}-* \textit{ in }  L^\infty ((0,\bar{t})  \times (\mathcal{S}_N)^c) \cap L^\infty ((0,\bar{t})  \times (\gamma \cap  (\mathcal{S}_N)^c)).
	\Ee	
Now the convergence of (\ref{af_l_1}) and (\ref{af_l_2}) is a direct consequence of strong convergence of~\eqref{eqn: Strong converge} and the weak$-*$ convergence of (\ref{converge_alpha_ell}):
\begin{equation}\label{convergence and fs}
\lim_{\ell\to \infty} (\ref{af_l_1})+(\ref{af_l_2})=-  \int^t_0\iint_{\O \times \R^3} \alpha^\beta_{f,\e}  f \nabla_{x,v}(w_{\tilde{\vartheta}}\psi) \dd x\dd v+ \int^t_0 \iint_{\gamma} n  \alpha^\beta_{f,\e}   f (w_{\tilde{\vartheta}}\psi).
\end{equation}

 We now show the convergence of~\eqref{af_l_3}.

		\vspace{4pt}

		\textit{Step 1. } Let us choose $(x,v) \in (\mathcal{S}_N)^c$. From~\eqref{alphaweight},
	\Be\label{alpha=1}
\text{If} \  \ \tb^{f^\ell} \geq t+ \e \  \ then \  \ 	\alpha_{f^\ell, \e}(t,x,v)=1.
	\Ee
	 From now we only consider that case
	\Be\label{tb_upperT^**}
	\tb^{f^\ell} (t,x,v) \leq \e +  t.
	\Ee
	
	If $|v|\geq 2 (\e + \bar{t})\sup_{\ell} \| \nabla \phi^\ell \|_\infty$ then
	\Be\notag
	\begin{split}
	|V^{f^\ell}(s;t,x,v)| &\geq |v| - \int^t_s \| \nabla \phi^\ell(\tau) \|_\infty \dd \tau \\
	&\geq  (\e + \bar{t})\sup_{\ell} \| \nabla \phi^\ell \|_\infty \ \ \ for \ all \  \ell  \ and \ s \in [-\e ,\bar{t}].
	\end{split}
	\Ee
Then we apply a velocity lemma derived in (3.32) of \cite{CKL}. We define
			\Be\label{beta}
			\tilde{\alpha}(t,x,v) : = \sqrt{\xi(x)^2  + |\nabla \xi(x) \cdot u|^2 - 2 (u\cdot \nabla_x^2 \xi(x) \cdot u) \xi(x)
			}.
			\Ee
			For $|u| \geq N$ and $t- \tb(t,x,u)\geq - \e/2$,
			\Be\label{velocity_N}
			 {\alpha}_{f,\e}(t,x,u)^2  \lesssim
				\tilde{\alpha} ( t,x,u)^2 \lesssim {\alpha}_{f,\e}(t,x,u)^2.
			\Ee
At $s=t-\tb^{f^{\ell}}(t,x,v)$, we obtain
\Be\label{alpha_|v|>}
|n(\xb^{f^\ell}) \cdot \vb^{f^\ell}|
 \geq \frac{e^{- \frac{C_\O }{\sup_{\ell} \| \nabla \phi^\ell  \|_\infty}}}{C_\O } \times \frac{1}{N} \ \ \ for \ all \ \ell.
\Ee

	 	\vspace{4pt}
	
	 \textit{Step 2. }  From now on we assume (\ref{tb_upperT^**}) and
	\Be\label{upper_|V|}
	\begin{split}
	&|v|\leq 2 (\e + \bar{t})\sup_{\ell} \| \nabla \phi^\ell \|_\infty,\\
	or& ,\ from \ (\ref{hamilton_ODE}), \ |V^{f^\ell}(s;t,x,v)|\leq 3 (\e + \bar{t})\sup_{\ell} \| \nabla \phi^\ell  \|_\infty \ \    for \  \ s \in [-\e,\bar{t}].
	\end{split}\Ee
	 Let $(X_n^{f^\ell}, X_\parallel^{f^\ell},V_n^{f^\ell}, V_\parallel^{f^\ell})$ satisfy (\ref{dot_Xn_Vn}), (\ref{def_V_parallel}), and (\ref{hamilton_ODE_perp}) with $E=- \nabla \phi^\ell$.
	
	 	 Let us define
		\Be \label{tau_1}
		\tau_1
		: = \sup  \big\{ \tau \geq  0: V_n^{f^\ell} (s;t,x,v)\geq 0 \ for \ all \ s \in [t-\tb^{f^\ell}(t,x,v),  \tau ]
		\big\}
		.
		\Ee
		Since $(X^{f^\ell} (s;t,x,v),V^{f^\ell} (s;t,x,v))$ is $C^1$ (note that $\nabla\phi^\ell \in C^1_{t,x}$) in $s$ we have $V_n^{f^\ell} (\tau_1 ;t,x,v)=0$.

	 We claim that, there exists some constant $\delta_{**} = O_{\e, \bar{t}, \sup_{\ell} \| \nabla \phi^\ell  \|_{C^1}}(\frac{1}{N})$ in (\ref{choice_delta**}) which does not depend on $\ell$ such that
	 \Be\label{claim_Vperp_growth}
	 \begin{split}
	 \text{If} \ \ &0\leq V_n^{f^\ell} (t-\tb^{f^\ell}(t,x,v);t,x,v) < \delta_{**} \text{ and } \  (\ref{upper_|V|}),
	 \\
	&then \ \ V_n^{f^\ell} (s;t,x,v) \leq e^{C|s-(t-\tb^{f^\ell}(t,x,v))|^2}V_n^{f^\ell} (t-\tb^{f^\ell}(t,x,v);t,x,v) \ \ for  \ s \in [ t-\tb^{f^\ell},\tau_1].
	 \end{split}\Ee
For the proof we regard the equations (\ref{dot_Xn_Vn}), (\ref{def_V_parallel}), and (\ref{hamilton_ODE_perp}) as the forward-in-time problem with an initial datum at $s=t-\tb^{f^\ell} (t,x,v)$. 	Clearly we have $X_n^{f^\ell}(t-\tb^{f^\ell}(t,x,v);t,x,v)=0$ and $V_n^{f^\ell}(t-\tb^{f^\ell}(t,x,v);t,x,v)\geq0$ from Lemma \ref{cannot_graze}. Again from Lemma \ref{cannot_graze}, if $V_n^{f^\ell}(t-\tb^{f^\ell}(t,x,v);t,x,v)=0$ then $X_n^{f^\ell}(s;t,x,v)=0$ for all $s\geq t-\tb^{f^\ell} (t,x,v)$. From now on we assume $V_n^{f^\ell} (t-\tb^{f^\ell}(t,x,v);t,x,v)] >0$. From (\ref{hamilton_ODE_perp}), as long as $t - \tb^{f^\ell} (t,x,v) \leq s \leq \bar{t}$ and
	\Be\label{small_s_V_n}
	V_n^{f^\ell}(s;t,x,v) \geq 0 \ \ and \ \
	X_n^{f^\ell} (s;t,x,v) \leq \frac{1}{N} \ll 1,
	\Ee
	then we have
	\Be\label{hamilton_ODE_perp_bound}\begin{split}
	\dot{V}_n^{f^\ell} (s)
		= & \  \underbrace{[V^{f^\ell}_\parallel (s)\cdot \nabla^2 \eta (X^{f^\ell}_\parallel(s)) \cdot V^{f^\ell}_\parallel(s) ] \cdot n(X^{f^\ell}_\parallel(s)) }_{\leq 0  \    from \ ~\eqref{eqn: convex}} \\
	- 	& \underbrace{\nabla \phi^\ell  (s , X^{f^\ell} (s ) ) \cdot [-n(X^{f^\ell}_\parallel(s)) ]}_{
		=O(1) \sup_\ell  \|  \nabla \phi^\ell \|_{C^1} \times X_n^{f^\ell} (s)
		 \   from \  (\ref{expansion_E})} \\
		  - &\underbrace{X_n^{f^\ell} (s) [V^{f^\ell}_\parallel(s) \cdot \nabla^2 n (X^{f^\ell}_\parallel(s)) \cdot V^{f^\ell}_\parallel(s)]  \cdot n(X^{f^\ell}_\parallel(s)) }_{
		=O(1)  \{3 (\e + \bar{t})\sup_{\ell} \| \nabla \phi^\ell  \|_\infty \}^2 \times X_n^{f^\ell} (s) \ from \
		(\ref{upper_|V|})
		}\\
		 \leq & \
		C (1+ \e  + \bar{t})^2 ( \sup_{\ell} \| \nabla \phi^\ell  \|_{C^1}  \sup_{\ell} \| \nabla \phi^\ell  \|_{\infty} )
		  \times X_n^{f^\ell} (s).
	\end{split}\Ee	
	
Let us consider (\ref{hamilton_ODE_perp_bound}) together with $\dot{X}^{f^\ell}_{n}(s;t,x,v) = V^{f^\ell}_{n}(s;t,x,v)$. Then, as long as $s$ satisfies (\ref{small_s_V_n}),
\Be
\begin{split}\notag
V_n^{f^\ell} (s)& =    V_n^{f^\ell} (t-\tb^{f^\ell} ) + \int^s_{t-\tb^{f^\ell}} \dot{V}_n^{f^\ell}(\tau) \dd \tau \\
&\leq V_n^{f^\ell} (t-\tb^{f^\ell} )+\int^s_{t-\tb^{f^\ell}}
C (1+ \e  + \bar{t})^2  ( \sup_{\ell} \| \nabla \phi^\ell  \|_{C^1}  \sup_{\ell} \| \nabla \phi^\ell  \|_{\infty} )
		  \times X_n^{f^\ell} (\tau)
\dd \tau\\
&= V_n^{f^\ell} (t-\tb^{f^\ell} )+\int^s_{t-\tb^{f^\ell}}
C (1+ \e  + \bar{t})^2  ( \sup_{\ell} \| \nabla \phi^\ell  \|_{C^1}  \sup_{\ell} \| \nabla \phi^\ell  \|_{\infty} )
	\int^\tau_{t-\tb^{f^\ell}} V_n^{f^\ell}(\tau^\prime) \dd \tau^\prime
\dd \tau \\
& \leq C (1+ \e  + \bar{t})^2  ( \sup_{\ell} \| \nabla \phi^\ell  \|_{C^1}  \sup_{\ell} \| \nabla \phi^\ell  \|_{\infty} )
\int^s_{t-\tb^{f^\ell}}
|s -  (t-\tb^{f^\ell}) |  V_n^{f^\ell}(\tau^\prime)  \dd \tau^\prime.
\end{split}
\Ee
From the Gronwall's inequality, we derive that, as long as (\ref{small_s_V_n}) holds, 	
\Be\label{upper_bound_Vn}
V_n^{f^\ell} (s;t,x,v) \leq  V_n^{f^\ell} (t-\tb^{f^\ell}(t,x,v)) e^{C (1+ \e  + \bar{t})^2 ( \sup_{\ell} \| \nabla \phi^\ell  \|_{C^1}  \sup_{\ell} \| \nabla \phi^\ell  \|_{\infty} ) \times |s -  (t-\tb^{f^\ell}(t,x,v)) | ^2}.
\Ee	

Now we verify the conditions of (\ref{small_s_V_n}) for all $- \e \leq t - \tb^{f^\ell} (t,x,v) \leq s \leq \bar{t}$. Note that we are only interested in the case of $V_n^{f^\ell} (t-\tb^{f^\ell}(t,x,v);t,x,v)< \delta_{**}$. From the argument of (\ref{hamilton_ODE_perp_bound}), ignoring negative curvature term,
		 \Be
		 \begin{split}\notag
		| X_n^{f^\ell} (s;t,x,v)| \leq& 
	\ 	(\e+ \bar{t}) | V_n^{f^\ell} (t-\tb^{f^\ell};t,x,v)  | \\
	&
		+   C[1 + (\e + \bar{t})^2 \sup_\ell \| \nabla \phi^\ell \|_\infty]  \sup_\ell\| \nabla \phi^\ell \|_{C^1}
		 \int^s_{t-\tb^{f^\ell}} \int^\tau_{t-\tb^{f^\ell}} |X_n^{f^\ell}(\tau;t,x,v)| \dd \tau  \dd s\\
		  \leq & \  (\e+ \bar{t}) | V_n^{f^\ell} (t-\tb^{f^\ell};t,x,v)  | +  C  \int_{t-\tb^{f^\ell}} ^s |\tau- (t-\tb^{f^\ell})| |X_n^{f^\ell}(\tau;t,x,v)|  \dd \tau.
		 \end{split}
		 \Ee
		 Then by the Gronwall's inequality we derive that, in case of (\ref{tb_upperT^**}),
		 \Be\label{upper_X_n}
		 | X_n^{f^\ell} (s;t,x,v)|\leq C_{\e+ \bar{t}} | V_n^{f^\ell} (t-\tb^{f^\ell};t,x,v)  | \ \ for  \ all \  -\e \leq t-\tb^{f^\ell} \leq s \leq t  \leq \bar{t}.
		 \Ee
If we choose
\Be\label{choice_delta**}
\delta_{**} = \frac{o(1)}{ |\bar{t}+ \e |}\times \frac{1}{N},
\Ee
then (\ref{upper_bound_Vn}) holds for $- \e \leq t - \tb^{f^\ell} (t,x,v) \leq s \leq \bar{t}$. Hence we complete the proof of (\ref{claim_Vperp_growth}).

	\hide

	\Be\notag
	\dot{X}_{n}(s;t,x,v) = V_{n}(s;t,x,v) , \ \  \dot{X}_{\parallel}(s;t,x,v) = V_{\parallel}(s;t,x,v).
	\Ee

	Then the proof of (\ref{no_graze}) asserts that
		\Be\label{alpha_lower_bound}
		|V_n^f(s;t,x,v)| \gtrsim 1 \ \ \text{for all } \ (t,x,v) \in [0,\frac{\e}{2}] \times  \text{supp} (\psi) .
		\Ee
		Now we consider $\alpha_{f^\ell, \e}(t,x,v)$.

		 On the other hand if $\tb^{f^\ell}(t,x,v)< 2\e$, since (\ref{uniform_h_ell}), from (\ref{hamilton_ODE_perp}),
		\Be
		\begin{split}\notag
			&\sup_{\ell \geq 0}|V^f_n(s;t,x,v) - V^{f^\ell }_n(s;t,x,v)|\\
			\lesssim & \ \sup_{\ell \geq 0}\max\{ \tb^{f}, \tb^{f^\ell}\} \times \left\{N^2  + \| \nabla \phi_f \|_\infty + \| \nabla \phi_{f^\ell} \|_\infty\right\}\\
			\lesssim & \ \e (1+ N^2).
		\end{split}\Ee
	Therefore, from (\ref{alpha_lower_bound})	 for small $\e>0$, we prove the lower bound
		\Be
	\inf_{\ell }	|n(\xb^{f^\ell})  \cdot \vb^{f^\ell}| \gtrsim 1 \ \ \text{for all } \ (t,x,v) \in [0,\frac{\e}{2}] \times  \text{supp} (\psi) .
		\Ee
		
		\unhide

		\vspace{4pt}

		\textit{Step 3. } Suppose that (\ref{upper_|V|}) holds and $0 \leq V_n^{f^\ell} (t-\tb^{f^\ell}(t,x,v);t,x,v) < \delta_{**}$ with $\delta_{**}$ of  (\ref{choice_delta**}). Recall the definition of $\tau_1$ in (\ref{tau_1}). Inductively we define $\tau_2
		: = \sup  \big\{ \tau \geq  0: V_n^{f^\ell} (s;t,x,v)\leq 0 \ for \ all \ s \in [\tau_1,  \tau ]
		\big\}$ and $\tau_3, \tau_4, \cdots$. Clearly such points can be countably many at most in an interval of $[t-\tb^{f^\ell},t]$. Suppose $ \lim_{k\rightarrow \infty}\tau_k=t   $. Then choose $k_0\gg1$ such that $|\tau_{k_0} -t| \ll  |V_n^{f^\ell} (t-\tb^{f^\ell} ;t,x,v)|$. Then, for $s \in [\tau_{k_0}, t] $, from (\ref{hamilton_ODE_perp_bound}) and (\ref{upper_|V|}),
		\Be\begin{split}\label{upper_V_n_0}
		|V_n^{f^\ell} (t;t,x,v)| \lesssim   {|V_n^{f^\ell} (t-\tb^{f^\ell} ;t,x,v)|} .
		\end{split}\Ee
		
		Now we assume that $\tau_{k_0} < t \leq \tau_{k_0+1}$. From the definition of $\tau_i$ in  (\ref{tau_1}) we split the case in two.
		
		\textit{\underline{Case 1:} Suppose $V_n^{f^\ell} (s;t,x,v)>0$ for $s \in (\tau_{k_0}, t)$. }

		From (\ref{hamilton_ODE_perp_bound}) and (\ref{upper_X_n})
		\Be
		\begin{split}\label{upper_V_n_1}
		 V_n^{f^\ell} (t;t,x,v)  \lesssim \int^{\bar{t}}_{\tau_{k_0}} X_n^{f^\ell}(s) \lesssim |V_n^{f^\ell} (t-\tb^{f^\ell};t,x,v)|.
		\end{split}
		\Ee

		\textit{\underline{Case 2:} Suppose $V_n^{f^\ell} (s;t,x,v)<0$ for $s \in (\tau_{k_0}, t)$. }
		
		Suppose
		\Be\label{assumption_lowerbound_V_n}
		-V_n^{f^\ell} (t;t,x,v)= |V_n^{f^\ell} (t;t,x,v)| \geq \frac{1}{\e}|V_n^{f^\ell} (t-\tb^{f^\ell};t,x,v)|^{1/2}.
		\Ee \hide	
In this step we claim that
\Be\label{max_Xn}
X_n^{f^\ell} (
t-\tb^{f^\ell}(t,x,v)+
\tau_*^\ell (t,x,v);t,x,v)\geq \frac{1}{2} \frac{| V_n^{f^\ell} (t- \tb^{f^\ell} (t,x,v);t,x,v)|^2}{1+\{3 (\e + T^{**})\sup_{\ell} \| \nabla \phi^\ell  \|_\infty\}^2 }.
\Ee		
\unhide
From (\ref{hamilton_ODE_perp_bound}), now taking account of the curvature term this time, we derive that 	\Be
	\begin{split}\notag
	-V_n^{f^\ell} (t;t,x,v)  \leq & \   \int_{\tau_{k_0}}^{t} (-1) [V^{f^\ell}_\parallel (s)\cdot \nabla^2 \eta (X^{f^\ell}_\parallel(s)) \cdot V^{f^\ell}_\parallel(s) ] \cdot n(X^{f^\ell}_\parallel(s))  \dd s\\
	& +C | V_n^{f^\ell} (t- \tb^{f^\ell} (t,x,v);t,x,v)|,
	\end{split}
	\Ee		
	where we have used (\ref{upper_|V|}) and (\ref{upper_X_n}). From (\ref{assumption_lowerbound_V_n}) the above inequality implies that, for $| V_n^{f^\ell} (t- \tb^{f^\ell} (t,x,v);t,x,v)|\ll 1$,
	\Be\notag
	\frac{1}{2\e}|V_n^{f^\ell} (t-\tb^{f^\ell};t,x,v)|^{1/2} \leq \int_{\tau_{k_0}}^{t} (-1) [V^{f^\ell}_\parallel (s)\cdot \nabla^2 \eta (X^{f^\ell}_\parallel(s)) \cdot V^{f^\ell}_\parallel(s) ] \cdot n(X^{f^\ell}_\parallel(s))  \dd s.
	\Ee
Note that $|\frac{d}{ds} V^{f^\ell}_\parallel (s)|$ and $|\frac{d}{ds}X^{f^\ell}_\parallel (s)|$ are all bound from $\nabla \phi^\ell \in C^1$, (\ref{upper_|V|}), and (\ref{upper_X_n}). By~\eqref{upper_|V|} and~\eqref{eqn: convex} we can take $\e$ to be sufficiently small such that
\[\int_{t- |V_n^{f^\ell} (t-\tb^{f^\ell};t,x,v)|^{1/2}}^{t} (-1) [V^{f^\ell}_\parallel (s)\cdot \nabla^2 \eta (X^{f^\ell}_\parallel(s)) \cdot V^{f^\ell}_\parallel(s) ] \cdot n(X^{f^\ell}_\parallel(s))  \dd s\]
\[\leq |V_n^{f^\ell} (t-\tb^{f^\ell};t,x,v)|^{1/2}C_\eta C_{\e,\bar{t},\sup_{\ell}\Vert \nabla \phi^\ell\Vert_\infty}\leq \frac{1}{4\e}|V_n^{f^\ell} (t-\tb^{f^\ell};t,x,v)|^{1/2}.\]
Hence we obtain
\Be\label{lower_bound_V_||^2}
	\frac{1}{4\e}|V_n^{f^\ell} (t-\tb^{f^\ell};t,x,v)|^{1/2} \leq \int_{\tau_{k_0}}^{t- |V_n^{f^\ell} (t-\tb^{f^\ell};t,x,v)|^{1/2}} (-1) [V^{f^\ell}_\parallel (s)\cdot \nabla^2 \eta (X^{f^\ell}_\parallel(s)) \cdot V^{f^\ell}_\parallel(s) ] \cdot n(X^{f^\ell}_\parallel(s))  \dd s.
	\Ee
	On the other hand, if $t- |V_n^{f^\ell} (t-\tb^{f^\ell};t,x,v)|^{1/2}\leq \tau_{k_0}$ then $|t-\tau_{k_0}| \leq |V_n^{f^\ell} (t-\tb^{f^\ell};t,x,v)|^{1/2}$, which implies that, from (\ref{hamilton_ODE_perp_bound}), (\ref{upper_|V|}), and (\ref{upper_X_n}),
	\Be\label{upper_V_n_2}
	|V_n^{f^\ell}(t;t,x,v)| \lesssim  |V_n^{f^\ell} (t-\tb^{f^\ell};t,x,v)|^{1/2}.
	\Ee
	
	Now we consider $X_n^{f^\ell} (t;t,x,v)$. From (\ref{hamilton_ODE_perp_bound}) and $\dot{X}^{f^\ell}_{n}(s;t,x,v) = V^{f^\ell}_{n}(s;t,x,v)$ together with (\ref{upper_X_n}) and (\ref{upper_|V|})
	\Be\label{X_n<0}
	\begin{split}
	&X_n^{f^\ell} (t;t,x,v)\\
	 \leq & \   (\bar{t}+\e)| V_n^{f^\ell} (t-\tb^{f^\ell};t,x,v)  |
	 + \int_{\tau_{k_0}}^t \int_{\tau_{k_0}}^\tau\underbrace{[V^{f^\ell}_\parallel (s)\cdot \nabla^2 \eta (X^{f^\ell}_\parallel(s)) \cdot V^{f^\ell}_\parallel(s) ] \cdot n(X^{f^\ell}_\parallel(s)) }_{\leq 0}  \dd s  \dd \tau \\
	\leq  & \   (\bar{t}+\e)| V_n^{f^\ell} (t-\tb^{f^\ell};t,x,v)  | \\
	&
	 +  |V_n^{f^\ell} (t-\tb^{f^\ell};t,x,v)|^{1/2} \int_{\tau_{k_0}}^{t- |V_n^{f^\ell} (t-\tb^{f^\ell};t,x,v)|^{1/2}}[V^{f^\ell}_\parallel (s)\cdot \nabla^2 \eta (X^{f^\ell}_\parallel(s)) \cdot V^{f^\ell}_\parallel(s) ] \cdot n(X^{f^\ell}_\parallel(s))   \dd s   \\
	 \leq & \  (\bar{t}+\e)| V_n^{f^\ell} (t-\tb^{f^\ell};t,x,v)  | -  \frac{1}{4\e}|V_n^{f^\ell} (t-\tb^{f^\ell};t,x,v)|^{1}
	 \ \ \ from \ (\ref{lower_bound_V_||^2})
	 \\
	  < & \ 0.
	\end{split}
	\Ee

 Clearly this cannot happen since $x \in \bar{\O}$ and $x_n\geq 0$. Therefore our assumption (\ref{assumption_lowerbound_V_n}) was wrong and we conclude (\ref{upper_V_n_2}).

		\vspace{4pt}

		\textit{Step 4 } From (\ref{claim_Vperp_growth}), (\ref{upper_V_n_0}), (\ref{upper_V_n_1}), and (\ref{upper_V_n_2}) in \textit{Step 1} and \textit{Step 2}, we conclude that the same estimate (\ref{upper_V_n_2}) for $|V_n^{f^\ell} (t-\tb^{f^\ell};t,x,v)|\ll 1$ in the case of (\ref{tb_upperT^**}) and (\ref{upper_|V|}). Finally from (\ref{alpha=1}), (\ref{alpha_|v|>}), (\ref{claim_Vperp_growth}), and (\ref{upper_V_n_2})
		Therefore we conclude that
		\Be\label{lower_bound_V_n_final}
		|V_n^{f^\ell} (t-\tb^{f^\ell} (t,x,v);t,x,v )|\gtrsim \left(\frac{1}{N^2}\right) \ \ (t,x,v) \in [0,\bar{t}] \times (\mathcal{S}_N)^c.
		\Ee

	\hide

 From (\ref{hamilton_ODE_perp}) and (\ref{upper_|V|})
		\Be\begin{split}\label{lower_t_*1}
		&\{  3 (\e + T^{**})\sup_{\ell} \| \nabla \phi^\ell  \|_\infty\}^2 \times  \tau_*^\ell (t,x,v)
		\\
		\geq & \ \int^{t-\tb^{f^\ell} +\tau_*^\ell }_{t-\tb^{f^\ell}} (-1) [V^{f^\ell}_\parallel (s)\cdot \nabla^2 \eta (X^{f^\ell}_\parallel(s)) \cdot V^{f^\ell}_\parallel(s) ] \cdot n(X^{f^\ell}_\parallel(s))  \dd s
		\\
\geq & \ 	V_n^{f^\ell} (\tau_*^\ell (t,x,v);t,x,v)	
+  \int^{t-\tb^{f^\ell} +\tau_*^\ell}_{t-\tb^{f^\ell}}  \nabla \phi^\ell  (s , X^{f^\ell} (s ) ) \cdot [-n(X^{f^\ell}_\parallel(s)) ] \dd s \\
& +   \int^{t-\tb^{f^\ell} +\tau_*^\ell }_{t-\tb^{f^\ell}}  X_n^{f^\ell} (s) [V^{f^\ell}_\parallel(s) \cdot \nabla^2 n (X^{f^\ell}_\parallel(s)) \cdot V^{f^\ell}_\parallel(s)]  \cdot n(X^{f^\ell}_\parallel(s)).
		\end{split}
		\Ee
	Note that, from (\ref{upper_bound_Vn}), the last two terms of the above estimates are bounded above by
	\Be\begin{split}\label{lower_t_*2}
&
\{ \| \nabla \phi^{\ell} \|_{C^1}
+9 (\e + T^{**})^2\sup_{\ell} \| \nabla \phi^\ell  \|_\infty^2 \}
\int^{t-\tb^{f^\ell} +\tau_*^\ell }_{t-\tb^{f^\ell}}  X_n^{f^\ell} (s)\\
\leq & \ \{ \| \nabla \phi^{\ell} \|_{C^1}
+9 (\e + T^{**})^2\sup_{\ell} \| \nabla \phi^\ell  \|_\infty^2 \}
 \int^{t-\tb^{f^\ell} +\tau_*^\ell }_{t-\tb^{f^\ell}}  \int^\tau_{t-\tb^{f^\ell}}
V_n^{f^\ell} (\tau;t,x,v)
 \dd \tau   \dd s \\
 \leq & \  C_* \times  |\tau^\ell_*(t,x,v)|^2 \times V_n^{f^\ell} (t-\tb^{f^\ell}(t,x,v)) ,
	\end{split}\Ee	
		where
		\Be\notag
		C_* := \{ \| \nabla \phi^{\ell} \|_{C^1}
+9 (\e + T^{**})^2\sup_{\ell} \| \nabla \phi^\ell  \|_\infty^2 \}
		e^{C (1+ \e  + T^{**})^2 (1+ \sup_{\ell} \| \nabla \phi^\ell  \|_{C^1} )^2 \times | \e + T^{**}| ^2}.
		\Ee
		By considering $\tau_*^\ell (t,x,v)\leq \e + T^{**} \ll1 $ and $\tau_*^\ell (t,x,v)\geq \e + T^{**}$ separately, from (\ref{lower_t_*1}) and (\ref{lower_t_*2}), we derive
		\Be\begin{split} \label{lower_t_*3}
		  \tau_*^\ell (t,x,v)
		 \geq  & \  \tau_{**}^\ell (t,x,v)\\& \ :=  \min \left\{ \e + T^{**} ,  \frac{(1- C_* ( \e + T^{**})^2)}{1+\{  3 (\e + T^{**})\sup_{\ell} \| \nabla \phi^\ell  \|_\infty\}^2} \times V_n^{f^\ell} (t- \tb^{f^\ell} (t,x,v);t,x,v) \right\}	.
	\end{split}	\Ee

		Now we study the lower bound of $X_n^{f^\ell} ( t- \tb^{f^\ell}+ \tau_{**}^\ell (t,x,v);t,x,v)$. From (\ref{upper_bound_Vn}) and (\ref{hamilton_ODE_perp_bound}),
		\Be\label{lower_bound_Xn}
		\begin{split}
		&X_n^{f^\ell} ( t- \tb^{f^\ell}(t,x,v) + \tau_{*}^\ell (t,x,v);t,x,v)\\
		\geq & \ X_n^{f^\ell} ( t- \tb^{f^\ell}(t,x,v) + \tau_{**}^\ell (t,x,v);t,x,v)\\
		\geq & \    \int^{ t- \tb^{f^\ell}(t,x,v) + \tau_{**}^\ell  (t,x,v) }_{t- \tb^{f^\ell}(t,x,v) } V_n ^{f^\ell} (s;t,x,v) \dd s\\
		\geq & \  V_n^{f^\ell} (t- \tb^{f^\ell} (t,x,v);t,x,v) \times   \tau_{**}^\ell (t,x,v) \\
		& \ -
		C (1+ \e  + T^{**})^2  ( \sup_{\ell} \| \nabla \phi^\ell  \|_{C^1}  \sup_{\ell} \| \nabla \phi^\ell  \|_{\infty} )\times  |\tau_{**}^\ell (t,x,v)|^2
		\\
		\geq & \
		(1- C_* ( \e + T^{**})^2) \big\{
		1-
		C (1- C_* ( \e + T^{**})^2) (1+ \e  + T^{**})^2  ( \sup_{\ell} \| \nabla \phi^\ell  \|_{C^1}  \sup_{\ell} \| \nabla \phi^\ell  \|_{\infty} )
		\big\}\\
		&
		\times
		\frac{| V_n^{f^\ell} (t- \tb^{f^\ell} (t,x,v);t,x,v)|^2}{1+\{3 (\e + T^{**})\sup_{\ell} \| \nabla \phi^\ell  \|_\infty\}^2 }\\
		\geq & \ \frac{1}{2} \frac{| V_n^{f^\ell} (t- \tb^{f^\ell} (t,x,v);t,x,v)|^2}{1+\{3 (\e + T^{**})\sup_{\ell} \| \nabla \phi^\ell  \|_\infty\}^2 },
		\end{split}
		\Ee
		as long as $\sup_{\ell} \| \nabla \phi^\ell  \|_{C^1}  \sup_{\ell} \| \nabla \phi^\ell  \|_{\infty}  \ll 1$ and $\e + T^{**} \ll 1.$

		\vspace{4pt}

		\textit{Step 8-c. } Suppose that (\ref{upper_|V|}) holds and $0 \leq V_n^{f^\ell} (t-\tb^{f^\ell}(t,x,v);t,x,v) < \delta_{**}$ with $\delta_{**}$ of  (\ref{choice_delta**}). Let us define
		\Be\label{tau^ell}
		\tau_*^\ell (t,x,v)
		: = \sup  \big\{ \tau \geq  0: V_n^{f^\ell} (s;t,x,v)\geq 0 \ for \ all \ s \in [t-\tb^{f^\ell}(t,x,v), 
		\tau ]
		\big\}
		.
		\Ee
		With two parameters $0< B < A \leq 1$,
		\Be\label{AB_choice}
		A, B
		\Ee
		 we split the case in three.
		
			\vspace{2pt}
			
			\textit{\underline{Case 1.} } Suppose $|V_\parallel^{f^\ell} (\tau_*^\ell (t,x,v);t,x,v)|\geq |V_n^{f^\ell} (t-\tb^{f^\ell}(t,x,v);t,x,v)|^A$. Set
			\Be\label{triangle_1}
			 \vartriangle =  (\e + T^{**} )^{1/2} \times |V_n^{f^\ell} (t-\tb^{f^\ell}(t,x,v);t,x,v)|^{\frac{1- 2A}{2}}.
			\Ee

			Note that
			\Be\notag
			|V_\parallel^{f^\ell} (\tau_*^\ell + \vartriangle ;t,x,v)|\geq  |V_n^{f^\ell} (t-\tb^{f^\ell}(t,x,v);t,x,v)|^A - \| \nabla \phi^\ell \|_\infty \times  \vartriangle,
			\Ee
			and, from (\ref{convexity_eta}),
			\Be
			\begin{split}
			&X^{f^\ell}_n(\tau_*^\ell + \vartriangle ;t,x,v)\\
			= & \ X^{f^\ell}_n(\tau_*^\ell   ;t,x,v) + \int^{\vartriangle}_0  \int_0^s \dot{V}_n^{f^\ell}(\tau_*^\ell + \tau;t,x,v) \dd \tau  \dd s \\
			\leq & \ (\e + T^{**})|V_n^{f^\ell} (  t-\tb^{f^\ell} ;t,x,v)| - C_\O \int^{\vartriangle}_0  \int_0^s  |V_n^{f^\ell}(t-\tb^{f^\ell};t,x,v)|^{2A} \dd \tau  \dd s\\
			\leq & \  0.
			\end{split}
			\Ee
			
			\vspace{2pt}
			
			\textit{\underline{Case 2.} } Suppose
			\Be
			\begin{split}\label{case_2}
		 	|V_\parallel^{f^\ell} (\tau_*^\ell (t,x,v);t,x,v)|&\leq |V_n^{f^\ell} (t-\tb^{f^\ell}(t,x,v);t,x,v)|^A\\
		 	 and    \  |\nabla \phi^\ell(\tau_*^\ell (t,x,v);t,x,v)|&\leq |V_n^{f^\ell} (t-\tb^{f^\ell}(t,x,v);t,x,v)|^B.
			\end{split}\Ee

			Note that, from (\ref{XV_ell}), for $\tau \geq 0$
			\Be\notag
			|V^{f^\ell}(\tau^\ell_*-\tau;t,x,v)| \leq |V^{f^\ell}(\tau^\ell_* ;t,x,v)|+ \int_0^{\tau}  \Big|
			\frac{V^{f^\ell}(\tau^\ell_*-\tau^\prime )}{|V^{f^\ell}(\tau^\ell_*-\tau^\prime )|} \cdot
			\nabla \phi (\tau^\ell_*-\tau^\prime, X^{f^\ell} (\tau^\ell_*-\tau^\prime ))
			\Big|\dd \tau^{\prime}.
			\Ee
			This, together with (\ref{case_2}), deduces that
			\Be
			\begin{split}
			&\Big|
			\frac{V^{f^\ell}(\tau^\ell_*-s )}{|V^{f^\ell}(\tau^\ell_*-s )|} \cdot
			\nabla \phi (\tau^\ell_*-s, X^{f^\ell} (\tau^\ell_*- s ))
			\Big|\\
			 \leq & \ |
			\nabla \phi (\tau^\ell_*-s, X^{f^\ell} (\tau^\ell_*- s ))
			 |\\
			 \leq&  \  |V_n^{f^\ell} (t-\tb^{f^\ell} ;t,x,v)|^B\\
			 & + \int
			\end{split}
			\Ee

			\Be
			\begin{split}
			&|\nabla \phi^\ell (X^{f^\ell}(\tau^\ell_*-s;t,x,v))| \\
			\leq&  \  |V_n^{f^\ell} (t-\tb^{f^\ell}(t,x,v);t,x,v)|^B
			\end{split}
			\Ee
			
			\vspace{2pt}
			
			\textit{\underline{Case 3.} }

		\vspace{4pt}

		\textit{Step 8-d. }

		Now we are ready to prove the claim (\ref{alpha_lower_bound_ell}). Choose $(x,v)$ in $(\mathcal{S}_N)^c$ in (\ref{test_function_condition}) and $- \e    \leq t \leq  T^{**}.$ If $|v|\leq 2 (\e + T^{**})\sup_{\ell} \| \nabla \phi^\ell \|_\infty$ and $0\leq V_n^{f^\ell} (t-\tb^{f^\ell}(t,x,v);t,x,v)< \delta_{**}$ then from (\ref{claim_Vperp_growth}) we have
		\Be\label{lower_Vn1}
  V_n^{f^\ell} (t-\tb^{f^\ell}(t,x,v);t,x,v)
\geq  e^{-C|T^{**} + \e |^2}	V_n^{f^\ell} (s;t,x,v)
\geq \frac{1}{e^{ C|T^{**} + \e |^2}} \times \frac{1}{N}.	
		\Ee

		Finally from (\ref{alpha=1}), (\ref{alpha_|v|>}), and (\ref{lower_Vn1}), we conclude that \Be\label{alpha_lower_bound_ell}
\inf_{\substack{ \ell \in \mathbb{N}, \ 0 \leq t \leq T^{**}
\\
 (x,v) \in (\mathcal{S}_N)^c  }}	\alpha_{f^\ell, \e} (t,x,v) \geq \frac{1}{N} \times \min \left\{    \frac{e^{- \frac{C_\O }{\sup_{\ell} \| \nabla \phi^\ell  \|_\infty}}}{C_\O }     ,\frac{e^{-C (1+ \e  + T^{**})^4 (1+ \sup_{\ell} \| \nabla \phi^\ell  \|_{C^1} )^2}}{ |T^{**} + \e |},  e^{ -C|T^{**} + \e |^2}\right\}.
	\Ee

			\unhide

		 From (2.36), (2.37), (2.40), and (2.41) in Lemma 2.4 in \cite{KL},
		\Be\notag
	\sup_{ \substack{ \ell \in \mathbb{N}
, \
 (x,v) \in (\mathcal{S}_N)^c , \\  - \e \leq t - \tb^{f^\ell} (t,x,v)   \leq t \leq  \bar{t}}}		|\nabla_{x,v} \alpha^\beta_{f^\ell,\e}(t,x,v)| \lesssim \frac{1}{|V_n^{f^\ell} (t-\tb^{f^\ell};t,x,v)|^{2- \beta}}  \lesssim_{\e,N,\bar{t}  } 1.
		\Ee
	Hence we extract another subsequence out of all previous steps (and redefine this as $\{\ell_N\}$) such that
	\Be\label{converge_D_alpha_ell}
	\nabla_{x,v} \alpha_{f^{\ell_N},\e}^\beta  \overset{\ast}{\rightharpoonup}  	\nabla_{x,v} \alpha_{f ,\e}^\beta  \textit{ weakly}-* \textit{ in } L^\infty
			( (-\e, \bar{t}) \times (\mathcal{S}_N)^c
			).
	\Ee	
Note that the limiting function is identified from (\ref{converge_alpha_ell}). Finally the trong convergence of~\eqref{eqn: Strong converge} and the weak$-*$ convergence of (\ref{converge_D_alpha_ell}) justifies the convergence of (\ref{af_l_3}):
\begin{equation}\label{convergence of third}
\lim_{\ell \to \infty}  (\ref{af_l_3})=-  {\int^t_0\iint_{\O \times \R^3} \nabla_{x,v}\alpha^\beta_{f,\e}  f (w_{\tilde{\vartheta}}\psi) \dd x\dd v}.
\end{equation}

		Now we extract the final subsequence $\{\ell_*\}$ from the previous subsequence: By the Cantor's diagonal argument we define
		\Be\label{ell_*}
		\ell_*= \ell_{ \ell} .
		\Ee
Combining~\eqref{convergence and fs} and~\eqref{convergence of third} we have (\ref{lim_Fell=F}) with this subsequence for any test function $\psi$. For any $\psi \in C^\infty_c (\bar{\O} \times \R^3 \backslash \gamma_0)$ there exists $N_{\psi} \in \mathbb{N}$ such that $supp (\psi) \subset (\mathcal{S}_{N_\psi})^c$. Hence \eqref{a_nabla_f_seq} follows from~\eqref{lim_Fell=F}.

Finally we obtain~\eqref{eqn: claim w1p}. Assumptions in Proposition \ref{Prop: L3L1 for nabla f} thus hold. Applying Proposition \ref{Prop: L3L1 for nabla f} \ref{L1+stability}, assuming $f_1$ and $f_2$ are both solutions, then
\[		\|e^{-\lambda t\langle v\rangle} \big[f_1(t) - f_2(t)\big] \|_{L^{1+\delta}(\O \times \R^3)}\lesssim \| f_1(0) - f_2(0) \|_{L^{1+\delta}(\O \times \R^3)},\]
so the solution is unique.

\end{proof}

\section{Appendix}
\begin{lemma}\label{Lemma: Prob measure}
For $R(u\to v;x,t)$ given by~\eqref{eqn: Formula for R}, given any $u$ such that $u\cdot n(x)>0$,
\begin{equation}\label{eqn: integrate 1}
  \int_{n(x)\cdot v<0}R(u\to v;x,t)dv=1.
\end{equation}

\end{lemma}

\begin{proof}

We can transform the basis from $\{n,\tau_1,\tau_2\}$ to the standard bases $\{e_1,e_2,e_3\}$. For the sake of simplicity, we assume $T_w(x)=1$. The integration over $\mathcal{V}_\parallel$, after the orthonormal transformation, becomes integration over $\mathbb{R}^2$. We have
\[\int_{\mathbb{R}^2} \frac{1}{r_\parallel(2-r_\parallel)}  \exp\Big(\frac{|v_\parallel-(1-r_\parallel)u_\parallel|^2}{r_\parallel(2-r_\parallel)} \Big)dv_\parallel,\]
which is obviously normalized.

Then we consider the integration over $\mathcal{V}_\perp$, which is $e_3<0$ after the transformation. We want to show

\begin{equation}\label{eqn: I0}
\frac{2}{r_\perp}\int_{-\infty}^0 -v_\perp e^{-\frac{|v_\perp|^2}{r_\perp}}e^{\frac{-(1-r_\perp)|u_\perp|^2}{r_\perp}}I_0(\frac{2(1-r_\perp)^{1/2}v_\perp u_\perp}{r_\perp})dv_\perp=1.
\end{equation}
The Bessel function reads
\[J_0(y)=\frac{1}{\pi}\int_0^{\pi} e^{iy\cos\theta}d\theta=\sum_{k=0}^\infty \frac{1}{\pi}\int_0^\pi \frac{(iy\cos\theta)^k}{k!} d\theta=\sum_{k=0}^\infty \int_0^\pi  \frac{(iy\cos\theta)^{2k}}{(2k)!}  d\theta\]
\[\sum_{k=0}^\infty   \int_0^\pi  \frac{(-1)^k (y)^{2k} (\cos\theta)^{2k}}{(2k)!}d\theta=\sum_{k=0}^\infty (-1)^k \frac{(\frac{1}{4}y^2)^k}{(k!)^2},\]
where we use the Fubini's theorem and the fact that
\[\int_0^\pi \cos^{2k}\theta=\frac{\pi}{2^{2k}}\left(
                                                 \begin{array}{c}
                                                   2k \\
                                                   k \\
                                                 \end{array}
                                               \right).
\]
Hence
\begin{equation}\label{eqn: I0 sequence}
I_0(y)=\frac{1}{\pi}\int_0^\pi e^{i(-iy)\cos \theta}d\theta =J_0(-iy)=\sum_{k=0}^\infty \frac{(\frac{1}{4}y^2)^k}{(k!)^2},\quad I_0(y)=I_0(-y).
\end{equation}
By taking the change of variable $v_\perp\to -v_\perp$, the LHS of \eqref{eqn: I0} can be written as
\[\frac{2}{r_\perp}\int_{0}^\infty v_\perp e^{-\frac{|v_\perp|^2}{r_\perp}}e^{\frac{-(1-r_\perp)|u_\perp|^2}{r_\perp}}I_0(\frac{2(1-r_\perp)^{1/2}v_\perp u_\perp}{r_\perp})dv_\perp.\]
Using~\eqref{eqn: I0 sequence} we rewrite the above term as
\begin{equation}\label{eqn: I_0 tough}
\sum_{k=0}^\infty \frac{2}{r_\perp}\int_0^\infty v_\perp e^{\frac{-|v_\perp|^2}{r_\perp}} e^{\frac{-(1-r_\perp)|u_\perp|^2}{r_\perp}}\frac{(1-r_\perp)^k v_\perp^{2k}u_\perp^{2k}}{(k!)^2r_\perp^{2k}}dv,
\end{equation}
where we use the Tonelli theorem. Rescale $v_\perp=\sqrt{r_\perp}v_\perp$ we have
\[\frac{2}{r_\perp}\int_0^\infty v_\perp e^{\frac{-|v_\perp|^2}{r_\perp}} e^{\frac{-(1-r_\perp)|u_\perp|^2}{r_\perp}}\frac{(1-r_\perp)^k v_\perp^{2k}u_\perp^{2k}}{(k!)^2r_\perp^{2k}}dv\]
\[=2\int_0^\infty v_\perp e^{-|v_\perp|^2} e^{\frac{-(1-r_\perp)|u_\perp|^2}{r_\perp}}\frac{(1-r_\perp)^k v_\perp^{2k}u_\perp^{2k}}{(k!)^2r_\perp^{k}}dv\]
\begin{equation}\label{eqn: appendix for I0}
=2\int_0^\infty v_\perp^{2k+1}e^{-|v_\perp|^2}dv    e^{\frac{-(1-r_\perp)|u_\perp|^2}{r_\perp}}\frac{(1-r_\perp)^k u_\perp^{2k}}{(k!)^2r_\perp^{k}}
\end{equation}
\[=2\frac{k!}{2}e^{\frac{-(1-r_\perp)|u_\perp|^2}{r_\perp}}\frac{(1-r_\perp)^k u_\perp^{2k}}{(k!)^2r_\perp^{k}}=e^{\frac{-(1-r_\perp)|u_\perp|^2}{r_\perp}}\frac{(1-r_\perp)^k u_\perp^{2k}}{k!r_\perp^{k}}.\]
Therefore, the LHS of~\eqref{eqn: I0} can be written as
\[e^{\frac{-(1-r_\perp)|u_\perp|^2}{r_\perp}}\sum_{k=0}^\infty \frac{(1-r_\perp)^k u_\perp^{2k}}{k!r_\perp^{k}} =e^{\frac{-(1-r_\perp)|u_\perp|^2}{r_\perp}}e^{\frac{(1-r_\perp)|u_\perp|^2}{r_\perp}}=1.\]

\end{proof}

\begin{lemma}\label{Lemma: abc}
For any $a>0,b>0,\e>0$ with $a+\e<b$,
\begin{equation}\label{eqn: coe abc}
\frac{b}{\pi}\int_{\mathbb{R}^2} e^{\e|v|^2}  e^{a|v|^2}e^{-b|v-w|^2}dv=\frac{b}{b-a-\e}e^{\frac{(a+\e)b}{b-a-\e}|w|^2}.
\end{equation}
And when $\delta\ll 1$,
\begin{eqnarray}
     \frac{b}{\pi}\int_{|v-\frac{b}{b-a-\e}w|>\delta^{-1}} e^{\e|v|^2}  e^{a|v|^2}e^{-b|v-w|^2}dv  &\leq  &e^{-(b-a-\e)\delta^{-2}} \frac{b}{b-a-\e} e^{\frac{(a+\e)b}{b-a-\e}|w|^2} \label{eqn: coe abc smaller} \\
   & \leq & \delta \frac{b}{b-a-\e}e^{\frac{(a+\e)b}{b-a-\e}|w|^2} \label{eqn: coe abc small}.
\end{eqnarray}

\end{lemma}

\begin{proof}
\begin{align*}
   & \frac{b}{\pi}\int_{\mathbb{R}^2} e^{\e|v|^2}  e^{a|v|^2}e^{-b|v-w|^2}dv = \frac{b}{\pi}\int_{\mathbb{R}^2} e^{(a+\e-b)|v|^2} e^{2bv\cdot w} e^{-b|w|^2}dv    \\
   & =\frac{b}{\pi}\int_{\mathbb{R}^2}   e^{(a+\e-b)|v+\frac{b}{a+\e-b}w|^2} e^{\frac{-b^2}{a+\e-b}|w|^2} e^{-b|w|^2}dv\\
   &=\frac{b}{\pi}\int_{\mathbb{R}^2} e^{(a+\e-b)|v|^2}dv e^{\frac{(a+\e)b}{b-a-\e}|w|^2}=\frac{b}{b-a-\e}e^{\frac{(a+\e)b}{b-a-\e}|w|^2},
\end{align*}
where we apply change of variable $v+\frac{b}{a+\e-b}w\to v$ in the first step of the last line, then we obtain~\eqref{eqn: coe abc}.

Following the same derivation
\begin{align*}
   & \frac{b}{\pi}\int_{|v-\frac{b}{b-a-\e}w|>\delta^{-1}} e^{\e|v|^2}  e^{a|v|^2}e^{-b|v-w|^2}dv \\
   & =\frac{b}{\pi}\int_{|v-\frac{b}{b-a-\e}w|>\delta^{-1}} e^{(a+\e-b)|v-\frac{b}{b-a-\e}w|^2}dv e^{\frac{(a+\e)b}{b-a-\e}|w|^2}\\
   &\leq  e^{-(b-a-\e)\delta^{-2}} \frac{b}{b-a-\e} e^{\frac{(a+\e)b}{b-a-\e}|w|^2} \leq \delta     \frac{b}{b-a-\e} e^{\frac{(a+\e)b}{b-a-\e}|w|^2},
\end{align*}
thus we obtain~\eqref{eqn: coe abc small}.

\end{proof}

\begin{lemma}\label{Lemma: perp abc}
For any $a>0,b>0,\e>0$ with $a+\e<b$,
\begin{equation}\label{eqn: coe abc perp}
2b\int_{\mathbb{R}^+}v e^{\e v^2}e^{av^2} e^{-bv^2}e^{-bw^2}I_0(2bv w)dv=\frac{b}{b-a-\e}e^{\frac{(a+\e)b}{b-a-\e}w^2}.
\end{equation}
And when $\delta\ll 1$,
\begin{equation}\label{eqn: coe abc perp small}
2b\int_{0< v<\delta}v e^{\e v^2}e^{av^2} e^{-bv^2}e^{-bw^2}I_0(2bv w)dv\leq \delta\frac{b}{b-a-\e}e^{\frac{(a+\e)b}{b-a-\e}w^2}.
\end{equation}

\end{lemma}

\begin{proof}
\begin{align*}
   & 2b\int_{\mathbb{R}^+}v e^{\e v^2}e^{av^2} e^{-bv^2}e^{-bw^2}I_0(2bv w)dv \\
   & =2b\int_{\mathbb{R}^+} v e^{(a+\e-b)v^2}I_0(2bv w) e^{\frac{b^2}{a+\e-b}w^2} e^{\frac{b^2}{b-a-\e}w^2} dv e^{-bw^2}\\
   &=2(b-a-\e)\int_{\mathbb{R}^+}v e^{(a+\e-b)v^2}I_0(2bv w)e^{\frac{(bw)^2}{a+\e-b}}dv \frac{b}{b-a-\e}e^{\frac{(a+\e)b}{b-a-\e}w^2}\\
   &=\frac{b}{b-a-\e}e^{\frac{(a+\e)b}{b-a-\e}w^2},
\end{align*}
where we use~\eqref{eqn: I0} in Lemma~\ref{Lemma: Prob measure} in the last line, then we obtain~\eqref{eqn: coe abc perp}.

Following the same derivation we have
\begin{align*}
   & 2b\int_{0< v< \delta}v e^{\e v^2}e^{av^2} e^{-bv^2}e^{-bw^2}I_0(2bv w)dv \\
  & =2(b-a-\e)\int_{0<v<\delta} v e^{(a+\e-b)v^2}I_0(2bv w)e^{\frac{(bw)^2}{a+\e-b}}dv \frac{b}{b-a-\e}e^{\frac{(a+\e)b}{b-a-\e}w^2}.
\end{align*}
Using the definition of $I_0$ we have
\[I_0(y)=\frac{1}{\pi}\int_{0}^{\pi} e^{y\cos\phi}d\phi\leq e^{y}.\]
Thus when $a-b+\e<0$,
\begin{align*}
 & 2(b-a-\e)\int_{0<v<\delta} v e^{(a+\e-b)v^2}I_0(2bv w)e^{\frac{(bw)^2}{a+\e-b}}dv \\
   & \leq 2(b-a-\e)\int_{0<v<\delta} v e^{(a-b+\e)v^2}e^{2v b w}e^{\frac{(bw)^2}{a-b+\e}}=2(b-a-\e)\int_{0<v<\delta} v e^{(a-b+\e)(v+\frac{bw}{a-b+\e})^2}dv\\
   &\leq 2(b-a-\e)\int_{0<v<\delta}vdv<\delta,
\end{align*}
where we use $\delta\ll 1$ in the last step, then we obtain~\eqref{eqn: coe abc perp small}. Then we derive~\eqref{eqn: coe perp small 2}.

\end{proof}

\begin{lemma}\label{Lemma: integrate normal small}
For any $m,n>0$, when $\delta\ll 1$, we have
\begin{equation}\label{eqn: smallness for i0}
2m^2\int_{\frac{n}{m}u_\perp+\delta^{-1}}^\infty     v_\perp e^{-m^2v_\perp^2}I_0(2mnv_\perp u_\perp)e^{-n^2u_\perp^2}dv_\perp \lesssim e^{-\frac{m^2}{4\delta^{2}}}.
\end{equation}
In consequence, for any $a>0,b>0,\e>0$ with $a+\e<b$,
\begin{eqnarray}
 2b\int_{\frac{b}{b-a-\e}w+\delta^{-1}}^\infty v e^{\e v^2}e^{av^2} e^{-bv^2}e^{-bw^2}I_0(2bv w)dv  &\leq  & e^{\frac{-(b-a-\e)}{4\delta^2}}\frac{b}{b-a-\e}e^{\frac{(a+\e)b}{b-a-\e}w^2} \label{eqn: coe perp smaller 2}\\
   &\leq  & \delta\frac{b}{b-a-\e}e^{\frac{(a+\e)b}{b-a-\e}w^2}. \label{eqn: coe perp small 2}
\end{eqnarray}

\end{lemma}

\begin{proof}
We discuss two cases. The first case is $v_\perp>2\frac{n}{m}u_\perp$. We bound $I_0$ as
\[I_0(2mnv_\perp u_\perp)\leq \frac{1}{\pi}\int_0^\pi \exp \Big( 2mnv_\perp u_\perp\Big) d\theta=\exp \Big(2mnv_\perp u_\perp\Big).\]
The LHS of~\eqref{eqn: smallness for i0} is bounded by
\[2m^2\int_{\max\{2\frac{n}{m}u_\perp,\frac{n}{m}u_\perp+\delta^{-1}\}}^\infty    ve^{-m^2(v_\perp-\frac{n}{m}u_\perp)^2}dv.\]
Using $v_\perp>2\frac{n}{m}u_\perp$ we have
\begin{equation*}
(v_\perp-\frac{n}{m}u_\perp)^2\geq (\frac{v_\perp}{2}+\frac{v_\perp}{2}-\frac{n}{m}u_\perp)^2 \geq \frac{v_\perp^2}{4}.
\end{equation*}
Thus we can further bound LHS of~\eqref{eqn: smallness for i0} by
\begin{equation*}
2m^2\int_{\max\{2\frac{n}{m}u_\perp,\frac{n}{m}u_\perp+\delta^{-1}\}}^\infty      v_\perp e^{-\frac{m^2 v_\perp^2}{4}} dv_\perp \lesssim e^{-\frac{m^2}{4\delta^2}}.
\end{equation*}

The second case is $0\leq v_\perp\leq 2\frac{n}{m}u_\perp$. Since $\frac{n}{m}u_\perp+\delta^{-1}<v_\perp$, without loss of generality, we can assume $u_\perp>\delta^{-1}$. We compare the Taylor series of $v_\perp I_0(2mnv_\perp u_\perp)$ and $\exp \Big(2mnv_\perp u_\perp \Big)$. We have
\begin{equation}\label{eqn: vI_0 taylor}
v_\perp I_0(2mnv_\perp u_\perp )=\sum_{k=0}^\infty \frac{m^{2k}n^{2k}v_\perp^{2k+1}u_\perp^{2k}}{(k!)^2},
\end{equation}
and
\begin{equation}\label{eqn: exp taylor}
\exp \Big(2mnv_\perp u_\perp \Big)=\sum_{k=0}^\infty \frac{2^k m^k n^k v_\perp^k u_\perp^k}{k!}.
\end{equation}
We choose $k_1$ such that when $k>k_1$, we can apply the Sterling formula such that
\[\frac{1}{2}\leq |\frac{k!}{k^ke^{-k}\sqrt{2\pi k}}|\leq 2.\]
Then we observe the quotient of the $k$-th term of~\eqref{eqn: vI_0 taylor} and the $2k+1$-th term of~\eqref{eqn: exp taylor},
\begin{align*}
   & \frac{m^{2k}n^{2k}v_\perp^{2k+1}u_\perp^{2k}}{(k!)^2}/\Big(\frac{2^{2k+1} m^{2k+1}n^{2k+1}v_\perp^{2k+1} u_\perp^{2k+1}}{(2k+1)!} \Big) \\
   & \leq \frac{4}{k^{2k}e^{-2k}2\pi k}/\Big(\frac{2^{2k+1} mn u_\perp}{(2k+1)^{2k+1}e^{-(2k+1)}\sqrt{2\pi (2k+1)}} \Big)\\
   &= \frac{4e}{2\pi mn}\Big(\frac{k+1/2}{k} \Big)^{2k+1}  \frac{\sqrt{2\pi (2k+1)}}{u_\perp}\\
   &= \frac{4e}{2\pi mn}\Big(\frac{2k+1}{2k} \Big)^{2k+1} \frac{\sqrt{2\pi (2k+1)}}{u_\perp}\leq \frac{4e^2}{\sqrt{\pi} mn} \frac{\sqrt{k}}{u_\perp}.
\end{align*}

Thus we can take $k_u=u_\perp^2$ such that when $k\leq k_u$,
\begin{equation}\label{k<ku}
\sum_{k=k_1}^{k_u} \frac{m^{2k}n^{2k}v_\perp^{2k+1}u_\perp^{2k}}{(k!)^2}\leq \frac{4e^2}{\sqrt{\pi}mn}\sum_{k=k_1}^{k_u}    \frac{2^{2k+1} m^{2k+1}n^{2k+1}v_\perp^{2k+1} u_\perp^{2k+1}}{(2k+1)!}.
\end{equation}
Similarly we observe the quotient of the $k$-th term of~\eqref{eqn: vI_0 taylor} and the $2k$-th term of~\eqref{eqn: exp taylor},
\[\frac{m^{2k}n^{2k}v_\perp^{2k+1}u_\perp^{2k}}{(k!)^2}/\Big(\frac{2^{2k} m^{2k}n^{2k}v_\perp^{2k} u_\perp^{2k}}{(2k)!} \Big)
\]
\[\leq \frac{4v_\perp}{k^{2k}e^{-2k}2\pi k}/\Big(\frac{2^{2k}}{(2k)^{2k}e^{-2k}\sqrt{4\pi k}} \Big)=\frac{4v_\perp}{\sqrt{\pi} \sqrt{k}}.\]
When $k>k_u=u_\perp^2$, by $u_\perp>\delta^{-1}$ and $v_\perp<2\frac{n}{m}u_\perp$ we have
\[\frac{4v_\perp}{\sqrt{\pi} \sqrt{k}}\leq \frac{4v_\perp}{\sqrt{\pi}u_\perp}\leq \frac{8n}{m\sqrt{\pi}}.\]
Thus we have
\begin{equation}\label{eqn: k>ku}
\sum_{k=k_u}^\infty \frac{m^{2k}n^{2k}v_\perp^{2k+1}u_\perp^{2k}}{(k!)^2}\leq \frac{8n}{m\sqrt{\pi}}\sum_{k=k_u}^\infty \frac{2^{2k} m^{2k}n^{2k}v_\perp^{2k} u_\perp^{2k}}{(2k)!}.
\end{equation}

Collecting~\eqref{eqn: k>ku}~\eqref{k<ku}, when $v_\perp<2\frac{n}{m}u_\perp$, we obtain
\begin{equation}\label{I0 <exp}
v_\perp I_0(2mnv_\perp u_\perp )\lesssim \exp \Big(\frac{2(1-r_\perp)^{1/2} v_\perp u_\perp}{r_\perp} \Big).
\end{equation}
By~\eqref{I0 <exp}, we have
\[\int_{\frac{n}{m}u_\perp+\delta^{-1}}^{2\frac{n}{m}u_\perp}v_\perp I_0(2mnv_\perp u_\perp)) e^{-m^2v_\perp^2}e^{n^2 v_\perp^2}dv\]
\begin{equation}\label{eqn: middle}
\lesssim \int_{\frac{n}{m}u_\perp+\delta^{-1}}^{2\frac{n}{m}u_\perp}    e^{-m^2(v_\perp-\frac{n}{m}u_\perp)^2} dv\leq e^{-m^2\delta^{-2}}.
\end{equation}
Collecting~\eqref{eqn: exp taylor} and~\eqref{eqn: middle} we prove~\eqref{eqn: smallness for i0}.

Then following the same derivation as~\eqref{eqn: coe abc perp},
\begin{align*}
   &  2b\int_{\frac{b}{b-a-\e}w+\delta^{-1}}^\infty v e^{\e v^2}e^{av^2} e^{-bv^2}e^{-bw^2}I_0(2bv w)dv\\
   & =2(b-a-\e)\int_{\frac{b}{b-a-\e}w+\delta^{-1}}^\infty v e^{(a+\e-b)v^2}I_0(2bv w)e^{\frac{(bw)^2}{a+\e-b}}dv \frac{b}{b-a-\e}e^{\frac{(a+\e)b}{b-a-\e}w^2}\\
   &\leq e^{\frac{-(b-a-\e)}{4\delta^2}}\frac{b}{b-a-\e}e^{\frac{(a+\e)b}{b-a-\e}w^2} \leq \delta \frac{b}{b-a-\e}e^{\frac{(a+\e)b}{b-a-\e}w^2},
\end{align*}
where we apply~\eqref{eqn: smallness for i0} in the first step in the third line and take $\delta\ll 1$ in the last step of the third line.

\end{proof}

\begin{lemma}
If $0<\frac{\theta}{4}<\rho$, if $0<\tilde{\rho}<  \rho- \frac{\theta}{4}$, $0\leq \lambda t< \theta$,
		\begin{equation}\label{k_theta}
		\mathbf{k}_{  \varrho}(v,u) \frac{e^{{\theta} |v|^2}}{e^{\mathcal{\theta} |u|^2}}\frac{e^{\lambda t\langle u\rangle}}{e^{\lambda t\langle v\rangle}} \lesssim  \mathbf{k}_{\tilde{\varrho}}(v,u) .
		\end{equation}
\end{lemma}

\begin{proof}
When $\langle u \rangle -\langle v\rangle\leq 1$,
\[\frac{e^{\lambda s\langle u\rangle }}{e^{\lambda s\langle v\rangle}}\leq e^{\lambda s}.\]
When $\langle u\rangle -\langle v\rangle \geq 1$,
\[\langle u\rangle ^2-\langle v\rangle^2=(\langle u\rangle -\langle v\rangle )(\langle u\rangle +\langle v\rangle)\geq \langle u\rangle-\langle v\rangle.\]
Thus by $\langle u\rangle^2=|u|^2+1$,
\[\frac{e^{\lambda s\langle u\rangle }}{e^{\lambda s\langle v\rangle}}\lesssim 1+\frac{e^{\lambda s | u|^2 }}{e^{\lambda s | v|^2}}.\]
Note
		\Be\notag
		\begin{split}
			\mathbf{k}_{  \varrho}(v,u) \frac{e^{\vartheta |v|^2}}{e^{\vartheta |u|^2}}
			=  \frac{1}{|v-u| } \exp\left\{- {\varrho} |v-u|^{2}
			-  {\varrho} \frac{ ||v|^2-|u|^2 |^2}{|v-u|^2} + \vartheta |v|^2 - \vartheta |u|^2
	\right\}.
		\end{split}\Ee

		Let $v-u=\eta $ and $u=v-\eta $. Then the exponent equals
		\begin{eqnarray*}
			&&- \varrho|\eta |^{2}-\varrho\frac{||\eta |^{2}-2v\cdot \eta |^{2}}{%
				|\eta |^{2}}-\vartheta \{|v-\eta |^{2}-|v|^{2}\}-\lambda t \{|v|-|v-\eta|\} \\
			&=&-2 \varrho |\eta |^{2}+ 4 \varrho v\cdot \eta - 4 \varrho\frac{|v\cdot
				\eta |^{2}}{|\eta |^{2}}-\vartheta \{|\eta |^{2}-2v\cdot \eta \} \\
			&=&(-2 \varrho-\vartheta  )|\eta |^{2}+(4 \varrho+2\vartheta )v\cdot \eta -%
			4 \varrho\frac{\{v\cdot \eta \}^{2}}{|\eta |^{2}}.
		\end{eqnarray*}%
		If $0<\vartheta <4 \varrho$ then the discriminant of the above quadratic form of
		$|\eta |$ and $\frac{v\cdot \eta }{|\eta |}$ is
		\begin{equation*}
		(4 \varrho+2\vartheta )^{2}-4
		(-2 \varrho-\vartheta  )(-%
		4 \varrho)
		=4\vartheta ^{2}- 16 \varrho \vartheta<0.
		\end{equation*}%
		Hence, the quadratic form is negative definite. We thus have, for $%
		0<\tilde{\varrho}< \varrho - \frac{\vartheta}{4}  $, the following perturbed quadratic form is still negative definite
		\[
		-(\varrho - \tilde{\varrho})|\eta |^{2}-(\varrho - \tilde{\varrho})\frac{||\eta
			|^{2}-2v\cdot \eta |^{2}}{|\eta |^{2}}-\vartheta \{|\eta |^{2}-2v\cdot \eta \}  \leq 0.\]

For
\[			\mathbf{k}_{  \varrho}(v,u) \frac{e^{\vartheta |v|^2}}{e^{\vartheta |u|^2}}\frac{e^{\lambda t\langle u\rangle^2}}{e^{\lambda t |v|^2}}
			=  \frac{1}{|v-u| } \exp\left\{- {\varrho} |v-u|^{2}
			-  {\varrho} \frac{ ||v|^2-|u|^2 |^2}{|v-u|^2} + (\theta-\lambda t) |v|^2 - (\theta-\lambda t) |u|^2
	\right\}.\]
We just need to replace $\theta$ by $\theta-\lambda t$ in the previous computation. By $\lambda t\ll \theta$,
		\[
		-(\varrho - \tilde{\varrho})|\eta |^{2}-(\varrho - \tilde{\varrho})\frac{||\eta
			|^{2}-2v\cdot \eta |^{2}}{|\eta |^{2}}-(\theta-\lambda t) \{|\eta |^{2}-2v\cdot \eta \}  \leq 0.\]
Therefore, we conclude the lemma.

	\end{proof}

\textbf{Acknowledgements.} Q.L. is support in part by National Science Foundation under award 1619778, 1750488. H.C. is support in part by Wisconsin Data Science Initiative. C.K is partly support in part by National Science Foundation under award NSF DMS-1501031, DMS-1900923.

\bibliographystyle{amsplain}
\bibliography{VPB}

\providecommand{\bysame}{\leavevmode\hbox to3em{\hrulefill}\thinspace}
\providecommand{\MR}{\relax\ifhmode\unskip\space\fi MR }
\providecommand{\MRhref}[2]{%
  \href{http://www.ams.org/mathscinet-getitem?mr=#1}{#2}
}
\providecommand{\href}[2]{#2}
\begin{thebibliography}{10}

\bibitem{ABDG}
Kazuo Aoki, Claude Bardos, Christian Dogbe, and Francois Golse, \emph{A note on
  the propagation of boundary induced discontinuities in kinetic theory},
  Mathematical Models and Methods in Applied Sciences \textbf{11} (2001),
  no.~09, 1581--1595.

\bibitem{CKL}
Yunbai Cao, Chanwoo Kim, and Donghyun Lee, \emph{Global strong solutions of the
  vlasov--poisson--boltzmann system in bounded domains}, Archive for Rational
  Mechanics and Analysis (2019), 1--104.

\bibitem{CC}
Carlo Cercignani, \emph{The boltzmann equation}, The Boltzmann equation and its
  applications, Springer, 1988, pp.~40--103.

\bibitem{CIP}
Carlo Cercignani, Reinhard Illner, and Mario Pulvirenti, \emph{The mathematical
  theory of dilute gases}, vol. 106, Springer Science \& Business Media, 2013.

\bibitem{CL}
Carlo Cercignani and Maria Lampis, \emph{{Kinetic models for gas-surface
  interactions}}, transport theory and statistical physics \textbf{1} (1971),
  no.~2, 101--114.

\bibitem{HC}
Hongxu Chen, \emph{{Cercignani-Lampis boundary in Boltzmann theory}}, arXiv
  preprint arXiv:1906.01808 (2019).

\bibitem{C}
TG~Cowling, \emph{{On the Cercignani-Lampis formula for gas-surface
  interactions}}, Journal of Physics D: Applied Physics \textbf{7} (1974),
  no.~6, 781.

\bibitem{EGKM}
R~Esposito, Y~Guo, C~Kim, and R~Marra, \emph{{Non-isothermal boundary in the
  Boltzmann theory and Fourier law}}, Communications in Mathematical Physics
  \textbf{323} (2013), no.~1, 177--239.

\bibitem{EGKMA}
Raffaele Esposito, Yan Guo, Chanwoo Kim, and Rossana Marra, \emph{{Stationary
  solutions to the Boltzmann equation in the hydrodynamic limit}}, Annals of
  PDE \textbf{4} (2018), no.~1, 1.

\bibitem{CS}
RDM Garcia and CE~Siewert, \emph{{The linearized Boltzmann equation with
  Cercignani--Lampis boundary conditions: Basic flow problems in a plane
  channel}}, European Journal of Mechanics-B/Fluids \textbf{28} (2009), no.~3,
  387--396.

\bibitem{Gar}
\bysame, \emph{{Viscous-slip, thermal-slip, and temperature-jump coefficients
  based on the linearized Boltzmann equation (and five kinetic models) with the
  Cercignani--Lampis boundary condition}}, European Journal of
  Mechanics-B/Fluids \textbf{29} (2010), no.~3, 181--191.

\bibitem{Gl}
Robert~T Glassey, \emph{{The Cauchy problem in kinetic theory}}, vol.~52, Siam,
  1996.

\bibitem{GVPB}
Yan Guo, \emph{{The Vlasov-Poisson-Boltzmann system near Maxwellians}},
  Communications on Pure and Applied Mathematics: A Journal Issued by the
  Courant Institute of Mathematical Sciences \textbf{55} (2002), no.~9,
  1104--1135.

\bibitem{G}
\bysame, \emph{{Decay and continuity of the Boltzmann equation in bounded
  domains}}, Archive for rational mechanics and analysis \textbf{197} (2010),
  no.~3, 713--809.

\bibitem{GKTT}
Yan Guo, Chanwoo Kim, Daniela Tonon, and Ariane Trescases, \emph{{Regularity of
  the Boltzmann equation in convex domains}}, Inventiones mathematicae
  \textbf{207} (2017), no.~1, 115--290.

\bibitem{K}
Chanwoo Kim, \emph{Formation and propagation of discontinuity for boltzmann
  equation in non-convex domains}, Communications in mathematical physics
  \textbf{308} (2011), no.~3, 641--701.

\bibitem{KL}
Chanwoo Kim and Donghyun Lee, \emph{The boltzmann equation with specular
  boundary condition in convex domains}, Communications on Pure and Applied
  Mathematics \textbf{71} (2018), no.~3, 411--504.

\bibitem{KB}
RF~Knackfuss and LB~Barichello, \emph{{Surface effects in rarefied gas
  dynamics: an analysis based on the Cercignani--Lampis boundary condition}},
  European Journal of Mechanics-B/Fluids \textbf{25} (2006), no.~1, 113--129.

\bibitem{L}
RG~Lord, \emph{{Some further extensions of the Cercignani--Lampis gas--surface
  interaction model}}, Physics of Fluids \textbf{7} (1995), no.~5, 1159--1161.

\bibitem{MA}
Andrew~J Majda, Andrew~J Majda, and Andrea~L Bertozzi, \emph{Vorticity and
  incompressible flow}, vol.~27, Cambridge university press, 2002.

\bibitem{SF}
Felix Sharipov, \emph{Application of the cercignani--lampis scattering kernel
  to calculations of rarefied gas flows. ii. slip and jump coefficients},
  European Journal of Mechanics-B/Fluids \textbf{22} (2003), no.~2, 133--143.

\bibitem{SF1}
\bysame, \emph{Application of the cercignani--lampis scattering kernel to
  calculations of rarefied gas flows. iii. poiseuille flow and thermal creep
  through a long tube}, European Journal of Mechanics-B/Fluids \textbf{22}
  (2003), no.~2, 145--154.

\bibitem{WR}
MS~Woronowicz and DFG Rault, \emph{{Cercignani-lampis-lord gas surface
  interaction model-comparisons between theory and simulation}}, Journal of
  Spacecraft and Rockets \textbf{31} (1994), no.~3, 532--534.

\end{thebibliography}

\end{document}